\documentclass[a4paper,12pt,twoside,openright]{report}

\usepackage[utf8]{inputenc}                  %
\usepackage[T1]{fontenc}                     %
\usepackage{lmodern}                         %
\usepackage{geometry}                        
\usepackage[square,numbers]{natbib}          
\usepackage[nottoc,notlot,notlof]{tocbibind} 
\usepackage{enumitem}                        
\usepackage{xparse}                          
\usepackage{xstring}                         
\usepackage{etoolbox}                        
\usepackage{parskip}                         
\usepackage{titling}

\usepackage{mathtools}              
\usepackage{amssymb}                
\usepackage{amsthm}                 
\usepackage{IEEEtrantools}          
\usepackage{tensor}                 
\usepackage{tikz}                   

\usepackage{hyperref}               
\usepackage{bookmark}               
\usepackage[capitalise]{cleveref}   
\usepackage[all]{hypcap}            

\apptocmd{\sloppy}{\hbadness 10000\relax}{}{}                       
\allowdisplaybreaks                                                 
\graphicspath{{./figures/}}                                         

\newlength{\alphabet}
\setlist[description]{font=\normalfont}
\setlist[enumerate]{font=\normalfont}
\setlist[enumerate,1]{label = {(\arabic*)}}
\setlist[enumerate,2]{label = {(\arabic{enumi}.\arabic*)}}

\newcounter{dummy}
\makeatletter
\newcommand\myitem[1][]{\item[#1]\refstepcounter{dummy}\def\@currentlabel{#1}}
\makeatother

\usetikzlibrary{decorations.pathreplacing} 
\usetikzlibrary{math}                      
\usetikzlibrary{calc}                      
\usetikzlibrary{cd}                        
\tikzset{
    symbol/.style={
        draw=none,
        every to/.append style={
            edge node={node [sloped, allow upside down, auto=false]{$#1$}}
        },
    },
}

\hypersetup{
    bookmarksnumbered=true,
    colorlinks=true,
    linkcolor=blue,
    citecolor=blue,
    urlcolor=blue
}

\theoremstyle{plain}      
\newtheorem{theorem}     {Theorem} [chapter]
\newtheorem{proposition} [theorem] {Proposition}
\newtheorem{lemma}       [theorem] {Lemma}
\newtheorem{corollary}   [theorem] {Corollary}
\newtheorem{conjecture}  [theorem] {Conjecture}

\theoremstyle{definition} 

\newtheorem{definition}  [theorem] {Definition}
\newtheorem{example}     [theorem] {Example}
\newtheorem{remark}      [theorem] {Remark}
\newtheorem{assumption}  [theorem] {Assumption}

\NewDocumentCommand{\plabel}{m}{\phantomsection\label{#1}}

\newcommand{\rmn}[1]{\mathrm{\MakeUppercase{\romannumeral #1}}}
\NewDocumentCommand{\signature}{}{\operatorname{sign}}

\NewDocumentCommand{\symp}      { }{\mathbf{Symp}}
\NewDocumentCommand{\liouvndg}  { }{\mathbf{Liouv}_{\mathrm{ndg}}^{\mathrm{gle}}}
\NewDocumentCommand{\liouvle}   { }{\mathbf{Liouv}_{\mathrm{ndg}}}
\NewDocumentCommand{\liouvgle}  { }{\mathbf{Liouv}^{\mathrm{gle}}}
\NewDocumentCommand{\modl}      { }{\mathbf{Mod}}
\NewDocumentCommand{\komp}      { }{\mathbf{Comp}}
\NewDocumentCommand{\comp}      { }{\mathbf{hComp}}
\NewDocumentCommand{\admissible}{m}{\mathbf{I}_{#1}}
\NewDocumentCommand{\stair}     {m}{\mathbf{I}_{#1}}
\NewDocumentCommand{\admstair}  {m}{\mathbf{K}_{#1}}

\NewDocumentCommand {\cgh} {m} {c^{\mathrm{GH}}_{#1}}
\NewDocumentCommand {\csh} {m} {c^{S^1}_{#1}}

\NewDocumentCommand{\shf}{}{S}
\NewDocumentCommand{\inc}{}{\tilde{i}}

\NewDocumentCommand{\union}           { }{\cup}                                          
\NewDocumentCommand{\bigunion}        { }{\bigcup}                                       
\NewDocumentCommand{\bigproduct}      { }{\prod}                                         
\NewDocumentCommand{\bigcoproduct}    { }{\coprod}                                       
\NewDocumentCommand{\tensorpr}        { }{\otimes}                                       
\NewDocumentCommand{\directsum}       { }{\oplus}                                        
\NewDocumentCommand{\bigdirectsum}    { }{\bigoplus}                                     

\NewDocumentCommand{\Z}{}{\mathbb{Z}} 
\NewDocumentCommand{\Q}{}{\mathbb{Q}} 
\NewDocumentCommand{\R}{}{\mathbb{R}} 
\NewDocumentCommand{\C}{}{\mathbb{C}} 

\NewDocumentCommand{\id} {}{\operatorname{id}} 
\NewDocumentCommand{\img}{}{\operatorname{im}} 

\NewDocumentCommand{\idm}{}{I}

\NewDocumentCommand{\Hom}   { }{\operatorname{Hom}}   
\NewDocumentCommand{\End}   { }{\operatorname{End}}   
\NewDocumentCommand{\coker} { }{\operatorname{coker}} 
\NewDocumentCommand{\colim} { }{\operatorname{colim}}
\NewDocumentCommand{\spn}   { }{\operatorname{span}}  
\NewDocumentCommand{\Ann}   { }{\operatorname{Ann}}   

\NewDocumentCommand{\itr} {}{\operatorname{int}}  
\NewDocumentCommand{\supp}{}{\operatorname{supp}} 

\NewDocumentCommand {\critpt}  { } {\operatorname{CritPt}}  

\NewDocumentCommand  {\dv}     {}    {\mathrm{D}}                          
\NewDocumentCommand  {\odv}    {m m} {\frac{\mathrm{d} #1}{\mathrm{d} #2}} 
\NewDocumentCommand  {\pdv}    {m m} {\frac{\partial #1}{\partial #2}}     
\NewDocumentCommand  {\edv}    {}    {\mathrm{d}}                          
\NewDocumentCommand  {\ldv}    {m}   {{L}_{#1}}                            
\NewDocumentCommand  {\del}    {}    {\partial}                            
\NewDocumentCommand  {\delbar} {}    {\overline{\partial}}                 

\NewDocumentCommand{\ind}          {}{\mu}
\NewDocumentCommand{\morse}        {}{\mu_{\operatorname{M}}}  
\NewDocumentCommand{\maslov}       {}{\mu}                     
\NewDocumentCommand{\conleyzehnder}{}{\mu_{\operatorname{CZ}}} 

\newcommand{\lpar}{(}
\newcommand{\rpar}{)}
\newcommand{\lsize}{}
\newcommand{\rsize}{}

\NewDocumentCommand{\SetParenthesisTypeSize}{m m}{
	\renewcommand{\lpar}{(}
	\renewcommand{\rpar}{)}
	\renewcommand{\lsize}{}
	\renewcommand{\rsize}{}
    \IfEq{#1}{(} { \renewcommand{\lpar}{(}       \renewcommand{\rpar}{)}       }{}
    \IfEq{#1}{()}{ \renewcommand{\lpar}{(}       \renewcommand{\rpar}{)}       }{}
	\IfEq{#1}{c} { \renewcommand{\lpar}{\{}      \renewcommand{\rpar}{\}}      }{}
	\IfEq{#1}{<} { \renewcommand{\lpar}{\langle} \renewcommand{\rpar}{\rangle} }{}
	\IfEq{#1}{[} { \renewcommand{\lpar}{[}       \renewcommand{\rpar}{]}       }{}
    \IfEq{#1}{[]}{ \renewcommand{\lpar}{[}       \renewcommand{\rpar}{]}       }{}
	\IfEq{#1}{|} { \renewcommand{\lpar}{\lvert}  \renewcommand{\rpar}{\rvert}  }{}
	\IfEq{#1}{||}{ \renewcommand{\lpar}{\lVert}  \renewcommand{\rpar}{\rVert}  }{}
    \IfEq{#1}{L} { \renewcommand{\lpar}{\lfloor} \renewcommand{\rpar}{\rfloor} }{}
    \IfEq{#1}{T} { \renewcommand{\lpar}{\lceil}  \renewcommand{\rpar}{\rceil}  }{}
    \IfEq{#2}{0}{ \renewcommand{\lsize}{}       \renewcommand{\rsize}{}       }{}
    \IfEq{#2}{1}{ \renewcommand{\lsize}{\bigl}  \renewcommand{\rsize}{\bigr}  }{}
    \IfEq{#2}{2}{ \renewcommand{\lsize}{\Bigl}  \renewcommand{\rsize}{\Bigr}  }{}
    \IfEq{#2}{3}{ \renewcommand{\lsize}{\biggl} \renewcommand{\rsize}{\biggr} }{}
    \IfEq{#2}{4}{ \renewcommand{\lsize}{\Biggl} \renewcommand{\rsize}{\Biggr} }{}
    \IfEq{#2}{a}{ \renewcommand{\lsize}{\left}  \renewcommand{\rsize}{\right} }{}
}

\NewDocumentCommand{\p}{m m m}{
    \IfEq{#1}{n}{}{\SetParenthesisTypeSize{#1}{#2} \lsize \lpar}
    #3
    \IfEq{#1}{n}{}{\SetParenthesisTypeSize{#1}{#2} \rsize \rpar}
}

\NewDocumentCommand{\sbn}{o m m}{
    \IfValueF{#1}{
        \{ #2 \ | \ #3 \}
    }{}
    \IfValueT{#1}{
        \IfEq{#1}{0}{        \{ #2 \       | \ #3        \} }{}
        \IfEq{#1}{1}{ \bigl  \{ #2 \ \big  | \ #3 \bigr  \} }{}
        \IfEq{#1}{2}{ \Bigl  \{ #2 \ \Big  | \ #3 \Bigr  \} }{}
        \IfEq{#1}{3}{ \biggl \{ #2 \ \bigg | \ #3 \biggr \} }{}
        \IfEq{#1}{4}{ \Biggl \{ #2 \ \Bigg | \ #3 \Biggr \} }{}
    }{}
}

\newcommand {\modifier}    {} 
\newcommand {\equivariant} {} 
\newcommand {\manifold}    {} 
\newcommand {\theory}      {} 
\newcommand {\complex}     {}
\newcommand {\filtration}  {}
\newcommand {\grading}     {}

\NewDocumentCommand{\homology}{m m m m m m m}{
    \renewcommand {\modifier}    {}
    \renewcommand {\equivariant} {}
    \renewcommand {\manifold}    {}
    \renewcommand {\theory}      {}
    \renewcommand {\complex}     {}
    \renewcommand {\filtration}  {}
    \renewcommand {\grading}     {}
    \renewcommand {\modifier}    {#1}
    \renewcommand {\equivariant} {#2} 
    \renewcommand {\manifold}    {#3}
    \renewcommand {\theory}      {#4}
    \renewcommand {\complex}     {#5}
    \renewcommand {\filtration}  {#6}
    \renewcommand {\grading}     {#7}
    \IfEq {#1} {}    {} {\renewcommand {\equivariant} {#1}} 
    \IfEq {#1} {L}   {\renewcommand {\equivariant} {}} {}
    \IfEq {#4} {sch} {\renewcommand {\theory} {} \renewcommand {\equivariant} {} \renewcommand {\filtration}  {\star}}   {}
    \IfEq {#4} {rsh} {\renewcommand {\theory} {} \renewcommand {\equivariant} {} \renewcommand {\filtration}  {\dagger}} {}
    \tensor*[]{\theory\complex}{^{\equivariant{\IfEq{#6}{}{}{,}}\filtration}_{\manifold\grading}}%
}

\NewDocumentEnvironment{copiedtheorem}
    {o m}
    {
        \theoremstyle{plain}
        \newtheorem*{copytheorem:#2}{\cref{#2}}
        \IfNoValueTF{#1}
        {
            \begin{copytheorem:#2}
        }
        {
            \begin{copytheorem:#2}[{#1}]
        }
    }
    {
        \end{copytheorem:#2}
    }

\NewDocumentEnvironment{secondcopy}
    {o m}
    {
        \IfNoValueTF{#1}
        {
            \begin{copytheorem:#2}
        }
        {
            \begin{copytheorem:#2}[{#1}]
        }
    }
    {
        \end{copytheorem:#2}
    }

\title{Equivariant symplectic homology, linearized contact homology and the Lagrangian capacity}
\author{Miguel Barbosa Pereira}
\date{\today}

\hypersetup{
    pdftitle={\thetitle},
    pdfauthor={\theauthor},
    pdflang={en-GB}
}

\begin{document}

\pagenumbering{roman}

\begin{titlepage}
	\centering
	\hspace{0pt}
	\vfill
    {\LARGE\bfseries \thetitle\par}
	\vspace{1.5cm}
	{\Large\bfseries Dissertation\par}
	\vspace{1.5cm}
	{\large zur Erlangung des akademischen Grades\par Dr. rer. nat.\par}
	\vspace{1.5cm}
	{\large eingereicht an der\par Mathematisch-Naturwissenschaftlich-Technischen Fakultät\par der Universität Augsburg\par}
	\vspace{1.5cm}
	{\large von\par}
	{\large\bfseries \theauthor\par}
	\vspace{2cm}
	{\large Augsburg, März 2022\par}
	\vspace{1cm}
	\includegraphics{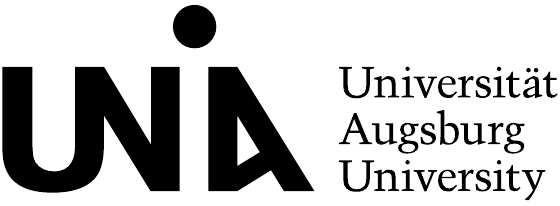}
\end{titlepage}

{
    \centering
    \hspace{0pt}
    \vfill
    \begin{tabular}{ r l }
        Betreuer:                   & Prof. Dr. Kai Cieliebak, Universität Augsburg \\
        Gutachter:                  & Prof. Dr. Urs Frauenfelder, Universität Augsburg \\
                                    & Prof. Dr. Klaus Mohnke, Humboldt-Universität zu Berlin \\ \\
    \end{tabular}
    \newline
    \begin{tabular}{ r l }
        Tag der mündlichen Prüfung: & 20.05.2022
    \end{tabular}
}

\cleardoublepage\pdfbookmark{Abstract}{abstract}
\chapter*{Abstract}

We establish computational results concerning the Lagrangian capacity from \cite{cieliebakPuncturedHolomorphicCurves2018}. More precisely, we show that the Lagrangian capacity of a 4-dimensional convex toric domain is equal to its diagonal. The proof involves comparisons between the Lagrangian capacity, the McDuff--Siegel capacities from \cite{mcduffSymplecticCapacitiesUnperturbed2022}, and the Gutt--Hutchings capacities from \cite{guttSymplecticCapacitiesPositive2018}. Working under the assumption that there is a suitable virtual perturbation scheme which defines the curve counts of linearized contact homology, we extend the previous result to toric domains which are convex or concave and of any dimension. For this, we use the higher symplectic capacities from \cite{siegelHigherSymplecticCapacities2020}. The key step is showing that moduli spaces of asymptotically cylindrical holomorphic curves in ellipsoids are transversely cut out.

\cleardoublepage\pdfbookmark{Acknowledgements}{acknowledgements}
\chapter*{Acknowledgements}

First and foremost, I would like to thank my advisor, Kai Cieliebak, for introducing me to this topic and for the guidance he gave me during this project. We had many fruitful discussions about the various details of this problem and I am very grateful for that.

Next, I want to thank my Mom Manuela, my Dad Manuel, and my Brother Pedro for their continued support during my PhD and their belief in me.

Finally, I want to thank the mathematical community at the University of Augsburg, for making it a pleasant place to work at. Special thanks go to Kathrin Helmsauer and Robert Nicholls for their help on several occasions, and to Yannis Bähni, Marián Poppr, Frederic Wagner, Thorsten Hertl, and Artem Nepechiy for listening to my talks about this subject and giving valuable feedback. I am also grateful to Kyler Siegel for productive discussions, and to Urs Frauenfelder and Klaus Mohnke for reading and refereeing my thesis.

\cleardoublepage\pdfbookmark{\contentsname}{contents}
\tableofcontents

\newpage

\pagenumbering{arabic}

\chapter{Introduction}

\section{Symplectic capacities and their uses}

A \textbf{symplectic manifold} is a pair $(X, \omega)$, where $X$ is a manifold and $\omega \in \Omega^2(X)$ is a closed and nondegenerate $2$-form on $X$. An example of a symplectic manifold is $\C^{n}$ with the canonical symplectic structure
\begin{IEEEeqnarray*}{c}
    \omega_0 \coloneqq \sum_{j=1}^{n} \edv x^j \wedge \edv y^j.
\end{IEEEeqnarray*}
An embedding $\phi \colon (X, \omega_X) \longrightarrow (Y, \omega_Y)$ between symplectic manifolds is \textbf{symplectic} if $\phi^* \omega_Y = \omega_X$. A \textbf{symplectomorphism} is a symplectic diffeomorphism. \textbf{Darboux' theorem} implies that any symplectic manifold $(X, \omega)$ is locally symplectomorphic to $(\C^n, \omega_0)$. We point out that the analogue of this theorem in Riemannian geometry is clearly false: such a theorem would imply that every Riemannian manifold is flat. Conversely, Darboux' theorem also implies that it is not possible to define local invariants of symplectic manifolds that are analogues of the curvature of a Riemannian manifold.

There are, however, examples of global invariants of symplectic manifolds, for example symplectic capacities. A \textbf{symplectic capacity} is a function $c$ that assigns to every symplectic manifold $(X,\omega)$ (in a restricted subclass of all symplectic manifolds) a number $c(X,\omega) \in [0,+\infty]$, satisfying
\begin{description}
    \item[(Monotonicity)] If there exists a symplectic embedding (possibly in a restricted subset of all symplectic embeddings) $(X, \omega_X) \longrightarrow (Y, \omega_Y)$, then $c(X, \omega_X) \leq c(Y, \omega_Y)$;
    \item[(Conformality)] If $\alpha > 0$ then $c(X, \alpha \omega_X) = \alpha \, c(X, \omega_X)$.
\end{description}
By the monotonicity property, symplectic capacities are symplectomorphism invariants of symplectic manifolds. There are many examples of symplectic capacities, a simple one being the \textbf{volume capacity} $c_{\mathrm{vol}}$, defined as follows for a $2n$-dimensional symplectic manifold $(X, \omega)$. Since $\omega$ is nondegenerate, $\omega^n / n!$ is a volume form on $X$. Define
\begin{IEEEeqnarray*}{rCl}
    \operatorname{vol}(X) & \coloneqq & \int_{X} \frac{\omega^n}{n!}, \\
    c_{\mathrm{vol}}(X)   & \coloneqq & \p{}{2}{\frac{\operatorname{vol}(X)}{\operatorname{vol}(B)}}^{1/n},
\end{IEEEeqnarray*}
where $B = \{z \in \C^n \mid \pi |z|^2 \leq 1 \}$. Symplectic capacities are especially relevant when discussing symplectic embedding problems. Notice that by the monotonicity property, a symplectic capacity can provide an obstruction to the existence of a symplectic embedding. We provide an example from physics. A \textbf{classical mechanical system} is a symplectic manifold $(X, \omega)$ together with a function $H$ called the \textbf{Hamiltonian}. The \textbf{Hamiltonian vector field} of $H$ is the unique vector field $X_H$ on $X$ such that
\begin{IEEEeqnarray*}{c}
    \edv H = - \iota_{X_H} \omega.
\end{IEEEeqnarray*}
Denote by $\phi^t_H$ the flow of $X_H$, which is a symplectomorphism. \textbf{Liouville's theorem} for a classical mechanical system says that for any subset $O \subset X$, the symplectic volume $c_{\mathrm{vol}}(\phi^t_H(O))$ is independent of $t$. The proof of this statement works for any capacity $c$ replacing the volume capacity. So, for every symplectic capacity we get a theorem analogous to Liouville's theorem, giving restrictions on what regions of the phase space flow onto other regions.

In more generality, one could say that \textbf{a symplectic capacity is a quantitative encoding of some specific property of symplectic manifolds}. To make this statement less vague, let us mention some symplectic capacities we will be working with in this thesis.
\begin{enumerate}
    \item If $(X, \omega)$ is a $2n$-dimensional symplectic manifold, a submanifold $L \subset (X, \omega)$ is \textbf{Lagrangian} if $\dim L = n$ and $\omega|_L = 0$. The \textbf{minimal symplectic area} of $L$ is given by
        \begin{IEEEeqnarray*}{c}
            A_{\mathrm{min}}(L) \coloneqq \inf \{ \omega(\sigma) \mid \sigma \in \pi_2(X,L), \, \omega(\sigma) > 0 \}.
        \end{IEEEeqnarray*}
        Cieliebak--Mohnke \cite[Section 1.2]{cieliebakPuncturedHolomorphicCurves2018} define the \textbf{Lagrangian capacity} of $(X, \omega)$ by
        \begin{IEEEeqnarray*}{c}
            c_L(X,\omega) \coloneqq \sup \{ A_{\mathrm{min}}(L) \mid L \subset X \text{ is an embedded Lagrangian torus}\}.
        \end{IEEEeqnarray*}
    \item If $(X, \lambda)$ is a nondegenerate \textbf{Liouville domain} (this implies that $X$ is a compact manifold with boundary together with a $1$-form $\lambda$ such that $(X, \edv \lambda)$ is symplectic, see \cref{def:liouville domain}), one can define its \textbf{$S^1$-equivariant symplectic homology}, denoted $\homology{}{S^1}{}{S}{H}{}{}(X,\lambda)$ (see \cref{sec:Floer homology}). This is a $\Q$-module which comes equipped with a filtration, i.e. for every $a \in \R$ we have a $\Q$-module $\homology{}{S^1}{}{S}{H}{a}{}(X,\lambda)$ and a map%
        \begin{equation*}
            \iota^a \colon \homology{}{S^1}{}{S}{H}{a}{}(X,\lambda) \longrightarrow \homology{}{S^1}{}{S}{H}{}{}(X,\lambda).
        \end{equation*}
        In particular, we can define the $S^1$-equivariant symplectic homology associated to intervals $(a,b] \subset \R$ and $(a, +\infty) \subset \R$ by taking the quotient:
        \begin{IEEEeqnarray*}{rCl}
            \homology{}{S^1}{}{S}{H}{(a,b]}{}(X,\lambda)       & \coloneqq & \homology{}{S^1}{}{S}{H}{b}{}(X,\lambda) / \iota^{b,a}(\homology{}{S^1}{}{S}{H}{a}{}(X,\lambda)), \\
            \homology{}{S^1}{}{S}{H}{(a,+\infty)}{}(X,\lambda) & \coloneqq & \homology{}{S^1}{}{S}{H}{}{} (X,\lambda) / \iota^{a}(\homology{}{S^1}{}{S}{H}{a}{}(X,\lambda)).
        \end{IEEEeqnarray*}
        The \textbf{positive $S^1$-equivariant symplectic homology} is given by $\homology{}{S^1}{}{S}{H}{+}{}(X,\lambda) = \homology{}{S^1}{}{S}{H}{(\varepsilon, + \infty)}{}(X,\lambda)$, where $\varepsilon > $ is a small number. The $S^1$-equivariant symplectic homology also comes with maps $U$ and $\delta$, which can be composed to obtain the map%
        \begin{equation*}
            \delta \circ U^{k-1} \circ \iota^a \colon \homology{}{S^1}{}{S}{H}{(\varepsilon,a]}{}(X) \longrightarrow H_\bullet(BS^1;\Q) \otimes H_\bullet(X, \partial X;\Q).
        \end{equation*}
        The $k$th \textbf{Gutt--Hutchings} capacity of $(X,\lambda)$ (\cite[Definition 4.1]{guttSymplecticCapacitiesPositive2018}) is given by
        \begin{IEEEeqnarray*}{c}
            \cgh{k}(X) \coloneqq \inf \{ a > 0 \mid [\mathrm{pt}] \otimes [X] \in \img (\delta \circ U^{k-1} \circ \iota^a) \}.
        \end{IEEEeqnarray*}
    \item Let $(X,\lambda)$ be a nondegenerate Liouville domain. There is a map
        \begin{equation*}
            \iota^{a,\varepsilon} \circ \alpha^{-1} \colon H_\bullet(BS^1;\Q) \otimes H_\bullet(X, \partial X;\Q) \longrightarrow \homology{}{S^1}{}{S}{H}{a}{}(X).
        \end{equation*}
        The $k$th $\textbf{$S^1$-equivariant symplectic homology capacity}$ was defined by Irie in \cite[Section 2.5]{irieSymplecticHomologyFiberwise2021}, and it is given by
        \begin{IEEEeqnarray*}{c}
            \csh{k}(X) \coloneqq \inf \{ a > 0 \mid \iota^{a,\varepsilon} \circ \alpha^{-1}([\C P^{k-1}] \otimes [X]) = 0 \}.
        \end{IEEEeqnarray*}
    \item Let $(X, \lambda)$ be a nondegenerate Liouville domain. Choose a point $x \in \itr X$ and a \textbf{symplectic divisor} (germ of a symplectic submanifold of codimension 2) $D \subset X$ through $x$. The boundary $(\partial X, \lambda|_{\partial X})$ is a \textbf{contact manifold} (\cref{def:contact manifold}) and therefore has a \textbf{Reeb vector field} (\cref{def:Reeb vector field}). The \textbf{completion} of $(X, \lambda)$ (\cref{def:completion of a Liouville domain}) is the exact symplectic manifold
        \begin{equation*}
            (\hat{X}, \hat{\lambda}) \coloneqq (X, \lambda) \cup_{\partial X} (\R_{\geq 0} \times \partial X, e^r \lambda|_{\partial X}).
        \end{equation*}
        Let $\mathcal{M}_X^J(\Gamma)\p{<}{}{\mathcal{T}^{(k)}x}$ denote the moduli space of $J$-holomorphic curves in $\hat{X}$ which are positively asymptotic to the tuple of Reeb orbits $\Gamma = (\gamma_1, \ldots, \gamma_p)$ and which have contact order $k$ to $D$ at $x$. Finally, for $\ell, k \in \Z_{\geq 1}$, the \textbf{McDuff--Siegel} capacities of $(X,\lambda)$ (\cite[Definition 3.3.1]{mcduffSymplecticCapacitiesUnperturbed2022}) are given by
        \begin{IEEEeqnarray*}{c}
            \tilde{\mathfrak{g}}^{\leq \ell}_k(X) \coloneqq \sup_{J \in \mathcal{J}(X,D)} \mathop{\inf\vphantom{\mathrm{sup}}}_{\Gamma_1, \dots, \Gamma_p} \sum_{i=1}^{p} \mathcal{A}(\Gamma_i),
        \end{IEEEeqnarray*}
        where $\mathcal{J}(X,D)$ is a set of almost complex structures on $\hat{X}$ which are cylindrical at infinity and compatible with $D$ (see \cref{sec:moduli spaces of holomorphic curves}) and the infimum is over tuples of Reeb orbits $\Gamma_1, \ldots, \Gamma_p$ such that there exist $k_1, \ldots, k_p \in \Z_{\geq 1}$ with 
        \begin{IEEEeqnarray*}{c+x*}
            \sum_{i=1}^{p} \# \Gamma_i \leq \ell, \qquad \sum_{i=1}^{p} k_i \geq k, \qquad \bigproduct_{i=1}^{p} \mathcal{M}_X^J(\Gamma_i)\p{<}{}{\mathcal{T}^{(k_i)}x} \neq \varnothing.
        \end{IEEEeqnarray*}
    \item Let $(X, \lambda)$ be a nondegenerate Liouville domain. If one assumes the existence of a suitable virtual perturbation scheme, one can define the \textbf{linearized contact homology} $\mathcal{L}_{\infty}$-algebra of $(X,\lambda)$, denoted $CC(X)[-1]$ (see \cref{def:l infinity algebra,def:linearized contact homology,def:lch l infinity}). We can then consider its \textbf{bar complex} $\mathcal{B}(CC(X)[-1])$ (see \cref{def:bar complex}) and the homology of the bar complex, $H(\mathcal{B}(CC(X)[-1]))$. There is an \textbf{augmentation map} (see \cref{def:augmentation map}) 
    \begin{IEEEeqnarray*}{c+x*}
        {\epsilon}_k \colon \mathcal{B}(CC(X)[-1]) \longrightarrow \Q
    \end{IEEEeqnarray*}
    which counts $J$-holomorphic curves satisfying a tangency constraint. For $\ell, k \in \Z_{\geq 1}$, Siegel \cite[Section 6.1]{siegelHigherSymplecticCapacities2020} defines the \textbf{higher symplectic capacities} by\footnote{To be precise, the definition we give may be slightly different from the one given in \cite{siegelHigherSymplecticCapacities2020}. This is due to the fact that we use an action filtration to define $\mathfrak{g}^{\leq \ell}_k(X)$, while the definition given in \cite{siegelHigherSymplecticCapacities2020} uses coefficients in a Novikov ring. See \cref{rmk:novikov coefficients} for further discussion.}
    \begin{IEEEeqnarray*}{c}
        \mathfrak{g}^{\leq \ell}_k(X) \coloneqq \inf \{ a > 0 \mid \epsilon_k \colon H(\mathcal{A}^{\leq a} \mathcal{B}^{\leq \ell}(CC(X)[-1])) \longrightarrow \Q \text{ is nonzero} \},
    \end{IEEEeqnarray*}
    where $\mathcal{A}^{\leq a}$ denotes the action filtration (\cref{def:action filtration lch}) and $\mathcal{B}^{\leq \ell}$ denotes the word length filtration (\cref{def:word length filtration}).
\end{enumerate}
The previous examples illustrate the fact that capacities can be defined using many tools that exist in symplectic geometry. If a capacity encodes a quantitative property between symplectic manifolds, then an inequality between two capacities encodes a relationship between said properties. So, capacities are also an efficient language to describe quantitative relations between properties of symplectic manifolds. Consider also that one can chain inequalities together to obtain new inequalities. In fact, one of the main goals of this thesis is to compute the Lagrangian capacity of convex or concave toric domains (a toric domain is a special type of Liouville domain, see \cref{def:toric domain}). We give two results in this direction (\cref{lem:computation of cl,thm:my main theorem}), and the proof of both results consists in composing together several inequalities between capacities (namely the capacities $\cgh{k}$, $\tilde{\mathfrak{g}}^{\leq 1}_k$ and $\mathfrak{g}^{\leq 1}_k$ which were defined above), where each of those inequalities is proven separately. Notice that in this case, we are able to compute the Lagrangian capacity of (some) toric domains, whose definition only concerns Lagrangian submanifolds, by considering other capacities whose definition concerns holomorphic curves in the toric domain. 

\section{Historical remarks}

The first symplectic capacity, the \textbf{Gromov width}, was constructed by Gromov \cite{gromovPseudoHolomorphicCurves1985}, although at this time the nomenclature of ``symplectic capacity'' had not been introduced. The notion of symplectic capacity was first introduced by Ekeland--Hofer in \cite{ekelandSymplecticTopologyHamiltonian1989}. In the sequel \cite{ekelandSymplecticTopologyHamiltonian1990}, the authors define the \textbf{Ekeland--Hofer capacities} $c_k^{\mathrm{EH}}$ (for every $k \in \Z_{\geq 1}$) using variational techniques for the symplectic action functional. The \textbf{Hofer--Zehnder capacity} \cite{hoferNewCapacitySymplectic1990,hoferSymplecticInvariantsHamiltonian2011} is another example of a capacity which can be defined by considering Hamiltonian systems. One can consider \textbf{spectral capacities}, which are generally defined as a minimal or maximal action of an orbit (Hamiltonian or Reeb) which is ``topologically visible''. The Gutt--Hutchings capacities \cite{guttSymplecticCapacitiesPositive2018}, $S^1$-equivariant symplectic homology capacities \cite{irieSymplecticHomologyFiberwise2021}, and Siegel's higher symplectic capacities \cite{siegelHigherSymplecticCapacities2020} mentioned above are examples of this principle. Other authors have used constructions like this, namely Hofer \cite{hoferEstimatesEnergySymplectic1993}, Viterbo \cite{viterboSymplecticTopologyGeometry1992,viterboFunctorsComputationsFloer1999}, Schwarz \cite{schwarzActionSpectrumClosed2000}, Oh \cite{ohChainLevelFloer2002,ohMinimaxTheorySpectral2002,ohSpectralInvariantsLength2005}, Frauenfelder--Schlenk \cite{frauenfelderHamiltonianDynamicsConvex2007}, Schlenk \cite{schlenkEmbeddingProblemsSymplectic2008} and Ginzburg--Shon \cite{ginzburgFilteredSymplecticHomology2018}. Using embedded contact homology (ECH), Hutchings \cite{hutchingsQuantitativeEmbeddedContact2011} defines the \textbf{ECH capacities} $c_k^{\mathrm{ECH}}$ (for every $k \in \Z_{\geq 1}$).

\section{Main results}

As explained before, one of the main goals of this thesis is to compute the Lagrangian capacity of (some) toric domains. A \textbf{toric domain} is a Liouville domain of the form $X_{\Omega} \coloneqq \mu^{-1}(\Omega) \subset \C^n$, where $\Omega \subset \R^n_{\geq 0}$ and $\mu(z_1,\ldots,z_n) = \pi(|z_1|^2,\ldots,|z_n|^2)$. The \textbf{ball}, the \textbf{cylinder} and the \textbf{ellipsoid}, which are defined by
\begin{IEEEeqnarray*}{rCrClCl}
    B^{2n}(a)              & \coloneqq & \{     z & = & (z_1,\ldots,z_n) \in \C^n & \mid  & \pi |z|^2 \leq a \}, \\
    Z^{2n}(a)              & \coloneqq & \{     z & = & (z_1,\ldots,z_n) \in \C^n & \mid  & \pi |z_1|^2 \leq a \}, \\
    E^{2n}(a_1,\ldots,a_n) & \coloneqq & \Big\{ z & = & (z_1,\ldots,z_n) \in \C^n & \Big| & \sum_{j=1}^{n} \frac{\pi |z_j|^2}{a_j} \leq 1 \Big\},
\end{IEEEeqnarray*}
are examples of toric domains.\footnote{Strictly speaking, the cylinder is noncompact, so it is not a toric domain. We will mostly ignore this small discrepancy in nomenclature, but sometimes we will refer to spaces like the cylinder as ``noncompact toric domains''.} The \textbf{diagonal} of a toric domain $X_{\Omega}$ is 
\begin{IEEEeqnarray*}{c}
    \delta_\Omega \coloneqq \max \{ a \mid (a,\ldots,a) \in \Omega \}.
\end{IEEEeqnarray*}
It is easy to show (see \cref{lem:c square leq c lag,lem:c square geq delta}) that $c_L(X_\Omega) \geq \delta_\Omega$ for any convex or concave toric domain $X_{\Omega}$. Cieliebak--Mohnke give the following results for the Lagrangian capacity of the ball and the cylinder.

\begin{copiedtheorem}[{\cite[Corollary 1.3]{cieliebakPuncturedHolomorphicCurves2018}}]{prp:cl of ball}
    The Lagrangian capacity of the ball is
    \begin{IEEEeqnarray*}{c+x*}
        c_L(B^{2n}(1)) = \frac{1}{n}.\footnote{In this introduction, we will be showcasing many results from the main text. The theorems appear here as they do on the main text, in particular with the same numbering. The numbers of the theorems in the introduction have hyperlinks to their corresponding location in the main text.}
    \end{IEEEeqnarray*}
\end{copiedtheorem}

\begin{copiedtheorem}[{\cite[p.~215-216]{cieliebakPuncturedHolomorphicCurves2018}}]{prp:cl of cylinder}
    The Lagrangian capacity of the cylinder is
    \begin{IEEEeqnarray*}{c+x*}
        c_L(Z^{2n}(1)) = 1.
    \end{IEEEeqnarray*}
\end{copiedtheorem}

In other words, if $X_{\Omega}$ is the ball or the cylinder then $c_L(X_{\Omega}) = \delta_\Omega$. This motivates the following conjecture by Cieliebak--Mohnke. 

\begin{copiedtheorem}[{\cite[Conjecture 1.5]{cieliebakPuncturedHolomorphicCurves2018}}]{conj:cl of ellipsoid}
    The Lagrangian capacity of the ellipsoid is%
    \begin{equation*}
        c_L(E(a_1,\ldots,a_n)) = \p{}{2}{\frac{1}{a_1} + \cdots + \frac{1}{a_n}}^{-1}.
    \end{equation*}
\end{copiedtheorem}

A more general form of the previous conjecture is the following.

\begin{copiedtheorem}{conj:the conjecture}
    If $X_{\Omega}$ is a convex or concave toric domain then 
    \begin{IEEEeqnarray*}{c+x*}
        c_L(X_{\Omega}) = \delta_\Omega.
    \end{IEEEeqnarray*}
\end{copiedtheorem}

The goal of this project is to prove \cref{conj:the conjecture}. We will offer two main results in this direction.
\begin{enumerate}
    \item In \cref{lem:computation of cl}, we prove that $c_L(X_\Omega) = \delta_\Omega$ whenever $X_{\Omega}$ is convex and $4$-dimensional.
    \item In \cref{thm:my main theorem}, using techniques from contact homology we prove that $c_L(X_\Omega) = \delta_\Omega$ for any convex or concave toric domain $X_{\Omega}$. More specifically, in this case we are working under the assumption that there is a virtual perturbation scheme such that the linearized contact homology of a nondegenerate Liouville domain can be defined (see \cref{sec:assumptions of virtual perturbation scheme}).
\end{enumerate}
Notice that by the previous discussion, we only need to prove the hard inequality $c_L(X_{\Omega}) \leq \delta_\Omega$. We now describe our results concerning the capacities mentioned so far. The key step in proving $c_L(X_{\Omega}) \leq \delta_\Omega$ is the following inequality between $c_L$ and $\tilde{\mathfrak{g}}^{\leq 1}_k$.

\begin{copiedtheorem}{thm:lagrangian vs g tilde}
    If $(X, \lambda)$ is a Liouville domain then 
    \begin{IEEEeqnarray*}{c+x*}
        c_L(X) \leq \inf_k^{} \frac{\tilde{\mathfrak{g}}_k^{\leq 1}(X)}{k}.
    \end{IEEEeqnarray*}
\end{copiedtheorem}

Indeed, this result can be combined with the following results from \cite{mcduffSymplecticCapacitiesUnperturbed2022} and \cite{guttSymplecticCapacitiesPositive2018}.

\begin{copiedtheorem}[{\cite[Proposition 5.6.1]{mcduffSymplecticCapacitiesUnperturbed2022}}]{prp:g tilde and cgh}
    If $X_{\Omega}$ is a $4$-dimensional convex toric domain then
    \begin{IEEEeqnarray*}{c+x*}
        \tilde{\mathfrak{g}}^{\leq 1}_k(X_\Omega) = \cgh{k}(X_\Omega).
    \end{IEEEeqnarray*}
\end{copiedtheorem}

\begin{copiedtheorem}[{\cite[Lemma 1.19]{guttSymplecticCapacitiesPositive2018}}]{lem:cgh of nondisjoint union of cylinders}
    $\cgh{k}(N^{2n}(\delta)) = \delta \, (k + n - 1)$.
\end{copiedtheorem}

Here, 
\begin{IEEEeqnarray*}{c}
    N^{2n}(\delta) \coloneqq \p{c}{2}{ (z_1,\ldots,z_n) \in \C^n \ \Big| \ \exists j=1,\ldots,n \colon \frac{\pi |z_j|^2}{\delta} \leq 1 }
\end{IEEEeqnarray*}
is the \textbf{nondisjoint union of cylinders}. Combining the three previous results, we get the following particular case of \cref{conj:the conjecture}. Since the proof is short, we present it here as well.

\begin{copiedtheorem}{lem:computation of cl}
    If $X_{\Omega}$ is a $4$-dimensional convex toric domain then
    \begin{IEEEeqnarray*}{c+x*}
        c_L(X_{\Omega}) = \delta_\Omega.
    \end{IEEEeqnarray*}
\end{copiedtheorem}
\begin{proof}
    For every $k \in \Z_{\geq 1}$,
    \begin{IEEEeqnarray*}{rCls+x*}
        \delta_\Omega
        & \leq & c_L(X_{\Omega})                                         & \quad [\text{by \cref{lem:c square geq delta,lem:c square leq c lag}}] \\
        & \leq & \frac{\tilde{\mathfrak{g}}^{\leq 1}_{k}(X_{\Omega})}{k} & \quad [\text{by \cref{thm:lagrangian vs g tilde}}] \\
        & =    & \frac{\cgh{k}(X_{\Omega})}{k}                           & \quad [\text{by \cref{prp:g tilde and cgh}}] \\
        & \leq & \frac{\cgh{k}(N(\delta_\Omega))}{k}                     & \quad [\text{$X_{\Omega}$ is convex, hence $X_{\Omega} \subset N(\delta_\Omega)$}] \\
        & =    & \frac{\delta_\Omega(k+1)}{k}                          & \quad [\text{by \cref{lem:cgh of nondisjoint union of cylinders}}].
    \end{IEEEeqnarray*}
    The result follows by taking the infimum over $k$.
\end{proof}

Notice that in the proof of this result, we used the Gutt--Hutchings capacities because the value $\cgh{k}(N^{2n}(\delta))$ is known and provides the desired upper bound for $c_L(X_{\Omega})$. Notice also that the hypothesis of the toric domain being convex and $4$-dimensional is present because we wish to use \cref{prp:g tilde and cgh} to compare $\tilde{\mathfrak{g}}^{\leq 1}_k$ and $\cgh{k}$. This suggests that we try to compare $c_L$ and $\cgh{k}$ directly.

\begin{copiedtheorem}{thm:main theorem}
    If $X$ is a Liouville domain, $\pi_1(X) = 0$ and $c_1(TX)|_{\pi_2(X)} = 0$, then%
    \begin{equation*}
        c_L(X,\lambda) \leq \inf_k \frac{\cgh{k}(X,\lambda)}{k}.
    \end{equation*}
\end{copiedtheorem}

We will try to prove \cref{thm:main theorem} by mimicking the proof of \cref{thm:lagrangian vs g tilde}. Unfortunately we will be unsuccessful, because we run into difficulties coming from the fact that in $S^1$-equivariant symplectic homology, the Hamiltonians and almost complex structures can depend on the domain and on a high dimensional sphere $S^{2N+1}$. Before we move on to the discussion about computations using contact homology, we show one final result which uses only the properties of $S^1$-equivariant symplectic homology.
\begin{copiedtheorem}{thm:ghc and s1eshc}
    If $(X, \lambda)$ is a Liouville domain, then
    \begin{enumerate}
        \item $\cgh{k}(X) \leq \csh{k}(X)$;
        \item $\cgh{k}(X) = \csh{k}(X)$ provided that $X$ is star-shaped.
    \end{enumerate}
\end{copiedtheorem}

We now present another approach that can be used to compute $c_L$, using linearized contact homology. This has the disadvantage that at the time of writing, linearized contact homology has not yet been defined in the generality that we need (see \cref{sec:assumptions of virtual perturbation scheme} and more specifically \cref{assumption}). Using linearized contact homology, one can define the higher symplectic capacities $\mathfrak{g}^{\leq \ell}_k$. The definition of $\mathfrak{g}^{\leq \ell}_k$ for any $\ell \in \Z_{\geq 1}$ relies on the $\mathcal{L}_{\infty}$-algebra structure of the linearized contact homology chain complex, as well as an $\mathcal{L}_{\infty}$-augmentation map $\epsilon_k$. However, to prove that $c_L(X_{\Omega}) \leq \delta_\Omega$, we will only need the capacity $\mathfrak{g}^{\leq 1}_k$, and for this the $\mathcal{L}_{\infty}$-algebra structure is not necessary. The key idea is that the capacities $\mathfrak{g}^{\leq 1}_k$ can be compared to $\tilde{\mathfrak{g}}^{\leq 1}_k$ and $\cgh{k}$.

\begin{copiedtheorem}[{\cite[Section 3.4]{mcduffSymplecticCapacitiesUnperturbed2022}}]{thm:g tilde vs g hat}
    If $X$ is a Liouville domain then 
    \begin{IEEEeqnarray*}{c+x*}
        \tilde{\mathfrak{g}}^{\leq \ell}_k(X) \leq {\mathfrak{g}}^{\leq \ell}_k(X).
    \end{IEEEeqnarray*}
\end{copiedtheorem}

\begin{copiedtheorem}{thm:g hat vs gh}
    If $X$ is a Liouville domain such that $\pi_1(X) = 0$ and $2 c_1(TX) = 0$ then%
    \begin{IEEEeqnarray*}{c+x*}
        {\mathfrak{g}}^{\leq 1}_k(X) = \cgh{k}(X).
    \end{IEEEeqnarray*}
\end{copiedtheorem}

These two results show that $\tilde{\mathfrak{g}}^{\leq 1}_k(X_\Omega) \leq \cgh{k}(X_\Omega)$ (under \cref{assumption}). Using the same proof as before, we conclude that $c_L(X_{\Omega}) = \delta_\Omega$.

\begin{copiedtheorem}{thm:my main theorem}
    Under \cref{assumption}, if $X_\Omega$ is a convex or concave toric domain then%
    \begin{IEEEeqnarray*}{c+x*}
        c_L(X_{\Omega}) = \delta_\Omega.
    \end{IEEEeqnarray*}
\end{copiedtheorem}

\section{Proof sketches}

In the last section, we explained our proof of $c_L(X_{\Omega}) = \delta_\Omega$ (first in the case where $X_{\Omega}$ is convex and $4$-dimensional, and second assuming that \cref{assumption} holds). In this section, we explain the proofs of the relations
\begin{IEEEeqnarray*}{rCls+x*}
    c_L(X)                                & \leq & \inf_k \frac{\tilde{\mathfrak{g}}^{\leq 1}_k(X)}{k}, \\
    \tilde{\mathfrak{g}}^{\leq \ell}_k(X) & \leq & \mathfrak{g}^{\leq \ell}_k(X), \\
    \mathfrak{g}_k^{\leq 1}(X)            & =    & \cgh{k}(X),
\end{IEEEeqnarray*}
which were mentioned without proof in the last section. Each of these relations will be proved in the main text, so the proof sketches of this section act as a way of showcasing what technical tools will be required for our purposes. In \cref{sec:symplectic capacities}, we study the question of extending the domain of a symplectic capacities from the class of nondegenerate Liouville domains to the class of Liouville domains which are possibly degenerate. By this discussion, it suffices to prove each theorem for nondegenerate Liouville domains only.

\begin{secondcopy}{thm:lagrangian vs g tilde}
    If $(X, \lambda)$ is a Liouville domain then 
    \begin{IEEEeqnarray*}{c+x*}
        c_L(X) \leq \inf_k^{} \frac{\tilde{\mathfrak{g}}_k^{\leq 1}(X)}{k}.
    \end{IEEEeqnarray*}
\end{secondcopy}
\begin{proof}[Proof sketch]
    Let $k \in \Z_{\geq 1}$ and $L \subset \itr X$ be an embedded Lagrangian torus. Denote $a \coloneqq \tilde{\mathfrak{g}}_k^{\leq 1}(X)$. We wish to show that there exists $\sigma \in \pi_2(X,L)$ such that $0 < \omega(\sigma) \leq a / k$. Choose a suitable Riemannian metric on $L$, given by \cref{lem:geodesics lemma CM abs} (which is a restatement of \cite[Lemma 2.2]{cieliebakPuncturedHolomorphicCurves2018}). Now, consider the unit cotangent bundle $S^* L$ of $L$. Choose a point $x$ inside the unit codisk bundle $D^* L$, a symplectic divisor $D$ through $x$, and a sequence $(J_t)_{t \in [0,1)}$ of almost complex structures on $\hat{X}$ realizing SFT neck stretching along $S^* L$. 
    
    By definition of $\tilde{\mathfrak{g}}_k^{\leq 1}(X) \eqqcolon a$, there exists a Reeb orbit $\gamma_0$ together with a sequence $(u_t)_t$ of $J_t$-holomorphic curves $u_t \in \mathcal{M}^{J_t}_X(\gamma_0)\p{<}{}{\mathcal{T}^{(k)}x}$. By the SFT-compactness theorem, the sequence $(u_t)_{t}$ converges to a holomorphic building $F = (F^1,\ldots,F^N)$, where each $F^{\nu}$ is a holomorphic curve. Denote by $C$ the component of $F^1 \subset T^* L$ which carries the tangency constraint. The choices of almost complex structures $J_t$ can be done in such a way that the simple curve corresponding to $C$ is regular, i.e. it is an element of a moduli space which is a manifold. Using the dimension formula for this moduli space, it is possible to conclude that $C$ must have at least $k + 1$ punctures (see \cref{thm:transversality with tangency,lem:punctures and tangency simple,lem:punctures and tangency}). This implies that $C$ gives rise to at least $k > 0$ disks $D_1, \ldots, D_k$ in $X$ with boundary on $L$. The total energy of the disks is less or equal to $a$. Therefore, one of the disks must have energy less or equal to $a/k$.

    We now address a small imprecision in the proof we just described. We need to show that $\omega(D_i) \leq a$ for some $i = 1, \ldots, k$. However, the above proof actually shows that $\tilde{\omega}(D_i) \leq a$, where $\tilde{\omega}$ is a piecewise smooth $2$-form on $\hat{X} \setminus L$ given as in \cref{def:energy of a asy cylindrical holomorphic curve}. This form has the property that $\omega = \tilde{\omega}$ outside $S^* L$. The solution then is to neck stretch along $S_{\delta}^* L$ for some small $\delta > 0$. In this case, one can bound $\omega(D_i)$ by $\tilde{\omega}(D_i)$ times a function of $\delta$ (see \cref{lem:energy wrt different forms}), and we can still obtain the desired bound for $\omega(D_i)$.    
\end{proof}

\begin{secondcopy}[\cite[Section 3.4]{mcduffSymplecticCapacitiesUnperturbed2022}]{thm:g tilde vs g hat}
    If $X$ is a Liouville domain then 
    \begin{IEEEeqnarray*}{c+x*}
        \tilde{\mathfrak{g}}^{\leq \ell}_k(X) \leq {\mathfrak{g}}^{\leq \ell}_k(X).
    \end{IEEEeqnarray*}
\end{secondcopy}
\begin{proof}[Proof sketch]
    Choose a point $x \in \itr X$ and a symplectic divisor $D$ through $x$. Let $J \in \mathcal{J}(X,D)$ and consider the bar complex $\mathcal{B}(CC(X)[-1])$, computed with respect to $J$. Suppose that $a > 0$ and $\beta \in H(\mathcal{A}^{\leq a} \mathcal{B}^{\leq \ell}(CC(X)[-1]))$ is such that $\epsilon_k(\beta) \neq 0$. By \cref{thm:g tilde two definitions}, 
    \begin{IEEEeqnarray*}{c+x*}
        \tilde{\mathfrak{g}}^{\leq \ell}_k(X) = \sup_{J \in \mathcal{J}(X,D)} \mathop{\inf\vphantom{\mathrm{sup}}}_{\Gamma} \mathcal{A}(\Gamma),
    \end{IEEEeqnarray*}
    where the infimum is taken over tuples of Reeb orbits $\Gamma = (\gamma_1, \ldots, \gamma_p)$ such that $p \leq \ell$ and $\overline{\mathcal{M}}^{J}_{X}(\Gamma)\p{<}{}{\mathcal{T}^{(k)}x} \neq \varnothing$. The class $\beta$ is a linear combination of words of Reeb orbits $\Gamma$ such that $\# \Gamma \leq \ell$ and $\mathcal{A}(\Gamma) \leq a$. Since $\epsilon_k(\beta) \neq 0$, one of the words in this linear combination, say $\Gamma$, is such that the virtual count of $\overline{\mathcal{M}}^{J}_{X}(\Gamma)\p{<}{}{\mathcal{T}^{(k)}x}$ is nonzero. By assumption on the virtual perturbation scheme, $\overline{\mathcal{M}}^{J}_{X}(\Gamma)\p{<}{}{\mathcal{T}^{(k)}x}$ is nonempty, which is the condition in the definition of $\tilde{\mathfrak{g}}^{\leq \ell}_k(X)$.
\end{proof}

\begin{secondcopy}{thm:g hat vs gh}
    If $X$ is a Liouville domain such that $\pi_1(X) = 0$ and $2 c_1(TX) = 0$ then%
    \begin{IEEEeqnarray*}{c+x*}
        {\mathfrak{g}}^{\leq 1}_k(X) = \cgh{k}(X).
    \end{IEEEeqnarray*}
\end{secondcopy}
\begin{proof}[Proof sketch]
    Choose a small ellipsoid $E$ such that there exists a strict exact symplectic embedding $\phi \colon E \longrightarrow X$. There are associated Viterbo transfer maps (see \cref{sec:viterbo transfer map of liouville embedding,sec:viterbo transfer map of exact symplectic embedding}, where we define the Viterbo transfer map of $S^1$-equivariant symplectic homology)
    \begin{IEEEeqnarray*}{rCls+x*}
        \phi_!^{S^1} \colon \homology{}{S^1}{}{S}{H}{}{}(X) & \longrightarrow & \homology{}{S^1}{}{S}{H}{}{}(E), \\
        \phi_!       \colon CH(X)                           & \longrightarrow & CH(E).
    \end{IEEEeqnarray*}
    Because of the topological conditions on $X$, the $S^1$-equivariant symplectic homology and the linearized contact homology have $\Z$-gradings given by the Conley--Zehnder index. In this context, one can offer an alternative definition of the Gutt--Hutchings capacities  via the Viterbo transfer map, namely $\cgh{k}(X)$ is the infimum over $a$ such that the map
    \begin{equation*}
        \begin{tikzcd}
            \homology{}{S^1}{}{S}{H}{(\varepsilon,a]}{n - 1 + 2k}(X) \ar[r, "\iota^{S^1,a}"] & \homology{}{S^1}{}{S}{H}{+}{n - 1 + 2k}(X) \ar[r, "\phi_!^{S^1}"] & \homology{}{S^1}{}{S}{H}{+}{n - 1 + 2k}(E)
        \end{tikzcd}
    \end{equation*}
    is nonzero (see \cref{def:ck alternative}). Bourgeois--Oancea \cite{bourgeoisEquivariantSymplecticHomology2016} define an isomorphism 
    \begin{IEEEeqnarray*}{c+x*}
        \Phi_{\mathrm{BO}} \colon \homology{}{S^1}{}{S}{H}{+}{}(X) \longrightarrow CH(X)
    \end{IEEEeqnarray*}
    between positive $S^1$-equivariant symplectic homology and linearized symplectic homology (whenever the latter is defined). All the maps we have just described assemble into the following commutative diagram.
    \begin{equation*}
        \begin{tikzcd}
            SH^{S^1,(\varepsilon,a]}_{n - 1 + 2k}(X) \ar[r, "\iota^{S^1,a}"] \ar[d, hook, two heads, swap, "\Phi_{\mathrm{BO}}^a"] & SH^{S^1,+}_{n - 1 + 2k}(X) \ar[r, "\phi_!^{S^1}"] \ar[d, hook, two heads, "\Phi_{\mathrm{BO}}"] & SH^{S^1,+}_{n - 1 + 2k}(E) \ar[d, hook, two heads, "\Phi_{\mathrm{BO}}"] \\
            CH^{a}_{n - 1 + 2k}(X) \ar[r, "\iota^{a}"] \ar[d, equals]                                                              & CH_{n - 1 + 2k}(X) \ar[r, "\phi_!"] \ar[d, equals]                                              & CH_{n - 1 + 2k}(E) \ar[d, "{\epsilon}^E_k"] \\
            CH^{a}_{n - 1 + 2k}(X) \ar[r, swap, "\iota^{a}"]                                                                       & CH_{n - 1 + 2k}(X) \ar[r, swap, "{\epsilon}_k^X"]                                               & \Q
        \end{tikzcd}
    \end{equation*}
    Here, the vertical arrows between the top two rows are the Bourgeois--Oancea isomorphism and the maps $\epsilon_k^X$ and $\epsilon_k^E$ are the augmentation maps of $X$ and $E$. Using this information, we can show that $\cgh{k}(X) \leq \mathfrak{g}^{\leq 1}_k(X)$:
    \begin{IEEEeqnarray*}{rCls+x*}
        \cgh{k}(X)
        & =    & \inf \{ a > 0 \mid \phi_!^{S^1} \circ \iota^{S^1,a} \neq 0 \} & \quad [\text{by the alternative definition of $\cgh{k}$}] \\
        & \leq & \inf \{ a > 0 \mid {\epsilon}_k^X \circ \iota^{a} \neq 0 \}   & \quad [\text{since the diagram commutes}] \\
        & =    & {\mathfrak{g}}^{\leq 1}_k(X)                                  & \quad [\text{by definition of $\mathfrak{g}^{\leq 1}_k$}].
    \end{IEEEeqnarray*}
    In this computation, the inequality in the second line is an equality if $\epsilon^E_k$ is an isomorphism. The proof of this statement is done in \cref{sec:augmentation map of an ellipsoid}, using the techniques from \cref{sec:cr operators,sec:functional analytic setup}. The key ideas are the following. One can show that $CH_{n - 1 + 2k}(E) \cong \Q$ (see \cref{lem:lch of ellipsoid}), and therefore it is enough to show that $\epsilon_k^E$ is nonzero. Recall that $\epsilon_k^E$ is given by the virtual count of holomorphic curves in $X$ satisfying a tangency constraint. We count those curves explicitly in \cref{lem:moduli spaces of ellipsoids have 1 element}. Notice that here we need to justify that the virtual count of curves equals the usual signed count. This follows by assumption on the virtual perturbation scheme and because in \cref{sec:augmentation map of an ellipsoid}, we also show that the moduli spaces are transversely cut out.
\end{proof}

\section{Outline of the thesis}

We now give a chapter by chapter outline of this thesis.

In \textbf{\cref{chp:symplectic manifolds}} we review the various types of manifolds that will show up in this thesis, i.e. symplectic manifolds and contact manifolds. We talk about the various types of vector fields in these manifolds (Hamiltonian vector field, Liouville vector field, Reeb vector field) and mention the properties of their flows. We give the definition of special types of symplectic manifolds, from less to more specific: Liouville domains, star-shaped domains, toric domains. Finally, we explain two constructions which will be present throughout: the symplectization of a contact manifold, and the completion of a Liouville domain.

In \textbf{\cref{chp:indices}} we give a review of the Conley--Zehnder indices. In order to list the properties of the Conley--Zehnder index, one needs to mention the Maslov index and the first Chern class, so we offer a review of those as well. We explain how to define the Conley--Zehnder index of an orbit in a symplectic or contact manifold by defining an induced path of symplectic matrices via a trivialization. Finally, we study the Conley--Zehnder index of a Reeb orbit in a unit cotangent bundle. The Conley--Zehnder index is needed for our purposes because it provides the grading of $S^1$-equivariant symplectic homology and of linearized contact homology.

\textbf{\cref{chp:holomorphic curves}} is about the analytic properties of holomorphic curves and Floer trajectories. We define punctured Riemann surfaces as the domains for such curves, and symplectic cobordisms as the targets for such curves. We prove the energy identity for holomorphic curves, as well as the maximum principle. Then, we discuss the known compactness and transversality for moduli spaces of asymptotically cylindrical holomorphic curves (these are the moduli spaces which are considered in linearized contact homology). The second half of this chapter is about solutions of the ``parametrized Floer equation'' (solutions to this equation are the trajectories which are counted in the differential of $S^1$-equivariant Floer chain complex). We prove an energy inequality for Floer trajectories, as well as three ``confinement lemmas'': the maximum principle, the asymptotic behaviour lemma, and the no escape lemma. Finally, we prove compactness and transversality for moduli spaces of solutions of the parametrized Floer equation using the corresponding results for moduli spaces of solutions of the Floer equation.

In \textbf{\cref{chp:floer}} we define the $S^1$-equivariant symplectic homology and establish its structural properties. First we define the $S^1$-equivariant Floer chain complex and its homology. The $S^1$-equivariant symplectic homology is then defined by taking the limit with respect to an increasing sequence of Hamiltonians of the $S^1$-equivariant Floer homology. We devote two sections to showing that $S^1$-equivariant symplectic homology is a functor, which amounts to defining the Viterbo transfer maps and proving their properties. Finally, we define a $\delta$ map, which enters the definition of the Gutt--Hutchings capacities.

\textbf{\cref{chp:symplectic capacities}} is about symplectic capacities. The first section is about generalities about symplectic capacities. We show how to extend a capacity of nondegenerate Liouville domains to a capacity of (possibly degenerate) Liouville domains. The next three sections are each devoted to defining and proving the properties of a specific capacity, namely the Lagrangian capacity $c_L$, the Gutt--Hutchings capacities $\cgh{k}$ and the $S^1$-equivariant symplectic homology capacities $\csh{k}$, and finally the McDuff--Siegel capacities $\tilde{\mathfrak{g}}^{\leq \ell}_k$. In the section about the Lagrangian capacity, we also state the conjecture that we will try to solve in the remainder of the thesis, i.e. $c_L(X_{\Omega}) = \delta_\Omega$ for a convex or concave toric domain $X_{\Omega}$. The final section is devoted to computations. We show that $c_L(X) \leq \inf_k^{} \tilde{\mathfrak{g}}^{\leq 1}_k(X) / k$. We use this result to prove the conjecture in the case where $X_{\Omega}$ is $4$-dimensional and convex.

\textbf{\cref{chp:contact homology}} introduces the linearized contact homology of a nondegenerate Liouville domain. The idea is that using the linearized contact homology, one can define the higher symplectic capacities, which will allow us to prove $c_L(X_{\Omega}) = \delta_\Omega$ for any convex or concave toric domain $X_{\Omega}$ (but under the assumption that linearized contact homology and the augmentation map are well-defined). We give a review of real linear Cauchy--Riemann operators on complex vector bundles, with a special emphasis on criteria for surjectivity in the case where the bundle has complex rank $1$. We use this theory to prove that moduli spaces of curves in ellipsoids are transversely cut out and in particular that the augmentation map of an ellipsoid is an isomorphism. The final section is devoted to computations. We show that $\mathfrak{g}^{\leq 1}_k(X) = \cgh{k}(X)$, and use this result to prove our conjecture (again, under \cref{assumption}).

\chapter{Symplectic and contact manifolds}
\label{chp:symplectic manifolds}

\section{Symplectic manifolds}

In this section, we recall some basics about symplectic manifolds.

\begin{definition}
    \label{def:symplectic manifold}
    A \textbf{symplectic manifold} is a manifold $X$ together with a $2$-form $\omega$ which is closed and nondegenerate. In this case we say that $\omega$ is a \textbf{symplectic form}. An \textbf{exact symplectic manifold} is a manifold $X$ together with a $1$-form $\lambda$ such that $\omega = \edv \lambda$ is a symplectic form. In this case we call $\lambda$ a \textbf{symplectic potential} for $\omega$.
\end{definition}

\begin{example}
    \label{exa:cn symplectic}
    Consider $\C^n$ with coordinates $(x^1, \ldots, x^n, y^1, \ldots, y^n)$, where $z^j = x^j + i y^j$ for every $j = 1, \ldots, n$. We define
    \begin{IEEEeqnarray*}{rCls+x*}
        \lambda & \coloneqq & \frac{1}{2} \sum_{j=1}^{n} (x^j \edv y^j - y^j \edv x^j), \\
        \omega  & \coloneqq & \edv \lambda = \sum_{j=1}^{n} \edv x^j \wedge \edv y^j.
    \end{IEEEeqnarray*}
    Then, $(\C^n, \lambda)$ is an exact symplectic manifold.        
\end{example}

\begin{example}
    \label{exa:cotangent bundle}
    Let $L$ be a manifold and consider the \textbf{cotangent bundle} of $L$, which is a vector bundle $\pi \colon T^*L \longrightarrow L$. As a set, $T^*L = \bigunion_{q \in L}^{} T^*_qL$. As a vector bundle, $T^*L$ is given as follows. For each coordinate chart $(U,q^1,\ldots,q^n)$ on $L$, there is a coordinate chart $(\pi ^{-1}(U),q^1 \circ \pi,\ldots,q^n \circ \pi,p_1,\ldots,p_n)$ on $T^*L$, where the $p_i$ are given by
    \begin{IEEEeqnarray*}{c}
        p_i(u) \coloneqq u \p{}{2}{ \pdv{}{q^i} \Big|_{\pi(u)} }
    \end{IEEEeqnarray*}
    for $u \in T^*L$. For simplicity, denote $q^i = q^i \circ \pi$. Define a 1-form $\lambda$ on $T^*L$, called the \textbf{canonical symplectic potential} or \textbf{Liouville $1$-form}, as follows. For each $u \in T^*L$, the linear map $\lambda _{u} \colon T _{u} T^*L \longrightarrow \R$ is given by $\lambda_{u} \coloneqq u \circ \dv \pi(u)$. The form $\omega \coloneqq \edv \lambda$ is the \textbf{canonical symplectic form}. In coordinates,
    \begin{IEEEeqnarray*}{rCls+x*}
        \lambda & = & \sum_{i=1}^{n} p_i \edv q^i, \\
        \omega  & = & \sum_{i=1}^{n} \edv p_i \wedge \edv q^i.
    \end{IEEEeqnarray*}
    Then, $(T^*L,\lambda)$ is an exact symplectic manifold. 
\end{example}

If $(X, \omega)$ is a symplectic manifold, then using symplectic linear algebra we conclude that $X$ must be even dimensional, i.e. $\dim X = 2n$ for some $n$ (see for example \cite[Theorem 1.1]{silvaLecturesSymplecticGeometry2008}). In particular, $\omega^n$ is a volume form on $X$.

\begin{definition}
    \label{def:types of embeddings}
    Let $(X,\omega_X)$, $(Y,\omega_Y)$ be symplectic manifolds and $\varphi \colon X \longrightarrow Y$ be an embedding. Then, $\varphi$ is \textbf{symplectic} if $\varphi^* \omega_Y = \omega_X$. A \textbf{symplectomorphism} is a symplectic embedding which is a diffeomorphism. We say that $\varphi$ is \textbf{strict} if $\varphi(X) \subset \itr Y$. If $(X,\lambda_X)$, $(Y,\lambda_Y)$ are exact, then we say that $\varphi$ is:
    \begin{enumerate}
        \item \label{def:types of embeddings 1} \textbf{symplectic} if $\varphi^* \lambda_Y - \lambda_X$ is closed (this is equivalent to the previous definition);
        \item \label{def:types of embeddings 2} \textbf{generalized Liouville} if $\varphi^* \lambda_Y - \lambda_X$ is closed and $(\varphi^* \lambda_Y - \lambda_X)|_{\partial X}$ is exact;
        \item \label{def:types of embeddings 3} \textbf{exact symplectic} if $\varphi^* \lambda_Y - \lambda_X$ is exact;
        \item \label{def:types of embeddings 4} \textbf{Liouville} if $\varphi^* \lambda_Y - \lambda_X = 0$.
    \end{enumerate}
\end{definition}

\begin{remark}
    \label{rmk:closed equivalent to exact}
    In the context of \cref{def:types of embeddings}, if $H^1_{\mathrm{dR}}(X) = 0$ then \ref{def:types of embeddings 1} $\Longleftrightarrow$ \ref{def:types of embeddings 2} $\Longleftrightarrow$ \ref{def:types of embeddings 3}.
\end{remark}

\begin{remark}
    The composition of generalized Liouville embeddings is not necessarily a generalized Liouville embedding. This means that exact symplectic manifolds together with generalized Liouville embeddings do not form a category.
\end{remark}

\begin{definition}
    Let $(X,\omega)$ be a symplectic manifold of dimension $2n$ and $\iota \colon L \longrightarrow X$ be an immersed submanifold of dimension $n$. Then, $L$ is \textbf{Lagrangian} if $\iota^* \omega = 0$. If $(X,\lambda)$ is exact, then we say that $L$ is:
    \begin{enumerate}
        \item \textbf{Lagrangian} if $\iota^* \lambda$ is closed (this is equivalent to the previous definition);
        \item \textbf{exact Lagrangian} if $\iota^* \lambda$ is exact.
    \end{enumerate}
\end{definition}

\begin{example}
    Let $L$ be a manifold and consider its cotangent bundle, $T^*L$. Then, the zero section $z \colon L \longrightarrow T^*L$ is an exact Lagrangian. In fact, $z^* \lambda = 0$.
\end{example}

\begin{lemma}[Moser's trick]
    \label{lem:mosers trick}
    Let $X$ be a manifold, $\alpha_t$ be a smooth $1$-parameter family of forms on $X$ and $Y_t$ be a complete time dependent vector field on $X$ with flow $\phi_t$. Then,%
    \begin{equation*}
        \phi^*_t \alpha_t^{} - \alpha_0^{} = \int_{0}^{t} \phi^*_s \p{}{1}{ \dot{\alpha}_s + \ldv{Y_s} \alpha_s } \edv s = \int_{0}^{t} \phi^*_s \p{}{1}{ \dot{\alpha}_s + \edv \iota _{Y_s} \alpha_s + \iota _{Y_s} \edv \alpha_s } \edv s.
    \end{equation*}
\end{lemma}
\begin{proof}
    \begin{IEEEeqnarray*}{rCls+x*}
        \IEEEeqnarraymulticol{3}{l}{\phi^*_t \alpha_t^{} - \alpha_0^{}}\\ \quad
        & = & \phi^*_t \alpha_t^{} - \phi^*_0 \alpha_0^{}                                                                      & \quad [\text{since $\phi_0 = \id$}]            \\
        & = & \int_{0}^{t} \odv{}{s} \phi^*_s \alpha_s \, \edv s                                                               & \quad [\text{by the fundamental theorem of calculus}] \\
        & = & \int_{0}^{t} \phi^*_s \p{}{1}{ \dot{\alpha}_s + \ldv{Y_s} \alpha_s } \edv s                                      & \quad [\text{by definition of Lie derivative}] \\
        & = & \int_{0}^{t} \phi^*_s \p{}{1}{ \dot{\alpha}_s + \edv \iota _{Y_s} \alpha_s + \iota _{Y_s} \edv \alpha_s } \edv s & \quad [\text{by the Cartan magic formula}].             & \qedhere
    \end{IEEEeqnarray*}
\end{proof}

\begin{theorem}[Darboux]
    Let $(X,\omega)$ be a symplectic manifold. Then, for every $p \in X$, there exists a coordinate neighbourhood $(U,x^1,\ldots,x^n,y^1,\ldots,y^n)$ of $p$ such that
    \begin{equation*}
        \omega = \sum_{i=1}^{n} \edv x^i \wedge \edv y^i.
    \end{equation*}
\end{theorem}
\begin{proof}
    Taking a coordinate chart on $X$, it is enough to assume that $\omega_0$, $\omega_1$ are symplectic forms on a neighbourhood of $0$ in $\C^n$ and to prove that there exists a local diffeomorphism $\phi$ of $\C^n$ such that $\phi^* \omega_1 = \omega_0$. Choosing the initial coordinate chart carefully, we may assume in addition that $\omega_j$ has a primitive $\lambda_j$, i.e. $\omega_j = \edv \lambda_j$, for $j = 0, 1$, and also that $\omega_0$ and $\omega_1$ are equal at $0 \in \C$, i.e. $\omega_0|_0 = \omega_1|_0$. Let
    \begin{IEEEeqnarray*}{rCls+x*}
        \lambda_t & \coloneqq & \lambda_0 + t (\lambda_1 - \lambda_0), \\
        \omega_t  & \coloneqq & \edv \omega_t = \omega_0 + t (\omega_1 - \omega_0).
    \end{IEEEeqnarray*}
    Since $\omega_t|_0 = \omega_0|_0$ is symplectic, possibly after passing to a smaller neighbourhood of $0$ we may assume that $\omega_t$ is symplectic. Let $Y_t$ be the unique time-dependent vector field such that $\dot{\lambda}_t + \iota_{Y_t} \omega_t = 0$ and denote by $\phi_t$ the flow of $Y_t$. Then,
    \begin{IEEEeqnarray*}{rCls+x*}
        \phi^*_t \omega_t^{} - \omega_0^{}
        & = & \int_{0}^{t} \phi^*_s \p{}{}{ \dot{\omega}_s + \edv \iota _{Y_s} \omega_s + \iota _{Y_s} \edv \omega_s } \edv s & \quad [\text{by Moser's trick (\cref{lem:mosers trick})}] \\
        & = & \int_{0}^{t} \phi^*_s \edv \p{}{}{ \dot{\lambda}_s + \edv \iota _{Y_s} \omega_s } \edv s                        & \quad [\text{since $\omega_t = \edv \lambda_t$}] \\
        & = & 0                                                                                                               & \quad [\text{by definition of $Y_t$}],
    \end{IEEEeqnarray*}
    which shows that $\phi_1$ is the desired local diffeomorphism.
\end{proof}

\begin{definition}
    \label{def:liouville vf}
    If $(X,\lambda)$ is an exact symplectic manifold, then the \textbf{Liouville vector field} of $(X,\lambda)$ is the unique vector field $Z$ such that
    \begin{IEEEeqnarray*}{c}
        \lambda = \iota_Z \omega.
    \end{IEEEeqnarray*}
\end{definition}

\begin{lemma}
    \label{lem:liouville vf}
    The Liouville vector field satisfies
    \begin{IEEEeqnarray*}{c}
        \ldv{Z} \lambda = \lambda.
    \end{IEEEeqnarray*}
\end{lemma}
\begin{proof}
    \begin{IEEEeqnarray*}{rCls+x*}
        \ldv{Z} \lambda
        & = & \edv \iota_Z \lambda + \iota_Z \edv \lambda & \quad [\text{by the Cartan magic formula}]     \\
        & = & \edv \iota_Z \lambda + \iota_Z \omega       & \quad [\text{since $\omega = \edv \lambda$}]     \\
        & = & \edv \iota_Z \iota_Z \omega + \lambda       & \quad [\text{by definition of Liouville vector field, $\lambda = \iota_Z \omega$}] \\
        & = & \lambda                                     & \quad [\text{since $\omega$ is antisymmetric, $\iota_Z \iota_Z \omega = 0$}].        & \qedhere
    \end{IEEEeqnarray*}
\end{proof}

\begin{definition}
    \label{def:Hamiltonian v field}
    Let $H \in C^\infty(X,\R)$ be a function on $X$. The \textbf{Hamiltonian vector field} of $H$, denoted $X_H$, is the unique vector field on $X$ satisfying 
    \begin{IEEEeqnarray*}{c}
        \edv H = -\iota _{X_H} \omega.
    \end{IEEEeqnarray*}
\end{definition}

\begin{proposition}
    \phantomsection\label{lem:hamiltonian vector field preserves symplectic form}
    The Hamiltonian vector field preserves the symplectic form, i.e.
    \begin{IEEEeqnarray*}{c}
        \ldv{X_H} \omega = 0.
    \end{IEEEeqnarray*}
\end{proposition}
\begin{proof}
    \begin{IEEEeqnarray*}{rCls+x*}
        \ldv{X_H} \omega
        & = & \edv \iota_{X_H} \omega + \iota_{X_H} \edv \omega & \quad [\text{by the Cartan magic formula}] \\
        & = & \edv \iota_{X_H} \omega                           & \quad [\text{since $\omega$ is closed}] \\
        & = & - \edv^2 H                                        & \quad [\text{by definition of $X_H$}] \\
        & = & 0                                                 & \quad [\text{since $\edv^2 = 0$}].           & \qedhere
    \end{IEEEeqnarray*}
\end{proof}

\begin{proposition}[Liouville's theorem]
    The Hamiltonian vector field preserves the symplectic volume form, i.e.
    \begin{equation*}
        \ldv{X_H} \p{}{2}{\frac{\omega^n}{n!}} = 0.
    \end{equation*}
\end{proposition}
\begin{proof}
    By \cref{lem:hamiltonian vector field preserves symplectic form} and the fact that Lie derivatives obey the Leibniz rule.
\end{proof}

\begin{proposition}[conservation of energy]
    \label{lem:conservation of energy}
    The Hamiltonian is constant along the Hamiltonian vector field, i.e.
    \begin{IEEEeqnarray*}{c}
        X_H(H) = 0.
    \end{IEEEeqnarray*}
\end{proposition}
\begin{proof}
    \begin{IEEEeqnarray*}{rCls+x*}
        X_H(H)
        & = & \edv H(X_H)                & \quad [\text{by definition of exterior derivative}] \\
        & = & - \iota_{X_H} \omega (X_H) & \quad [\text{by definition of $X_H$}] \\
        & = & - \omega(X_H, X_H)         & \quad [\text{by definition of interior product}] \\
        & = & 0                          & \quad [\text{since $\omega$ is a form}].              & \qedhere
    \end{IEEEeqnarray*}
\end{proof}

\section{Contact manifolds}

In this section, we recall some basics about contact manifolds.

\begin{definition}
    \label{def:contact manifold}
    A \textbf{contact manifold} is a pair $(M,\xi)$, where $M$ is a smooth manifold and $\xi$ is a distribution on $M$ of codimension 1, called the \textbf{contact structure}, such that for all locally defining forms $\alpha \in \Omega^1(U)$ for $\xi$ (i.e. such that $\xi = \ker \alpha$), $\edv \alpha |_{\xi}$ is nondegenerate. In this case we call $\alpha$ a \textbf{local contact form} for $M$. In the case where $\alpha \in \Omega^1(M)$ we say that $\alpha$ is a \textbf{global contact form} for $M$. A \textbf{strict contact manifold} is a pair $(M,\alpha)$ such that $(M,\ker \alpha)$ is a contact manifold.
\end{definition}

The following lemma characterizes the linear algebra of contact manifolds.

\begin{lemma}
    \label{lem:contact manifold}
    Let $M$ be an $m$-dimensional manifold, $\alpha \in \Omega^1(M)$ be nonvanishing and $\xi = \ker \alpha$. Then, the following are equivalent:
    \begin{enumerate}
        \item \label{lem:contact manifold 1} The form $\edv \alpha |_{\xi}$ is nondegenerate, i.e. $(M,\alpha)$ is a contact manifold;
        \item \label{lem:contact manifold 3} The tangent bundle of $M$ decomposes as $T M = \ker \edv \alpha \directsum \ker \alpha$;
        \item \label{lem:contact manifold 2} There exists an $n \in \Z_{\geq 0}$ such that $m = 2n + 1$ and $\alpha \wedge (\edv \alpha)^{n}$ is a volume form.
    \end{enumerate}
\end{lemma}
\begin{proof}
    {\ref{lem:contact manifold 1}} $\Longrightarrow$ {\ref{lem:contact manifold 3}}: We show that $\ker \edv \alpha \cap \ker \alpha = 0$. For this, it suffices to assume that $v \in \ker \edv \alpha \cap \ker \alpha$ and to prove that $v = 0$. Since $\edv \alpha|_{\ker \alpha}(v) = 0$ and $\edv \alpha|_{\ker \alpha}$ is nondegenerate we conclude that $v = 0$.

    We show that $\dim TM = \dim \ker \edv \alpha + \dim \ker \alpha$. Since $\alpha$ is nonvanishing, $\dim \ker \alpha = \dim TM - 1$. Since $\ker \edv \alpha \cap \ker \alpha = 0$, this implies that $\dim \ker \edv \alpha \in \{0,1\}$. Considering that $\edv \alpha|_{\ker \alpha}$ is nondegenerate and that $\dim TM = \dim \ker \alpha + 1$, we conclude that $\edv \alpha|_{TM}$ is degenerate. Therefore, $\dim \ker \edv \alpha = 1$.    

    {\ref{lem:contact manifold 3}} $\Longrightarrow$ {\ref{lem:contact manifold 2}}: Since $T M = \ker \edv \alpha \oplus \ker \alpha$, we conclude that the forms $\alpha|_{\ker \edv \alpha}$ and $\edv \alpha|_{\ker \alpha}$ are nondegenerate. In particular, $\ker \alpha$ is even dimensional, i.e. $\dim \ker \alpha = 2n$ for some $n$, and $(\edv \alpha|_{\ker \alpha})^n$ is a volume form on $\ker \alpha$. So, $\alpha \wedge (\edv \alpha)^n$ is a volume form on $M$.

    {\ref{lem:contact manifold 2}} $\Longrightarrow$ {\ref{lem:contact manifold 1}}: If $v \in \xi = \ker \alpha$ is such that $v \in \ker \edv \alpha|_{\xi}$, then $\iota_v (\alpha \wedge (\edv \alpha)^n) = 0$, which implies that $v = 0$.
\end{proof}

\begin{definition}
    Let $(M,\xi_M)$, $(N,\xi_N)$ be contact manifolds. A \textbf{contactomorphism} from $M$ to $N$ is a diffeomorphism $\phi \colon M \longrightarrow N$ such that $T \phi(\xi_M) = \xi_N$. If $(M,\alpha_M)$, $(N,\alpha_N)$ are strict contact manifolds, a \textbf{strict contactomorphism} from $M$ to $N$ is a diffeomorphism $\phi \colon M \longrightarrow N$ such that $\phi^* \alpha_N = \alpha_M$.
\end{definition}

\begin{remark}
    We will consider only strict contact manifolds and strict contactomorphisms, and for simplicity we will drop the word ``strict'' from our nomenclature.
\end{remark}

\begin{definition}
    \label{def:Reeb vector field}
    The \textbf{Reeb vector field} of $(M,\alpha)$ is the unique vector field $R$ satisfying%
    \begin{IEEEeqnarray*}{rCls+x*}
        \iota_R \edv \alpha & = & 0, \\
        \iota_R \alpha      & = & 1.
    \end{IEEEeqnarray*}
\end{definition}

\begin{remark}
    \cref{lem:contact manifold} {\ref{lem:contact manifold 3}} can also be written as $TM = \p{<}{}{R} \directsum \xi$.
\end{remark}

\begin{lemma}
    \label{lem:reeb vf preserves contact form}
    The Reeb vector field preserves the contact form, i.e.
    \begin{IEEEeqnarray*}{c+x*}
        \ldv{R} \alpha = 0.
    \end{IEEEeqnarray*}
\end{lemma}
\begin{proof}
    \begin{IEEEeqnarray*}{rCls+x*}
        \ldv{R} \alpha
        & = & \iota _{R} \edv \alpha + \edv \iota _{R} \alpha & \quad [\text{by the Cartan magic formula}] \\
        & = & 0 + \edv 1                                      & \quad [\text{by definition of $R$}]      \\
        & = & 0.                                              & \quad                                  & \qedhere
    \end{IEEEeqnarray*}
\end{proof}

We now consider contact manifolds which are hypersurfaces of symplectic manifolds.

\begin{definition}
    \label{def:hypersurface of contact type}
    Let $(X,\omega)$ be a symplectic manifold of dimension $2n$, $(M, \alpha)$ be a contact manifold of dimension $2n - 1$ such that $M \subset X$, and denote by $\iota \colon M \longrightarrow X$ the inclusion. We say that $M$ is a \textbf{hypersurface of contact type} if $\edv \alpha = \iota^* \omega$. In this case, the \textbf{Liouville vector field} is the unique vector field $Z \in C^{\infty}(\iota^* TX)$ such that
    \begin{IEEEeqnarray*}{c+x*}
        \iota_Z \omega = \alpha.
    \end{IEEEeqnarray*}
\end{definition}

\begin{example}
    Let $(L,g)$ be a Riemannian manifold. Recall that $(T^*L, \lambda)$ is an exact symplectic manifold. Consider the \textbf{unit cotangent bundle}
    \begin{IEEEeqnarray*}{c+x*}
        S^* L \coloneqq \{ u \in T^* L \mid \| u \| = 1 \}.
    \end{IEEEeqnarray*}
    The form $\alpha \coloneqq \lambda|_{S^*L}$ is a contact form on $S^* L$. Therefore, $(S^*L, \alpha) \subset (T^*L, \lambda)$ is a hypersurface of contact type. More generally, we can also define the cotangent bundle of radius $r > 0$ by $S^*_r L \coloneqq \{ u \in T^* L \mid \| u \| = r \}$, which is also a hypersurface of contact type.
\end{example}

\begin{lemma}
    \label{lem:decomposition coming from contact hypersurface}
    We have the decompositions
    \begin{IEEEeqnarray*}{rCls+x*}
        \iota^* TX & = & \p{<}{}{Z} \directsum \p{<}{}{R} \directsum \xi, \\
        TM         & = & \p{<}{}{R} \directsum \xi, \\
        \xi^\perp  & = & \p{<}{}{Z} \directsum \p{<}{}{R}.
    \end{IEEEeqnarray*}
\end{lemma}
\begin{proof}
    By \cref{lem:contact manifold}, we have that $TM = \p{<}{}{R} \directsum \xi$. To show that $\xi^\perp = \p{<}{}{Z} \directsum \p{<}{}{R}$, by considering the rank of the vector bundles it suffices to show that $\p{<}{}{Z} \directsum \p{<}{}{R} \subset \xi^\perp$. Let $v \in \xi_p = \ker \alpha_p$. We wish to show that $\omega(Z_p, v) = 0$ and $\omega(R_p, v) = 0$.
    \begin{IEEEeqnarray*}{rCls+x*}
        \omega(Z_p, v)
        & = & \alpha(v)           & \quad [\text{by definition of $Z$}] \\
        & = & 0                   & \quad [\text{since $v \in \ker \alpha_p$}], \\ \\
        \omega(R_p, v)
        & = & \edv \alpha(R_p, v) & \quad [\text{by definition of hypersurface of contact type}] \\
        & = & 0                   & \quad [\text{by definition of Reeb vector field}].
    \end{IEEEeqnarray*}
    Then, as oriented vector bundles, $\iota^* TX = \xi^\perp \directsum \xi = \p{<}{}{Z} \directsum \p{<}{}{R} \directsum \xi$.
\end{proof}

\begin{lemma}
    \label{lem:HR flow}
    Let $H \colon X \longrightarrow \R$ and assume that $M$ is the preimage of $H$ under a regular value $c \in \R$, i.e. $M = H^{-1}(c)$. Then, there exists a unique vector field $X_H^M$ on $M$ which is $\iota$-related to $X_H$. In addition, $X_H^M = \alpha(X_H^M) R$.
\end{lemma}
\begin{proof}
    To prove the first statement, it suffices to show that $X_H|_p \in T_p M$ for every $p \in M$. By conservation of energy (\cref{lem:conservation of energy}), we have that
    \begin{IEEEeqnarray*}{rCls+x*}
        X_H|_p
        & \in & \ker \edv H(p) \\
        & = & T_p (H ^{-1}(c)) \\
        & = & T_p M.
    \end{IEEEeqnarray*}
    We now show that $\iota_{X_H^M} \edv \alpha = 0$.
    \begin{IEEEeqnarray*}{rCls+x*}
        \iota _{X_H^ M} \edv \alpha
        & = & \iota _{X_H^ M} \iota^* \omega & \quad [\text{by definition of hypersurface of contact type}]                        \\
        & = & \iota^* \iota _{X_H} \omega    & \quad [\text{since $X_H^M$ is $\iota$-related to $X_H$}]         \\
        & = & - \iota^* \edv H               & \quad [\text{by definition of Hamiltonian vector field}]      \\
        & = & - \edv \iota^* H               & \quad [\text{by naturality of $\edv$}]                        \\
        & = & 0                              & \quad [\text{since $H$ is constant equal to $c$ on $M$}].
    \end{IEEEeqnarray*}
    By definition of Reeb vector field, we conclude that $X_H^M$ and $R$ are collinear, and in particular $X_H^M = \alpha(X_H^M) R$.
\end{proof}

We now compare the dynamics from the points of view of Riemannian, symplectic and contact geometry. Let $(L,g)$ be a Riemannian manifold of dimension $n$. The manifold $L$ has a tangent bundle $TL$ and a cotangent bundle $T^*L$, and the map $\tilde{g} \colon TL \longrightarrow T^*L$ given by $\tilde{g}(v) = g(v,\cdot)$ is a vector bundle isomorphism. Consider the unit cotangent bundle $\iota \colon S^*L \longrightarrow T^*L$, which has a Reeb vector field $R$, and the function
\begin{IEEEeqnarray*}{rrCl}
    H \colon & T^*L & \longrightarrow & \R \\
             & u    & \longmapsto     & \frac{1}{2} \p{||}{}{u}_{}^2.
\end{IEEEeqnarray*}

\begin{definition}
    We define a vector field $G$ on $TL$, called the \textbf{geodesic field}, as follows. At $v \in TL$, $G _{v}$ is given by
    \begin{equation*}
        G _{v} \coloneqq \odv{}{t}\Big|_{t=0} \dot{\gamma}(t),
    \end{equation*}
    where $\gamma \colon I \longrightarrow L$ is the unique geodesic with $\dot{\gamma}(0) = v$ and $\dot{\gamma} \colon I \longrightarrow TL$ is the lift of $\gamma$.
\end{definition}

A curve $\gamma$ in $L$ is a geodesic if and only if its lift $\dot{\gamma}$ to $TL$ is a flow line of $G$.

\begin{theorem}
    \label{thm:flow geodesic vs hamiltonian}
    The vector field $G$ is $\tilde{g}$-related to $X_H$.
\end{theorem}
\begin{proof}
    See for example \cite[Theorem 1.5.2]{geigesIntroductionContactTopology2008} or \cite[Theorem 2.3.1]{frauenfelderRestrictedThreeBodyProblem2018}.
\end{proof}

\begin{theorem}
    \label{thm:flow reeb vs hamiltonian}
    The vector field $R$ is $\iota$-related to $X_H$.
\end{theorem}
\begin{proof}
    Notice that $S^*L = H^{-1}(2)$. By \cref{lem:HR flow}, it suffices to show that $\lambda(X_H) \circ \iota = 1$. Let $(q^1, \ldots, q^n)$ be coordinates on $L$, with induced coordinates $(q^1, \ldots, q^n, p_1, \ldots, p_n)$ on $T^* L$. With respect to these coordinates, $X_H$ can be written as
    \begin{IEEEeqnarray}{rCls+x*}
        X_H
        & = & \sum_{i = 1}^{n} \p{}{2}{ \pdv{H}{p_i} \pdv{}{q^i} - \pdv{H}{q^i} \pdv{}{p_i} } \IEEEnonumber \\
        & = & \sum_{i = 1}^{n} \p{}{2}{ \sum_{j=1}^{n} g^{ij} p_j \pdv{}{q^i} - \sum_{j,k=1}^{n} \pdv{g^{jk}}{q^i} p_j p_k \pdv{}{p_i} }. \plabel{eq:hamiltonian vector field in coordinates}
    \end{IEEEeqnarray}
    We show that $\p{<}{}{\dv \pi(u) X_H|_{u}, \cdot } = u$.
    \begin{IEEEeqnarray*}{rCls+x*}
        \p{<}{}{\dv \pi (u) X_{H}|_{u}, v}
        & = & \sum_{i,j=1}^{n} g _{ij} (\dv \pi (u) X_{H}|_{u})^i v^j \\
        & = & \sum_{i,j,k=1}^{n} g _{ij} g ^{ik} p_k v^j \\
        & = & \sum_{j,k=1}^{n} \delta^k_j p_k v^j \\
        & = & \sum_{j=1}^{n} p_j v^j \\
        & = & \sum_{i=1}^{n} p_i \edv q^i \p{}{2}{ \sum_{j=1}^{n} v^j \pdv{}{q^j} } \\
        & = & u(v).
    \end{IEEEeqnarray*}
    We show that $\lambda(X_H) = 2 H$:
    \begin{IEEEeqnarray*}{rCls+x*}
        \lambda(X_{H})|_{u}
        & = & u (\dv \pi (u) X_{H}|_{u})                               & \quad [\text{by definition of $\lambda$}]                 \\
        & = & \p{<}{}{ \dv \pi (u) X_{H}|_{u},\dv \pi (u) X_{H}|_{u} } & \quad [\text{since $u = \p{<}{}{\dv \pi(u) X_H|_{u}, \cdot }$}] \\
        & = & \p{||}{}{ \dv \pi (u) X_{H}|_{u} }^2                     & \quad [\text{by definition of the norm}]     \\
        & = & \p{||}{}{u}^2                                            & \quad [\text{since $u = \p{<}{}{\dv \pi(u) X_H|_{u}, \cdot }$}] \\
        & = & 2 H (u)                                                  & \quad [\text{by definition of $H$}].
    \end{IEEEeqnarray*}
    By definition of $H$, this implies that $\lambda(X_H) \circ \iota = 1$, as desired.
\end{proof}

\section{Liouville domains}

In this section we introduce Liouville domains, which are going to be the main type of symplectic manifold we will work with.

\begin{definition}
    \label{def:liouville domain}
    A \textbf{Liouville domain} is a pair $(X,\lambda)$, where $X$ is a compact, connected smooth manifold with boundary $\del X$ and $\lambda \in \Omega^1(X)$ is such that $\edv \lambda \in \Omega^2(X)$ is symplectic, $\lambda|_{\del X}$ is contact and the orientations on $\del X$ coming from $(X,\edv \lambda)$ and coming from $\lambda|_{\del X}$ are equal.
\end{definition}

\begin{example}
    Let $(L,g)$ be a Riemannian manifold. The \textbf{unit codisk bundle},
    \begin{IEEEeqnarray*}{c+x*}
        D^* L \coloneqq \{ u \in T^*L \mid \| u \| \leq 1 \},
    \end{IEEEeqnarray*}
    is a Liouville domain. More generally, we can define the codisk bundle of radius $r > 0$ by $D^*_r L \coloneqq \{ u \in T^*L \mid \| u \| \leq r \}$, which is also a Liouville domain.
\end{example}


\begin{definition}
    \label{def:star shaped}
    A \textbf{star-shaped domain} is a compact, connected $2n$-dimensional submanifold $X$ of $\C^{n}$ with boundary $\del X$ such that $(X,\lambda)$ is a Liouville domain, where $\lambda$ is the symplectic potential of \cref{exa:cn symplectic}.
\end{definition}

\begin{definition}
    \label{def:moment map}
    The \textbf{moment map} is the map $\mu \colon \C^n \longrightarrow \R^n _{\geq 0}$ given by 
    \begin{IEEEeqnarray*}{c+x*}
        \mu(z_1,\ldots,z_n) \coloneqq \pi(|z_1|^2,\ldots,|z_n|^2).   
    \end{IEEEeqnarray*}
    Define also
    \begin{IEEEeqnarray*}{rCrClClrCl}
        \Omega_X        & \coloneqq & \Omega(X)      & \coloneqq & \hphantom{{}^{-1}} \mu(X) \subset \R_{\geq 0}^n, & \qquad & \text{for every } & X      & \subset & \C^n, \\
        X_{\Omega}      & \coloneqq & X(\Omega)      & \coloneqq & \mu^{-1}(\Omega) \subset \C^n,                   & \qquad & \text{for every } & \Omega & \subset & \R^{n}_{\geq 0}, \\
        \delta_{\Omega} & \coloneqq & \delta(\Omega) & \coloneqq & \sup \{ a \mid (a, \ldots, a) \in \Omega \},     & \qquad & \text{for every } & \Omega & \subset & \R^{n}_{\geq 0}.
    \end{IEEEeqnarray*}
    We call $\delta_\Omega$ the \textbf{diagonal} of $\Omega$.
\end{definition}

\begin{definition}
    \label{def:toric domain}
    A \textbf{toric domain} is a star-shaped domain $X$ such that $X = X(\Omega(X))$. A toric domain $X = X _{\Omega}$ is 
    \begin{enumerate}
        \item \textbf{convex} if $\hat{\Omega} \coloneqq \{ (x_1, \ldots, x_n) \in \R^n \mid (|x_1|,\ldots,|x_n|) \in \Omega \} $ is convex;
        \item \textbf{concave} if $\R^n _{\geq 0} \setminus \Omega$ is convex.
    \end{enumerate}
\end{definition}

\begin{example}
    \phantomsection\label{exa:toric domains}
    Here we give some examples of toric domains. See \cref{fig:Toric domains} for a picture of the examples given below.
    \begin{enumerate}
        \item The \textbf{ellipsoid} is the convex and concave toric domain given by
            \begin{IEEEeqnarray*}{rCls+x*}
                E(a_1,\ldots,a_n)        & \coloneqq & \p{c}{2}{ (z_1,\ldots,z_n) \in \C^n           \ \Big| \ \sum_{j=1}^{n} \frac{\pi |z_j|^2}{a_j} \leq 1 } \\
                \Omega_E(a_1,\ldots,a_n) & \coloneqq & \p{c}{2}{ (x_1,\ldots,x_n) \in \R^n _{\geq 0} \ \Big| \ \sum_{j=1}^{n} \frac{x_j}{a_j} \leq 1 }.
            \end{IEEEeqnarray*}
            Its limit shape, the \textbf{ball}, is $B^{2n}(a) \coloneqq B(a) \coloneqq E(a,\ldots,a)$.
        \item The \textbf{polydisk} is the convex ``toric domain with corners'' given by
            \begin{IEEEeqnarray*}{rCls+x*}
                P(a_1,\ldots,a_n)        & \coloneqq & \p{c}{2}{ (z_1,\ldots,z_n) \in \C^n           \ \Big| \ \forall j=1,\ldots,n \colon \frac{\pi |z_j|^2}{a_j} \leq 1 } \\
                \Omega_P(a_1,\ldots,a_n) & \coloneqq & \p{c}{2}{ (x_1,\ldots,x_n) \in \R^n _{\geq 0} \ \Big| \ \forall j=1,\ldots,n \colon \frac{x_j}{a_j}         \leq 1 }.
            \end{IEEEeqnarray*}
            Its limit shape, the \textbf{cube}, is $P^{2n}(a) \coloneqq P(a) \coloneqq P(a,\ldots,a)$.
        \item The \textbf{nondisjoint union of cylinders} is the concave ``noncompact toric domain with corners'' given by
            \begin{IEEEeqnarray*}{rCls+x*}
                N(a_1,\ldots,a_n)        & \coloneqq & \p{c}{2}{ (z_1,\ldots,z_n) \in \C^n           \ \Big| \ \exists j=1,\ldots,n \colon \frac{\pi |z_j|^2}{a_j} \leq 1 } \\
                \Omega_N(a_1,\ldots,a_n) & \coloneqq & \p{c}{2}{ (x_1,\ldots,x_n) \in \R^n _{\geq 0} \ \Big| \ \exists j=1,\ldots,n \colon \frac{x_j}{a_j}         \leq 1 }.
            \end{IEEEeqnarray*}
            Its limit shape is denoted $N^{2n}(a) \coloneqq N(a) \coloneqq N(a,\ldots,a)$.
        \item The \textbf{cylinder} is the convex and concave ``noncompact toric domain'' given by
            \begin{IEEEeqnarray*}{rCls+x*}
                Z(a)        & \coloneqq & \p{c}{2}{ (z_1,\ldots,z_n) \in \C^n           \ \Big| \ \frac{\pi |z_1|^2}{a_1} \leq 1 } \\
                \Omega_Z(a) & \coloneqq & \p{c}{2}{ (x_1,\ldots,x_n) \in \R^n _{\geq 0} \ \Big| \ \frac{x_1}{a_1} \leq 1 }.
            \end{IEEEeqnarray*}
            Note that $Z^{2n}(a) \coloneqq Z(a) = E(a,\infty,\ldots,\infty) = P(a,\infty,\ldots,\infty)$.
    \end{enumerate}
\end{example}

\begin{figure}[ht]
    \centering

    \begin{tikzpicture}
        [
        nn/.style={thick, color = gray},
        zz/.style={thick, color = gray},
        pp/.style={thick, color = gray},
        bb/.style={thick, color = gray}
        ]
        \tikzmath{
            \x = 1.5;
            \y = 3;
            \z = 1.0;
            coordinate \o, \a, \b, \c, \d, \e, \r, \s, \q;
            \o{ball} = (0 , 0 ) + 0*(\y+\z,0);
            \a{ball} = (\x, 0 ) + 0*(\y+\z,0);
            \b{ball} = (0 , \x) + 0*(\y+\z,0);
            \c{ball} = (\x, \x) + 0*(\y+\z,0);
            \d{ball} = (\x, \y) + 0*(\y+\z,0);
            \e{ball} = (\y, \x) + 0*(\y+\z,0);
            \r{ball} = (\y, 0 ) + 0*(\y+\z,0);
            \s{ball} = (0 , \y) + 0*(\y+\z,0);
            \q{ball} = (\y, \y) + 0*(\y+\z,0);
            \o{cube} = (0 , 0 ) + 1*(\y+\z,0);
            \a{cube} = (\x, 0 ) + 1*(\y+\z,0);
            \b{cube} = (0 , \x) + 1*(\y+\z,0);
            \c{cube} = (\x, \x) + 1*(\y+\z,0);
            \d{cube} = (\x, \y) + 1*(\y+\z,0);
            \e{cube} = (\y, \x) + 1*(\y+\z,0);
            \r{cube} = (\y, 0 ) + 1*(\y+\z,0);
            \s{cube} = (0 , \y) + 1*(\y+\z,0);
            \q{cube} = (\y, \y) + 1*(\y+\z,0);
            \o{cyld} = (0 , 0 ) + 2*(\y+\z,0);
            \a{cyld} = (\x, 0 ) + 2*(\y+\z,0);
            \b{cyld} = (0 , \x) + 2*(\y+\z,0);
            \c{cyld} = (\x, \x) + 2*(\y+\z,0);
            \d{cyld} = (\x, \y) + 2*(\y+\z,0);
            \e{cyld} = (\y, \x) + 2*(\y+\z,0);
            \r{cyld} = (\y, 0 ) + 2*(\y+\z,0);
            \s{cyld} = (0 , \y) + 2*(\y+\z,0);
            \q{cyld} = (\y, \y) + 2*(\y+\z,0);
            \o{ndju} = (0 , 0 ) + 3*(\y+\z,0);
            \a{ndju} = (\x, 0 ) + 3*(\y+\z,0);
            \b{ndju} = (0 , \x) + 3*(\y+\z,0);
            \c{ndju} = (\x, \x) + 3*(\y+\z,0);
            \d{ndju} = (\x, \y) + 3*(\y+\z,0);
            \e{ndju} = (\y, \x) + 3*(\y+\z,0);
            \r{ndju} = (\y, 0 ) + 3*(\y+\z,0);
            \s{ndju} = (0 , \y) + 3*(\y+\z,0);
            \q{ndju} = (\y, \y) + 3*(\y+\z,0);
        }
        \foreach \domain in {ball, cube, cyld, ndju}{
            \draw[->] (\o{\domain}) -- (\r{\domain});
            \draw[->] (\o{\domain}) -- (\s{\domain});
            \node[anchor = north] at (\a{\domain}) {$1$};
            \node[anchor =  east] at (\b{\domain}) {$1$};
        }

        \node[anchor = north east] at (\q{ball}) {$\Omega_B(1)$};
        \fill[bb, opacity=0.5] (\o{ball}) -- (\a{ball}) -- (\b{ball}) -- cycle;
        \draw[bb] (\o{ball}) -- (\a{ball}) -- (\b{ball}) -- cycle;

        \node[anchor = north east] at (\q{cube}) {$\Omega_P(1)$};
        \fill[pp, opacity=0.5]  (\o{cube}) -- (\a{cube}) -- (\c{cube}) -- (\b{cube}) -- cycle;
        \draw[pp] (\o{cube}) -- (\a{cube}) -- (\c{cube}) -- (\b{cube}) -- cycle;
        
        \node[anchor = north east] at (\q{cyld}) {$\Omega_Z(1)$};
        \fill[zz, opacity=0.5] (\o{cyld}) -- (\a{cyld}) -- (\d{cyld}) -- (\s{cyld});
        \draw[zz] (\s{cyld}) -- (\o{cyld}) -- (\a{cyld}) -- (\d{cyld});
        
        \node[anchor = north east] at (\q{ndju}) {$\Omega_N(1)$};
        \fill[nn, opacity=0.5] (\o{ndju}) -- (\s{ndju}) -- (\d{ndju}) -- (\c{ndju}) -- (\e{ndju}) -- (\r{ndju}) -- cycle;
        \draw[nn] (\d{ndju}) -- (\c{ndju}) -- (\e{ndju});
        \draw[nn] (\s{ndju}) -- (\o{ndju}) -- (\r{ndju});
    \end{tikzpicture}

    \caption{Toric domains}
    \label{fig:Toric domains}
\end{figure}
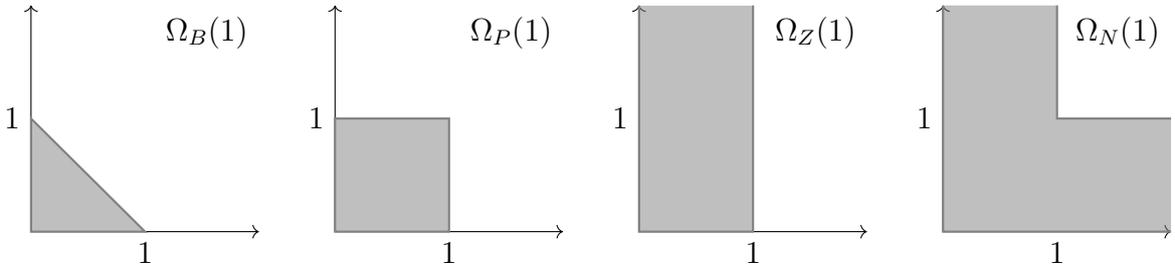

\section{Symplectization of a contact manifold}

Let $(M,\alpha)$ be a contact $(2n - 1)$-dimensional manifold.

\begin{definition}
    \label{def:symplectization}
    The \textbf{symplectization} of $(M,\alpha)$ is the exact symplectic manifold $(\R \times M, e^r \alpha)$, where $r$ is the coordinate on $\R$.
\end{definition}

\begin{lemma}
    \label{lem:symplectization form}
    The form $\edv (e^r \alpha)$ is symplectic.
\end{lemma}
\begin{proof}
    The form $\edv (e^r \alpha)$ is exact, so it is closed. We show that $\edv (e^r \alpha)$ is nondegenerate.
    \begin{IEEEeqnarray*}{rCls+x*}
        \IEEEeqnarraymulticol{3}{l}{( \edv (e^r \alpha) )^n}\\ \quad
        & =    & ( e^r \edv r \wedge \alpha + e^r \edv \alpha )^n                                        & \quad [\text{by the Leibniz rule}]                                          \\
        & =    & e^{nr} \sum_{k=0}^{n} \binom{n}{k} ( \edv r \wedge \alpha)^k \wedge (\edv \alpha)^{n-k} & \quad [\text{by the binomial theorem}]                                      \\
        & =    & e^{n r} \edv r \wedge \alpha \wedge (\edv \alpha)^{n-1}                                 & \quad [\text{since $\alpha^2 = 0$ and $(\edv \alpha)^n = 0$}]               \\
        & \neq & 0                                                                                       & \quad [\text{since $\alpha \wedge (\edv \alpha)^{n-1}$ is a volume form on $M$}]. & \qedhere
    \end{IEEEeqnarray*}
\end{proof}

\begin{lemma}
    \label{lem:symplectization lvf}
    The Liouville vector field of $(\R \times M, e^r \alpha)$ is $Z = \partial_r$.
\end{lemma}
\begin{proof}
    By definition of Liouville vector field, we need to show that $\iota_{\partial_r} \edv (e^r \alpha) = e^r \alpha$.
    \begin{IEEEeqnarray*}{rCls+x*}
        \iota_{\partial_r} \edv (e^r \alpha)
        & = & \iota_{\partial_r} (e^r \edv r \wedge \alpha + e^r \edv \alpha)                               & \quad [\text{by the Leibniz rule}]              \\
        & = & e^r (\edv r (\partial_r) \alpha - \alpha(\partial_r) \edv r + \iota_{\partial_r} \edv \alpha) & \quad [\text{since $\iota_Z$ is a derivation}]  \\
        & = & e^r \alpha                                                                                    & \quad [\text{since $\alpha$ is a form on $M$}].   & \qedhere
    \end{IEEEeqnarray*}
\end{proof}

\begin{example}
    Let $(L,g)$ be a Riemannian manifold. Recall that $(T^*L,\lambda)$ is an exact symplectic manifold and that $(S^*L, \alpha)$ is a hypersurface of contact type. Consider the symplectization of $S^*L$, which is $(\R \times S^*L, e^r \alpha)$. Then, the map $\R \times S^*L \longrightarrow T^*L \setminus L$ given by $(r,u) \longmapsto e^r u$ is a Liouville diffeomorphism.
\end{example}

Defining $R_{(r,x)} = R_x$ we can view the Reeb vector field of $M$ as a vector field in $\R \times M$. Analogously, we define a distribution $\xi$ on $\R \times M$ by $\xi_{(r,x)} = \xi_x$. Then, $T(\R \times M) = \p{<}{}{Z} \directsum \p{<}{}{R} \directsum \xi$. Let $H \colon \R \times M \longrightarrow \R$ be a function which only depends on $\R$, (i.e. $H(r,x) = H(r)$). Define $h \coloneqq H \circ \exp^{-1} \colon \R_{> 0} \longrightarrow \R$ and $T(r) \coloneqq H'(r) / e^r = h'(e^r)$.

\begin{lemma}
    \label{lem:reeb equals hamiltonian on symplectization}
    The Hamiltonian vector field of $H$ satisfies $\alpha(X_H) = T$ and $X_H = T R$.
\end{lemma}
\begin{proof}
    By \cref{lem:HR flow}, $X_H$ and $R$ are collinear. By definition of Reeb vector field, this implies that $X_H = \alpha(X_H) R$. It remains to show that $\alpha(X_H) = T$. For this, we compute%
    \begin{IEEEeqnarray*}{rCls+x*}
        H' \edv r
        & = & \edv H                                                                     & \quad [\text{by definition of exterior derivative}] \\
        & = & - \iota _{X_H} \edv (e^r \alpha)                                           & \quad [\text{by definition of Hamiltonian v.f.}]    \\
        & = & - \iota _{X_H} (e^r \edv r \wedge \alpha + e^r \edv \alpha)                & \quad [\text{Leibniz rule for exterior derivative}] \\
        & = & - e^r (\edv r(X_H) \alpha - \alpha(X_H) \edv r + \iota _{X_H} \edv \alpha) & \quad [\text{interior product is a derivation}].
    \end{IEEEeqnarray*}
    Therefore, $H' \edv r = e^r \alpha(X_H) \edv r$, which implies that $\alpha(X_H) = H'/\exp = T$.
\end{proof}

\begin{corollary}
    \phantomsection\label{cor:hamiltonian orbits are reeb orbits}
    Suppose that $\gamma = (r,\rho) \colon S^1 \longrightarrow \R \times M$ is a $1$-periodic orbit of $X_H$, i.e. $\dot{\gamma}(t) = X_H(\gamma(t))$. Then: 
    \begin{enumerate}
        \item $r \colon S^1 \longrightarrow \R$ is constant;
        \item $\rho \colon S^1 \longrightarrow M$ is a $T(r)$-periodic orbit of $R$, i.e. $\dot{\rho}(t) = T(r) R(\rho(t))$.
    \end{enumerate}
\end{corollary}
\begin{proof}
    The function $r \colon S^1 \longrightarrow \R$ is constant because $X_H$ is tangent to $\{r\} \times M$. Since $\dot{\gamma}(t) = X_H(\gamma(t))$ and by \cref{lem:reeb equals hamiltonian on symplectization}, we conclude that $\dot{\rho}(t) = T(r) R(\rho(t))$.
\end{proof}

\begin{lemma}
    \label{lem:action in symplectization}
    Let $\gamma = (r,\rho) \colon S^1 \longrightarrow \R \times M$ be a $1$-periodic orbit of $X_H$ and consider its action, given by
    \begin{IEEEeqnarray*}{c+x*}
        \mathcal{A}_H(\gamma) = \int_{S^1}^{} \gamma^* (e^r \alpha) - \int_{S^1}^{} H(\gamma(t)) \, \edv t.
    \end{IEEEeqnarray*}
    Then, $\mathcal{A}_H(\gamma) \eqqcolon \mathcal{A}_H(r)$ only depends on $r$, and we have the following formulas for $\mathcal{A}_H$ and $\mathcal{A}'_H$ (as functions of $r$):
    \begin{IEEEeqnarray*}{rClCl}
        \mathcal{A}_H (r) & = & H' (r) - H (r) & = & e^{ r} h' (e^r) - h(e^r), \\
        \mathcal{A}'_H(r) & = & H''(r) - H'(r) & = & e^{2r} h''(e^r).
    \end{IEEEeqnarray*}
\end{lemma}
\begin{proof}
    We show only that $\mathcal{A}_H(\gamma) = H'(r) - H(r)$, since the other formulas follow from this one by elementary calculus.
    \begin{IEEEeqnarray*}{rCls+x*}
        \mathcal{A}_H(\gamma)
        & = & \int_{S^1}^{} \gamma^* ( e^r \alpha) - \int_{S^1}^{} H(\gamma(t)) \, \edv t & \quad [\text{by definition of action}]                          \\
        & = & \int_{S^1}^{} e^r \rho^* \alpha - \int_{0}^{1} H(r, \rho(t)) \, \edv t      & \quad [\text{since $\gamma(t) = (r, \rho(t))$}]                 \\
        & = & e^r \int_{S^1}^{} \rho^* \alpha - \int_{0}^{1} H(r) \, \edv t               & \quad [\text{since $H = H(r)$}]                                 \\
        & = & e^r T(\rho) - H(r)                                                          & \quad [\text{by \cref{cor:hamiltonian orbits are reeb orbits}}] \\
        & = & H'(r) - H(r)                                                                & \quad [\text{by definition of $T$}].                              & \qedhere
    \end{IEEEeqnarray*}
\end{proof}

\begin{definition}
    \label{def:J cylindrical}
    Let $J$ be an almost complex structure on $(\R \times M, e^r \alpha)$. We say that $J$ is \textbf{cylindrical} if $J(\partial_r) = R$, if $J(\xi) \subset \xi$, and if the almost complex structure $J \colon \xi \longrightarrow \xi$ is compatible with $\edv \alpha$ and independent of $r$. We denote by $\mathcal{J}(M)$ the set of such $J$.
\end{definition}

\begin{lemma}
    \label{lem:J cylindrical forms}
    If $J$ is cylindrical then $\alpha \circ J = \edv r$.
\end{lemma}
\begin{proof}
    It suffices to show that $\alpha \circ J = \edv r$ on $\partial_r$, $R$ and $V \in \xi$.
    \begin{IEEEeqnarray*}{rCrClCl+x*}
        \alpha \circ J (\partial_r) & = & \alpha (R)            & = & 1 & = & \edv r (\partial_r) \\
        \alpha \circ J (R)          & = & - \alpha (\partial_r) & = & 0 & = & \edv r (R) \\
        \alpha \circ J (V)          & = & \alpha(J(V))          & = & 0 & = & \edv r (V).           & \qedhere
    \end{IEEEeqnarray*}
\end{proof}

\section{Completion of a Liouville domain}
\label{sec:completion of liouville domain}

Let $(X,\lambda)$ be a Liouville domain and $\omega = \edv \lambda$. Our goal in this section is to define the completion of $(X,\lambda)$, which is an exact symplectic manifold denoted by $(\hat{X}, \hat{\lambda})$. Recall that $(\del X, \lambda|_{\del X})$ is contact. Consider the symplectization $(\R \times \del X, e^r \lambda|_{\del X})$ of $(\del X, \lambda|_{\del X})$. Let $Z$ be the Liouville vector field of $(X, \lambda)$, which is given by $\lambda = \iota_Z \omega$. Denote the flow of $Z$ by
\begin{IEEEeqnarray*}{rrCl}
    \Phi_Z \colon & \R_{\leq 0} \times \del X & \longrightarrow & X \\
                  & (t,x)                     & \longmapsto     & \phi^t_Z(x).
\end{IEEEeqnarray*}
Since the vector field $Z$ is outward pointing at $\partial X$, the map $\Phi_Z$ is well-defined. Also, since $\Phi_Z$ is given by flowing along the vector field $Z$, it is an embedding.

\begin{lemma}
    \label{lem:flow of liouville}
    The map $\Phi_Z$ is a Liouville embedding, i.e. $\Phi_Z^* \lambda = e^r \lambda|_{\del X}$.
\end{lemma}
\begin{proof}
    If $(t,x) \in \R_{\leq 0} \times \partial X$ and $(u,v) \in T_{(t,x)} (\R_{\leq 0} \times \partial X) = \R \oplus T_x \partial X$, then
    \begin{IEEEeqnarray*}{rCls+x*}
        \IEEEeqnarraymulticol{3}{l}{(\Phi_Z^* \lambda)(u,v)} \\
        \quad & = & \lambda(\dv \Phi_Z(t,x)(u,v))                                                & \quad [\text{by definition of pullback}] \\
              & = & \lambda(\dv \Phi_Z(t,x)(0,v)) + \lambda(\dv \Phi_Z(t,x)(u,0))                & \quad [\text{by linearity of the derivative}] \\
              & = & \lambda(\dv \phi^t_Z (x)(v))  + u \, \lambda(Z_{\phi^t_Z(x)})                & \quad [\text{by definition of $\Phi_Z$}]\\
              & = & \lambda(\dv \phi^t_Z (x)(v))  + u \, \omega(Z_{\phi^t_Z(x)},Z_{\phi^t_Z(x)}) & \quad [\text{by definition of $Z$}] \\
              & = & \lambda(\dv \phi^t_Z (x)(v))                                                 & \quad [\text{since $\omega$ is antisymmetric}]\\
              & = & ((\phi^t_Z)^* \lambda)(v)                                                    & \quad [\text{by definition of pullback}] \\
              & = & e^t \lambda (v)                                                              & \quad [\text{by \cref{lem:mosers trick,lem:liouville vf}}]. & \qedhere
    \end{IEEEeqnarray*}
\end{proof}

\begin{definition}
    \label{def:completion of a Liouville domain}
    We define an exact symplectic manifold $(\hat{X},\hat{\lambda})$ called the \textbf{completion} of $(X,\lambda)$, as follows. As a smooth manifold, $\hat{X}$ is the gluing of $X$ and $\R \times \del X$ along the map $\Phi _{Z} \colon \R_{\leq 0} \times \del X \longrightarrow \Phi_Z(\R_{\leq 0} \times \del X)$. This gluing comes with embeddings 
    \begin{IEEEeqnarray*}{rCls+x*}
        \iota_X \colon X                                 & \longrightarrow & \hat{X}, \\
        \iota_{\R \times \del X} \colon \R \times \del X & \longrightarrow & \hat{X}.
    \end{IEEEeqnarray*}
    The form $\hat{\lambda}$ is the unique $1$-form on $\hat{X}$ such that 
    \begin{IEEEeqnarray*}{rCls+x*}
        \iota_X^* \hat{\lambda}                   & = & \lambda, \\
        \iota _{\R \times \del X}^* \hat{\lambda} & = & e^r \lambda|_{\del X}.
    \end{IEEEeqnarray*}
    The symplectic form of $\hat{X}$ is given by $\hat{\omega} \coloneqq \edv \hat{\lambda}$, which satisfies
    \begin{IEEEeqnarray*}{rCls+x*}
        \iota_X^* \hat{\omega}                   & = & \omega, \\
        \iota _{\R \times \del X}^* \hat{\omega} & = & \edv (e^r \lambda|_{\del X}).
    \end{IEEEeqnarray*}
    The Liouville vector field of $\hat{X}$ is the unique vector field $\hat{Z}$ such that $\iota_{\hat{Z}} \hat{\omega} = \hat{\lambda}$, which satisfies
    \begin{IEEEeqnarray*}{rRls+x*}
        Z          & \text{ is $\iota_X$-related to }                      & \hat{Z}, \\
        \partial_r & \text{ is $\iota_{\R \times \partial X}$-related to } & \hat{Z}.
    \end{IEEEeqnarray*}
\end{definition}

\begin{example}
    Let $(L,g)$ be a Riemannian manifold. Recall that $T^*L$ is an exact symplectic manifold, $S^*L$ is a hypersurface of contact type and that $D^*L$ is a Liouville domain. Also recall that there is a Liouville embedding $\varphi \colon \R \times S^* L \longrightarrow T^*L$ given by $\varphi(r,u) = e^r u$. Then, we can define a Liouville diffeomorphism $\hat{\varphi} \colon \widehat{D^*L} \longrightarrow T^*L$ as the unique map such that the following diagram commutes:
    \begin{IEEEeqnarray*}{c+x*}
        \begin{tikzcd}
            \widehat{D^* L} \ar[dr, hook, two heads, "\hat{\varphi}"] & \R \times S^* L \ar[l, hook'] \ar[d, hook, "\varphi"] \\
            D^* L \ar[u, hook] \ar[r, hook] & T^* L
        \end{tikzcd}
    \end{IEEEeqnarray*}
\end{example}

\begin{lemma}
    \label{lem:properties of completion}
    The diagram
    \begin{IEEEeqnarray*}{c}
        \begin{tikzcd}[ampersand replacement = \&]
            \R_{\leq 0} \times \del X \ar[d, swap, hook, "\Phi_Z"] \ar[r, hookrightarrow] \& \R \times \del X \ar[d, hookrightarrow, "\iota _{\R \times \del X}"] \ar[r, hookrightarrow] \& \R \times \hat{X} \ar[d, two heads, "\Phi _{\hat{Z}}"] \\
            X \ar[r, swap, hookrightarrow, "\iota_X"] \& \hat{X} \ar[r, equals] \& \hat{X}
        \end{tikzcd}
    \end{IEEEeqnarray*}
    commutes.
\end{lemma}
\begin{proof}
    The left square commutes by definition of $\hat{X}$. To prove that the right square commutes, let $(t,x) \in \R \times \del X$. We wish to show that $\Phi_{\hat{Z}}(t,x) = \iota_{\R \times \del X}(t,x)$. 
    \begin{IEEEeqnarray*}{rCls+x*}
        \iota_{\R \times \partial X} (t, x)
        & = & \iota_{\R \times \partial X} \circ \phi^t_{\partial_r} (0, x) & \quad [\text{by definition of flow of $\partial_r$}] \\
        & = & \phi^t_{\hat{Z}} \circ \iota_{\R \times \partial X}(0, x)     & \quad [\text{since $\partial_r$ is $\iota_{\R \times \partial X}$-related to $\hat{Z}$}] \\
        & = & \phi^t_{\hat{Z}} \circ \iota_X(x)                             & \quad [\text{by definition of completion}] \\
        & = & \Phi_{\hat{Z}}(t,x)                                           & \quad [\text{by definition of $\Phi_{\hat{Z}}$}].                                          & \qedhere
    \end{IEEEeqnarray*}
\end{proof}

\begin{lemma}
    \label{lem:codim 0 liouville emb preserves lvf}
    If $(X, \lambda_X)$ and $(Y, \lambda_Y)$ are Liouville domains and $\varphi \colon X \longrightarrow Y$ is a Liouville embedding of codimension $0$ then $Z_X$ is $\varphi$-related to $Z_Y$.
\end{lemma}
\begin{proof}
    For any $x \in X$ and $v \in T_x X$,
    \begin{IEEEeqnarray*}{rCls+x*}
        \IEEEeqnarraymulticol{3}{l}{\omega_Y (\dv \varphi(x) (Z_X|_x) - Z_Y|_{\varphi(x)}, \dv \varphi(x)(v))}\\
        \quad & = & (\iota_{Z_X} \varphi^* \omega_Y - \varphi^* \iota_{Z_Y} \omega_Y) (v) & \quad [\text{by the definitions of $\iota_{Z_X}$, $\iota_{Z_Y}$, and $\varphi^*$}] \\
        \quad & = & (\iota_{Z_X} \omega_X - \varphi^* \iota_{Z_Y} \omega_Y) (v)           & \quad [\text{since $\varphi$ is a Liouville embedding}] \\
        \quad & = & (\lambda_X - \varphi^* \lambda_X) (v)                                 & \quad [\text{by definition of Liouville vector field}] \\
        \quad & = & 0                                                                     & \quad [\text{since $\varphi$ is a Liouville embedding}].
    \end{IEEEeqnarray*}
    Since $\omega_Y$ is nondegenerate and $\varphi$ is a $0$-codimensional embedding, the result follows.
\end{proof}

We will now explain how to view the construction of taking the completion of a Liouville domain as a functor. Let $(X,\lambda_X)$, $(Y,\lambda_Y)$ be Liouville domains and $\varphi \colon X \longrightarrow Y$ be a Liouville embedding such that $Z_X$ is $\varphi$-related to $Z_Y$ (by \cref{lem:codim 0 liouville emb preserves lvf}, this is true whenever $\varphi$ is $0$-codimensional, although here we assume only that the Liouville vector fields are related). We wish to define an embedding $\varphi \colon \hat{X} \longrightarrow \hat{Y}$, using the following diagram as a guide (we will show that this diagram commutes in \cref{lem:diagram for map on completions commutes}):
\begin{IEEEeqnarray}{c}
    \plabel{eq:diagram for induced map on completions}
    \begin{tikzcd}[ampersand replacement = \&, row sep=scriptsize, column sep=0.2em]
        \& \R_{\leq 0} \times \del X \ar[dd, near end, swap, "\Phi_{Z_X}"] \ar[rr, "{\iota _{\R_{\leq 0}} \times \id_{\del X}}"] \& \& \R \times \del X \ar[dd, near start, swap, "{\iota _{\R \times \del X}}"] \ar[rr, "{\id \times \iota _{ \del X }}"] \& \& \R \times X \ar[ld, swap, "\id \times {\varphi}"] \ar[dd, near end] \ar[rr, "{\id \times \iota_X}"] \& \& \R \times \hat{X} \ar[ld,swap, "\id \times \hat{\varphi}"]\ar[dd, "\Phi _{\hat{Z}_X}"] \\
        \R_{\leq 0} \times \del Y \ar[dd, swap, "\Phi_{Z_Y}"] \ar[rr, crossing over] \& \& \R \times \del Y \ar[rr, crossing over] \& \& \R \times Y \ar[rr, crossing over, near end, "\hphantom{-}\id \times \iota_Y"] \& \& \R \times \hat{Y} \& \\
        \& X \ar[ld, "{\varphi}"] \ar[rr, near end, "\iota_X"] \& \& \hat{X} \ar[ld, "\hat{\varphi}"] \ar[rr, equals] \& \& \hat{X} \ar[ld, "\hat{\varphi}"]\ar[rr, equals] \& \& \hat{X} \ar[ld, "\hat{\varphi}"]\\
        Y \ar[rr, swap, "\iota_Y"] \& \& \hat{Y} \ar[uu, crossing over, near start, leftarrow, "{\iota _{\R \times \del Y}}"]\ar[rr, equals] \& \& \hat{Y} \ar[uu, near start, crossing over, leftarrow]\ar[rr, equals] \& \& \hat{Y} \ar[uu, near start, crossing over, leftarrow, "\Phi _{\hat{Z}_Y}"]\&
    \end{tikzcd}
    \IEEEeqnarraynumspace
\end{IEEEeqnarray}

\begin{definition}
    \label{def:embedding on completions coming from Liouville embedding}
    We define an embedding $\hat{\varphi} \colon \hat{X} \longrightarrow \hat{Y}$ by
    \begin{IEEEeqnarray*}{rCls+x*}
        \hat{\varphi} \circ \iota_X & \coloneqq & \iota_Y \circ \varphi, \\
        \hat{\varphi} \circ \iota_{\R \times \del X} & \coloneqq & \Phi_{\hat{Z}_Y} \circ (\id_ \R \times (\iota_Y \circ \varphi \circ \iota_{\partial X})).
    \end{IEEEeqnarray*}
\end{definition}

For $\hat{\varphi}$ to be well-defined, we need to check that the definitions of $\varphi$ on each region agree on the overlap.

\begin{lemma}
    \label{def:map on completions is well defined}
    The map $\hat{\varphi}$ is well-defined, i.e.
    \begin{IEEEeqnarray*}{c}
        \iota_Y \circ \varphi \circ \Phi _{Z_X} = \Phi_{\hat{Z}_Y} \circ (\id_ \R \times (\iota_Y \circ \varphi \circ \iota_{\partial X})) \circ (\iota _{\R_{\leq 0}} \times \id _{\del X}).
    \end{IEEEeqnarray*}
\end{lemma}
\begin{proof}
    It suffices to assume that $(t,x) \in \R_{\leq 0} \times \del X$ and to prove that $\iota_Y \circ \varphi \circ \Phi _{Z_X}(t,x) = \Phi _{\hat{Z}_Y}(t,\iota_Y(\varphi(x)))$.
    \begin{IEEEeqnarray*}{rCls+x*}
        \iota_Y \circ \varphi \circ \Phi _{Z_X}(t,x)
        & = & \iota_Y \circ \varphi \circ \phi^t _{Z_X}(x)       & \quad [\text{by definition of $\Phi _{Z_X}$}] \\
        & = & \iota_Y \circ \phi^t _{Z_Y} \circ \varphi(x)       & \quad [\text{since $Z_X$ is $\varphi$-related to $Z_Y$}] \\
        & = & \phi^t _{\hat{Z}_Y} \circ \iota_Y \circ \varphi(x) & \quad [\text{since $Z_Y$ is $\iota_Y$-related to $\hat{Z}_Y$}] \\
        & = & \Phi _{\hat{Z}_Y}(t,\iota_Y(\varphi(x)))           & \quad [\text{by definition of $\Phi _{\hat{Z}_Y}$}]. & \qedhere
    \end{IEEEeqnarray*}
\end{proof}

\begin{lemma}
    \label{def:map on completions is liouville embedding}
    The map $\hat{\varphi}$ is a Liouville embedding, i.e. $\hat{\varphi}^* \hat{\lambda}_Y = \hat{\lambda}_X$.
\end{lemma}
\begin{proof}
    We need to show that $\hat{\varphi}^* \hat{\lambda}_Y = \hat{\lambda}_X$, which is equivalent to
    \begin{IEEEeqnarray}{rCls+x*}
        \iota_X^* \hat{\varphi}^* \hat{\lambda}_Y                  & = & \iota_X^* \hat{\lambda}_X,                  \plabel{eq:map on completion is liouville embedding 1} \\
        \iota_{\R \times \del X}^* \hat{\varphi}^* \hat{\lambda}_Y & = & \iota_{\R \times \del X}^* \hat{\lambda}_X. \plabel{eq:map on completion is liouville embedding 2}
    \end{IEEEeqnarray}
    We prove Equation \eqref{eq:map on completion is liouville embedding 1}.
    \begin{IEEEeqnarray*}{rCls+x*}
        \iota_X^* \hat{\varphi}^* \hat{\lambda}_Y
        & = & (\hat{\varphi} \circ \iota_X)^* \hat{\lambda}_Y & \quad [\text{by functoriality of pullbacks}] \\
        & = & (\iota_Y \circ \varphi)^* \hat{\lambda}_Y       & \quad [\text{by definition of $\hat{\varphi}$}] \\
        & = & \varphi^* \iota_Y^* \hat{\lambda}_Y             & \quad [\text{by functoriality of pullbacks}] \\
        & = & \varphi^* \lambda_Y                             & \quad [\text{by definition of $\hat{\lambda}_Y$}] \\
        & = & \lambda_X                                       & \quad [\text{since $\varphi$ is a Liouville embedding}] \\
        & = & \iota_X^* \hat{\lambda}_X                       & \quad [\text{by definition of $\hat{\lambda}_X$}].
    \end{IEEEeqnarray*}
    We prove Equation \eqref{eq:map on completion is liouville embedding 2}.
    \begin{IEEEeqnarray*}{rCls+x*}
        \IEEEeqnarraymulticol{3}{l}{\iota _{\R \times \del X}^* \hat{\varphi}^* \hat{\lambda}_Y}\\
        \quad & = & (\hat{\varphi} \circ \iota _{\R \times \del X})^* \hat{\lambda}_Y                                            & \quad [\text{by functoriality of pullbacks}] \\
              & = & ( \Phi _{\hat{Z}_Y} \circ (\id_ \R \times (\iota_Y \circ \varphi \circ \iota _{\del X})) )^* \hat{\lambda}_Y & \quad [\text{by definition of $\hat{\varphi}$}] \\
              & = & (\id_ \R \times (\iota_Y \circ \varphi \circ \iota _{\del X}))^* \Phi _{\hat{Z}_Y}^* \hat{\lambda}_Y         & \quad [\text{by functoriality of pullbacks}] \\
              & = & (\id_ \R \times (\iota_Y \circ \varphi \circ \iota _{\del X}))^* e^r \hat{\lambda}_Y                         & \quad [\text{by \cref{lem:mosers trick,lem:liouville vf}}] \\
              & = & e^r \iota _{\del X}^* \varphi^* \iota_Y^* \hat{\lambda}_Y                                                    & \quad [\text{by functoriality of pullbacks}] \\
              & = & e^r \iota _{\del X}^* \varphi^* \lambda_Y                                                                    & \quad [\text{by definition of $\hat{\lambda}_Y$}] \\
              & = & e^r \iota _{\del X}^* \lambda_X                                                                              & \quad [\text{since $\varphi$ is a Liouville embedding}] \\
              & = & \iota^* _{\R \times \del X} \hat{\lambda}_X                                                                  & \quad [\text{by definition of $\hat{\lambda}_X$}].        & \qedhere
    \end{IEEEeqnarray*}
\end{proof}

\begin{lemma}
    \label{lem:liouville vector fields on completion are related}
    The Liouville vector fields $\hat{Z}_X$ and $\hat{Z}_Y$ are $\hat{\varphi}$-related.
\end{lemma}
\begin{proof}
    We need to show that
    \begin{IEEEeqnarray}{Rls+x*}
        Z_X        \text{ is $(\iota_Y \circ \varphi)$-related to }                                                                    & \hat{Z}_Y, \plabel{eq:liouville vector fields on completion are related 1} \\
        \partial_r \text{ is $(\Phi_{\hat{Z}_Y} \circ (\id_ \R \times (\iota_Y \circ \varphi \circ \iota_{\partial X})))$-related to } & \hat{Z}_Y. \plabel{eq:liouville vector fields on completion are related 2}
    \end{IEEEeqnarray}
    Here, \eqref{eq:liouville vector fields on completion are related 1}, follows because $Z_X$ is $\varphi$-related to $Z_Y$. To prove \eqref{eq:liouville vector fields on completion are related 2}, notice that for every $(t,x) \in \R \times \partial X$, we have $\partial_r = (1,0) \in \R \oplus T_x \partial X$ and therefore
    \begin{IEEEeqnarray*}{rCls+x*}
        \IEEEeqnarraymulticol{3}{l}{\dv ( \Phi_{\hat{Z}_Y} \circ (\id_ \R \times (\iota_Y \circ \varphi \circ \iota_{\partial X})) )(t,x) (1,0)}\\
        \quad & = & \dv \Phi_{\hat{Z}_Y} (t, \varphi(x)) (1, 0) & \quad [\text{by the chain rule}] \\
              & = & \hat{Z}_Y(t, \varphi(x))                    & \quad [\text{by definition of $\Phi_{\hat{Z}_Y}$}]. & \qedhere
    \end{IEEEeqnarray*}
\end{proof}

\begin{lemma}
    \label{lem:diagram for map on completions commutes}
    Diagram \eqref{eq:diagram for induced map on completions} commutes.    
\end{lemma}
\begin{proof}
    We have already proven in \cref{lem:properties of completion} that the squares on the front and back commute. The first square on the bottom commutes by definition of $\hat{\varphi}$. The other two squares on the bottom commute trivially. The top square commutes because $\hat{\varphi} \circ \iota_X = \iota_Y \circ \varphi$ by definition of $\hat{\varphi}$. We prove that the right square commutes. For $(t,x) \in \R \times \hat{X}$,%
    \begin{IEEEeqnarray*}{rCls+x*}
        \hat{\varphi} \circ \Phi _{\hat{Z}_X}(t,x)
        & = & \hat{\varphi} \circ \phi^t _{\hat{Z}_X}(x)                & \quad [\text{by definition of $\Phi _{\hat{Z}_X}$}] \\
        & = & \phi^t _{\hat{Z}_Y} \circ \hat{\varphi} (x)               & \quad [\text{by \cref{lem:liouville vector fields on completion are related}}] \\
        & = & \Phi _{\hat{Z}_Y} (t, \hat{\varphi}(x))                   & \quad [\text{by definition of $\Phi _{\hat{Z}_Y}$}] \\
        & = & \Phi _{\hat{Z}_Y} \circ (\id_ \R \times \hat{\varphi})(x) & \quad [\text{by definition of $\id_ \R \times \hat{\varphi}$}]. & \qedhere
    \end{IEEEeqnarray*}
\end{proof}

Finally, we check that the induced maps on the completions behave nicely with respect to compositions.

\begin{proposition}
    \phantomsection\label{prop:completion is a functor}
    The operation of taking the completion is a functor.
\end{proposition}
\begin{proof}
    We show that identities are preserved. Let $(X,\lambda)$ be a Liouville domain. We wish to prove that $\widehat{\id_X} = \id _{\hat{X}} \colon \hat{X} \longrightarrow \hat{X}$, which is equivalent to
    \begin{IEEEeqnarray}{rCls+x*}
        \widehat{\id_X} \circ \iota_X                  & = & \id_{\hat{X}} \circ \iota_X,                  \plabel{eq:completion functor identity 1} \\
        \widehat{\id_X} \circ \iota_{\R \times \del X} & = & \id_{\hat{X}} \circ \iota_{\R \times \del X}. \plabel{eq:completion functor identity 2}
    \end{IEEEeqnarray}
    We prove Equation \eqref{eq:completion functor identity 1}.
    \begin{IEEEeqnarray*}{rCls+x*}
        \widehat{\id_X} \circ \iota_X
        & = & \iota_X \circ \id_X          & \quad [\text{by definition of $\widehat{\id_X}$}] \\
        & = & \iota_X                      & \quad [\text{since $\id_X$ is the identity map}] \\
        & = & \id _{\hat{X}} \circ \iota_X & \quad [\text{since $\id_{\hat{X}}$ is the identity map}].
    \end{IEEEeqnarray*}
    We prove Equation \eqref{eq:completion functor identity 2}.
    \begin{IEEEeqnarray*}{rCls+x*}
        \widehat{\id_X} \circ \iota _{\R \times \del X}
        & = & \Phi_{\hat{Z}} \circ (\id_\R \times (\iota_X \circ \id_X \circ \iota_{\partial X})) & \quad [\text{by definition of $\widehat{\id_X}$}] \\
        & = & \id_{\hat{X}} \circ \iota_{\R \times \del X}                                        & \quad [\text{by \cref{lem:properties of completion}}].
    \end{IEEEeqnarray*}

    Now, we prove that compositions are preserved. Let $(X,\lambda_X)$, $(Y,\lambda_Y)$ and $(W,\lambda_W)$ be Liouville domains and $f \colon X \longrightarrow Y$ and $g \colon Y \longrightarrow W$ be Liouville embeddings. We wish to prove that $\widehat{g \circ f} = \hat{g} \circ \hat{f}$, which is equivalent to
    \begin{IEEEeqnarray}{rCls+x*}
        \widehat{g \circ f} \circ \iota_X                  & = & \hat{g} \circ \hat{f} \circ \iota_X,                  \plabel{eq:completion functor composition 1} \\
        \widehat{g \circ f} \circ \iota_{\R \times \del X} & = & \hat{g} \circ \hat{f} \circ \iota_{\R \times \del X}. \plabel{eq:completion functor composition 2}
    \end{IEEEeqnarray}
    We prove Equation \eqref{eq:completion functor composition 1}.
    \begin{IEEEeqnarray*}{rCls+x*}
        \widehat{g \circ f} \circ \iota_X
        & = & \iota_W \circ g \circ f             & \quad [\text{by definition of $\widehat{g \circ f}$}] \\
        & = & \hat{g} \circ \iota_Y \circ f       & \quad [\text{by definition of $\hat{g}$}]\\
        & = & \hat{g} \circ \hat{f} \circ \iota_X & \quad [\text{by definition of $\hat{f}$}].
    \end{IEEEeqnarray*}
    We prove Equation \eqref{eq:completion functor composition 2}.
    \begin{IEEEeqnarray*}{rCls+x*}
        \IEEEeqnarraymulticol{3}{l}{\widehat{g \circ f} \circ \iota _{\R \times \del X}} \\
        \quad & = & \Phi_{\hat{Z}_W} \circ (\id_{\R} \times (\iota_W \circ g \circ f \circ \iota_{\partial X}))                         & \quad [\text{by definition of $\widehat{g \circ f}$}] \\
              & = & \Phi_{\hat{Z}_W} \circ (\id_{\R} \times (\hat{g} \circ \iota_Y \circ f \circ \iota_{\partial X}))                   & \quad [\text{by definition of $\hat{g}$}]\\
              & = & \Phi_{\hat{Z}_W} \circ (\id_{\R} \times \hat{g}) \circ (\id_{\R} \times (\iota_Y \circ f \circ \iota_{\partial X})) & \\
              & = & \hat{g} \circ \Phi_{\hat{Z}_Y} \circ (\id_{\R} \times (\iota_Y \circ f \circ \iota_{\partial X}))                   & \quad [\text{by diagram \eqref{eq:diagram for induced map on completions}}] \\
              & = & \hat{g} \circ \hat{f} \circ \iota _{\R \times \del X}                                                               & \quad [\text{by definition of $\hat{f}$}].              & \qedhere
    \end{IEEEeqnarray*}
\end{proof}

\chapter{Indices}
\label{chp:indices}

\section{Maslov indices}
\label{sec:maslov indices}

In this section, our goal is to define the Maslov index of a loop of symplectic matrices and the Maslov index of a loop of Lagrangian subspaces. Our presentation is based on \cite{mcduffIntroductionSymplecticTopology2017}. We start by recalling relevant facts and notation about symplectic linear algebra. Let $V$ be a finite dimensional vector space. The vector spaces $V \directsum V^*$ and $V^* \oplus V$ admit symplectic structures given by
\begin{IEEEeqnarray*}{rCls+x*}
    \omega_{V \directsum V^*}((a,\alpha),(b,\beta)) & = & \beta(a) - \alpha(b), \\
    \omega_{V^* \directsum V}((\alpha,a),(\beta,b)) & = & \alpha(b) - \beta(a).
\end{IEEEeqnarray*}
If $V$ has an inner product $\p{<}{}{\cdot,\cdot}$, then we define a symplectic structure on $V \directsum V$ by
\begin{IEEEeqnarray}{c+x*}
    \plabel{eq:symplectic structure on v + v}
    \omega_{V \directsum V}((u,v),(x,y)) = \p{<}{}{u,y} - \p{<}{}{v,x}.
\end{IEEEeqnarray}
In this case, the maps
\begin{IEEEeqnarray*}{rrClCrrCl}
    \phi \colon & V \directsum V & \longrightarrow & V \directsum V^*      & \qquad & \psi \colon & V \directsum V & \longrightarrow & V^* \directsum V \\
                & (x,y)          & \longmapsto     & (x,\p{<}{}{y,\cdot}), &        &             & (x,y)          & \longmapsto     & (\p{<}{}{x,\cdot},y)
\end{IEEEeqnarray*}
are isomorphisms of symplectic vector spaces. For each $n$, define the $2n \times 2n$ matrices
\begin{IEEEeqnarray*}{c+x*}
    J_0 = \begin{bmatrix}
        0 & -\idm \\
        \idm & 0
    \end{bmatrix}, \quad
    \Omega_0 = \begin{bmatrix}
        0 & \idm \\
        -\idm & 0
    \end{bmatrix}.
\end{IEEEeqnarray*}
The canonical symplectic structure of $\R^{2n} = \R^n \directsum \R^n$, denoted $\omega_0$, is defined as in Equation \eqref{eq:symplectic structure on v + v} (where we use the Euclidean inner product). For $\mathbf{u} = (u,v) \in \R^{2n}$ and $\mathbf{x} = (x,y) \in \R^{2n}$, $\omega_0(\mathbf{u},\mathbf{v})$ is given by
\begin{IEEEeqnarray*}{rCls+x*}
    \omega_0((u,v),(x,y))
        & = & \p{<}{}{u,y} - \p{<}{}{v,x} \\
        & = & \mathbf{u}^T \Omega_0 \mathbf{v}.
\end{IEEEeqnarray*}
The \textbf{symplectic group} is given by
\begin{IEEEeqnarray*}{c+x*}
    \operatorname{Sp}(2n) \coloneqq \{ A \in \operatorname{GL}(2n,\R) \ | \ A^T \Omega_0 A = \Omega_0 \}.
\end{IEEEeqnarray*}
Denote by $C(S^1,\operatorname{Sp}(2n))$ the set of continuous maps from $S^1$ to $\operatorname{Sp}(2n)$, i.e. the set of loops of symplectic matrices.

\begin{theorem}[{\cite[Theorem 2.2.12]{mcduffIntroductionSymplecticTopology2017}}]
    \phantomsection\label{thm:maslov sympl properties}
    There exists a unique function
    \begin{IEEEeqnarray*}{c+x*}
        \maslov \colon C(S^1,\operatorname{Sp}(2n)) \longrightarrow \Z,
    \end{IEEEeqnarray*}
    called the \emph{\textbf{Maslov index}}, which satisfies the following properties:
    \begin{description}
        \item[(Homotopy)] The Maslov index descends to an isomorphism $\maslov \colon \pi_1(\operatorname{Sp}(2n)) \longrightarrow \Z$.
        \item[(Product)] If $A_1,A_2 \in C(S^1, \operatorname{Sp}(2n))$ then $\maslov(A_1 A_2) = \maslov(A_1) + \maslov(A_2)$.
        \item[(Direct sum)] If $A_i \in C(S^1, \operatorname{Sp}(2 n_i))$ for $i=1,2$ then $\maslov(A_1 \directsum A_2) = \maslov(A_1) + \maslov(A_2)$.
        \item[(Normalization)] If $A \in C(S^1, \operatorname{Sp}(2))$ is given by
            \begin{IEEEeqnarray*}{c+x*}
                A(t) =
                \begin{bmatrix}
                    \cos(2 \pi t) & -\sin(2 \pi t) \\
                    \sin(2 \pi t) & \cos(2 \pi t)
                \end{bmatrix}
            \end{IEEEeqnarray*}
            then $\maslov(A) = 1$.
    \end{description}
\end{theorem}

Let $(V,\omega)$ be a symplectic vector space. A subspace $W$ of $V$ is \textbf{Lagrangian} if $\dim W = 1/2 \dim V$ and $\omega|_W = 0$. The \textbf{Lagrangian Grassmannian} of $(V,\omega)$, denoted $\mathcal{L}(V,\omega)$, is the set of Lagrangian subspaces of $(V,\omega)$. Denote $\mathcal{L}(n) = \mathcal{L}(\R ^{2n},\omega_0)$.

\begin{theorem}[{\cite[Theorem 2.3.7]{mcduffIntroductionSymplecticTopology2017}}]
    \label{thm:maslov lagrangian properties}
    There exists a unique function
    \begin{IEEEeqnarray*}{c+x*}
        \maslov \colon C(S^1,\mathcal{L}(n)) \longrightarrow \Z,
    \end{IEEEeqnarray*}
    called the \emph{\textbf{Maslov index}}, which satisfies the following properties:
    \begin{description}
        \item[(Homotopy)] The Maslov index descends to an isomorphism $\maslov \colon \pi_1(\mathcal{L}(n)) \longrightarrow \Z$.
        \item[(Product)] If $W \in C(S^1,\mathcal{L}(n))$ and $A \in C(S^1,\operatorname{Sp}(2 n))$ then $\mu(AW) = \mu(W) + 2 \mu(A)$.
        \item[(Direct sum)] If $W_i \in C(S^1,\mathcal{L}(n_i))$ for $i = 1,2$ then $\mu(W_1 \directsum W_2) = \mu(W_1) + \mu(W_2)$.
        \item[(Normalization)] If $W \in C(S^1, \mathcal{L}(n))$ is given by $W(t) = e^{\pi i t} \R \subset \C$ then $\mu(W) = 1$.
        \item[(Zero)] A constant loop has Maslov index zero.
    \end{description}
\end{theorem}

\section{Conley--Zehnder index}

In this section we define the Conley--Zehnder index of a path of symplectic matrices. We define
\begin{IEEEeqnarray*}{rCls+x*}
    \operatorname{Sp}^\star(2n) & \coloneqq & \{ A \in \operatorname{Sp}(2n)  \ | \ \det(A - \idm) \neq 0 \},     \\
    \mathrm{SP}(n)              & \coloneqq &
    \left\{
        A \colon [0,1] \longrightarrow \mathrm{Sp}(2n)
        \ \middle\vert
        \begin{array}{l}
            A \text{ is continuous, } \\
            A(0) = \idm, \\
            A(1) \in \mathrm{Sp}^{\star}(2n)
        \end{array}
    \right\}.
\end{IEEEeqnarray*}

The following theorem characterizes the Conley--Zehnder index of a path of symplectic matrices. Originally, this result has appeared in \cite{salamonMorseTheoryPeriodic1992} and \cite{salamonLecturesFloerHomology1999}. However, we will use a restatement from \cite{guttConleyZehnderIndex2012}. Recall that if $S$ is a symmetric matrix, its \textbf{signature}, denoted by $\signature S$, is the number of positive eigenvalues of $S$ minus the number of negative eigenvalues of $S$.

\begin{theorem}[{\cite[Propositions 35 and 37]{guttConleyZehnderIndex2012}}]
    \phantomsection\label{thm:properties of cz}
    There exists a unique function
    \begin{IEEEeqnarray*}{c+x*}
        \conleyzehnder \colon \operatorname{SP}(n) \longrightarrow \Z,
    \end{IEEEeqnarray*}
    called the \emph{\textbf{Conley--Zehnder index}}, which satisfies the following properties:
    \begin{description}
        \item[(Naturality)] If $B \colon [0,1] \longrightarrow \operatorname{Sp}(2n)$ is a continuous path, then $\conleyzehnder(B A B ^{-1}) = \conleyzehnder(A)$;
        \item[(Homotopy)] $\conleyzehnder$ is constant on each component of $\operatorname{SP}(n)$;
        \item[(Zero)] If $A(s)$ has no eigenvalue on the unit circle for $s > 0$ then $\conleyzehnder(A) = 0$;
        \item[(Product)] If $A_i \in \operatorname{SP}(n_i)$ for $i=1,2$ then $\conleyzehnder(A_1 \directsum A_2) = \conleyzehnder(A_1) + \conleyzehnder(A_2)$;
        \item[(Loop)] If $B \in C(S^1, \operatorname{Sp}(2n))$ and $B(0) = B(1) = \idm$ then $\conleyzehnder(B A) = \conleyzehnder(A) + 2 \maslov(B)$.
        \item[(Signature)] If $S$ is a symmetric nondegenerate $2n \times 2n$-matrix with operator norm $\p{||}{}{S} < 2 \pi$ and $A(t) = \exp(J_0 S t)$, then $\conleyzehnder(A) = \frac{1}{2} \signature (S)$;
        \item[(Determinant)] ${n - \conleyzehnder(A)}$ is even if and only if $\det (\idm - A(1)) > 0$;
        \item[(Inverse)] $\conleyzehnder(A ^{-1}) = \conleyzehnder (A^T) = - \conleyzehnder(A)$.
    \end{description}
\end{theorem}

\begin{remark}
    By \cite[Proposition 37]{guttConleyZehnderIndex2012}, the homotopy, loop and signature properties are enough to determine the Conley--Zehnder index uniquely.
\end{remark}

We finish this section with a result which we will use later on to compute a Conley--Zehnder index.

\begin{proposition}[{\cite[Proposition 41]{guttConleyZehnderIndex2012}}]
    \label{prp:gutts cz formula}
    Let $S$ be a symmetric, nondegenerate $2 \times 2$-matrix and $T > 0$ be such that $\exp(T J_0 S) \neq \idm$. Consider the path of symplectic matrices $A \colon [0,T] \longrightarrow \operatorname{Sp}(2)$ given by
    \begin{IEEEeqnarray*}{c+x*}
        A(t) \coloneqq \exp(t J_0 S).
    \end{IEEEeqnarray*}
    Let $a_1$ and $a_2$ be the eigenvalues of $S$ and $\signature S$ be its signature. Then,
    \begin{IEEEeqnarray*}{c+x*}
        \conleyzehnder(A) = 
        \begin{cases}
            \p{}{1}{\frac{1}{2} + \p{L}{1}{\frac{\sqrt{a_1 a_2} T}{2 \pi}}} \signature S & \text{if } \signature S \neq 0, \\
            0                                                                            & \text{if } \signature S = 0.
        \end{cases}
    \end{IEEEeqnarray*}
\end{proposition}

\section{First Chern class}

Denote by $\mathbf{Man}^2$ the category of manifolds which are $2$-dimensional, connected, compact, oriented and with empty boundary. We will give a definition of the first Chern class of a symplectic vector bundle $E \longrightarrow \Sigma$ where $\Sigma \in \mathbf{Man}^2$. Our presentation is based on \cite{mcduffIntroductionSymplecticTopology2017}. We will start by setting up some categorical language. Define a contravariant functor $\mathbf{Man}^2 \longrightarrow \mathbf{Set}$:
\begin{IEEEeqnarray*}{rrCl}
    \mathcal{E} \colon & \mathbf{Man}^2 & \longrightarrow & \mathbf{Set}                                             \\
                       & \Sigma         & \longmapsto     & \mathcal{E}(\Sigma) \coloneqq \{ \text{symplectic vector bundles with base $\Sigma$} \}/\sim \\
                       & f \downarrow   & \longmapsto     & \uparrow f^*                                             \\
                       & \Sigma'        & \longmapsto     & \mathcal{E}(\Sigma') \coloneqq \{ \text{symplectic vector bundles with base $\Sigma'$} \}/\sim,
\end{IEEEeqnarray*}
where $\sim$ is the equivalence relation coming from isomorphisms of symplectic vector bundles. Define also the following contravariant functors $\mathbf{Man}^2 \longrightarrow \mathbf{Set}$:
\begin{IEEEeqnarray*}{rrCl}
    H^2 \coloneqq H^2(-;\Z) \colon                          & \mathbf{Man}^2 & \longrightarrow & \mathbf{Set}, \\ \\
    H_2^* \coloneqq \operatorname{Hom}(H_2(-;\Z),\Z) \colon & \mathbf{Man}^2 & \longrightarrow & \mathbf{Set}, \\ \\
    \mathcal{Z} \colon                                      & \mathbf{Man}^2 & \longrightarrow & \mathbf{Set}                \\
                                                            & \Sigma         & \longmapsto     & \mathcal{Z}(\Sigma) \coloneqq \Z \\
                                                            & f \downarrow   & \longmapsto     & \uparrow \times \deg f           \\
                                                            & \Sigma'        & \longmapsto     & \mathcal{Z}(\Sigma') \coloneqq \Z.
\end{IEEEeqnarray*}
We have a natural transformation $\alpha \colon H^2 \longrightarrow H_2^*$ which is given by
\begin{IEEEeqnarray*}{rrCl}
    \alpha_\Sigma \colon & H^2(\Sigma;\Z) & \longrightarrow & \operatorname{Hom}(H_2(\Sigma;\Z),\Z) \\
                         & [\omega]       & \longmapsto     & \alpha_\Sigma([\omega]),
\end{IEEEeqnarray*}
where $\alpha_\Sigma([\omega])([\sigma]) = [\omega(\sigma)]$. By the universal coefficient theorem for cohomology (see for example \cite{rotmanIntroductionHomologicalAlgebra2009}), $\alpha_\Sigma$ is surjective. Both $H^2(\Sigma;\Z)$ and $\operatorname{Hom}(H_2(\Sigma;\Z),\Z)$ are isomorphic to $\Z$, since $\Sigma \in \mathbf{Man}^2$. Therefore, $\alpha$ is a natural isomorphism. We also have a natural isomorphism $\operatorname{ev} \colon H_2^* \longrightarrow \mathcal{Z}$, given by
\begin{IEEEeqnarray*}{rrCl}
    \operatorname{ev}_\Sigma \colon & \operatorname{Hom}(H_2(\Sigma;\Z),\Z) & \longrightarrow & \Z \\
                                    & \phi                                  & \longmapsto     & \phi([\Sigma]).
\end{IEEEeqnarray*}

As we will see, the first Chern class is a natural transformation $c_1 \colon \mathcal{E} \longrightarrow H^2$ and the first Chern number is a natural transformation (which we denote by the same symbol) $c_1 \colon \mathcal{E} \longrightarrow \mathcal{Z}$. These functors and natural transformations will all fit into the following commutative diagram:
\begin{IEEEeqnarray*}{c+x*}
    \begin{tikzcd}[ampersand replacement = \&]
        \mathcal{E} \ar[r, "c_1"] \ar[rrr, bend right=50, swap, "c_1"] \& H^2 \ar[r, hook, two heads, "\alpha"] \& H_2^* \ar[r, hook, two heads, "\operatorname{ev}"] \& \mathcal{Z}.
    \end{tikzcd}
\end{IEEEeqnarray*}
Therefore, the first Chern class determines and is determined by the first Chern number. More precisely, if $E \longrightarrow \Sigma$ is a symplectic vector bundle then the first Chern number of $E$ equals the first Chern class of $E$ evaluated on $\Sigma$:
\begin{IEEEeqnarray}{c+x*}
    \plabel{eq:first chern class vs number}
    c_1(E) = c_1(E)[\Sigma].
\end{IEEEeqnarray}

\begin{definition}[{\cite[Section 2.7]{mcduffIntroductionSymplecticTopology2017}}]
    \label{def:c1}
    Let $\Sigma \in \mathbf{Man}^2$ (i.e. $\Sigma$ is $2$-dimensional, connected, compact, oriented, with empty boundary) and $E \longrightarrow \Sigma$ be a symplectic vector bundle. We define the \textbf{first Chern number} of $E$, $c_1(E) \in \Z$, as follows. Choose embedded $0$-codimensional manifolds $\Sigma_1$ and $\Sigma_2$ of $\Sigma$ such that
    \begin{IEEEeqnarray*}{c+x*}
        S \coloneqq \del \Sigma_1 = \del \Sigma_2 = \Sigma_1 \cap \Sigma_2
    \end{IEEEeqnarray*}
    and $\Sigma$ is the gluing of $\Sigma_1$ and $\Sigma_2$ along $S$. Orient $S$ as the boundary of $\Sigma_1$. For $i=1,2$, denote by $\iota_i \colon \Sigma_i \longrightarrow \Sigma$ the inclusion and choose a symplectic trivialization
    \begin{IEEEeqnarray*}{c+x*}
        \tau^i \colon \iota_i^* E \longrightarrow \Sigma_i \times \R ^{2n}.
    \end{IEEEeqnarray*}
    Define the overlap map $A \colon S \longrightarrow \operatorname{Sp}(2n)$ by $A(x) = \tau^1_x \circ (\tau^2_x)^{-1}$. Denote by $S_1, \ldots, S_k$ the connected components of $S$ and parametrize each component by a loop $\gamma_i \colon S^1 \longrightarrow S_i$ such that $\dot{\gamma}_i(t)$ is positively oriented. Finally, let 
    \begin{IEEEeqnarray*}{c+x*}
        c_1(E) \coloneqq \sum_{i=1}^{k} \mu(A \circ \gamma_i),
    \end{IEEEeqnarray*}
    where $\mu$ is the Maslov index as in \cref{thm:maslov sympl properties}.
\end{definition}

\begin{theorem}[{\cite[Theorem 2.7.1]{mcduffIntroductionSymplecticTopology2017}}]
    The first Chern number is well-defined and it is the unique natural transformation $c_1 \colon \mathcal{E} \longrightarrow \mathcal{Z}$ which satisfies the following properties:%
    \begin{description}
        \item[(Classification)] If $E, E' \in \mathcal{E}(\Sigma)$ then $E$ and $E'$ are isomorphic if and only if $\operatorname{rank} E = \operatorname{rank} E'$ and $c_1(E) = c_1(E')$.
        \item[(Naturality)] If $f \colon \Sigma \longrightarrow \Sigma'$ is a smooth map and $E \in \mathcal{E}(\Sigma)$ then $c_1(f^*E) = \deg(f) c_1(E)$.
        \item[(Additivity)] If $E, E' \in \mathcal{E}(\Sigma)$ then $c_1(E \directsum E') = c_1(E) + c_1(E')$.
        \item[(Normalization)] The first Chern number of $T \Sigma$ is $c_1(T\Sigma) = 2 - 2g$.
    \end{description}
\end{theorem}

\section{Conley--Zehnder index of a periodic orbit}

Let $(X,\omega)$ be a symplectic manifold of dimension $2n$ and $H \colon S^1 \times X \longrightarrow \R$ be a time-dependent Hamiltonian. For each $t \in S^1$ we denote by $H_t$ the map $H_t = H(t,\cdot) \colon X \longrightarrow \R$. The Hamiltonian $H$ has a corresponding time-dependent Hamiltonian vector field $X_H$ which is uniquely determined by 
\begin{IEEEeqnarray*}{c+x*}
    \edv H_t = - \iota_{X_{H_t}} \omega.
\end{IEEEeqnarray*}
We denote by $\phi^t_{X_H}$ the time-dependent flow of $X_{H}$.

\begin{definition}
    \label{def:orbit of hamiltonian}
    A \textbf{$1$-periodic orbit} of $H$ is a map $\gamma \colon S^1 \longrightarrow X$ such that
    \begin{IEEEeqnarray*}{c+x*}
        \dot{\gamma}(t) = X_{H_t} (\gamma(t))
    \end{IEEEeqnarray*}
    for every $t \in S^1$. If $\lambda$ is a symplectic potential for $(X,\omega)$, then the \textbf{action} of $\gamma$ is
    \begin{IEEEeqnarray*}{c+x*}
        \mathcal{A}_H(\gamma) \coloneqq \int_{S^1}^{} \gamma^* \lambda - \int_{S^1}^{} H(t, \gamma(t)) \edv t.
    \end{IEEEeqnarray*}
\end{definition}

\begin{definition}
    \label{def:nondegenerate hamiltonian orbit}
    Let $\gamma$ be a $1$-periodic orbit of $H$. We say that $\gamma$ is \textbf{nondegenerate} if the linear map
    \begin{IEEEeqnarray*}{c+x*}
        \dv \phi^{1}_{X_H} \colon T_{\gamma(0)} X \longrightarrow T_{\gamma(1)} X = T_{\gamma(0)} X
    \end{IEEEeqnarray*}
    does not have $1$ as an eigenvalue. We say that the Hamiltonian $H$ is \textbf{nondegenerate} if every $1$-periodic orbit of $H$ is nondegenerate.
\end{definition}

\begin{definition}
    \phantomsection\label{def:cz of hamiltonian orbit wrt trivialization}
    Let $\gamma$ be a $1$-periodic orbit of $H$ and $\tau$ be a symplectic trivialization of $\gamma^* TX$. We define the \textbf{Conley--Zehnder index} of $\gamma$ with respect to $\tau$, denoted $\conleyzehnder^{\tau}(\gamma)$, as follows. First, define a path of symplectic matrices $A^{\gamma,\tau} \colon [0,1] \longrightarrow \operatorname{Sp}(2n)$ by the equation $A^{\gamma,\tau}(t) \coloneqq \tau_t \circ \dv \phi^t_{X_H}(\gamma(0)) \circ \tau_{0}^{-1}$. In other words, $A^{\gamma,\tau}(t)$ is the unique linear map such that the diagram
    \begin{IEEEeqnarray*}{c+x*}
        \begin{tikzcd}
            T_{\gamma(0)} X \ar[d, swap, "\dv \phi^t_{X_{H}}(\gamma(0))"] \ar[r, "\tau_0"] & \R^{2n} \ar[d, "A^{\gamma,\tau}(t)"] \\
            T_{\gamma(t)} \ar[r, swap, "\tau_t"] & \R^{2n}
        \end{tikzcd}
    \end{IEEEeqnarray*}
    commutes. Notice that since $\gamma$ is nondegenerate, $A^{\gamma,\tau} \in \operatorname{SP}(n)$. Then, define
    \begin{IEEEeqnarray*}{c+x*}
        \conleyzehnder^{\tau}(\gamma) \coloneqq \conleyzehnder(A^{\gamma,\tau}).
    \end{IEEEeqnarray*}
\end{definition}

Let $D = \{ z \in \C \mid |z| \leq 1 \}$ be the disk and denote by $\iota_{D,S^1} \colon S^1 \longrightarrow D$ the inclusion on the boundary, i.e. $\iota_{D,S^1}(t) = e^{2 \pi i t}$.

\begin{lemma}
    \label{lem:cz of hamiltonian is independent of triv over filling disk}
    Let $\gamma$ be a $1$-periodic orbit of $H$. For $i = 1,2$, let $u_i \colon D \longrightarrow X$ be a filling disk for $\gamma$ (i.e. $\gamma = u_i \circ \iota_{D,S^1}$) and $\tau^i$ be a symplectic trivialization of $u_i^* TX$. If $c_1(TX)|_{\pi_2(X)} = 0$, then
    \begin{IEEEeqnarray*}{c+x*}
        \conleyzehnder^{\tau^1}(\gamma) = \conleyzehnder^{\tau^2}(\gamma).
    \end{IEEEeqnarray*}
\end{lemma}
\begin{proof}
    Consider the diagram
    \begin{IEEEeqnarray}{c+x*}
        \plabel{eq:diagram cz indep choices}
        \begin{tikzcd}
            \R^{2n} \ar[d, swap, "A^{\gamma,\tau^1}(t)"] & T_{\gamma(0)} X \ar[d, "\dv \phi^t_{X_H}(\gamma(0))"] \ar[l, swap, "\tau^1_0"] \ar[r, "\tau^2_0"] & \R ^{2n} \ar[ll, bend right=50, swap, "B(0)"] \ar[d, "A^{\gamma,\tau^2}(t)"] \\
            \R^{2n}                                      & T_{\gamma(t)} X \ar[l, "\tau^1_t"] \ar[r, swap, "\tau^2_t"]                                       & \R ^{2n} \ar[ll, bend left=50, "B(t)"] \\
        \end{tikzcd}
    \end{IEEEeqnarray}
    where we have defined $B(t) \coloneqq \tau^1_t \circ (\tau^2_t)^{-1}$. Let $\sigma \colon S^2 \longrightarrow X$ be the gluing of the disks $u_1$ and $u_2$ along their common boundary $\gamma$. Then,
    \begin{IEEEeqnarray*}{rCls+x*}
        \IEEEeqnarraymulticol{3}{l}{\conleyzehnder^{\tau^1}(\gamma) - \conleyzehnder^{\tau^2}(\gamma)}\\ \quad
        & = & \conleyzehnder(A^{\gamma,\tau^1}) - \conleyzehnder(A^{\gamma,\tau^2})             & \quad [\text{by \cref{def:cz of hamiltonian orbit wrt trivialization}}]\\
        & = & \conleyzehnder(B A^{\gamma,\tau^2} B(0)^{-1}) - \conleyzehnder(A^{\gamma,\tau^2}) & \quad [\text{by diagram \eqref{eq:diagram cz indep choices}}] \\
        & = & \conleyzehnder(B(0)^{-1} B A^{\gamma,\tau^2}) - \conleyzehnder(A^{\gamma,\tau^2}) & \quad [\text{by naturality of $\conleyzehnder$}] \\
        & = & 2 \mu(B(0)^{-1} B)                                                                & \quad [\text{by the loop property of $\conleyzehnder$}] \\
        & = & 2 \mu(B)                                                                          & \quad [\text{by homotopy invariance of $\maslov$}] \\
        & = & 2 c_1(\sigma^* TX)                                                                & \quad [\text{by definition of the first Chern number}] \\
        & = & 2 c_1 (TX) ([\sigma])                                                             & \quad [\text{by Equation \eqref{eq:first chern class vs number}}] \\
        & = & 0                                                                                 & \quad [\text{by assumption}].                                             & \qedhere
    \end{IEEEeqnarray*}
\end{proof}

Let $(M,\alpha)$ be a contact manifold of dimension $2n + 1$ with Reeb vector field $R$. Our goal is to repeat the discussion of the first part of this section in the context of periodic orbits of $R$.

\begin{definition}
    A \textbf{Reeb orbit} is a map $\gamma \colon \R / T \Z \longrightarrow M$ such that
    \begin{IEEEeqnarray*}{c+x*}
        \dot{\gamma}(t) = R(\gamma(t))
    \end{IEEEeqnarray*}
    for every $t \in S^1$. In this case, we call $T$ the \textbf{period} of $\gamma$. The \textbf{multiplicity} of $\gamma$, which we will usually denote by $m$, is the degree of the map $\gamma \colon \R / T \Z \longrightarrow \img \gamma$. The \textbf{action} of $\gamma$ is
    \begin{IEEEeqnarray*}{c+x*}
        \mathcal{A}(\gamma) \coloneqq \int_{0}^{T} \gamma^* \lambda = T. 
    \end{IEEEeqnarray*}
\end{definition}

\begin{remark}
    Alternatively, a $T$-periodic Reeb orbit can be seen as a map $\gamma \colon S^1 \longrightarrow M$ such that $\dot{\gamma}(t) = T R(\gamma(t))$. We will use the two possible descriptions interchangeably.
\end{remark}

Since $\ldv{R} \alpha = 0$ (by \cref{lem:reeb vf preserves contact form}) and using \cref{lem:mosers trick}, we conclude that $(\phi^t_R)^* \alpha = \alpha$. In particular, $\dv \phi^t_R(p) (\xi_p) \subset \xi_{\phi^t_R(p)}$ and
\begin{IEEEeqnarray*}{c+x*}
    \dv \phi^t_R(p) \colon \xi_p \longrightarrow \xi_{\phi^t_R(p)}
\end{IEEEeqnarray*}
is a symplectic linear map.

\begin{definition}
    A Reeb orbit $\gamma$ of $M$ is \textbf{nondegenerate} if the linear map
    \begin{IEEEeqnarray*}{c+x*}
        \dv \phi^1_R(\gamma(0)) \colon \xi_{\gamma(0)} \longrightarrow \xi_{\gamma(1)} = \xi_{\gamma(0)}
    \end{IEEEeqnarray*}
    does not have $1$ as an eigenvalue. We say that $(M, \alpha)$ is \textbf{nondegenerate} if every Reeb orbit in $M$ is nondegenerate. If $(X, \lambda)$ is a Liouville domain, then $(X, \lambda)$ is \textbf{nondegenerate} if $(\partial X, \lambda|_{\partial X})$ is nondegenerate.
\end{definition}

\begin{definition}
    \label{def:cz of reeb orbit wrt trivialization}
    Let $\gamma$ be a periodic orbit of $R$ and $\tau$ be a symplectic trivialization of $\gamma^* \xi$. The \textbf{Conley--Zehnder index} of $\gamma$ is given by
    \begin{IEEEeqnarray*}{c+x*}
        \conleyzehnder^{\tau}(\gamma) \coloneqq \conleyzehnder(A^{\gamma,\tau}),
    \end{IEEEeqnarray*}
    where $A^{\gamma,\tau} \colon [0,1] \longrightarrow \operatorname{Sp}(2n)$ is the path of symplectic matrices given by the equation $A^{\gamma,\tau}(t) \coloneqq \tau_t \circ \dv \phi^t_{R}(\gamma(0)) \circ \tau_{0}^{-1}$.
\end{definition}

\begin{lemma}
    \label{lem:cz of reeb is independent of triv over filling disk}
    Let $(X, \lambda)$ be a Liouville domain and $\gamma \colon S^1 \longrightarrow \partial X$ be a Reeb orbit. For $i = 1,2$, let $u_i \colon D \longrightarrow X$ be a filling disk for $\gamma$ (i.e. $\iota_{X,\partial X} \circ \gamma = u_i \circ \iota_{D,S^1}$). Let $\tau^i$ be a symplectic trivialization of $u_i^* TX$ and denote also by $\tau^i$ the induced trivialization of $(\iota_{X,\partial X} \circ \gamma)^* TX$. Assume that
    \begin{IEEEeqnarray*}{rClCl}
        \tau^i_{t}(Z_{\gamma(t)}) & = & e_1     & \in & \R^{2n}, \\
        \tau^i_{t}(R_{\gamma(t)}) & = & e_{n+1} & \in & \R^{2n},
    \end{IEEEeqnarray*}
    for every $t \in S^1$. If $2 c_1(TX) = 0$, then
    \begin{IEEEeqnarray*}{c+x*}
        \conleyzehnder^{\tau^1}(\gamma) = \conleyzehnder^{\tau^2}(\gamma).
    \end{IEEEeqnarray*}
\end{lemma}
\begin{proof}
    By the assumptions on $\tau^i$, the diagram
    \begin{IEEEeqnarray}{c+x*}
        \plabel{eq:diagram cz reeb indep triv}
        \begin{tikzcd}
            \xi_{\gamma(t)} \ar[r] \ar[d, swap, "\tau^i_t"] & T_{\gamma(t)} X \ar[d, "\tau^i_t"] & \xi^{\perp}_{\gamma(t)} \ar[d, "\tau^i_t"] \ar[l] \\
            \R^{2n-2} \ar[r, swap, "\iota_{\R^{2n-2}}"]     & \R^{2n}                            & \R^{2} \ar[l, "\iota_{\R^{2}}"]
        \end{tikzcd}
    \end{IEEEeqnarray}
    commutes, where
    \begin{IEEEeqnarray*}{rCls+x*}
        \iota_{\R^{2n-2}}(x^2,\ldots,x^n,y^2,\ldots,y^n) & = & (0,x^2,\ldots,x^n,0,y^2,\ldots,y^n), \\
        \iota_{\R^{2}}(x,y)                              & = & (x,0,\ldots,0,y,0,\ldots,0).
    \end{IEEEeqnarray*}
    Define 
    \begin{IEEEeqnarray*}{rCcCrCl}
        B^{2n}(t)   & \coloneqq & \tau^1_t \circ (\tau^2_t)^{-1} & \colon & \R^{2n}   & \longrightarrow & \R^{2n}, \\
        B^{2n-2}(t) & \coloneqq & \tau^1_t \circ (\tau^2_t)^{-1} & \colon & \R^{2n-2} & \longrightarrow & \R^{2n-2},
    \end{IEEEeqnarray*}
    By the assumptions on $\tau^i$, and diagram \eqref{eq:diagram cz reeb indep triv},
    \begin{IEEEeqnarray}{c+x*}
        \plabel{eq:decomposition of b}
        B^{2n}(t) =
        \begin{bmatrix}
            \id_{\R^2} & 0 \\
            0          & B^{2n-2}
        \end{bmatrix}.
    \end{IEEEeqnarray}
    Let $\sigma \colon S^2 \longrightarrow X$ be the gluing of the disks $u_1$ and $u_2$ along their common boundary $\gamma$. Finally, we compute
    \begin{IEEEeqnarray*}{rCls+x*}
        \conleyzehnder^{\tau^1}(\gamma) - \conleyzehnder^{\tau^2}(\gamma)
        & = & 2 \mu (B^{2n-2})   & \quad [\text{by the same computation as in \cref{lem:cz of hamiltonian is independent of triv over filling disk}}] \\
        & = & 2 \mu (B^{2n})     & \quad [\text{by Equation \eqref{eq:decomposition of b} and \cref{thm:maslov sympl properties}}] \\
        & = & 2 c_1(\sigma^* TX) & \quad [\text{by definition of first Chern class}] \\
        & = & 0                  & \quad [\text{by assumption}].                                                                                        & \qedhere
    \end{IEEEeqnarray*}
\end{proof}

\begin{remark}
    \label{rmk:notation for tuples of orbits}
    Suppose that $\Gamma = (\gamma_1, \ldots, \gamma_p)$ is a tuple of (Hamiltonian or Reeb) orbits and $\tau$ is a trivialization of the relevant symplectic vector bundle over each orbit. We will frequently use the following notation:
    \begin{IEEEeqnarray*}{rCls+x*}
        \mathcal{A}(\Gamma)           & \coloneqq & \sum_{i=1}^{p} \mathcal{A}(\gamma_i), \\
        \conleyzehnder^{\tau}(\Gamma) & \coloneqq & \sum_{i=1}^{p} \conleyzehnder^{\tau}(\gamma_i).
    \end{IEEEeqnarray*}
    If $\beta = \sum_{i=1}^{m} a_i \Gamma_i$ is a formal linear combination of tuples of orbits, then we denote
    \begin{IEEEeqnarray*}{c+x*}
        \mathcal{A}(\beta) \coloneqq \max_{i = 1, \ldots, m} \mathcal{A}(\Gamma_i).
    \end{IEEEeqnarray*}
    The action of a formal linear combination is going to be relevant only in \cref{chp:contact homology}, where we will consider the action filtration on linearized contact homology.
\end{remark}

\section{Periodic Reeb orbits in a unit cotangent bundle}

Let $(L, g)$ be an orientable Riemannian manifold of dimension $n$. Recall that $L$ has a cotangent bundle $\pi \colon T^* L \longrightarrow L$, which is an exact symplectic manifold with symplectic potential $\lambda \in \Omega^1(T^* L)$, symplectic form $\omega \coloneqq \edv \lambda$ and Liouville vector field $Z$ given by $\iota_Z \omega = \lambda$. We will denote by $z \colon L \longrightarrow T^*L$ the zero section. Consider the unit cotangent bundle $\pi \colon S^* L \longrightarrow L$ and denote by $\iota \colon S^* L \longrightarrow L$ the inclusion. Then, $\alpha \coloneqq \iota^* \lambda$ is a contact form on $S^* L$, with associated contact distribution $\xi = \ker \alpha \subset T S^* L$ and Reeb vector field $R \in \mathfrak{X}(S^* L)$. The Riemannian metric $g$ defines a vector bundle isomorphism $\tilde{g} \colon TL \longrightarrow T^*L$ given by $\tilde{g}(v) = g(v, \cdot)$.

Let $\ell > 0$ and $c \colon \R / \ell \Z \longrightarrow L$ be a curve which is parametrized by arclength. Define $\gamma \coloneqq \tilde{g} \circ \dot{c} \colon \R / \ell \Z \longrightarrow S^* L$. Then, by \cref{thm:flow geodesic vs hamiltonian,thm:flow reeb vs hamiltonian}, the curve $c$ is a geodesic (of length $\ell$) if and only if $\gamma$ is a Reeb orbit (of period $\ell$). We will assume that this is the case. The goal of this section is to study specific sets of trivializations and maps between these sets (see diagram \eqref{eq:diagram of maps of trivializations}), which can be used to define the Conley--Zehnder index of $\gamma$ (see \cref{thm:index of geodesic or reeb orbit isometric triv}).

Since $T^* L$ is a symplectic manifold, $T T^* L \longrightarrow T^* L$ is a symplectic vector bundle. The hyperplane distribution $\xi$ is a symplectic subbundle of $\iota^* T T^* L \longrightarrow S^* L$. We can consider the symplectic complement of $\xi$, which by \cref{lem:decomposition coming from contact hypersurface} is given by
\begin{IEEEeqnarray*}{c+x*}
    \xi^{\perp}_{u} = \p{<}{}{Z_u} \oplus \p{<}{}{R_u}
\end{IEEEeqnarray*}
for every $u \in S^* L$. Finally, $T^* L \oplus T L \longrightarrow L$ is a symplectic vector bundle, with symplectic structure given by
\begin{IEEEeqnarray*}{c+x*}
    \omega_{T^* L \oplus TL}((u,v), (x,y)) = u(y) - x(v).
\end{IEEEeqnarray*}

\begin{remark}
    \label{rmk:connections}
    Let $\pi \colon E \longrightarrow B$ be a vector bundle. Consider the vector bundles $\pi^* E$, $TE$ and $\pi^* TB$ over $E$. There is a short exact sequence
    \begin{IEEEeqnarray*}{c+x*}
        \phantomsection\label{eq:short exact sequence of vector bundles}
        \begin{tikzcd}
            0 \ar[r] & \pi^* E \ar[r, "I^V"] & TE \ar[r, "P^H"] & \pi^* T B \ar[r] & 0
        \end{tikzcd} 
    \end{IEEEeqnarray*}
    of vector bundles over $E$, where
    \begin{IEEEeqnarray*}{rClCrClCl}
        I^V_e & \coloneqq & \dv \iota_e(e)                              & \colon & E_{\pi(e)} & \longrightarrow & T_e E,       & \quad & \text{where } \iota_e \colon E_{\pi(e)} \longrightarrow E \text{ is the inclusion,} \\
        P^H_e & \coloneqq & \dv \parbox{\widthof{$\iota_e$}}{$\pi$} (e) & \colon & T_e E      & \longrightarrow & T_{\pi(e)} B,
    \end{IEEEeqnarray*}
    for every $e \in E$. Recall that a \textbf{Koszul connection} on $E$ is a map
    \begin{IEEEeqnarray*}{c+x*}
        \nabla \colon \mathfrak{X}(B) \times \Gamma(E) \longrightarrow \Gamma(E)
    \end{IEEEeqnarray*}
    which is $C^{\infty}$-linear on $\mathfrak{X}(B)$ and satisfies the Leibniz rule on $\Gamma(E)$. A \textbf{linear Ehresmann connection} on $E$ is a vector bundle map $P^V \colon TE \longrightarrow \pi^* E$ such that $P^V \circ I^V = \id_{\pi^* TB}$ and $P^V \circ T m_{\lambda} = m_{\lambda} \circ P^V$ for every $\lambda \in \R$, where $m_{\lambda} \colon E \longrightarrow E$ is the map which multiplies by $\lambda$. The sets of Koszul connections on $E$ and of linear Ehresmann connections on $E$ are in bijection. If $\nabla$ is a Koszul connection on $E$, the corresponding linear Ehresmann connection is given as follows. Let $I^H \colon \pi^* TB \longrightarrow TE$ be the map which is given by
    \begin{IEEEeqnarray*}{c+x*}
        I^H_e(u) \coloneqq \dv s (\pi(e)) u - I^V_e(\nabla_u^{} s)
    \end{IEEEeqnarray*}
    for every $e \in E$ and $u \in T_{\pi(e)} B$, where $s$ in any choice of section of $\pi \colon E \longrightarrow B$ such that $s(\pi(e)) = e$. The map $I^H$ is independent of the choice of section $s$ and satisfies $P^H \circ I^H = \id_{\pi^* TB}$. Let $P^V \colon TE \longrightarrow \pi^* E$ be the map which is given by
    \begin{IEEEeqnarray*}{c+x*}
        P^V_e(w) \coloneqq (I^V_e)^{-1} (w - I^H_e \circ P^H_e (w))
    \end{IEEEeqnarray*}
    for every $e \in E$ and $w \in T_e E$. We point out that this definition is well-posed, since $w - I^H_e \circ P^H_e (w) \in \ker P^H_e = \img I^V_e$. As before, $P^V \circ I^V = \id_{\pi^* E}$. Finally, the maps
    \begin{IEEEeqnarray*}{rCrCrCl}
        I & \coloneqq & I^V & \oplus & I^H & \colon & \pi^* E \oplus \pi^* T B \longrightarrow TE, \\
        P & \coloneqq & P^V & \times & P^H & \colon & TE \longrightarrow \pi^* E \oplus \pi^* T B,
    \end{IEEEeqnarray*}
    are isomorphisms and inverses of one another.
\end{remark}

Consider the Levi-Civita connection on $L$, which is a Koszul connection on $T L$. There is an induced Koszul connection on $T^* L$ given by
\begin{IEEEeqnarray*}{c+x*}
    (\nabla_X \beta)(Y) \coloneqq X(\beta(Y)) - \beta(\nabla_X Y),
\end{IEEEeqnarray*}
for every $X, Y \in \mathfrak{X}(L)$ and $\beta \in \Gamma(T^* L) = \Omega^1(L)$. By \cref{rmk:connections} (with $B = L$ and $E = T^*L$), there is an induced linear Ehresmann connection on $\pi \colon T^*L \longrightarrow L$ which is given by maps
\begin{IEEEeqnarray*}{rCrCrCl}
    I & \coloneqq & I^V & \oplus & I^H & \colon & \pi^* T^* L \oplus \pi^* T L \longrightarrow T T^* L, \\
    P & \coloneqq & P^V & \times & P^H & \colon & T T^* L \longrightarrow \pi^* T^* L \oplus \pi^* T L.
\end{IEEEeqnarray*}

\begin{lemma}
    \label{prop:properties of p}
    The maps $I$ and $P$ are isomorphisms of symplectic vector bundles. Moreover,
    \begin{IEEEeqnarray}{rClCl}
        P(Z_u) & = & (u,0),                 & \quad & \text{ for every } u \in T^* L, \plabel{eq:p of vfs 1} \\
        P(R_u) & = & (0,\tilde{g}^{-1}(u)), & \quad & \text{ for every } u \in S^* L. \plabel{eq:p of vfs 2}
    \end{IEEEeqnarray}
\end{lemma}
\begin{proof}
    Let $q \coloneqq \pi(u)$ and choose normal coordinates $(q^1,\ldots,q^n)$ on $L$ centred at $q$ (this means that with respect to these coordinates, $g_{ij}(q) = \delta_{ij}$ and $\partial_k g_{ij} (q) = 0$). Let $(q^1, \ldots, q^n, p_1, \ldots, p_n)$ be the induced coordinates on $T^* L$. Then, the vector spaces $T_u T^*L$ and $T^*_q L \directsum T_q L$ have the following symplectic bases:
    \begin{IEEEeqnarray}{rCls+x*}
        T_ u T^*L                & = & \spn \p{c}{2}{ \pdv{}{p_1}\Big|_{u}, \cdots, \pdv{}{p_n}\Big|_{u}, \pdv{}{q^1}\Big|_{u}, \cdots, \pdv{}{q^n}\Big|_{u} },          \plabel{eq:basis 1} \\
        T^*_q L \directsum T_q L & = & \spn \p{c}{1}{ \edv q^1|_q, \ldots, \edv q^n|_q } \directsum \spn \p{c}{2}{ \pdv{}{q^1}\Big|_{q}, \cdots, \pdv{}{q^n}\Big|_{q} }. \plabel{eq:basis 2}
    \end{IEEEeqnarray}
    By the definitions of $P$ and $I$ in \cref{rmk:connections}, we have
    \begin{IEEEeqnarray}{rCls+x*}
        I^V_u (\edv q^i|_q)                   & = & \pdv{}{p_i}\Big|_u,                              \IEEEnonumber\\
        P^H_u \p{}{2}{ \pdv{}{q^i}\Big|_{u} } & = & \pdv{}{q^i}\Big|_{q},                            \plabel{eq:p horizontal in coordinates} \\
        P^V_u \p{}{2}{ \pdv{}{p_i}\Big|_{u} } & = & P^V_u \circ I^V_u (\edv q^i|_{q}) = \edv q^i|_q, \plabel{eq:p vertical in coordinates}
    \end{IEEEeqnarray}
    which implies that $P$ is the identity matrix when written with respect to the bases \eqref{eq:basis 1} and \eqref{eq:basis 2}. Since these bases are symplectic, $P$ is a symplectic linear map. With respect to the coordinates $(q^1, \ldots, q^n, p_1, \ldots, p_n)$, the Liouville vector field is given by
    \begin{IEEEeqnarray}{c+x*}
        Z = \sum_{i=1}^{n} p_i \pdv{}{p_i}. \plabel{eq:liouville vector field in coordinates}
    \end{IEEEeqnarray}
    By \cref{thm:flow reeb vs hamiltonian} and Equation \eqref{eq:hamiltonian vector field in coordinates}, and since the coordinates are normal, the Reeb vector field is given by
    \begin{IEEEeqnarray}{rCl}
        R_u
        & = & \sum_{i=1}^{n} p_i(u) \pdv{}{q^i}\Big|_{u}. \plabel{eq:reeb vector field in coordinates}
    \end{IEEEeqnarray}
    Equations \eqref{eq:liouville vector field in coordinates} and \eqref{eq:reeb vector field in coordinates} together with equations \eqref{eq:p horizontal in coordinates} and \eqref{eq:p vertical in coordinates} imply Equations \eqref{eq:p of vfs 1} and \eqref{eq:p of vfs 2}.
\end{proof}

Define
\begin{IEEEeqnarray*}{rCls+x*}
    \mathcal{T}(c^* TL)                & \coloneqq & 
    \left\{
        \kappa
        \ \middle\vert
        \begin{array}{l}
            \kappa \text{ is an isometric trivialization of } c^* TL \\
            \text{such that } \kappa_t (\dot{c}(t)) = e_1 \in \R^n \text{ for every } t \in \R / \ell \Z
        \end{array}
    \right\}, \\
    \mathcal{T}(\gamma^* \xi)          & \coloneqq & \{ \tau \mid \tau \text{ is a symplectic trivialization of } \gamma^* \xi \}, \\
    \mathcal{T}((z \circ c)^* T T^* L) & \coloneqq & \{ \sigma \mid \sigma \text{ is a symplectic trivialization of } (z \circ c)^* T T^* L \}.
\end{IEEEeqnarray*}
We will define maps $\tau$, $\sigma_0$ and $\sigma$ (see \cref{def:map of trivializations tau,def:map of trivializations sigma 0,def:map of trivializations sigma}) which fit into the following diagram.
\begin{IEEEeqnarray}{c+x*}
    \plabel{eq:diagram of maps of trivializations}
    \begin{tikzcd}
        \mathcal{T}(c^* TL) \ar[d, swap, "\tau"] \ar[dr, "\sigma"] \\
        \mathcal{T}(\gamma^* \xi) \ar[r, swap, "\sigma_0"] & \mathcal{T}((z \circ c)^* T T^* L)
    \end{tikzcd}
\end{IEEEeqnarray}
We will check that this diagram commutes in \cref{lem:diagram of maps of trivalizations commutes}. Consider the following diagram of symplectic vector spaces and symplectic linear maps.
\begin{IEEEeqnarray*}{c+x*}
    \begin{tikzcd}
        \xi_{\gamma(t)}^{} \ar[r, "\iota_{\xi_{\gamma(t)}}"] & \xi^{\perp}_{\gamma(t)} \oplus \xi_{\gamma(t)}^{} \ar[r, equals] & T_{\gamma(t)}^{} T^* L \ar[r, "P_{\gamma(t)}"] & T^*_{c(t)} L \oplus T_{c(t)}^{} L & T_{z \circ c(t)}^{} T^* L \ar[l, swap, "P_{z \circ c(t)}"]
    \end{tikzcd}
\end{IEEEeqnarray*}
We now define the maps $\tau$, $\sigma_0$ and $\sigma$.

\begin{definition}
    \phantomsection\label{def:map of trivializations tau}
    For every $\kappa \in \mathcal{T}(c^* TL)$, we define $\tau(\kappa) \in \mathcal{T}(\gamma^* \xi)$ by 
    \begin{IEEEeqnarray*}{c+x*}
        \tau(\kappa)_t \coloneqq \pi_{\R^{2n-2}} \circ \tilde{\kappa}_t \circ P_{\gamma(t)} \circ \iota_{\xi_{\gamma(t)}},
    \end{IEEEeqnarray*}
    where $\tilde{\kappa}_t \colon T^*_{c(t)} L \oplus T_{c(t)}^{} L \longrightarrow \R^n \oplus \R^n$ and $\pi_{\R^{2n-2}} \colon \R^{2n} \longrightarrow \R^{2n-2}$ are given by
    \begin{IEEEeqnarray*}{rCl}
        \tilde{\kappa}_t(u,v)                          & \coloneqq & (\kappa_t \circ \tilde{g}^{-1}_{c(t)}(u), \kappa_t(v)), \\
        \pi_{\R^{2n-2}}(x^1,\ldots,x^n,y^1,\ldots,y^n) & \coloneqq & (x^2,\ldots,x^n,y^2,\ldots,y^n).
    \end{IEEEeqnarray*}
\end{definition}

For \cref{def:map of trivializations tau} to be well-posed, we need $\tilde{\kappa}_t$ to be a symplectic linear map. We check this in \cref{lem:kappa tl is symplectic} below.

\begin{definition}
    \phantomsection\label{def:map of trivializations sigma 0}
    For every $\tau \in \mathcal{T}(\gamma^* \xi)$, we define $\sigma_0(\tau) \in \mathcal{T}((z \circ c)^* T T^*L)$ by
    \begin{IEEEeqnarray*}{c+x*}
        \sigma_0 (\tau)_t \coloneqq \tilde{\tau}_t \circ P^{-1}_{\gamma(t)} \circ P_{z \circ c(t)},
    \end{IEEEeqnarray*}
    where $\tilde{\tau}_t \colon \xi^{\perp}_{\gamma(t)} \oplus \xi_{\gamma(t)}^{} \longrightarrow \R^{2n}$ is the symplectic linear map given by
    \begin{IEEEeqnarray*}{rCls+x*}
        \tilde{\tau}_t (Z_{\gamma(t)}) & = & e_1, \\
        \tilde{\tau}_t (R_{\gamma(t)}) & = & e_{n+1}, \\
        \tilde{\tau}_t (v)             & = & \iota_{\R^{2n-2}} \circ \tau_t(v), \quad \text{for every } v \in \xi_{\gamma(t)},
    \end{IEEEeqnarray*}
    and $\iota_{\R^{2n-2}} \colon \R^{2n-2} \longrightarrow \R^{2n}$ is given by
    \begin{IEEEeqnarray*}{c+x*}
        \iota_{\R^{2n-2}}(x^2,\ldots,x^n,y^2,\ldots,y^n) = (0,x^2,\ldots,x^n,0,y^2,\ldots,y^n).
    \end{IEEEeqnarray*}
\end{definition}
    
\begin{definition}
    \label{def:map of trivializations sigma}
    For every $\kappa \in \mathcal{T}(c^* TL)$, we define $\sigma(\kappa) \in \mathcal{T}((z \circ c)^* T T^*L)$ by
    \begin{IEEEeqnarray*}{c+x*}
        \sigma(\kappa)_t \coloneqq \tilde{\kappa}_t \circ P_{z \circ c(t)}.
    \end{IEEEeqnarray*}
\end{definition}
        
\begin{lemma}
    \label{lem:kappa tl is symplectic}
    The map $\tilde{\kappa}_t$ from \cref{def:map of trivializations tau,def:map of trivializations sigma} is symplectic.
\end{lemma}
\begin{proof}
    For $(u,v), (x,y) \in T^*_{c(t)} L \oplus T_{c(t)}^{} L$, we have
    \begin{IEEEeqnarray*}{rCls+x*}
        \IEEEeqnarraymulticol{3}{l}{\omega_{\R^n \oplus \R^n} \p{}{1}{ \tilde{\kappa}_t \p{}{}{u,v}, \tilde{\kappa}_t \p{}{}{x,y} } }\\ \
        & = & \omega_{\R^n \oplus \R^n} \p{}{1}{ \p{}{1}{ \kappa_t \circ \tilde{g}_{c(t)}^{-1} (u), \kappa_t (v)}, \p{}{1}{ \kappa_t \circ \tilde{g}_{c(t)}^{-1} (x), \kappa_t (y)} } & \quad [\text{by definition of $\tilde{\kappa}_t$}] \\
        & = & \p{<}{1}{ \kappa_t \circ \tilde{g}_{c(t)}^{-1} (u), \kappa_t (y) }_{\R^n} - \p{<}{1}{ \kappa_t \circ \tilde{g}_{c(t)}^{-1} (x), \kappa_t (v) }_{\R^n}                   & \quad [\text{by definition of $\omega_{\R^n \oplus \R^n}$}] \\
        & = & \p{<}{1}{ \tilde{g}_{c(t)}^{-1} (u), y }_{TL} - \p{<}{1}{ \tilde{g}_{c(t)}^{-1} (x), v }_{TL}                                                                           & \quad [\text{since $\kappa_t$ is an isometry}] \\
        & = & u(y) - x(v)                                                                                                                                                             & \quad [\text{by definition of $\tilde{g}$}] \\
        & = & \omega_{T^*L \oplus TL} \p{}{1}{(u,v),(x,y)}                                                                                                                            & \quad [\text{by definition of $\omega_{T^*L \oplus TL}$}].    & \qedhere
    \end{IEEEeqnarray*}
\end{proof}

\begin{lemma}
    \label{lem:diagram of maps of trivalizations commutes}
    Diagram \eqref{eq:diagram of maps of trivializations} commutes, i.e. $\sigma = \sigma_0 \circ \tau$.
\end{lemma}
\begin{proof}
    By \cref{def:map of trivializations tau,def:map of trivializations sigma 0,def:map of trivializations sigma},
    \begin{IEEEeqnarray*}{rCls+x*}
        \sigma(\kappa)_t       & = & \tilde{\kappa}_t \circ P_{z \circ c(t)}, \\
        \sigma_0(\tau(\kappa)) & = & \widetilde{\tau(\kappa)}_t \circ P_{\gamma(t)}^{-1} \circ P_{z \circ c(t)}.
    \end{IEEEeqnarray*}
    Therefore, it is enough to show that $\tilde{\kappa}_t \circ P_{\gamma(t)} = \widetilde{\tau(\kappa)}_t \colon T_{\gamma(t)} T^*L \longrightarrow \R^{2n}$. We show that $\tilde{\kappa}_t \circ P_{\gamma(t)}(Z_{\gamma(t)}) = \widetilde{\tau(\kappa)}_t(Z_{\gamma(t)})$.
    \begin{IEEEeqnarray*}{rCls+x*}
        \tilde{\kappa}_{t} \circ P_{\gamma(t)} (Z_{\gamma(t)})
        & = & \tilde{\kappa}_t(\gamma(t), 0)                       & \quad [\text{by \cref{prop:properties of p}}] \\
        & = & (\kappa_t \circ \tilde{g}^{-1}_{c(t)}(\gamma(t)), 0) & \quad [\text{by definition of $\tilde{\kappa}_t$}] \\
        & = & (\kappa_t(\dot{c}(t)), 0)                            & \quad [\text{by definition of $\gamma$}] \\
        & = & (e_1,0)                                              & \quad [\text{since $\kappa \in \mathcal{T}(c^* TL)$}] \\
        & = & \widetilde{\tau(\kappa)}_t (Z_{\gamma(t)})           & \quad [\text{by definition of $\widetilde{\tau(\kappa)}_t$}].
    \end{IEEEeqnarray*}
    We show that $\tilde{\kappa}_t \circ P_{\gamma(t)}(R_{\gamma(t)}) = \widetilde{\tau(\kappa)}_t(R_{\gamma(t)})$.
    \begin{IEEEeqnarray*}{rCls+x*}
        \tilde{\kappa}_{t} \circ P_{\gamma(t)} (R_{\gamma(t)})
        & = & \tilde{\kappa}_t(0, \tilde{g}^{-1}_{c(t)}(\gamma(t))) & \quad [\text{by \cref{prop:properties of p}}] \\
        & = & (0, \kappa_t \circ \tilde{g}^{-1}_{c(t)}(\gamma(t)))  & \quad [\text{by definition of $\tilde{\kappa}_t$}] \\
        & = & (0, \kappa_t(\dot{c}(t)))                             & \quad [\text{by definition of $\gamma$}] \\
        & = & (0,e_1)                                               & \quad [\text{since $\kappa \in \mathcal{T}(c^* TL)$}] \\
        & = & \widetilde{\tau(\kappa)}_t (R_{\gamma(t)})            & \quad [\text{by definition of $\widetilde{\tau(\kappa)}_t$}].
    \end{IEEEeqnarray*}
    The previous computations show that
    \begin{IEEEeqnarray*}{c+x*}
        P_{\gamma(t)} \circ \tilde{\kappa}_t (\xi_{\gamma(t)}^{\perp}) = \ker \pi_{\R^{2n-2}},
    \end{IEEEeqnarray*}
    which in turn implies that
    \begin{IEEEeqnarray}{c+x*}
        \plabel{eq:image of p kappa}
        P_{\gamma(t)} \circ \tilde{\kappa}_t (\xi_{\gamma(t)}) = (\ker \pi_{\R^{2n-2}})^{\perp} = \img \iota_{\R^{2n - 2}}.
    \end{IEEEeqnarray}
    Finally, we show that $\tilde{\kappa}_t \circ P_{\gamma(t)}(v) = \widetilde{\tau(\kappa)}_t(v)$ for every $v \in \xi_{\gamma(t)}$.
    \begin{IEEEeqnarray*}{rCls+x*}
        \widetilde{\tau(\kappa)}_t (v)
        & = & \iota_{\R^{2n-2}} \circ \tau(\kappa)_t (v)                                                                           & \quad [\text{by definition of $\widetilde{\tau(\kappa)}_t$}] \\
        & = & \iota_{\R^{2n-2}} \circ \pi_{\R^{2n-2}} \circ \tilde{\kappa}_t \circ P_{\gamma(t)} \circ \iota_{\xi_{\gamma(t)}} (v) & \quad [\text{by definition of $\tau$}] \\
        & = & \tilde{\kappa}_t \circ P_{\gamma(t)}(v)                                                                              & \quad [\text{by Equation \eqref{eq:image of p kappa}}].        & \qedhere
    \end{IEEEeqnarray*}
\end{proof}

This finishes the ``construction'' of diagram \eqref{eq:diagram of maps of trivializations}. Our goal is to show that $\conleyzehnder^{\tau(\kappa)}(\gamma)$ is independent of the choice of $\kappa \in \mathcal{T}(c^* TL)$ (see \cref{thm:index of geodesic or reeb orbit isometric triv}). Indeed, we will actually show that $\conleyzehnder^{\tau(\kappa)}(\gamma) = \morse(c)$. To make sense of this statement, we start by explaining the meaning of the Morse index of a geodesic.

\begin{remark}
    \label{rmk:morse theory for geodesics}
    Define $X \coloneqq W^{1,2}(\R / \ell \Z,L)$ (maps from $\R / \ell \Z$ to $L$ of Sobolev class $W ^{1,2}$). Then, $X$ is a Hilbert manifold. At $c \in X$, the tangent space of $X$ is
    \begin{IEEEeqnarray*}{c+x*}
        T_{c} X = W ^{1,2}(\R / \ell \Z,c^* TL),
    \end{IEEEeqnarray*}
    which is a Hilbert space. We can define the \textbf{Energy functional} by
    \begin{IEEEeqnarray*}{rrCl}
        E \colon & X & \longrightarrow & \R \\
                 & c & \longmapsto     & \frac{1}{2} \int_{\R / \ell \Z}^{} \p{||}{}{ \dot{c}(t) }^2 \edv t.
    \end{IEEEeqnarray*}
    Then, $c \in X$ is a critical point of $E$ if and only if $c$ is smooth and a geodesic in $L$. We say that $c$ is \textbf{nondegenerate} if the kernel of the map
    \begin{IEEEeqnarray*}{c+x*}
        \operatorname{Hess} E (c) \colon T _{c} X \longrightarrow T _{c}^* X
    \end{IEEEeqnarray*}
    is $\ker \operatorname{Hess} E(c) = \p{<}{}{\dot{c}}$. If $c$ is a critical point of $E$, i.e. a geodesic, then we define the \textbf{Morse index} of $c$ by
    \begin{IEEEeqnarray*}{c+x*}
        \morse(c) = \sup
        \left\{
            \dim V
            \ \middle\vert
            \begin{array}{l}
                V \text{ is a subspace of } T _{c} X, \\
                \operatorname{Hess} E (c)|_V \colon V \times V \longrightarrow \R \text{ is negative definite}
            \end{array}
        \right\}.
    \end{IEEEeqnarray*}
    Recall that $c$ is a geodesic if and only if $\gamma \coloneqq \tilde{g} \circ \dot{c}$ is a Reeb orbit. In this case, $c$ is a nondegenerate critical point of $E$ if and only if ${\gamma}$ is a nondegenerate Reeb orbit.
\end{remark}

\begin{definition}
    \phantomsection\label{lem:maslov index of a geodesic}
    For $\sigma \in \mathcal{T}((z \circ c)^* T T^* L)$, we define the \textbf{Maslov index} of $c$ with respect to $\sigma$, denoted $\maslov^{\sigma}(c)$, as follows. First, let $W^{c,\sigma}$ be the loop of Lagrangian subspaces of $\R^{2n}$ given by
    \begin{IEEEeqnarray*}{c+x*}
        W^{c,\sigma}(t) \coloneqq \sigma_t \circ \dv z(c(t)) (T_{c(t)} L).
    \end{IEEEeqnarray*}
    Then, define $\maslov^{\sigma}(c)$ to be the Maslov index of $W^{c,\sigma}$ in the sense of \cref{thm:maslov lagrangian properties}.
\end{definition}

\begin{lemma}
    \label{lem:maslov index of a geodesic is zero}
    For any $\kappa \in \mathcal{T}(c^* TL)$,
    \begin{IEEEeqnarray*}{c+x*}
        \maslov^{\sigma(\kappa)}(c) = 0.
    \end{IEEEeqnarray*}
\end{lemma}
\begin{proof}
    We will show that $W^{c,\sigma(\kappa)} = \{0\} \oplus \R^{n}$. By the zero property of the Maslov index for a path of Lagrangian subspaces, this implies the result. We start by showing that $P^V_{z(x)} \circ \dv z(x) = 0$ for any $x \in L$. For any $w \in T_x L$,
    \begin{IEEEeqnarray*}{rCls+x*}
        \IEEEeqnarraymulticol{3}{l}{P^V_{z(x)} \circ \dv z(x) w}\\ \quad
        & = & (I^V_{z(x)})^{-1} (\dv z(x) w - I^H_{z(x)} \circ P^H_{z(x)} (\dv z(x) w))       & \quad [\text{by definition of $P^V$}] \\
        & = & (I^V_{z(x)})^{-1} (\dv z(x) w - \dv z(x) \circ \dv \pi (z(x)) \circ \dv z(x) w) & \quad [\text{by definition of $I^H$ and $P^H$}] \\
        & = & 0                                                                               & \quad [\text{since $\pi \circ z = \id_L$}].
    \end{IEEEeqnarray*}
    We compute $W^{c,\sigma(\kappa)}$.
    \begin{IEEEeqnarray*}{rCls+x*}
        W^{c,\sigma(\kappa)}
        & = & \sigma(\kappa)_t \circ \dv z(c(t)) (T_{c(t)} L)                          & \quad [\text{by definition of $W^{c,\sigma(\kappa)}$}] \\
        & = & \tilde{\kappa}_t \circ P_{z \circ c(t)} \circ \dv z(c(t))(T_{c(t)} L)    & \quad [\text{by definition of $\sigma(\kappa)$}] \\
        & = & \tilde{\kappa}_t (0, P^H_{z \circ c(t)} \circ \dv z(c(t)) (T_{c(t)} L) ) & \quad [\text{since $P^V_{z(c(t))} \circ \dv z(c(t)) = 0$}] \\
        & = & (0, \kappa_t \circ P^H_{z \circ c(t)} \circ \dv z(c(t)) (T_{c(t)} L) )   & \quad [\text{by definition of $\tilde{\kappa}_t$}] \\
        & = & (0, \kappa_t(T_{c(t)} L))                                                & \quad [\text{since $P^H_{z \circ c(t)} = \dv \pi(z \circ c(t))$}] \\
        & = & \{0\} \oplus \R^n                                                        & \quad [\text{since $\kappa_t$ is an isomorphism}].                  & \qedhere
    \end{IEEEeqnarray*}
\end{proof}

The following theorem was originally proven in \cite{viterboNewObstructionEmbedding1990}, but we will use a restatement of it from \cite{cieliebakPuncturedHolomorphicCurves2018}.

\begin{theorem}[{\cite[Lemma 2.1]{cieliebakPuncturedHolomorphicCurves2018}}]
    \label{thm:index of geod reeb}
    For any $\tau \in \mathcal{T}(\gamma^* \xi)$,
    \begin{IEEEeqnarray*}{c+x*}
        \conleyzehnder^{\tau}({\gamma}) + \maslov^{\sigma_0(\tau)}(c) = \morse(c).
    \end{IEEEeqnarray*}
\end{theorem}

\begin{theorem}
    \label{thm:index of geodesic or reeb orbit isometric triv}
    For any $\kappa \in \mathcal{T}(c^* TL)$,
    \begin{IEEEeqnarray*}{c+x*}
        \conleyzehnder^{\tau(\kappa)}({\gamma}) = \morse(c).
    \end{IEEEeqnarray*}
\end{theorem}
\begin{proof}
    By \cref{lem:diagram of maps of trivalizations commutes,lem:maslov index of a geodesic is zero,thm:index of geod reeb}.
\end{proof}

Finally, we state a result which will be necessary to prove \cref{thm:lagrangian vs g tilde}.

\begin{lemma}[{\cite[Lemma 2.2]{cieliebakPuncturedHolomorphicCurves2018}}]
    \label{lem:geodesics lemma CM abs}
    Let $L$ be a compact $n$-dimensional manifold without boundary. Let $\mathrm{Riem}(L)$ be the set of Riemannian metrics on $L$, equipped with the $C^2$-topology. If $g_0 \in \mathrm{Riem}(L)$ is a Riemannian metric of nonpositive sectional curvature and $\mathcal{U} \subset \mathrm{Riem}(L)$ is an open neighbourhood of $g_0$, then for all $\ell_0 > 0$ there exists a Riemannian metric $g \in \mathcal{U}$ on $L$ such that with respect to $g$, any closed geodesic $c$ in $L$ of length $\ell(c) \leq \ell_0$ is noncontractible, nondegenerate, and such that $0 \leq \morse(c) \leq n - 1$.
\end{lemma}

\chapter{Holomorphic curves}
\label{chp:holomorphic curves}

\section{Holomorphic curves}

In this section we define asymptotically cylindrical holomorphic curves (see \cref{def:asy cyl holomorphic curve}). The domain of such a curve is a punctured Riemann surface (see \cref{def:punctures asy markers cyl ends}), and the target is a symplectic cobordism (see \cref{def:symplectic cobordism}).

\begin{definition}
    \label{def:punctures asy markers cyl ends}
    Let $(\Sigma, j)$ be a Riemann surface. A \textbf{puncture} on $\Sigma$ is a point $z \in \Sigma$. Denote by $D$ the closed unit disk in $\C$ and by $Z^{\pm}$ the positive or negative half-cylinders:%
    \begin{IEEEeqnarray*}{rCls+x*}
        Z^+ & \coloneqq & \R_{\geq 0} \times S^1, \\
        Z^- & \coloneqq & \R_{\leq 0} \times S^1,
    \end{IEEEeqnarray*}
    with coordinates $(s,t) \in Z^{\pm}$ and complex structure $j$ given by $j(\partial_s) = \partial_t$. Consider the holomorphic maps
    \begin{IEEEeqnarray*}{rClCrCl}
        \psi^\pm \colon Z^{\pm} & \longrightarrow & D \setminus \{0\}, & \quad & \psi^\pm(s,t) & = & \exp(\mp 2 \pi (s + i t)).
    \end{IEEEeqnarray*} 
    A positive or negative \textbf{cylindrical end} near $z$ is a holomorphic embedding $\phi^{\pm} \colon Z^{\pm} \longrightarrow \Sigma \setminus \{z\}$ of the form $\phi^{\pm} \coloneqq \varphi \circ \psi^\pm$, where $\varphi \colon D \longrightarrow \Sigma$ is a holomorphic embedding such that $\varphi(0) = z$. In this case, we say that $(s,t)$ are \textbf{cylindrical coordinates} near $z$. A \textbf{punctured Riemann surface} is a Riemann surface $(\Sigma, j)$ together with sets
    \begin{IEEEeqnarray*}{rClCrCl}
        \mathbf{z} & = & \mathbf{z}^+ \cup \mathbf{z}^-, & \quad & \mathbf{z}^{\pm} & = & \{z^{\pm}_1,\ldots,z^{\pm}_{p^{\pm}}\} \subset \Sigma, \quad \mathbf{z}^+ \cap \mathbf{z}^- = \varnothing,
    \end{IEEEeqnarray*}
    of positive and negative punctures. In this case, we denote $\dot{\Sigma} \coloneqq \Sigma \setminus \mathbf{z}$. Whenever we talk about cylindrical coordinates near a puncture, it is implicit that we mean the cylindrical coordinates induced from a positive of negative cylindrical end, in accordance to whether the puncture is positive or negative.
\end{definition}

\begin{definition}
    \label{def:symplectic cobordism}
    A \textbf{symplectic cobordism} is a compact symplectic manifold $(X, \omega)$ with boundary $\partial X$, together with a $1$-form $\lambda$ defined on an open neighbourhood of $\partial X$, such that $\edv \lambda = \omega$ and the restriction of $\lambda$ to $\partial X$ is a contact form. Let $\partial^+ X$ (respectively $\partial^- X$) be the subset of $\partial X$ where the orientation defined by $\lambda|_{\partial X}$ as a contact form agrees with the boundary orientation (respectively negative boundary orientation).
\end{definition}

\begin{definition}
    \phantomsection\label{def:liouville cobordism}
    A \textbf{Liouville cobordism} is a symplectic cobordism $(X,\omega,\lambda)$ such that $\lambda$ is defined on $X$.
\end{definition}

\begin{example}
    A Liouville domain is a Liouville cobordism whose negative boundary is empty.
\end{example}

\begin{remark}
    We can define the completion of a symplectic cobordism $(X,\omega,\lambda)$ like in \cref{sec:completion of liouville domain}, with the difference that now we attach half-symplectizations to the negative and positive boundaries:
    \begin{IEEEeqnarray*}{c+x*}
        \hat{X} \coloneqq \R_{\leq 0} \times \partial^- X \cup_{\partial^- X} X \cup_{\partial^+ X} \R_{\geq 0} \times \partial^+ X. 
    \end{IEEEeqnarray*}
\end{remark}

\begin{definition}
    \label{def:admissible}
    Let $(X,\omega,\lambda)$ be a symplectic cobordism and consider its completion $\hat{X}$. An almost complex structure $J$ on $\hat{X}$ is \textbf{cylindrical} if $J$ is compatible with $\hat{\omega}$ and $J$ is cylindrical on $\R_{\geq 0} \times \partial^+ X$ and $\R_{\leq 0} \times \partial^- X$. Denote by $\mathcal{J}(X)$ the set of such $J$.
\end{definition}

\begin{definition}
    \label{def:asy cyl holomorphic curve}
    Let $(X, \omega, \lambda)$ be a symplectic cobordism, $J \in \mathcal{J}(X)$ be a cylindrical almost complex structure on $\hat{X}$ and $\Gamma^{\pm} = (\gamma^{\pm}_1, \ldots, \gamma^{\pm}_{p^{\pm}})$ be tuples of Reeb orbits in $\partial^{\pm} X$. Let $T_{i}^{\pm}$ denote the period of $\gamma_i^{\pm}$. An \textbf{asymptotically cylindrical holomorphic curve} in $\hat{X}$ from $\Gamma^-$ to $\Gamma^+$ is given by a Riemann surface $(\Sigma, j)$ with punctures $\mathbf{z}^{\pm} = \{z_1^{\pm}, \ldots, z^{\pm}_{p^{\pm}}\}$ together with a $J$-holomorphic map $u \colon \dot{\Sigma} \longrightarrow \hat{X}$, such that:
    \begin{enumerate}
        \item $u$ is positively asymptotic to $\gamma^{+}_i$ at $z^{+}_{i}$, i.e. there exist cylindrical coordinates $(s,t)$ near $z_i^+$ such that $u(s,t) \in \R_{\geq 0} \times \partial^+ X$ for $s$ big enough and
            \begin{IEEEeqnarray*}{rrCl}
                \lim_{s \to + \infty} & \pi_{\R} \circ u(s,t)           & = & + \infty, \\
                \lim_{s \to + \infty} & \pi_{\partial^+ X} \circ u(s,t) & = & \gamma^+_i(t T^+_i);
            \end{IEEEeqnarray*} 
        \item $u$ is negatively asymptotic to $\gamma^{-}_i$ at $z^{-}_{i}$, i.e. there exist cylindrical coordinates $(s,t)$ near $z_i^-$ such that $u(s,t) \in \R_{\leq 0} \times \partial^- X$ for $s$ small enough and
            \begin{IEEEeqnarray*}{rrCl}
                \lim_{s \to - \infty} & \pi_{\R} \circ u(s,t)           & = & - \infty, \\
                \lim_{s \to - \infty} & \pi_{\partial^- X} \circ u(s,t) & = & \gamma^-_i(t T^-_i).
            \end{IEEEeqnarray*} 
    \end{enumerate}
\end{definition}

We now explain some analytical properties of asymptotically cylindrical holomorphic curves. The key results are the maximum principle (\cref{thm:maximum principle holomorphic}) and a lemma comparing the energy of such a curve and the action of the asymptotic Reeb orbits (\cref{lem:action energy for holomorphic}). The following lemma is an auxiliary result which will allow us to prove that the energy (see \cref{def:energy of a asy cylindrical holomorphic curve}) is a nonnegative number.

\begin{lemma}
    \label{lem:holomorphic curves in symplectizations}
    Let $(M, \alpha)$ be a contact manifold and $J$ be a cylindrical almost complex structure on $\R \times M$. If $u = (a, f) \colon \dot{\Sigma} \longrightarrow \R \times M$ is a holomorphic curve, then $f^* \edv \alpha \geq 0$ and%
    \begin{IEEEeqnarray}{rCls+x*}
        - \edv a \circ j              & = & f^* \alpha                                \plabel{eq:holomorphic curves in symplectizations 1} \\
        \pi_{\xi} \circ \dv f \circ j & = & J_{\xi}({f}) \circ \pi_{\xi} \circ \dv f. \plabel{eq:holomorphic curves in symplectizations 2}
    \end{IEEEeqnarray}
\end{lemma}
\begin{proof}
    We prove equation \eqref{eq:holomorphic curves in symplectizations 1}:
    \begin{IEEEeqnarray*}{rCls+x*}
        - \edv a \circ j
        & = & - \edv r \circ \dv u \circ j      & \quad [\text{by definition of $a$}] \\
        & = & - \edv r \circ J({u}) \circ \dv u & \quad [\text{${u}$ is holomorphic}] \\
        & = & \alpha \circ \dv u                & \quad [\text{by \cref{lem:J cylindrical forms}}] \\
        & = & f^* \alpha                        & \quad [\text{by definition of pullback}].
    \end{IEEEeqnarray*}
    Equation \eqref{eq:holomorphic curves in symplectizations 2} follows by applying $\pi_{\xi} \colon T(\R \times M) \longrightarrow \xi$ to the equation $J \circ Tu = Tu \circ j$. We show that $f^* \edv \alpha \geq 0$:
    \begin{IEEEeqnarray*}{rCls+x*}
        \IEEEeqnarraymulticol{3}{l}{f^*\edv \alpha(S, j (S))}\\ \quad 
        & =    & \edv \alpha (\dv f (S), \dv f \circ j (S))                                          & \quad [\text{by definition of pullback}]                                                     \\
        & =    & \edv \alpha (\pi_{\xi} \circ \dv f (S), \pi_{\xi} \circ \dv f \circ j (S))          & \quad [\text{$TM = \p{<}{}{R} \directsum \xi = \ker \edv \alpha \directsum \ker \alpha$}] \\
        & =    & \edv \alpha (\pi_{\xi} \circ \dv f (S), J_{\xi}(f) \circ \pi_{\xi} \circ \dv f (S)) & \quad [\text{by Equation \eqref{eq:holomorphic curves in symplectizations 2}}]            \\
        & =    & \| \pi_{\xi} \circ \dv f (S) \|^2_{J_{\xi}({f}), \edv \alpha}                       & \quad [\text{since $J$ is cylindrical}]                                          \\
        & \geq & 0.                                                                                  &                                                                                             & \qedhere
    \end{IEEEeqnarray*}
\end{proof}

\begin{lemma}
    \label{lem:laplacian}
    Let $\omega_{\dot{\Sigma}}$ be a symplectic form on $\dot{\Sigma}$ such that $g_{\dot{\Sigma}} \coloneqq \omega_{\dot{\Sigma}}(\cdot, j \cdot)$ is a Riemannian metric. Denote by $\operatorname{dvol}_{\dot{\Sigma}}$ the Riemannian volume element of $\dot{\Sigma}$. Let $a$ be a function on $\dot{\Sigma}$ and consider the Laplacian of $a$, $\Delta a \coloneqq \operatorname{div} (\nabla a)$. Then, $\omega_{\dot{\Sigma}} = \operatorname{dvol}_{\dot{\Sigma}}$ and 
    \begin{IEEEeqnarray*}{c+x*}
        \Delta a \, \omega_{\dot{\Sigma}} = - \edv (\edv a \circ j).
    \end{IEEEeqnarray*}
\end{lemma}
\begin{proof}
    For any unit vector $S \in T \dot{\Sigma}$, if we define $T \coloneqq j (S)$ then $\{S, T\}$ is an orthonormal basis of $T \dot{\Sigma}$ and $\omega_{\dot{\Sigma}}(S, T) = 1$, which implies $\omega_{\dot{\Sigma}} = \operatorname{dvol}_{\dot{\Sigma}}$. We now prove the formula for the Laplacian.
    \begin{IEEEeqnarray*}{rCls+x*}
        \Delta a \, \omega_{\dot{\Sigma}}
        & = & \operatorname{div} (\nabla a) \omega_{\dot{\Sigma}} & \quad [\text{by definition of Laplacian}]                                                       \\
        & = & \ldv{\nabla a} \omega_{\dot{\Sigma}}                & \quad [\text{by definition of divergence and $\omega_{\dot{\Sigma}} = \operatorname{dvol}_{\dot{\Sigma}}$}] \\
        & = & \edv \iota_{\nabla a} \omega_{\dot{\Sigma}}         & \quad [\text{by the Cartan magic formula}].
    \end{IEEEeqnarray*}
    It remains to show that $\iota_{\nabla a} \omega_{\dot{\Sigma}} = - \edv a \circ j$.
    \begin{IEEEeqnarray*}{rCls+x*}
        \iota_{\nabla a} \omega_{\dot{\Sigma}} (S)
        & = & \omega_{\dot{\Sigma}} (\nabla a, S)               & \quad [\text{by definition of interior product}]         \\
        & = & - \omega_{\dot{\Sigma}} (\nabla a, j \circ j (S)) & \quad [\text{by definition of almost complex structure}] \\
        & = & - g_{\dot{\Sigma}} (\nabla a, j (S))              & \quad [\text{by definition of $g_{\dot{\Sigma}}$}]             \\
        & = & - \edv a \circ j (S)                              & \quad [\text{by definition of gradient}].                        & \qedhere
    \end{IEEEeqnarray*}
\end{proof}

\begin{lemma}[maximum principle]
    \label{thm:maximum principle holomorphic}
    Assume that $\dot{\Sigma}$ is connected. Let $(M, \alpha)$ be a contact manifold and $J$ be a cylindrical almost complex structure on $\R \times M$. If 
    \begin{IEEEeqnarray*}{c+x*}
        u = (a, f) \colon \dot{\Sigma} \longrightarrow \R \times M
    \end{IEEEeqnarray*}
    is a holomorphic curve and $a \colon \dot{\Sigma} \longrightarrow \R$ has a local maximum then $a$ is constant.
\end{lemma}
\begin{proof}
    Define $L = -\Delta$. The operator $L$ is a linear elliptic partial differential operator (as in \cite[p.~312]{evansPartialDifferentialEquations2010}). We show that $L a \leq 0$. For this, choose $\omega_{\dot{\Sigma}}$ a symplectic structure on $\dot{\Sigma}$ such that $g_{\dot{\Sigma}} \coloneqq \omega_{\dot{\Sigma}}(\cdot, j \cdot)$ is a Riemannian metric.
    \begin{IEEEeqnarray*}{rCls+x*}
        L a \, \omega_{\dot{\Sigma}}
        & =    & - \Delta a \, \omega_{\dot{\Sigma}} & \quad [\text{by definition of $L$}] \\
        & =    & \edv (\edv a \circ j)               & \quad [\text{by \cref{lem:laplacian}}]                                                    \\
        & =    & - \edv f^* \alpha                   & \quad [\text{by \cref{lem:holomorphic curves in symplectizations}}] \\
        & =    & - f^* \edv \alpha                   & \quad [\text{by naturality of exterior derivative}] \\
        & \leq & 0                                   & \quad [\text{by \cref{lem:holomorphic curves in symplectizations}}].
    \end{IEEEeqnarray*}
    This shows that $L a \leq 0$. By the strong maximum principle for elliptic partial differential operators in \cite[p.~349-350]{evansPartialDifferentialEquations2010}, if $a$ has a local maximum then $a$ is constant.
\end{proof}

\begin{lemma}
    \label{lem:integrand of energy is well-defined}
    Let $(V,j)$ be a complex vector space of real dimension 2, $(W,J,\omega,g)$ be a complex vector space with a symplectic form $\omega$ and inner product $g = \omega(\cdot,J \cdot)$, and $\phi \colon V \longrightarrow W$ be a linear map. For each choice of $s \in V$, define 
    \begin{IEEEeqnarray*}{rCls+x*}
        t                & \coloneqq & js, \\
        \{\sigma, \tau\} & \coloneqq & \text{basis of } V^* \text{ dual to } \{s,t\}, \\
        \omega_V         & \coloneqq & \sigma \wedge \tau, \\
        \| \phi \|^2     & \coloneqq & \| \phi s \|^2 + \|\phi t\|^2.
    \end{IEEEeqnarray*}
    Then,
    \begin{IEEEeqnarray*}{c+x*}
        \frac{1}{2} \| \phi \|^2 \omega_V = (\phi ^{1,0})^* \omega - (\phi ^{0,1})^* \omega,
    \end{IEEEeqnarray*}
    which is independent of the choice of $s$.
\end{lemma}
\begin{proof}
    Recall the definitions of $\phi^{1,0}$ and $\phi^{0,1}$:
    \begin{IEEEeqnarray*}{rCls+x*}
        \phi^{1,0} & \coloneqq & \frac{1}{2} (\phi - J \circ \phi \circ j), \\
        \phi^{0,1} & \coloneqq & \frac{1}{2} (\phi + J \circ \phi \circ j).
    \end{IEEEeqnarray*}
    These equations imply that $\phi^{1,0}$ is holomorphic, while $\phi^{0,1}$ is anti-holomorphic:
    \begin{IEEEeqnarray}{c+x*}
        \plabel{eq:phi holo and anti holo}
        \phi^{1,0} \circ j = J \circ \phi^{1,0}, \qquad \phi^{0,1} \circ j = - J \circ \phi^{0,1}.
    \end{IEEEeqnarray}
    Finally, we compute
    \begin{IEEEeqnarray*}{rCls+x*}
        \IEEEeqnarraymulticol{3}{l}{\| \phi \|^2 \omega_V(s,js)} \\ \quad
        & = & \| \phi (s) \|^2 + \| \phi \circ j (s) \|^2                                                               & \quad [\text{definitions of $\|\phi\|$, $\omega_V$}]  \\
        & = & \| \phi ^{1,0} (s) + \phi ^{0,1} (s) \|^2 + \| \phi ^{1,0} \circ j (s) + \phi ^{0,1} \circ j (s) \|^2     & \quad [\text{since $\phi = \phi^{1,0} + \phi^{0,1}$}] \\
        & = & \| \phi ^{1,0} (s) + \phi ^{0,1} (s) \|^2 + \| J \circ \phi ^{1,0} (s) - J \circ \phi ^{0,1} (s) \|^2     & \quad [\text{by \eqref{eq:phi holo and anti holo}}]   \\
        & = & \| \phi ^{1,0} (s) + \phi ^{0,1} (s) \|^2 + \| \phi ^{1,0} (s) - \phi ^{0,1} (s) \|^2                     & \quad [\text{since $g = \omega(\cdot, J \cdot)$}]     \\
        & = & 2 \| \phi ^{1,0} (s) \|^2 + 2 \| \phi ^{0,1} (s) \|^2                                                     & \quad [\text{by the parallelogram law}]               \\
        & = & 2 \omega (\phi ^{1,0} (s), J \circ \phi ^{1,0} (s)) + 2 \omega (\phi ^{0,1} (s), J \circ \phi ^{0,1} (s)) & \quad [\text{since $g = \omega(\cdot, J \cdot)$}]     \\
        & = & 2 \omega (\phi ^{1,0} (s), \phi ^{1,0} \circ j (s)) - 2 \omega (\phi ^{0,1} (s), \phi ^{0,1} \circ j (s)) & \quad [\text{by \eqref{eq:phi holo and anti holo}}]   \\
        & = & 2 (\phi ^{1,0})^* \omega (s,js) - 2 (\phi ^{0,1})^* \omega (s,js)                                         & \quad [\text{by definition of pullback}].             & \qedhere
    \end{IEEEeqnarray*}    
\end{proof}

\begin{definition}
    \phantomsection\label{def:energy of a asy cylindrical holomorphic curve}
    Define a piecewise smooth $2$-form $\tilde{\omega} \in \Omega^2(\hat{X})$ by
    \begin{IEEEeqnarray*}{c+x*}
        \tilde{\omega}
        \coloneqq
        \begin{cases}
            \edv \lambda|_{\partial^+ X} & \text{on } \R_{\geq 0} \times \partial^+ X, \\
            \omega                       & \text{on } X, \\
            \edv \lambda|_{\partial^- X} & \text{on } \R_{\leq 0} \times \partial^- X.
        \end{cases}
    \end{IEEEeqnarray*}
    If $u$ is an asymptotically cylindrical holomorphic curve, its \textbf{energies} are given by
    \begin{IEEEeqnarray*}{rClCl}
        E_{\hat{\omega}}(u)   & \coloneqq & \int_{\dot{\Sigma}}^{} u^* \hat{\omega}, \\
        E_{\tilde{\omega}}(u) & \coloneqq & \int_{\dot{\Sigma}}^{} u^* \tilde{\omega}.
    \end{IEEEeqnarray*}
\end{definition}

We point out that if $u$ has positive punctures, then $E_{\hat{\omega}}(u) = + \infty$. Whenever we talk about the energy of an asymptotically cylindrical holomorphic curve, we mean the $E_{\tilde{\omega}}$ energy, unless otherwise specified. We included $E_{\hat{\omega}}$ in the definition above because we will need to use it in \cref{thm:lagrangian vs g tilde} to compare the Lagrangian and the McDuff--Siegel capacities. In \cref{lem:energy wrt different forms}, we compare $E_{\hat{\omega}}$ and $E_{\tilde{\omega}}$.

\begin{lemma}
    \label{lem:action energy for holomorphic}
    If $(X, \omega, \lambda)$ is a Liouville cobordism then 
    \begin{IEEEeqnarray*}{c+x*}
        0 \leq E_{\tilde{\omega}}(u) = \mathcal{A}(\Gamma^+) - \mathcal{A}(\Gamma^-).
    \end{IEEEeqnarray*}
\end{lemma}
\begin{proof}
    Since $(X, \omega, \lambda)$ is a Liouville cobordism, $E_{\tilde{\omega}}(u)$ is given by
    \begin{IEEEeqnarray*}{rCls+x*}
        E_{\tilde{\omega}}(u)
        & = & \int_{\dot{\Sigma}}^{} u^* \tilde{\omega} \\
        & = & \int_{u^{-1}(\R_{\leq 0} \times \partial^- X)} u^* \edv \lambda|_{\partial^- X} + \int_{u^{-1}(X)} u^* \edv \lambda + \int_{u^{-1}(\R_{\geq 0} \times \partial^+ X)} u^* \edv \lambda|_{\partial^+ X}.
    \end{IEEEeqnarray*}
    Here, the first and third terms are nonnegative by \cref{lem:holomorphic curves in symplectizations}, while the second term is nonnegative by \cref{lem:integrand of energy is well-defined}. This shows that $E_{\tilde{\omega}}(u) \geq 0$. Since $u$ is asymptotic to $\Gamma^{\pm}$ and by Stokes' theorem, $E_{\tilde{\omega}}(u) = \mathcal{A}(\Gamma^+) - \mathcal{A}(\Gamma^-)$.
\end{proof}

\begin{lemma}
    \label{lem:energy wrt different forms}
    Assume that $\Sigma$ has no positive punctures. Let $(X, \omega, \lambda)$ be a symplectic cobordism, and $J \in \mathcal{J}(X)$ be a cylindrical almost complex structure on $\hat{X}$. Assume that the canonical symplectic embedding
    \begin{align*}
        (\R_{\leq 0} \times \partial^- X, \edv (e^r \lambda|_{\partial^- X})) \longrightarrow (\hat{X}, \hat{\omega}) & \\
        \intertext{can be extended to a symplectic embedding}
        (\R_{\leq K} \times \partial^- X, \edv (e^r \lambda|_{\partial^- X})) \longrightarrow (\hat{X}, \hat{\omega}) &     
    \end{align*}
    for some $K > 0$. Let $u \colon \dot{\Sigma} \longrightarrow \hat{X}$ be a $J$-holomorphic curve which is negatively asymptotic to a tuple of Reeb orbits $\Gamma$ of $\partial^- X$. Consider the energies $E_{\hat{\omega}}(u)$ and $E_{\tilde{\omega}}(u)$ of \cref{def:energy of a asy cylindrical holomorphic curve}. Then, 
    \begin{IEEEeqnarray}{rCls+x*}
        \mathcal{A}(\Gamma) & \leq & \frac{1  }{e^K - 1} E_{\tilde{\omega}}(u), \plabel{eq:action is bounded by vertical energy} \\
        E_{\hat{\omega}}(u) & \leq & \frac{e^K}{e^K - 1} E_{\tilde{\omega}}(u). \plabel{eq:energy is bounded by vertical energy}
    \end{IEEEeqnarray}
\end{lemma}
\begin{proof}
    It is enough to show that
    \begin{IEEEeqnarray}{rCls+x*}
        E_{\hat{\omega}}(u) - E_{\tilde{\omega}}(u) & =    & \mathcal{A}(\Gamma), \plabel{eq:vertical energy bounds 1} \\
        E_{\hat{\omega}}(u)                         & \geq & e^K \mathcal{A}(\Gamma), \plabel{eq:vertical energy bounds 2}
    \end{IEEEeqnarray}
    since these equations imply Equations \eqref{eq:action is bounded by vertical energy} and \eqref{eq:energy is bounded by vertical energy}. Since $u$ has no positive punctures, the maximum principle (\cref{thm:maximum principle holomorphic}) implies that $u$ is contained in $\R_{\leq 0} \times \partial^- X \cup X$. We prove Equation \eqref{eq:vertical energy bounds 1}. For simplicity, denote $M = \partial^- X$ and $\alpha = \lambda|_{\partial^- X}$.
    \begin{IEEEeqnarray*}{rCls+x*}
        E_{\hat{\omega}}(u) - E_{\tilde{\omega}}(u)
        & = & \int_{\dot{\Sigma}}^{} u^* (\hat{\omega} - \tilde{\omega})         & \quad [\text{by definition of $E_{\hat{\omega}}$ and $E_{\tilde{\omega}}$}] \\
        & = & \int_{u^{-1}(\R_{\leq 0} \times M)}^{} u^* \edv ((e^r - 1) \alpha) & \quad [\text{by definition of $\hat{\omega}$ and $\tilde{\omega}$}] \\
        & = & \mathcal{A}(\Gamma)                                                & \quad [\text{by Stokes' theorem}].
    \end{IEEEeqnarray*}
    We prove Equation \eqref{eq:vertical energy bounds 2}.
    \begin{IEEEeqnarray*}{rCls+x*}
        E_{\hat{\omega}}(u)
        & =    & \int_{\dot{\Sigma}}^{} u^* \hat{\omega}                                               & \quad [\text{by definition of $E_{\hat{\omega}}$}] \\
        & \geq & \int_{u^{-1}(\R_{\leq K} \times M)}^{} u^* \edv (e^r \alpha)                          & \quad [\text{by definition of $\hat{\omega}$ and $u^* \hat{\omega} \geq 0$}] \\
        & =    & e^K \int_{u^{-1}( \{K\} \times M)}^{} u^* \alpha                                      & \quad [\text{by Stokes' theorem}] \\
        & =    & e^K \int_{u^{-1}( \R_{\leq K} \times M)}^{} u^* \edv \alpha + e^K \mathcal{A}(\Gamma) & \quad [\text{by Stokes' theorem}] \\
        & \geq & e^K \mathcal{A}(\Gamma)                                                               & \quad [\text{by \cref{lem:holomorphic curves in symplectizations}}].           & \qedhere
    \end{IEEEeqnarray*}
\end{proof}

\section{Moduli spaces of Holomorphic curves}
\label{sec:moduli spaces of holomorphic curves}

If $(M, \alpha)$ is a contact manifold, we denote by $\mathcal{J}(M)$ the set of cylindrical almost complex structures on $\R \times M$ (see \cref{def:J cylindrical}). If $(X, \omega, \lambda)$ is a symplectic cobordism, we denote by $\mathcal{J}(X)$ the set of cylindrical almost complex structures on $\hat{X}$ (see \cref{def:admissible}). If $J^{\pm} \in \mathcal{J}(\partial^{\pm} X)$ is a cylindrical almost complex structure on $\R \times \partial^{\pm} X$, then we define the following subsets of $\mathcal{J}(X)$:
\begin{IEEEeqnarray*}{rCls+x*}
    \mathcal{J}^{J^+}(X)                  & \coloneqq & \{ J \in \mathcal{J}(X) \mid J = J^{+} \text{ on } \R_{\geq 0} \times \partial^+ X \}, \\
    \mathcal{J}_{J^-}^{\hphantom{J^+}}(X) & \coloneqq & \{ J \in \mathcal{J}(X) \mid J = J^{-} \text{ on } \R_{\leq 0} \times \partial^- X \}, \\
    \mathcal{J}^{J^+}_{J^-}(X)            & \coloneqq & \{ J \in \mathcal{J}(X) \mid J = J^{+} \text{ on } \R_{\geq 0} \times \partial^+ X \text{ and } J = J^{-} \text{ on } \R_{\leq 0} \times \partial^- X \}.
\end{IEEEeqnarray*}

Let $\Gamma^{\pm} = (\gamma^{\pm}_1, \ldots, \gamma^{\pm}_{p ^{\pm}})$ be a tuple of Reeb orbits in $\partial^{\pm} X$ and $J \in \mathcal{J}(X)$ be a cylindrical almost complex structure on $\hat{X}$. Define a moduli space
\begin{IEEEeqnarray*}{c+x*}
    \mathcal{M}^{J}_{X}(\Gamma^+, \Gamma^-)
    \coloneqq
    \left\{
        (\Sigma, u)
        \ \middle\vert
        \begin{array}{l}
            \Sigma \text{ is a connected closed Riemann surface} \\
            \text{of genus $0$ with punctures $\mathbf{z}^{\pm} = \{z^{\pm}_1, \ldots, z^{\pm}_{p ^{\pm}}\}$,} \\
            u \colon \dot{\Sigma} \longrightarrow \hat{X} \text{ is as in \cref{def:asy cyl holomorphic curve}}
        \end{array}
        \right\} / \sim,
\end{IEEEeqnarray*}
where $(\Sigma_0, u_0) \sim (\Sigma_1, u_1)$ if and only if there exists a biholomorphism $\phi \colon \Sigma_0 \longrightarrow \Sigma_1$ such that $u_1 \circ \phi = u_0$ and $\phi(z^{\pm}_{0,i}) = z^{\pm}_{1,i}$ for every $i = 1,\ldots,p ^{\pm}$. If $\Gamma^{\pm} = (\gamma^{\pm}_1, \ldots, \gamma^{\pm}_{p ^{\pm}})$ is a tuple of Reeb orbits on a contact manifold $M$ and $J \in \mathcal{J}(M)$, we define a moduli space $\mathcal{M}_{M}^{J}(\Gamma^+, \Gamma^-)$ of holomorphic curves in $\R \times M$ analogously. Since $J$ is invariant with respect to translations in the $\R$ direction, $\mathcal{M}_{M}^{J}(\Gamma^+, \Gamma^-)$ admits an action of $\R$ by composition on the target by a translation.

One can try to show that the moduli space $\mathcal{M}_{X}^{J}(\Gamma^+, \Gamma^-)$ is transversely cut out by showing that the relevant linearized Cauchy--Riemann operator is surjective at every point of the moduli space. In this case, the moduli space is an orbifold whose dimension is given by the Fredholm index of the linearized Cauchy--Riemann operator. However, since the curves in $\mathcal{M}_{X}^{J}(\Gamma^+, \Gamma^-)$ are not necessarily simple, this proof will in general not work, and we cannot say that the moduli space is an orbifold. However, the Fredholm theory part of the proof still works, which means that we still have a dimension formula. In this case the expected dimension given by the Fredholm theory is usually called a virtual dimension. For the moduli space above, the virtual dimension at a point $u$ is given by (see \cite[Section 4]{bourgeoisCoherentOrientationsSymplectic2004})
\begin{IEEEeqnarray*}{c}
    \operatorname{virdim}_u \mathcal{M}_{X}^{J}(\Gamma^+, \Gamma^-) = (n - 3)(2 - p^+ - p^-) + c_1^{\tau}(u^* T \hat{X}) + \conleyzehnder^{\tau} (\Gamma^+) - \conleyzehnder^{\tau} (\Gamma^-),
\end{IEEEeqnarray*}
where $\tau$ is a unitary trivialization of the contact distribution over each Reeb orbit. 

We now discuss curves satisfying a tangency constraint. Our presentation is based on \cite[Section 2.2]{mcduffSymplecticCapacitiesUnperturbed2022} and \cite[Section 3]{cieliebakPuncturedHolomorphicCurves2018}. Let $(X,\omega,\lambda)$ be a symplectic cobordism and $x \in \itr X$. A \textbf{symplectic divisor} through $x$ is a germ of a $2$-codimensional symplectic submanifold $D \subset X$ containing $x$. A cylindrical almost complex structure $J \in \mathcal{J}(X)$ is \textbf{compatible} with $D$ if $J$ is integrable near $x$ and $D$ is holomorphic with respect to $J$. We denote by $\mathcal{J}(X,D)$ the set of such almost complex structures. In this case, there are complex coordinates $(z^1, \ldots, z^n)$ near $x$ such that $D$ is given by $h(z_1,\ldots,z_n) = 0$, where $h(z_1,\ldots,z_n) = z_1$. Let $u \colon \Sigma \longrightarrow X$ be a $J$-holomorphic curve together with a marked point $w \in \Sigma$. For $k \geq 1$, we say that $u$ has \textbf{contact order $k$} to $D$ at $x$ if $u(w) = x$ and%
\begin{IEEEeqnarray*}{c+x*}
    (h \circ u \circ \varphi)^{(1)}(0) = \cdots = (h \circ u \circ \varphi)^{(k-1)}(0) = 0,
\end{IEEEeqnarray*}
for some local biholomorphism $\varphi \colon (\C,0) \longrightarrow (\Sigma, w)$. We point out that the condition of having ``contact order $k$'' as written above is equal to the condition of being ``tangent of order $k-1$'' as defined in \cite[Section 3]{cieliebakPuncturedHolomorphicCurves2018}. Following \cite{mcduffSymplecticCapacitiesUnperturbed2022}, we will use the notation $\p{<}{}{\mathcal{T}^{(k)}x}$ to denote moduli spaces of curves which have contact order $k$, i.e. we will denote them by $\mathcal{M}_{X}^{J}(\Gamma^+, \Gamma^-)\p{<}{}{\mathcal{T}^{(k)}x}$ and $\mathcal{M}_{M}^{J}(\Gamma^+, \Gamma^-)\p{<}{}{\mathcal{T}^{(k)}x}$. The virtual dimension is given by (see \cite[Equation (2.2.1)]{mcduffSymplecticCapacitiesUnperturbed2022})
\begin{IEEEeqnarray*}{l}
    \operatorname{virdim}_u \mathcal{M}_{X}^{J}(\Gamma^+, \Gamma^-)\p{<}{}{\mathcal{T}^{(k)}x} \\
    \quad = (n - 3)(2 - p^+ - p^-) + c_1^{\tau}(u^* T \hat{X}) + \conleyzehnder^{\tau} (\Gamma^+) - \conleyzehnder^{\tau} (\Gamma^-) - 2n - 2k + 4.
\end{IEEEeqnarray*}

The following theorem says that moduli spaces of simple, asymptotically cylindrical holomorphic curves are transversely cut out.

\begin{theorem}[{\cite[Proposition 6.9]{cieliebakSymplecticHypersurfacesTransversality2007}}]
    \label{thm:transversality with tangency}
    Let $(X,\omega,\lambda)$ be a symplectic cobordism, $x \in \itr X$ and $D$ be a symplectic divisor at $x$. There exists a comeagre set $\mathcal{J}_{\mathrm{reg}}(X,D) \subset \mathcal{J}(X,D)$ with the following property. If $J \in \mathcal{J}_{\mathrm{reg}}(X,D)$ is a regular almost complex structure, $\Gamma^{\pm} = (\gamma^\pm_1,\ldots,\gamma^\pm_{p^{\pm}})$ is a tuple of Reeb orbits of $\partial^{\pm} X$ and $A \in H_2(X,\Gamma^+ \cup \Gamma^-)$, then the moduli space $\mathcal{M}_{X,A,s}^J(\Gamma^+,\Gamma^-)\p{<}{}{\mathcal{T}^{(k)}x} \subset \mathcal{M}_{X}^J(\Gamma^+,\Gamma^-)\p{<}{}{\mathcal{T}^{(k)}x}$ of simple curves representing the homology class $A$ is a manifold of dimension
    \begin{IEEEeqnarray*}{l}
        \dim \mathcal{M}_{X,A,s}^J(\Gamma^+,\Gamma^-)\p{<}{}{\mathcal{T}^{(k)}x} \\
        \quad = (n-3)(2 - p^+ - p^-) + 2 c_1^{\tau}(TX) \cdot A + \conleyzehnder^{\tau}(\Gamma^+) - \conleyzehnder^{\tau}(\Gamma^-) - 2n - 2k + 4.
    \end{IEEEeqnarray*}
\end{theorem}

We will now use this transversality result to state two lemmas from \cite{cieliebakPuncturedHolomorphicCurves2018}, namely \cref{lem:punctures and tangency,lem:punctures and tangency simple}, which we will use in the proof of \cref{thm:lagrangian vs g tilde}. For the sake of completeness, we will also give proofs of the results. We point out that in order to achieve the conditions in the statement of the lemmas, we can use a metric as in \cref{lem:geodesics lemma CM abs}. Finally, notice that \cref{lem:punctures and tangency} generalizes \cref{lem:punctures and tangency simple} to the case where the curve is not necessarily simple.

\begin{lemma}[{\cite[Lemma 3.2]{cieliebakPuncturedHolomorphicCurves2018}}]
    \phantomsection\label{lem:punctures and tangency simple}
    Let $(L,g)$ be an $n$-dimensional Riemannian manifold with the property that for some $\ell_0 > 0$, all closed geodesics $\gamma$ of length $\ell(\gamma) \leq \ell_0$ are noncontractible and nondegenerate and have Morse index $\morse(\gamma) \leq n - 1$. Let $x \in T^*L$ and $D$ be a symplectic divisor through $x$. For generic $J$ every simple punctured $J$-holomorphic sphere $C$ in $T^*L$ which is asymptotic at the punctures to geodesics of length $\leq \ell_0$ and which has contact order $k$ to $D$ at $x$ must have at least $k + 1$ punctures.
\end{lemma}
\begin{proof}
    Let $(\gamma_1, \ldots, \gamma_p)$ be the tuple of asymptotic Reeb orbits of $C$, which have corresponding geodesics also denoted by $(\gamma_1, \ldots, \gamma_p)$. By assumption, $\morse(\gamma_i) \leq n - 1$ for every $i = 1,\ldots,p$. Choose a trivialization $\tau$ of $C^* T T^*L$ such that the induced trivialization over the asymptotic Reeb orbits is as in \cref{thm:index of geodesic or reeb orbit isometric triv}. We show that $p \geq k + 1$.
    \begin{IEEEeqnarray*}{rCls+x*}
        0
        & \leq & \dim_{C} \mathcal{M}_{X,s}^J(\Gamma^+,\Gamma^-)\p{<}{}{\mathcal{T}^{(k)}x}                             \\
        & =    & (n-3)(2-p) + 2 c_1^{\tau}(TX) \cdot [C] + \sum_{i=1}^{p} \conleyzehnder^{\tau}(\gamma_i) - 2n - 2k + 4 \\
        & =    & (n-3)(2-p) + \sum_{i=1}^{p} \morse(\gamma_i) - 2n - 2k + 4                                             \\
        & \leq & (n-3)(2-p) + \sum_{i=1}^{p} (n-1) - 2n - 2k + 4                                                        \\
        & =    & 2 (p - 1 - k).                                                                                         & & \qedhere
    \end{IEEEeqnarray*}
\end{proof}

\begin{lemma}[{\cite[Corollary 3.3]{cieliebakPuncturedHolomorphicCurves2018}}]
    \label{lem:punctures and tangency}
    Let $(L,g)$ be an $n$-dimensional Riemannian manifold with the property that for some $\ell_0 > 0$, all closed geodesics $\gamma$ of length $\ell(\gamma) \leq \ell_0$ are noncontractible and nondegenerate and have Morse index $\morse(\gamma) \leq n - 1$. Let $x \in T^*L$ and $D$ be a symplectic divisor through $x$. For generic $J$ every (not necessarily simple) punctured $J$-holomorphic sphere $\tilde{C}$ in $T^*L$ which is asymptotic at the punctures to geodesics of length $\leq \ell_0$ and which has contact order $\tilde{k}$ to $D$ at $x$ must have at least $\tilde{k} + 1$ punctures.
\end{lemma}
\begin{proof}
    Let $\tilde{z}_1,\ldots,\tilde{z}_{\tilde{p}}$ be the punctures of $\tilde{C}$. Then $\tilde{C}$ is a map $\tilde{C} \colon S^2 \setminus \{\tilde{z}_1,\ldots,\tilde{z}_{\tilde{p}}\} \longrightarrow T^*L$ which has contact order $\tilde{k}$ at $\tilde{z}_0$ to $D$, for some $\tilde{z}_0 \in S^2 \setminus \{\tilde{z}_1,\ldots,\tilde{z}_{\tilde{p}}\}$. There exists a $d$-fold branched cover $\phi \colon S^2 \longrightarrow S^2$ and a simple punctured $J$-holomorphic sphere $C$ with $p$ punctures $\{z_1,\ldots,z_p\}$ which has contact order $k$ at $z_0 = \phi(\tilde{z}_0)$ to $D$, such that the following diagram commutes:
    \begin{IEEEeqnarray*}{c+x*}
        \begin{tikzcd}
            S^2 \setminus \{\tilde{z}_1,\ldots,\tilde{z}_{\tilde{p}}\} \ar[d, swap, "\phi"] \ar[rd, "\tilde{C}"] \\
            S^2 \setminus \{z_1,\ldots,z_p\} \ar[r, swap, "C"] & T^*L
        \end{tikzcd}
    \end{IEEEeqnarray*}
    
    Define $b = \operatorname{ord}(\tilde{z}_0)$. Since the asymptotic Reeb orbits of $\tilde{C}$ are multiples of the asymptotic Reeb orbits of $C$, we have that the Reeb orbits of $C$ all have period less or equal to $\ell_0$. Therefore, applying \cref{lem:punctures and tangency simple} to $C$ we conclude that $p - 1 \geq k$.
    
    We show that $k b \geq \tilde{k}$. For this, choose holomorphic coordinates centred at $z_0 \in S^2$, $\tilde{z}_0 \in S^2$, and $x \in X$ such that $D$ is given by $h(z_1,\ldots,z_n) = 0$, where $h(z_1,\ldots,z_n) = z_1$. Then, with respect to these coordinates
    \begin{IEEEeqnarray*}{rCls+x*}
        \phi(z)      & = & z^b, \\
        h \circ C(z) & = & \sum_{j=1}^{+\infty} a_j z^j,
    \end{IEEEeqnarray*}
    and therefore
    \begin{IEEEeqnarray*}{c+x*}
        h \circ \tilde{C}(z) = h \circ C \circ \phi(z) = \sum_{j=1}^{+\infty} a_j z^{b j}.
    \end{IEEEeqnarray*}
    Since $\tilde{C}$ has contact order $\tilde{k}$ to $D$,
    \begin{IEEEeqnarray*}{c+x*}
        0 = (h \circ \tilde{C})^{(r)}(0) = \sum_{j=1}^{+\infty} a_j (b j)^r z^{b j - r} \Big|_{z = 0}
    \end{IEEEeqnarray*}
    for every $r = 1,\ldots,\tilde{k}-1$. Therefore, for every $j \in \Z_{\geq 1}$ if there exists $r = 1,\ldots,\tilde{k}-1$ such that if $b j - r = 0$, then $a_j = 0$. In other words $a_1 = \cdots = a_\ell = 0$, where
    \begin{IEEEeqnarray*}{rCll}
        \ell
        & = & \max & \{ j \in \Z_{\geq 1} \mid b j \leq \tilde{k} - 1 \} \\
        & = & \min & \{ j \in \Z_{\geq 1} \mid b (j+1) \geq \tilde{k} \}.
    \end{IEEEeqnarray*}
    So, we conclude that $b k \geq b (\ell + 1) \geq \tilde{k}$.

    We show that $\tilde{p} \geq (p - 2) d + b + 1$.
    \begin{IEEEeqnarray*}{rCls+x*}
        2 d - 2
        & =    & \sum_{\tilde{z} \in S^2}^{} (\operatorname{ord}(\tilde{z}) - 1)                                    & \quad [\text{by the Riemann-Hurwitz formula}] \\
        & \geq & \sum_{i=1}^{\tilde{p}} (\operatorname{ord}(\tilde{z}_i) - 1) + \operatorname{ord}(\tilde{z}_0) - 1 & \quad [\text{since $\operatorname{ord}(z) \geq 1$ for every $z \in S^2$}] \\
        & =    & p d - \tilde{p} + \operatorname{ord}(\tilde{z}_0) - 1                                              & \quad [\text{since $\phi(\{\tilde{z}_1,\ldots,\tilde{z}_{\tilde{p}}\}) = \{z_1,\ldots,z_p\}$}] \\
        & =    & p d - \tilde{p} + b - 1                                                                            & \quad [\text{by definition of $b$}].
    \end{IEEEeqnarray*}

    Since $\phi$ is a $d$-fold covering, $d \geq b$. Combining all the facts which we have proven, we conclude that
    \begin{IEEEeqnarray*}{rCls+x*}
        \tilde{p}
        & \geq & (p-2)d + b + 1 & \quad [\text{by the last computation}] \\
        & \geq & (k-1)d + b + 1 & \quad [\text{since $p - 1 \geq k$}] \\
        & \geq & k b + 1        & \quad [\text{since $d \geq b$}] \\
        & \geq & \tilde{k} + 1  & \quad [\text{since $k b \geq \tilde{k}$}]. & \qedhere
    \end{IEEEeqnarray*}
\end{proof}

\section{SFT compactness}
\label{sec:sft compactness}

In this section we present the SFT compactness theorem, which describes the compactifications of the moduli spaces of the previous section. This theorem was first proven by Bourgeois--Eliashberg--Hofer--Wysocki--Zehnder \cite{bourgeoisCompactnessResultsSymplectic2003}. Cieliebak--Mohnke \cite{cieliebakCompactnessPuncturedHolomorphic2005} have given a proof of this theorem using different methods. Our presentation is based primarily on \cite{cieliebakPuncturedHolomorphicCurves2018} and \cite{mcduffSymplecticCapacitiesUnperturbed2022}.

\begin{definition}
    \label{def:nodal riemann surface}
    A \textbf{nodal Riemann surface} is a Riemann surface $(\Sigma, j)$ together with a set $\mathbf{n}$ of \textbf{nodes} of the form $\mathbf{n} = \{n_1^+, n_1^-, \ldots, n_k^+, n_k^-\}$.
\end{definition}

\begin{definition}
    \label{def:nodal holomorphic curve}
    Let $(\Sigma, j)$ be a Riemann surface with a set $\mathbf{n} = \{n_1^+, n_1^-, \ldots, n_k^+, n_k^-\}$ of nodes and $(X, J)$ be an almost complex manifold. A \textbf{nodal $J$-holomorphic curve} is a $J$-holomorphic curve $u \colon (\Sigma, j) \longrightarrow (X, J)$ such that $u(n^+_i) = u(n^-_i)$ for every $i = 1, \ldots, k$.%
\end{definition}

Let $(X, \omega, \lambda)$ be a symplectic cobordism and choose almost complex structures $J^{\pm} \in \mathcal{J}(\partial^{\pm} X)$ and $J \in \mathcal{J}^{J^+}_{J^-}(X)$. Let $\Gamma^{\pm} = (\gamma^{\pm}_1, \ldots, \gamma^{\pm}_{p ^{\pm}})$ be a tuple of Reeb orbits in $\partial^{\pm} X$.

\begin{definition}
    \label{def:sft compactification}
    For $1 \leq L \leq N$, let $\alpha^{\pm} \coloneqq \lambda|_{\partial^{\pm} X}$ and define
    \begin{IEEEeqnarray*}{rCl}
        (X^{\nu}, \omega^\nu, \tilde{\omega}^{\nu}, J^{\nu})
        & \coloneqq & 
        \begin{cases}
            (\R \times \partial^- X, \edv(e^r \alpha^-), \edv \alpha^-  , J^-) & \text{if } \nu = 1   , \ldots, L - 1, \\
            (\hat{X}               , \hat{\omega}      , \tilde{\omega} , J  ) & \text{if } \nu = L   , \\
            (\R \times \partial^+ X, \edv(e^r \alpha^+), \edv \alpha^+  , J^+) & \text{if } \nu = L+1 ,\ldots ,N     ,
        \end{cases} \\
        (X^*, \omega^*, \tilde{\omega}^*, J^*) & \coloneqq & \bigcoproduct_{\nu = 1}^N (X^{\nu}, \omega^\nu, \tilde{\omega}^{\nu}, J^{\nu}).
    \end{IEEEeqnarray*}
    The moduli space of \textbf{holomorphic buildings}, denoted $\overline{\mathcal{M}}^{J}_X(\Gamma^+, \Gamma^-)$, is the set of tuples $F = (F^1, \ldots, F^N)$, where $F^{\nu} \colon \dot{\Sigma}^\nu \longrightarrow X^\nu$ is an asymptotically cylindrical nodal $J^{\nu}$-holomorphic curve in $X^{\nu}$ with sets of asymptotic Reeb orbits $\Gamma^{\pm}_{\nu}$. Here, each $F^{\nu}$ is possibly disconnected and if $X^{\nu}$ is a symplectization then $F^{\nu}$ is only defined up to translation in the $\R$ direction. We assume in addition that $F$ satisfies the following conditions.
    \begin{enumerate}
        \item The sets of asymptotic Reeb orbits $\Gamma_{\nu}^{\pm}$ are such that
            \begin{IEEEeqnarray*}{rCls+x*}
                \Gamma^+_{\nu} & = & \Gamma^-_{\nu + 1} \quad \text{for every } \nu = 1, \ldots, N - 1, \\
                \Gamma^-_1     & = & \Gamma^-, \\
                \Gamma^+_N     & = & \Gamma^+.
            \end{IEEEeqnarray*}
        \item Define the graph of $F$ to be the graph whose vertices are the components of $F^1, \ldots, F^N$ and whose edges are determined by the asymptotic Reeb orbits. Then the graph of $F$ is a tree.
        \item The building $F$ has no symplectization levels consisting entirely of trivial cylinders, and any constant component of $F$ has negative Euler characteristic after removing all special points.
    \end{enumerate}
\end{definition}

\begin{definition}
    The \textbf{energy} of a holomorphic building $F = (F^1, \ldots, F^N)$ is 
    \begin{IEEEeqnarray*}{c+x*}
        E_{\tilde{\omega}^*}(F) \coloneqq \sum_{\nu = 1}^{N} E_{\tilde{\omega}^{\nu}}(F^{\nu}),
    \end{IEEEeqnarray*}
    where $E_{\tilde{\omega}^{\nu}}(F^{\nu})$ is given as in \cref{def:energy of a asy cylindrical holomorphic curve}.
\end{definition}

The moduli space $\overline{\mathcal{M}}_X^J(\Gamma^+, \Gamma^-)$ admits a metrizable topology (see \cite[Appendix B]{bourgeoisEquivariantSymplecticHomology2016}). With this language, the SFT compactness theorem can be stated as follows.

\begin{theorem}[SFT compactness]
    The moduli space $\overline{\mathcal{M}}_X^J(\Gamma^+, \Gamma^-)$ is compact.
\end{theorem}

We now consider the case where the almost complex structure on $\hat{X}$ is replaced by a family of almost complex structures obtained via \textbf{neck stretching}. Let $(X^{\pm}, \omega^{\pm}, \lambda^{\pm})$ be symplectic cobordisms with common boundary 
\begin{IEEEeqnarray*}{c+x*}
    (M, \alpha) = (\partial^- X^{+}, \lambda^+|_{\partial^- X^+}) = (\partial^+ X^-, \lambda^-|_{\partial^+ X^-}).
\end{IEEEeqnarray*}
Choose almost complex structures
\begin{IEEEeqnarray*}{rCls+x*}
    J_M & \in & \mathcal{J}(M), \\
    J_+ & \in & \mathcal{J}_{J_M}(X^+), \\
    J_- & \in & \mathcal{J}^{J_M}(X^-),
\end{IEEEeqnarray*}
and denote by $J_{\partial^{\pm} X^{\pm}} \in \mathcal{J}(\partial^{\pm} X^{\pm})$ the induced cylindrical almost complex structure on $\R \times \partial^{\pm} X^{\pm}$. Let $(X, \omega, \lambda) \coloneqq (X^-, \omega^-, \lambda^-) \circledcirc (X^+, \omega^+, \lambda^+)$ be the gluing of $X^-$ and $X^+$ along $M$. We wish to define a family of almost complex structures $(J_t)_{t \in \R_{\geq 0}} \subset \mathcal{J}(X)$. For every $t \geq 0$, let
\begin{IEEEeqnarray*}{c+x*}
    X_t \coloneqq X^- \cup_M [-t, 0] \times M \cup_M X^+.
\end{IEEEeqnarray*}
There exists a canonical diffeomorphism $\phi_t \colon X \longrightarrow X_t$. Define an almost complex structure $J_t$ on $X_t$ by
\begin{IEEEeqnarray*}{c+x*}
    J_t \coloneqq
    \begin{cases}
        J^{\pm} & \text{on } X^{\pm}, \\
        J_M     & \text{on } [-t, 0] \times M.
    \end{cases}
\end{IEEEeqnarray*}
Denote also by $J_t$ the pullback of $J_t$ to ${X}$, as well as the induced almost complex structure on the completion $\hat{X}$. Finally, consider the moduli space
\begin{IEEEeqnarray*}{c+x*}
    \mathcal{M}_X^{(J_t)_t}(\Gamma^+, \Gamma^-) \coloneqq \bigcoproduct_{t \in \R_{\geq 0}} \mathcal{M}^{J_t}_{X}(\Gamma^+, \Gamma^-).
\end{IEEEeqnarray*}

\begin{definition}
    \phantomsection\label{def:sft compactification neck stretching}
    For $1 \leq L^- < L^+ \leq N$, let $\alpha^{\pm} \coloneqq \lambda^{\pm}|_{\partial^{\pm} X^\pm}$ and define
    \begin{IEEEeqnarray*}{rCls+x*}
        (X^{\nu}, \omega^\nu, \tilde{\omega}^{\nu}, J^{\nu})
        & \coloneqq & 
        \begin{cases}
            (\R \times \partial^- X^-, \edv(e^r \alpha^-)  , \edv \alpha^-   , J_{\partial^- X^-}) & \text{if } \nu = 1        , \ldots, L^- - 1, \\
            (X^-                     , \omega^-            , \tilde{\omega}^-, J^-)                & \text{if } \nu = L^-, \\
            (\R \times M             , \edv(e^r \alpha)    , \edv \alpha     , J_M)                & \text{if } \nu = L^- + 1    , \ldots, L^+ - 1, \\
            (X^+                     , \omega^+            , \tilde{\omega}^+, J^+)                & \text{if } \nu = L^+, \\
            (\R \times \partial^+ X^+, \edv (e^r \alpha^+) , \edv \alpha^+   , J_{\partial^+ X^+}) & \text{if } \nu = L^+ + 1  , \ldots, N      , \\
        \end{cases} \\
        (X^*, \omega^*, \tilde{\omega}^*, J^*) & \coloneqq & \bigcoproduct_{\nu = 1}^N (X^{\nu}, \omega^\nu, \tilde{\omega}^{\nu}, J^{\nu}).
    \end{IEEEeqnarray*}
    Define $\overline{\mathcal{M}}^{(J_t)_t}_X(\Gamma^+, \Gamma^-)$ to be the set of tuples $F = (F^1, \ldots, F^N)$, where $F^{\nu} \colon \dot{\Sigma}^\nu \longrightarrow X^\nu$ is an asymptotically cylindrical nodal $J^{\nu}$-holomorphic curve in $X^{\nu}$ with sets of asymptotic Reeb orbits $\Gamma^{\pm}_{\nu}$, such that $F$ satisfies conditions analogous to those of \cref{def:sft compactification}.
\end{definition}

\begin{theorem}[SFT compactness]
    The moduli space $\overline{\mathcal{M}}^{(J_t)_t}_X(\Gamma^+, \Gamma^-)$ is compact.
\end{theorem}

\begin{remark}
    \label{rmk:compactifications with tangency}
    The discussion above also applies to compactifications of moduli spaces of curves satisfying tangency constraints. The compactification $\overline{\mathcal{M}}^{J}_{X}(\Gamma^+,\Gamma^-)\p{<}{}{\mathcal{T}^{(k)}x}$ consists of buildings $F = (F^1, \ldots, F^N) \in \overline{\mathcal{M}}^J_X(\Gamma^+, \Gamma^-)$ such that exactly one component $C$ of $F$ inherits the tangency constraint $\p{<}{}{\mathcal{T}^{(k)}x}$, and which satisfy the following additional condition. Consider the graph obtained from the graph of $F$ by collapsing adjacent constant components to a point. Let $C_1, \ldots, C_p$ be the (necessarily nonconstant) components of $F$ which are adjacent to $C$ in the new graph. Then we require that there exist $k_1, \ldots, k_p \in \Z_{\geq 1}$ such that $k_1 + \cdots + k_p \geq k$ and $C_i$ satisfies the constraint $\p{<}{}{\mathcal{T}^{(k_i)}x}$ for every $i = 1, \ldots, p$. This definition is natural to consider by \cite[Lemma 7.2]{cieliebakSymplecticHypersurfacesTransversality2007}. We can define $\overline{\mathcal{M}}^{(J_t)_t}_X(\Gamma^+, \Gamma^-)\p{<}{}{\mathcal{T}^{(k)}x}$ analogously.
\end{remark}

\begin{remark}
    We point out that in \cite[Definition 2.2.1]{mcduffSymplecticCapacitiesUnperturbed2022}, the compactification of \cref{rmk:compactifications with tangency} is denoted by $\overline{\overline{\mathcal{M}}}^{J}_{X}(\Gamma^+,\Gamma^-)\p{<}{}{\mathcal{T}^{(k)}x}$, while the notation $\overline{\mathcal{M}}^{J}_{X}(\Gamma^+,\Gamma^-)\p{<}{}{\mathcal{T}^{(k)}x}$ is used to denote the moduli space of buildings $F = (F^1, \ldots, F^N) \in \overline{\mathcal{M}}^J_X(\Gamma^+, \Gamma^-)$ such that exactly one component $C$ of $F$ inherits the tangency constraint $\p{<}{}{\mathcal{T}^{(k)}x}$, but which do not necessarily satisfy the additional condition of \cref{rmk:compactifications with tangency}.
\end{remark}

\begin{lemma}
    \label{lem:no nodes}
    Suppose that $\Gamma^- = \varnothing$ and $\Gamma^+ = (\gamma)$ consists of a single Reeb orbit. Let $F$ be a holomorphic building of genus $0$ in any of the following compactified moduli spaces:%
    \begin{IEEEeqnarray*}{lCl}
        \overline{\mathcal{M}}^J_X(\gamma),         & \quad & \overline{\mathcal{M}}^J_X(\gamma)\p{<}{}{\mathcal{T}^{(k)}x}, \\
        \overline{\mathcal{M}}^{(J_t)_t}_X(\gamma), & \quad & \overline{\mathcal{M}}^{(J_t)_t}_X(\gamma)\p{<}{}{\mathcal{T}^{(k)}x}.
    \end{IEEEeqnarray*}
    Then $F$ has no nodes.
\end{lemma}
\begin{proof}
    Assume by contradiction that $F$ has a node. Let $\overline{\Sigma}$ be the topological space obtained by gluing the $\Sigma^{\nu}$ along the matching punctures. Let $\overline{X}$ be the topological space obtained by gluing the $X^{\nu}$ along the matching ends. The space $\overline{X}$ is homeomorphic to $\hat{X}$, and therefore we can identify homology classes in $\overline{X}$ and $\hat{X}$. The holomorphic building $F$ defines a continuous map $\overline{F} \colon \overline{\Sigma} \longrightarrow \overline{X}$ (for more details on the definitions of $\overline{F} \colon \overline{\Sigma} \longrightarrow \overline{X}$, see \cite[Section 2.6]{cieliebakPuncturedHolomorphicCurves2018}). By the assumptions on $F$ and since $F$ has a node, it is possible to decompose $\overline{F}$ along the node into two continuous maps
    \begin{IEEEeqnarray*}{rCls+x*}
        \overline{F}_0 \colon \overline{\Sigma}_0 & \longrightarrow & \overline{X}, \\
        \overline{F}_1 \colon \overline{\Sigma}_1 & \longrightarrow & \overline{X},
    \end{IEEEeqnarray*}
    where $\overline{F}_0$ is a plane and $\overline{F}_1$ is a sphere. Since $\overline{F}_1$ is a sphere, it defines a homology class $[\overline{F}_1] \in H_2(\hat{X}; \Z)$. Then,
    \begin{IEEEeqnarray*}{rCls+x*}
        0
        & = & \edv \hat{\lambda}([\overline{F}_1]) & \quad [\text{since $\edv \hat{\lambda} = 0 \in H^2_{\mathrm{dR}}(\hat{X})$}] \\
        & > & 0                                    & \quad [\text{by \cite[Lemma 2.8]{cieliebakPuncturedHolomorphicCurves2018}}],
    \end{IEEEeqnarray*}
    which gives the desired contradiction.
\end{proof}

\section{Solutions of the parametrized Floer equation}
\label{sec:floer trajectories}

The goal of this section is to introduce the trajectories that appear in $S^1$-equivariant symplectic homology (see \cref{def:floer trajectory abstract}). We will write these trajectories as maps whose domain is any punctured Riemann surface, but we point out that in \cref{chp:floer}, where we discuss $S^1$-equivariant symplectic homology, all trajectories have as domain the cylinder $\R \times S^1$. Let $(\Sigma, j)$ be a Riemann surface with punctures 
\begin{IEEEeqnarray*}{c+x*}
    \mathbf{z} = \mathbf{z}^+ \cup \mathbf{z}^-, \qquad \mathbf{z}^{\pm} = \{z^{\pm}_1, \ldots, z^{\pm}_{p^{\pm}}\}.
\end{IEEEeqnarray*}
We assume that near every puncture $z$, there are cylindrical coordinates $(s,t)$ as in \cref{def:punctures asy markers cyl ends}. Let $\sigma, \tau \in \Omega^1(\dot{\Sigma})$ be $1$-forms such that for every (positive or negative) puncture $z$, if we denote by $(s,t)$ the coordinates on the cylindrical end of $\dot{\Sigma}$ near $z$, then%
\begin{IEEEeqnarray*}{rCls+x*}
    \sigma & = & A \, \edv s, \\
    \tau   & = & B \, \edv t,
\end{IEEEeqnarray*}
for some $A, B > 0$. Finally, we assume that there is an action 
\begin{IEEEeqnarray*}{c+x*}
    S^1 \times \dot{\Sigma} \longrightarrow \dot{\Sigma}
\end{IEEEeqnarray*}
of $S^1$ on $\dot{\Sigma}$ which preserves $j$, $\sigma$ and $\tau$ and such that if $t' \in S^1$ and $(s,t)$ belongs to any cylindrical coordinate neighbourhood, then
\begin{IEEEeqnarray*}{c+x*}
    t' \cdot (s, t) = (s, t + t').
\end{IEEEeqnarray*} 

\begin{example}
    \label{exa:sphere and cylinder}
    Consider the cylinder $\R \times S^1$ with coordinates $(s,t)$ and almost complex structure given by $j(\partial_s) = \partial_t$. We have the $1$-forms $\sigma \coloneqq \edv s$ and $\tau \coloneqq \edv t$. The cylinder is biholomorphic to the sphere $S^2$ with the north and south poles removed. There is an action of $S^1$ on $\R \times S^1$ given by $t' \cdot (s,t) = (s,t + t')$. Therefore, $\R \times S^1$ can be seen as a special case of the assumptions above. In this case, we will typically denote $\dot{\Sigma} = \R \times S^1$ and $\Sigma = S^2$.
\end{example}

Let $(S,g^S)$ be a Riemannian manifold together with an action $S^1 \times S \longrightarrow S$ which is free, proper and by isometries. Define $C = S / S^1$ and denote the projection by $\pi \colon S \longrightarrow C$. Since the action is by isometries, there exists a unique Riemannian metric $g^C$ on $C$ such that $\pi \colon S \longrightarrow C$ is a Riemannian submersion. Let $f \colon C \longrightarrow \R$ be a Morse function and define $\tilde{f} \coloneqq f \circ \pi \colon S \longrightarrow \R$, which is Morse--Bott. 

\begin{example}
    For $N \in \Z_{\geq 1}$, let 
    \begin{IEEEeqnarray*}{rCls+x*}
        S & \coloneqq & S^{2N+1}, \\
        C & \coloneqq & \C P^N, \\
        f & \coloneqq & f_N,
    \end{IEEEeqnarray*}
    where
    \begin{IEEEeqnarray*}{c+x*}
        f_N([w_0:\cdots:w_N]) \coloneqq \frac{ \sum_{j=0}^{N} j |w_j|^2 }{ \sum_{j=0}^{N} |w_j|^2 }.
    \end{IEEEeqnarray*}
    As we will discuss in \cref{sec:action functional}, $S$, $C$ and $f$ given above are as in the previous paragraph.
\end{example}

Finally, let $(X,\lambda)$ be a Liouville domain.

\begin{definition}
    \label{def:admissible hamiltonian abstract}
    An \textbf{admissible Hamiltonian} is a map $H \colon \dot{\Sigma} \times S \times \hat{X} \longrightarrow \R$ such that:
    \begin{enumerate}
        \item \label{def:admissible hamiltonian abstract 1} For every puncture $z$, the restriction of $H$ to the cylindrical end near $z$ is independent of $s$ for $s$ large enough. In other words, there is a map $H_z \colon S^1 \times S \times \hat{X} \longrightarrow \R$ such that $H(s,t,w,x) = H_z(t,w,x)$ for $s$ large enough.
        \item \label{def:admissible hamiltonian abstract 2} For every critical point $w$ of $\tilde{f}$, there exists a neighbourhood $V$ of $w$ in $S$ such that the restriction $H \colon \dot{\Sigma} \times V \times \hat{X} \longrightarrow \R$ is independent of $V$.
        \item Consider the action of $S^1$ on $\dot{\Sigma} \times S \times \hat{X}$ given by $t \cdot (z, w, x) = (t \cdot z, t \cdot w, x)$. Then, the Hamiltonian $H$ is invariant under the action of $S^1$.
        \item For every puncture $z$, there exist $D \in \R$, $C \in \R_{> 0} \setminus \operatorname{Spec}(\partial X, \lambda|_{\partial X})$ and $\delta > 0$ such that on $S^1 \times S \times [\delta,+\infty) \times \partial X$, we have that $H_z(t,w,r,x) = C e^r + D$.
        \item For every puncture $z$ and critical point $w$ of $\tilde{f}$ the Hamiltonian $H_{z,w} \colon S^1 \times \hat{X} \longrightarrow \R$ is nondegenerate.
        \item \label{def:admissible hamiltonian abstract 3} For every $(z,w,x) \in \dot{\Sigma} \times S \times \hat{X}$ we have
            \begin{IEEEeqnarray*}{rCls+x*}
                H_{w,x} \, \edv \tau                                                             & \leq & 0, \\
                \edv_{\dot{\Sigma}} H_{w,x} \wedge \tau                                          & \leq & 0, \\
                \p{<}{}{ \nabla_S H_{z,x}(w), \nabla \tilde{f} (w) } \, \sigma_z \wedge \tau_z & \leq & 0.
            \end{IEEEeqnarray*}
    \end{enumerate}
\end{definition}

\begin{definition}
    \label{def:admissible acs abstract}
    An \textbf{admissible almost complex structure} on $\hat{X}$ is a section $J \colon \dot{\Sigma} \times S \times \hat{X} \longrightarrow \End(T \hat{X})$ such that $J^2 = - \id_{TX}$ and:
    \begin{enumerate}
        \item \label{def:admissible acs abstract 1} For every puncture $z$, the restriction of $J$ to the cylindrical end near $z$ is independent of $s$ for $s$ large enough. In other words, there is a function $J_z \colon S^1 \times S \times \hat{X} \longrightarrow \End(T \hat{X})$ such that $J(s,t,w,x) = J_z(t,w,x)$ for $s$ large enough.
        \item \label{def:admissible acs abstract 2} For every critical point $w$ of $\tilde{f}$, there exists a neighbourhood $V$ of $w$ in $S$ such that the restriction $J \colon \dot{\Sigma} \times V \times \hat{X} \longrightarrow \End(T \hat{X})$ is independent of $V$.
        \item The almost complex structure $J$ is $S^1$-invariant.
        \item $J$ is \textbf{compatible}, i.e. $g \coloneqq \omega(\cdot, J \cdot) \colon \dot{\Sigma} \times S \times \hat{X} \longrightarrow T^* \hat{X} \otimes T^* \hat{X}$ is a Riemannian metric on $X$ parametrized by $\dot{\Sigma} \times S$.
        \item $J$ is \textbf{cylindrical}, i.e. if $(z,w) \in \dot{\Sigma} \times S$ then $J_{z,w}$ is cylindrical on $\R_{\geq 0} \times \partial X$.
    \end{enumerate}
\end{definition}

\begin{definition}
    \label{def:floer trajectory abstract}
    Let $w \colon \dot{\Sigma} \longrightarrow S$ and $u \colon \dot{\Sigma} \longrightarrow \hat{X}$ be maps. We will denote by $\mathbf{u}$ the map $\mathbf{u} \coloneqq (\id_{\dot{\Sigma}}, w, u) \colon \dot{\Sigma} \longrightarrow \dot{\Sigma} \times S \times \hat{X}$. We say that $(w,u)$ is a solution of the \textbf{parametrized Floer equation} if
    \begin{IEEEeqnarray}{rCls+x*}
        \dv w - \nabla \tilde{f} (w) \otimes \sigma                     & = & 0, \phantomsection\label{eq:parametrized floer equation 1} \\
        (\dv u - X_H(\mathbf{u}) \otimes \tau)^{0,1}_{J(\mathbf{u}), j} & = & 0. \phantomsection\label{eq:parametrized floer equation 2}
    \end{IEEEeqnarray}
\end{definition}

\begin{example}
    Suppose that $(\dot{\Sigma}, j, \sigma, \tau) = (\R \times S^1, j, \edv s, \edv t)$ is the cylinder from \cref{exa:sphere and cylinder}. Then, $(w,u)$ is a solution of the parametrized Floer equation if and only if $w \colon \R \times S^1 \longrightarrow S$ is independent of $t \in S^1$, thus defining a map $w \colon \R \longrightarrow S$, and
    \begin{IEEEeqnarray*}{rCls+x*}
        \pdv{w}{s}(s)   & = & \nabla \tilde{f}(w(s)), \\
        \pdv{u}{s}(s,t) & = & - J(s, t, w(s), u(s,t)) \p{}{2}{ \pdv{u}{t}(s,t) - X_{H}(s, t,w(s),u(s,t)) }.
    \end{IEEEeqnarray*}
\end{example}

\begin{definition}
    \label{def:1 periodic orbit abstract}
    Let $z$ be a puncture and $B > 0$ be such that $\tau = B \, \edv t$, where $(s,t)$ are the cylindrical coordinates near $z$. A \textbf{$1$-periodic orbit} of $H$ at $z$ is a pair $(w ,\gamma)$ such that $w \in S$ is a critical point of $\tilde{f}$ and $\gamma$ is a $1$-periodic orbit of $H_{z,w} \colon S^1 \times \hat{X} \longrightarrow \R$. Denote by $\mathcal{P}(H,z)$ the set of such pairs. The \textbf{action} of $(w, \gamma)$ is
    \begin{IEEEeqnarray*}{c+x*}
        \mathcal{A}_{H}(w,\gamma) \coloneqq \mathcal{A}_{B H_{z,w}}(\gamma) = \int_{S^1}^{} \gamma^* \hat{\lambda} - B \int_{S^1}^{} H_{z,w} (t, \gamma(t)) \edv t.
    \end{IEEEeqnarray*}
\end{definition}

\begin{definition}
    \label{def:asymptotic}
    Let $(w,u)$ be a solution of the parametrized Floer equation. We say that $(w,u)$ is \textbf{asymptotic} at $z^{\pm}_i$ to $(w^{\pm}_i, \gamma^{\pm}_i) \in \mathcal{P}(H, z^{\pm}_i)$ if
    \begin{IEEEeqnarray*}{rCls+x*}
        \lim_{s \to \pm \infty} w(s)   & = & w^{\pm}_i, \\
        \lim_{s \to \pm \infty} u(s,t) & = & \gamma^{\pm}_i,
    \end{IEEEeqnarray*}
    where $(s,t)$ are the cylindrical coordinates near $z^{\pm}_i$.
\end{definition}

\begin{definition}
    \label{def:energy of floer trajectory}
    The \textbf{energy} of $(w,u)$ is
    \begin{IEEEeqnarray*}{c+x*}
        E(u) \coloneqq \frac{1}{2} \int_{\dot{\Sigma}}^{} \| \dv u - X_H(\mathbf{u}) \otimes \tau \|^2_{J(\mathbf{u}), \hat{\omega}} \, \omega_{\Sigma}.
    \end{IEEEeqnarray*}
\end{definition}

We will now state the analytical results about solutions of the parametrized Floer equation. Some results we will state are analogous to previous results about solutions of a pseudoholomorphic curve equation. Namely, in \cref{lem:action energy for floer trajectories} we compare the energy of a solution with the action at the asymptotes, and in \cref{lem:maximum principle} we show that solutions satisfy a maximum principle.

\begin{lemma}
    \phantomsection\label{lem:action energy for floer trajectories}
    If $(w,u)$ is a solution of the parametrized Floer equation which is asymptotic at $z^{\pm}_i$ to $(w^{\pm}_i, \gamma^{\pm}_i) \in \mathcal{P}(H, z^{\pm}_i)$, then
    \begin{IEEEeqnarray*}{c+x*}
        0 \leq E(u) \leq \sum_{i=1}^{p^+} \mathcal{A}_H(w^+_i, \gamma^+_i) - \sum_{i=1}^{p^-} \mathcal{A}_H(w^-_i, \gamma^-_i).
    \end{IEEEeqnarray*}
\end{lemma}
\begin{proof}
    We show that $1/2 \| \dv u - X_H(\mathbf{u}) \otimes \tau \|^{2}_{J(\mathbf{u}),j} \, \omega_{\dot{\Sigma}} = u^* \hat{\omega} - u^* \edv_{\hat{X}} H(\mathbf{u}) \wedge \tau$.
    \begin{IEEEeqnarray*}{rCls+x*}
        \IEEEeqnarraymulticol{3}{l}{\frac{1}{2} \| \dv u - X_H(\mathbf{u}) \otimes \tau \|^{2}_{J(\mathbf{u}), \hat{\omega}} \, \omega_{\dot{\Sigma}}(S, T)}\\
        \quad & = & (\dv u - X_H(\mathbf{u}) \otimes \tau)^* \hat{\omega}(S, T)                                            \\
              & = & \hat{\omega}(\dv u (S) - X_{H}(\mathbf{u}) \tau(S), \dv u (T) - X_{H}(\mathbf{u}) \tau(T))             \\
              & = & \hat{\omega} (\dv u (S), \dv u (T)) - \hat{\omega} (\dv u (S), X_{H}(\mathbf{u})) \tau(T) - \hat{\omega} (X_{H}(\mathbf{u}), \dv u (T)) \tau(S)              \\
              & = & u^* \hat{\omega} (S,T) + u^* \iota_{X_H(\mathbf{u})} \hat{\omega} \wedge \tau (S,T)                    \\
        \quad & = & u^* \hat{\omega} (S,T) - u^* \edv_{\hat{X}} H(\mathbf{u}) \wedge \tau (S,T),
    \end{IEEEeqnarray*}
    Where in the first equality we used \cref{lem:integrand of energy is well-defined} and the fact that $\dv u - X_H(\mathbf{u}) \otimes \tau$ is holomorphic, and in the last equality we used the definition of Hamiltonian vector field. We show that $u^* \hat{\omega} - u^* \edv_{\hat{X}} H (\mathbf{u}) \wedge \tau \leq u^* \hat{\omega} - \edv(\mathbf{u}^* H \wedge \tau)$.
    \begin{IEEEeqnarray*}{rCls+x*}
        \edv (\mathbf{u}^* H \wedge \tau)
        & =    & \mathbf{u}^* H \wedge \edv \tau + \mathbf{u}^* \edv H \wedge \tau                                                                                                      \\
        & =    & \mathbf{u}^* H \wedge \edv \tau + \edv_{\dot{\Sigma}} H (\mathbf{u}) \wedge \tau + w^* \edv_S H(\mathbf{u}) \wedge \tau + u^* \edv_{\hat{X}} H(\mathbf{u}) \wedge \tau \\
        & =    & \mathbf{u}^* H \wedge \edv \tau + \edv_{\dot{\Sigma}} H (\mathbf{u}) \wedge \tau + \p{<}{}{\nabla_S H(\mathbf{u}), \nabla \tilde{f}(w)} \, \sigma \wedge \tau + u^* \edv_{\hat{X}} H(\mathbf{u}) \wedge \tau \\
        & \leq & u^* \edv_{\hat{X}} H (\mathbf{u}) \wedge \tau                                                                                                                          
    \end{IEEEeqnarray*}
    Here, in the third equality we used Equation \eqref{eq:parametrized floer equation 1} and in the last line of the computation we used the fact that $H$ is admissible. Combining these results,
    \begin{IEEEeqnarray*}{rCls+x*}
        0
        & \leq & E(u)                                                                                                      \\
        & \leq & \int_{\dot{\Sigma}}^{} u^* \edv \hat{\lambda} - \int_{\dot{\Sigma}}^{} \edv (\mathbf{u}^* H \wedge \tau)  \\
        & =    & \sum_{i=1}^{p^+} \mathcal{A}_H(w^+_i, \gamma^+_i) - \sum_{i=1}^{p^-} \mathcal{A}_H(w^-_i, \gamma^-_i),
    \end{IEEEeqnarray*}
    where in the last line we used Stokes' theorem.
\end{proof}

\begin{lemma}
    \label{lem:floer eq proj}
    Suppose that $(M, \alpha)$ is a contact manifold, $H \colon \dot{\Sigma} \times S \times \R \times M \longrightarrow \R$ is a Hamiltonian which is independent of $M$ and $J \colon \dot{\Sigma} \times S \times \R \times M \longrightarrow \End(T(\R \times M))$ is a cylindrical almost complex structure. If 
    \begin{IEEEeqnarray*}{c+x*}
        \mathbf{u} = (\id_{\dot{\Sigma}}, w, u) = (\id_{\dot{\Sigma}}, w, (a, f)) \colon \dot{\Sigma} \longrightarrow \dot{\Sigma} \times S \times \R \times M
    \end{IEEEeqnarray*}
    is a solution of the parametrized Floer equation, then $f^* \edv \alpha \geq 0$ and
    \begin{IEEEeqnarray}{rCls+x*}
        - \edv a \circ j              & = & f^* \alpha - \alpha(X_H(\mathbf{u})) \tau        \plabel{eq:floer eq proj 1} \\
        \pi_{\xi} \circ \dv f \circ j & = & J_{\xi}(\mathbf{u}) \circ \pi_{\xi} \circ \dv f. \plabel{eq:floer eq proj 2}
    \end{IEEEeqnarray}
\end{lemma}
\begin{proof}
    We prove equation \eqref{eq:floer eq proj 1}:
    \begin{IEEEeqnarray*}{rCls+x*}
        - \edv a \circ j
        & = & - \edv r \circ \dv u \circ j                                                & \quad [\text{by definition of $a$}]                           \\
        & = & - \edv r \circ (\dv u - X_H(\mathbf{u}) \tensorpr \tau) \circ j             & \quad [\text{$H$ is independent of $M$}]                      \\
        & = & - \edv r \circ J(\mathbf{u}) \circ (\dv u - X_H(\mathbf{u}) \tensorpr \tau) & \quad [\text{$\dv u - X_H(\mathbf{u}) \tensorpr \tau$ is holomorphic}] \\
        & = & \alpha \circ (\dv u - X_H(\mathbf{u}) \tensorpr \tau)                       & \quad [\text{by \cref{lem:J cylindrical forms}}]              \\
        & = & f^* \alpha - \alpha(X_H(\mathbf{u})) \tau                                   & \quad [\text{by definition of pullback}].
    \end{IEEEeqnarray*}
    Equation \eqref{eq:floer eq proj 2} follows by applying $\pi_{\xi} \colon T(\R \times M) \longrightarrow \xi$ to $(\dv u - X_H(\mathbf{u}) \tensorpr \tau)^{0,1}_{J(\mathbf{u}),j} = 0$. The proof of $f^* \edv \alpha \geq 0$ is equal to the one presented in \cref{lem:holomorphic curves in symplectizations}.
\end{proof}

The following is an adaptation to solutions of the parametrized Floer equation of the maximum principle from \cref{thm:maximum principle holomorphic}. Other authors have proven similar results about solutions of a Floer equation satisfying a maximum principle, namely Viterbo \cite[Lemma 1.8]{viterboFunctorsComputationsFloer1999}, Oancea \cite[Lemma 1.5]{oanceaSurveyFloerHomology2004}, Seidel \cite[Section 3]{seidelBiasedViewSymplectic2008} and Ritter \cite[Lemma D.1]{ritterTopologicalQuantumField2013}.

\begin{lemma}[maximum principle]
    \label{lem:maximum principle}
    Under the assumptions of \cref{lem:floer eq proj}, define 
    \begin{IEEEeqnarray*}{rClCrCl}
        h \colon \dot{\Sigma} \times S \times \R & \longrightarrow & \R, & \quad & h(z,w,\rho) & = & H(z,w,\ln(\rho)), \\
        \rho \colon \dot{\Sigma}                 & \longrightarrow & \R, & \quad & \rho        & = & \exp \circ a.
    \end{IEEEeqnarray*}
    If
    \begin{IEEEeqnarray}{rCl}
        \partial_{\rho} h(z,w,\rho) \, \edv \tau                                                     & \leq & 0, \plabel{eq:maximum principle 1} \\
        \edv_{\dot{\Sigma}} (\partial_{\rho} h(z,w,\rho)) \wedge \tau                                & \leq & 0, \plabel{eq:maximum principle 2} \\
        \p{<}{}{\nabla_{S} \partial_{\rho} h(z,w,\rho), \nabla \tilde{f} (w) } \, \sigma \wedge \tau & \leq & 0, \plabel{eq:maximum principle 3}
    \end{IEEEeqnarray}
    and $a \colon \dot{\Sigma} \longrightarrow \R$ has a local maximum then $a$ is constant.
\end{lemma}
\begin{proof}
    Choose a symplectic structure $\omega_{\dot{\Sigma}}$ on $\dot{\Sigma}$ such that $g_{\dot{\Sigma}} \coloneqq \omega_{\dot{\Sigma}}(\cdot, j \cdot)$ is a Riemannian metric. Define $L \colon C^{\infty}(\dot{\Sigma}, \R) \longrightarrow C^{\infty}(\dot{\Sigma}, \R)$ by
    \begin{IEEEeqnarray*}{c+x*}
        L \nu = - \Delta \nu - \rho \, \partial^2_{\rho} h (z,w,\rho) \frac{\edv \nu \wedge \tau}{\omega_{\dot{\Sigma}}},
    \end{IEEEeqnarray*}
    for every $\nu \in C^{\infty}(\dot{\Sigma}, \R)$. The map $L$ is a linear elliptic partial differential operator (as in \cite[p.~312]{evansPartialDifferentialEquations2010}). We wish to show that $L \rho \leq 0$. For this, we start by computing $\Delta \rho \, \omega_{\dot{\Sigma}}$.%
    \begin{IEEEeqnarray*}{rCls+x*}
        - \Delta \rho \, \omega_{\dot{\Sigma}}
        & = & \edv (\edv \rho \circ j)                                                                       & \quad [\text{by \cref{lem:laplacian}}] \\
        & = & - \edv (u^*(e^r \alpha) - \rho \, \alpha(X_H(\mathbf{u})) \, \tau)                             & \quad [\text{by \cref{lem:floer eq proj}}] \\
        & = & - u^* \edv (e^r \alpha) + \edv (\rho \, \partial_{\rho} h (z,w,\rho) \, \tau)                  & \quad [\text{by \cref{lem:reeb equals hamiltonian on symplectization}}] \\
        & = & - u^* \edv (e^r \alpha) + \partial_{\rho} h (z,w,\rho) \, \edv \rho \wedge \tau                & \quad [\text{by the Leibniz rule}] \\
        &   & \hphantom{- u^* \edv (e^r \alpha)} + \rho \, \edv (\partial_{\rho} h (z,w,\rho)) \wedge \tau \\
        &   & \hphantom{- u^* \edv (e^r \alpha)} + \rho \, \partial_{\rho} h (z,w,\rho) \, \edv \tau.
    \end{IEEEeqnarray*}
    By Equation \eqref{eq:maximum principle 1}, the last term on the right is nonnegative. We show that the sum of the first two terms on the right is nonnegative.
    \begin{IEEEeqnarray*}{rCls+x*}
        \IEEEeqnarraymulticol{3}{l}{- u^* \edv (e^r \alpha) + \partial_{\rho} h (z,w,\rho) \, \edv \rho \wedge \tau}\\ \quad
        & =    & - u^* \edv (e^r \alpha) + u^* \edv_{\R \times M} H(\mathbf{u}) \wedge \tau                                            & \quad [\text{by definition of $h$}] \\
        & =    & - \frac{1}{2} \| \dv u - X_H(\mathbf{u}) \otimes \tau \|^2_{J(\mathbf{u}), \edv(e^r \alpha)} \, \omega_{\dot{\Sigma}} & \quad [\text{by the computation in \cref{lem:action energy for floer trajectories}}] \\
        & \leq & 0.
    \end{IEEEeqnarray*}
    Finally, we show that $\rho \, \edv (\partial_{\rho} h (z,w,\rho)) \wedge \tau \leq \rho \, \partial^2_{\rho} h(z,w,\rho) \, \edv \rho \wedge \tau$:
    \begin{IEEEeqnarray*}{rCls+x*}
        \IEEEeqnarraymulticol{3}{l}{\rho \, \edv (\partial_{\rho} h (z,w,\rho)) \wedge \tau}\\ \quad
        & =    & \rho \, \edv_{\dot{\Sigma}} \partial_{\rho} h(z,w,\rho) \wedge \tau + \rho \, \p{<}{}{\nabla_{S} \partial_{\rho} h(z,w,\rho), \nabla \tilde{f}(w)} \, \sigma \wedge \tau + \rho \, \partial^2_{\rho} h(z,w,\rho) \, \edv \rho \wedge \tau \\
        & \leq & \rho \, \partial^2_{\rho} h(z,w,\rho) \, \edv \rho \wedge \tau,
    \end{IEEEeqnarray*}
    where in the last line we used Equations \eqref{eq:maximum principle 2} and \eqref{eq:maximum principle 3}. This shows that $L \rho \leq 0$. By the strong maximum principle in \cite[p.~349-350]{evansPartialDifferentialEquations2010}, if $\rho$ has a local maximum then $\rho$ is constant. Since $\rho = \exp \circ a$, the same is true for $a$.
\end{proof}

The next lemma is an adaptation to our setup of an argument by Bourgeois--Oancea which first appeared in \cite[p.~654-655]{bourgeoisExactSequenceContact2009}. The same argument was also used by Cieliebak--Oancea \cite[Lemma 2.3]{cieliebakSymplecticHomologyEilenberg2018} in a different setup.

\begin{lemma}[asymptotic behaviour]
    \label{lem:asymptotic behaviour}
    Consider the half-cylinder $Z^{\pm}$ of \cref{def:punctures asy markers cyl ends}, with $1$-forms $\sigma \coloneqq \edv s$ and $\tau \coloneqq \edv t$. Assume the same conditions as in \cref{lem:floer eq proj}, but with $\dot{\Sigma}$ replaced by $Z^{\pm}$. Suppose that $\mathbf{u}$ is asymptotic at $\pm \infty$ to a $1$-periodic orbit $(z_{\pm}, \gamma_{\pm})$ of $H_{\pm \infty}$ of the form $\gamma_{\pm}(t) = (r_{\pm}, \rho_{\pm}(t))$, where $z_{\pm}$ is a critical point of $\tilde{f}$, $r_{\pm} \in \R$ and $\rho_{\pm} \colon S^1 \longrightarrow M$ is a periodic Reeb orbit in $M$. Define $h \colon Z^{\pm} \times S \times \R \longrightarrow \R$ by $h(s,t,z,r) = H(s,t,z,\ln(r))$ (recall that $H$ is independent of $M$). If 
    \begin{IEEEeqnarray}{rCls+x*}
        \pm \del_r^2 h(s,t,z_{\pm},e^{r_{\pm}})                                             & <    & 0  \plabel{lem:asymptotic behaviour gen 1} \\
        \p{<}{}{ \nabla_S \del_r h(s, t, z_{\pm}, e^{r_{\pm}}), \nabla \tilde{f}(z_{\pm}) } & <    & 0  \plabel{lem:asymptotic behaviour gen 2} \\
        \del_s \del_r h(s,t,z_{\pm},e^{r_{\pm}})                                            & \leq & 0, \plabel{lem:asymptotic behaviour gen 3}
    \end{IEEEeqnarray}
    then either there exists $(s_0,t_0) \in Z^{\pm}$ such that $a(s_0, t_0) > r_{\pm}$ or $\mathbf{u}$ is of the form $\mathbf{u}(s,t) = (s,t, w(s), r_{\pm}, \rho_{\pm}(t))$.
\end{lemma}
\begin{proof}
    It suffices to assume that $a(s,t) \leq r_{\pm}$ for all $(s,t) \in Z^{\pm}$ and to prove that $a(s,t) = r_{\pm}$ and $f(s,t) = \rho_{\pm}(t)$ for all $(s,t) \in Z^{\pm}$. After replacing $Z^{\pm}$ by a smaller half-cylinder we may assume the following analogues of \eqref{lem:asymptotic behaviour gen 1} and \eqref{lem:asymptotic behaviour gen 2}:
    \begin{IEEEeqnarray}{rCls+x*}
        \pm \del_r^2 h(s,t,w(s),e^{a(s,t)})                                           & \leq & 0, \plabel{lem:asymptotic behaviour gen 1b} \\
        \p{<}{}{ \nabla_S \del_r h(s, t, w(s), e^{r_{\pm}}), \nabla \tilde{f}(w(s)) } & \leq & 0. \plabel{lem:asymptotic behaviour gen 2b}
    \end{IEEEeqnarray}
    Define the average of $a$, which we denote by $\overline{a} \colon \R^{\pm}_0 \longrightarrow \R$, by 
    \begin{IEEEeqnarray*}{c+x*}
        \overline{a}(s) \coloneqq \int_{0}^{1} a(s,t) \edv t.
    \end{IEEEeqnarray*}
    Then,
    \begin{IEEEeqnarray*}{rCls+x*}
        \IEEEeqnarraymulticol{3}{l}{\pm \del_s \overline{a}(s)}\\ \quad
        & =    & \pm \int_{0}^{1} \del_s a(s,t) \edv t                                                                                                            & \quad [\text{by definition of $\overline{a}$}]                     \\
        & =    & \pm \int_{0}^{1} f_s^* \alpha \mp \int_{0}^{1} \alpha(X_H(\mathbf{u}(s,t))) \edv t                                                               & \quad [\text{by \cref{lem:floer eq proj}}]                         \\
        & =    & \pm \int_{0}^{1} \rho_{\pm}^* \alpha \mp \int_{s}^{\pm \infty} \int_{0}^{1} f^* \edv \alpha \mp \int_{0}^{1} \alpha(X_H(\mathbf{u}(s,t))) \edv t & \quad [\text{by Stokes' theorem}]                                  \\
        & \leq & \pm \int_{0}^{1} \rho_{\pm}^* \alpha \mp \int_{0}^{1} \alpha(X_H(\mathbf{u}(s,t))) \edv t                                                        & \quad [\text{since $f^* \edv \alpha \geq 0$}]                         \\
        & =    & \pm \int_{0}^{1} \del_r h(\pm \infty, t, z_{\pm}, e^{r_{\pm}}) \edv t \mp \int_{0}^{1} \del_r h(s, t, w(s), e^{a(s,t)}) \edv t                   & \quad [\text{by \cref{lem:reeb equals hamiltonian on symplectization}}]             \\
        & \leq & \pm \int_{0}^{1} \del_r h(\pm \infty, t, z_{\pm}, e^{r_{\pm}}) \edv t \mp \int_{0}^{1} \del_r h(s, t, w(s), e^{r_{\pm}}) \edv t                  & \quad [\text{by Equation \eqref{lem:asymptotic behaviour gen 1b}}] \\
        & \leq & \pm \int_{0}^{1} \del_r h(\pm \infty, t, z_{\pm}, e^{r_{\pm}}) \edv t \mp \int_{0}^{1} \del_r h(s, t, z_{\pm}, e^{r_{\pm}}) \edv t               & \quad [\text{by Equation \eqref{lem:asymptotic behaviour gen 2b}}] \\
        & \leq & 0                                                                                                                                                & \quad [\text{by Equation \eqref{lem:asymptotic behaviour gen 3}}].
    \end{IEEEeqnarray*}
    Since $\pm \del_s \overline{a}(s) \leq 0$ and $\overline{a}(\pm \infty) = r_{\pm}$, we have that $\overline{a}(s) \geq r_{\pm}$ for all $s$. By assumption, $a(s,t) \leq r_{\pm}$, and therefore $a(s,t) = r_{\pm}$ for all $(s,t) \in Z^{\pm}$. This implies that every inequality in the previous computation is an equality, and in particular $f^* \edv \alpha = 0$. Therefore, $f$ is independent of $s$ and $f(s,t) = \rho_{\pm}(t)$ for all $(s,t) \in Z^{\pm}$.
\end{proof}

The following lemma is an adaptation of a result originally proven by Abouzaid--Seidel \cite[Lemma 7.2]{abouzaidOpenStringAnalogue2010}. Other authors have proven variations of this result, namely Ritter \cite[Lemma D.3]{ritterTopologicalQuantumField2013}, Gutt \cite[Theorem 3.1.6]{guttMinimalNumberPeriodic2014} and Cieliebak--Oancea \cite[Lemma 2.2]{cieliebakSymplecticHomologyEilenberg2018}.

\begin{lemma}[no escape]
    \label{lem:no escape}
    Let $V \subset (X, \lambda)$ be a Liouville domain such that $\iota \colon V \longrightarrow (X, \lambda)$ is a strict Liouville embedding, $H \colon \dot{\Sigma} \times S \times \hat{X} \longrightarrow \R$ be an admissible Hamiltonian, $J \colon \dot{\Sigma} \times S \times \hat{X} \longrightarrow \End(T \hat{X})$ be a compatible almost complex structure and $\mathbf{u} = (\id_{\dot{\Sigma}}, w, u) \colon \dot{\Sigma} \longrightarrow \dot{\Sigma} \times S \times \hat{X}$ be a solution of the parametrized Floer equation such that all the asymptotic $1$-periodic orbits of $\mathbf{u}$ are inside $V$. Assume that there exists $\varepsilon > 0$ such that:
    \begin{enumerate}
        \item The restriction of $H$ to $\dot{\Sigma} \times S \times (-\varepsilon, \varepsilon) \times \del V$ is independent of $\del V$.
        \item The restriction of \parbox{\widthof{$H$}}{$J$} to $\dot{\Sigma} \times S \times (-\varepsilon, \varepsilon) \times \del V$ is cylindrical.
        \item If $\mathcal{A}_{H} \colon \dot{\Sigma} \times S \times (-\varepsilon,\varepsilon) \longrightarrow \R$ is given by $\mathcal{A}_H(z,w,r) \coloneqq \lambda(X_H)(z,w,r) - H(z,w,r)$, then for every $(z,w,r) \in \dot{\Sigma} \times S \times (-\varepsilon,\varepsilon)$,
            \begin{IEEEeqnarray*}{rCls+x*}
                \mathcal{A}_H(z,w,r) \, \edv \tau                                                 & \leq & 0, \plabel{eq:no escape eq 1} \\
                \edv_{\dot{\Sigma}} \mathcal{A}_H(z,w,r) \wedge \tau                              & \leq & 0, \plabel{eq:no escape eq 2} \\
                \p{<}{}{\nabla_S \mathcal{A}_H(z,w,r), \nabla \tilde{f}(w)} \, \sigma \wedge \tau & \leq & 0. \plabel{eq:no escape eq 3}
            \end{IEEEeqnarray*} 
    \end{enumerate}
    Then, $\img u \subset V$.
\end{lemma}
\begin{proof}
    Assume by contradiction that $\img u$ is not contained in $V$. After changing $V$ to $\hat{V} \setminus \{ (r,x) \in \R \times \del V \mid r > r_0 \}$, for some $r_0 \in (-\varepsilon,\varepsilon)$, we may assume without loss of generality that $\img u$ is not contained in $V$ and that $u$ is transverse to $\del V$. Then, ${\Sigma_V} \coloneqq u ^{-1}(\hat{X} \setminus \itr V)$ is a compact surface with boundary. We show that $E({u}|_{\Sigma_V}) = 0$.
    \begin{IEEEeqnarray*}{rCls+x*}
        0
        & \leq & \frac{1}{2} \int_{\Sigma_V}^{} \| \dv u - X_{H} (\mathbf{u}) \tensorpr \tau \|^2_{J(\mathbf{u}), \edv \lambda} \, \omega _{\Sigma_V} & \quad [\text{by positivity of norms}]                                         \\
        & \leq & \int_{{\Sigma_V}}   \edv (u^* \lambda - H(\mathbf{u}) \, \tau)                                                                       & \quad [\text{by the computation in \cref{lem:action energy for floer trajectories}}]    \\
        & =    & \int_{\del {\Sigma_V}}^{} u^* \lambda - H(\mathbf{u}) \, \tau                                                                        & \quad [\text{by Stokes' theorem}]                                          \\
        & \leq & \int_{\del {\Sigma_V}}^{} u^* \lambda - \lambda(X_H(\mathbf{u})) \, \tau                                                             & \quad [\text{(a), proven below}]                                           \\
        & =    & \int_{\del {\Sigma_V}}^{} \lambda \circ (\dv u - X_H(\mathbf{u}) \tensorpr \tau)                                                     & \quad [\text{by definition of pullback}]                                   \\
        & =    & - \int_{\del {\Sigma_V}}^{} \lambda \circ J(\mathbf{u}) \circ (\dv u - X_H(\mathbf{u}) \tensorpr \tau) \circ j                       & \quad [\text{$\dv u - X_H(\mathbf{u}) \tensorpr \tau$ is holomorphic}]              \\
        & =    & - \int_{\del {\Sigma_V}}^{} \edv \exp \circ (\dv u - X_H(\mathbf{u}) \tensorpr \tau) \circ j                                         & \quad [\text{$J$ is cylindrical near $u(\del {\Sigma_V}) \subset \del V$}] \\
        & =    & - \int_{\del {\Sigma_V}}^{} \edv \exp \circ \dv u \circ j                                                                            & \quad [\text{$H$ is independent of $\del V$}]                              \\
        & \leq & 0                                                                                                                                    & \quad [\text{(b), proven below}].
    \end{IEEEeqnarray*}
    The proof of (a) is the computation
    \begin{IEEEeqnarray*}{rCls+x*}
        \IEEEeqnarraymulticol{3}{l}{\int_{\del {\Sigma_V}}^{} ( \lambda(X_H(\mathbf{u})) - H(\mathbf{u}) ) \, \tau}\\ \quad
        & =    & \int_{\del {\Sigma_V}}^{} \mathcal{A}_H(z,w,r_0) \, \tau                & \quad [\text{by definition of $\mathcal{A}_H$ and $u(\del {\Sigma_V}) \subset \del V$}] \\
        & =    & \int_{{\Sigma_V}}^{} \edv_{{\Sigma_V}} (\mathcal{A}_H(z,w,r_0) \, \tau) & \quad [\text{by Stokes' theorem}]                                                       \\
        & \leq & 0                                                                       & \quad [\text{by the assumptions on $\mathcal{A}_H$}].
    \end{IEEEeqnarray*}
    Statement (b) is true because if $\xi$ is a vector tangent to $\del {\Sigma_V}$ giving the boundary orientation, then $j (\xi)$ points into ${\Sigma_V}$, therefore $\dv u \circ j (\xi)$ points out of $V$. Then, we conclude that $E({u}|_{\Sigma_V}) = 0$ and that $\dv u = X_H(\mathbf{u}) \tensorpr \tau$, and since $X_H(\mathbf{u})$ is tangent to $\del V$ it follows that $\img u \subset \del V$. This contradicts the fact that $u$ is not contained in $V$.
\end{proof}

\section{Compactness for solutions of the parametrized Floer equation}

In this section, we assume that $(\dot{\Sigma}, j, \sigma, \tau) = (\R \times S^1, j, \edv s, \edv t)$ is the cylinder from \cref{exa:sphere and cylinder}. Suppose that $H \colon \dot{\Sigma} \times S \times \hat{X} \longrightarrow \R$ is an admissible Hamiltonian as in \cref{def:admissible hamiltonian abstract}. In this case, there exist Hamiltonians $H^{\pm} \colon S^1 \times S \times \hat{X} \longrightarrow \R$ such that $H(s,t,w,x) = H^{\pm}(t,w,x)$ for $\pm s \geq s_0$. Assume also that $J \colon \dot{\Sigma} \times S \times \hat{X} \longrightarrow \End(T \hat{X})$ is an admissible almost complex structure as in \cref{def:admissible acs abstract}, which has associated limit almost complex structures $J^{\pm} \colon S^1 \times S \times \hat{X} \longrightarrow \End(T \hat{X})$. Note that since $\dot{\Sigma} = \R \times S^1$, we can also view $H^{\pm}$ and $J^{\pm}$ as maps whose domain is $\dot{\Sigma}$. For $N \in \Z_{\geq 1}$ and $L, \nu = 1,\ldots,N$, define 
\begin{IEEEeqnarray*}{c+x*}
    H^{L,\nu} \coloneqq
    \begin{cases}
        H^{+} & \text{if } \nu > L, \\
        H     & \text{if } \nu = L, \\
        H^{-} & \text{if } \nu < L,
    \end{cases}
    \quad
    J^{L,\nu} \coloneqq
    \begin{cases}
        J^{+} & \text{if } \nu > L, \\
        J     & \text{if } \nu = L, \\
        J^{-} & \text{if } \nu < L.
    \end{cases}
\end{IEEEeqnarray*}
Finally, let $(H_m)_m$ be a sequence of admissible Hamiltonians converging to $H$, $(J_m)_m$ be a sequence of admissible almost complex structures converging to $J$, and for every $m \in \Z_{\geq 1}$ let $(w_m, u_m)$ be a solution of the parametrized Floer equation with respect to $H_m, J_m$ with asymptotes $(z^\pm_m, \gamma^\pm_m)$.

\begin{definition}
    \label{def:broken floer cylinder}
    Let $(z^{\pm}, \gamma^{\pm})$ be Hamiltonian $1$-periodic orbits of $H^{\pm}$. A \textbf{broken Floer trajectory} from $(z^-, \gamma^-)$ to $(z^+, \gamma^+)$ is given by:
    \begin{enumerate}
        \item Numbers $N \in \Z_{\geq 1}$ and $L = 1, \ldots, N$;
        \item Hamiltonian $1$-periodic orbits $(z^-, \gamma^-) = (z^1, \gamma^1), \ldots, (z^L, \gamma^L)$ of $H^-$ and Hamiltonian $1$-periodic orbits $(z^{L+1}, \gamma^{L+1}), \ldots, (z^{N+1}, \gamma^{N+1}) = (z^+, \gamma^+)$ of $H^+$;
        \item For every $\nu = 1, \ldots, N$, a Floer trajectory $(w^\nu,u^\nu)$ with respect to $H^{L,\nu}, J^{L,\nu}$ with negative asymptote $(z^\nu, \gamma^\nu)$ and positive asymptote $(z^{\nu+1}, \gamma^{\nu+1})$.
    \end{enumerate}
\end{definition}

\begin{definition}
    We say that $(w_m, u_m)_{m}$ \textbf{converges} to $(w^{\nu}, u^{\nu})_{\nu}$ if there exist numbers $s^1_m \leq \cdots \leq s^N_m$ such that
    \begin{IEEEeqnarray*}{rCls+x*}
        \lim_{m \to +\infty} s^L_m                        & \in & \R, \\
        \lim_{m \to +\infty} (s^{\nu + 1}_m - s^\nu_m)    & =   & + \infty, \\
        \lim_{m \to +\infty} w_m( \cdot + s^\nu_m)        & =   & w^\nu_m, \\
        \lim_{m \to +\infty} u_m( \cdot + s^\nu_m, \cdot) & =   & u^\nu_m.
    \end{IEEEeqnarray*}
\end{definition}

\begin{theorem}
    \label{thm:compactness in s1eft}
    There exists a subsequence (whose index we still denote by $m$) and a broken Floer trajectory $(w^{\nu}, u^{\nu})_{\nu}$ such that $(w_m, u_m)_m$ converges to $(w^{\nu}, u^{\nu})_{\nu}$.
\end{theorem}
\begin{proof}
    Since $f \colon C \longrightarrow \R$ is Morse and $H_{z,w} \colon S^1 \times \hat{X} \longrightarrow \R$ is nondegenerate for every puncture $z$ and critical point $w$ of $\tilde{f}$, we conclude that we can pass to a subsequence such that $(z_m^{\pm}, \gamma_m^{\pm})$ converges to $(z^{\pm}, \gamma^{\pm})$. By compactness in Morse theory, there exists a further subsequence and a broken Morse trajectory $(w^\nu)_{\nu = 1,\ldots,N}$, where $w^{\nu} \colon \R \longrightarrow S$ is a Morse trajectory from $z^{\nu}$ to $z^{\nu + 1}$, $z^1 = z^-$ and $z^{N+1} = z^+$, such that $(w_m)_m$ converges in the sense of Morse theory to $(w^{\nu})_{\nu}$. More precisely, this means that there exist numbers $s^1_m \leq \cdots \leq s^N_m$ and $L \leq N$ such that
    \begin{IEEEeqnarray*}{rCls+x*}
        \lim_{m \to +\infty} s^L_m                   & \in & \R, \\
        \lim_{m \to +\infty} (s^{\nu+1}_m - s^\nu_m) & =   & + \infty, \\
        \lim_{m \to +\infty} w_m(\cdot + s^\nu_m)    & =   & w^\nu.
    \end{IEEEeqnarray*}
    Possibly after reparametrizing the $w^\nu$, we may assume that $s^L_m = 0$ for every $m$. Now, for $\nu = 1,\ldots,N$, define
    \begin{IEEEeqnarray*}{rCLCRCl}
        u^\nu_m \colon \R \times S^1                & \longrightarrow & \hat{X},         & \quad & u^\nu_m(s,t)   & = & u_m(s + s^\nu_m, t),                \\
        H^\nu_m \colon \R \times S^1 \times \hat{X} & \longrightarrow & \R,              & \quad & H^\nu_m(s,t,x) & = & H_m(s + s^\nu_m, t, w_m(s + s^\nu_m), x), \\
        H^\nu   \colon \R \times S^1 \times \hat{X} & \longrightarrow & \R,              & \quad & H^\nu(s,t,x)   & = & H^{L,\nu}(s, t, w^\nu(s), x), \\
        J^\nu_m \colon \R \times S^1 \times \hat{X} & \longrightarrow & \End(T \hat{X}), & \quad & J^\nu_m(s,t,x) & = & J_m(s + s^\nu_m, t, w_m(s + s^\nu_m), x), \\
        J^\nu   \colon \R \times S^1 \times \hat{X} & \longrightarrow & \End(T \hat{X}), & \quad & J^\nu(s,t,x)   & = & J^{L,\nu}(s, t, w^\nu(s), x).
    \end{IEEEeqnarray*}
    Then, $u^\nu_m$ is a solution of the equation 
    \begin{IEEEeqnarray*}{c+x*}
        \pdv{u^\nu_m}{s} = - J^\nu_m(s,t,u^\nu_m) \p{}{2}{ \pdv{u^\nu_m}{t} - X_{H^\nu_m}(s,t,u^\nu_m) },      
    \end{IEEEeqnarray*}
    and 
    \begin{IEEEeqnarray*}{rCls+x*}
        \lim_{m \to + \infty} H^\nu_m & = & H^\nu, \\
        \lim_{m \to + \infty} J^\nu_m & = & J^\nu.
    \end{IEEEeqnarray*}
    By compactness in Floer theory, there exists a further subsequence such that for every $\nu = 1,\ldots,N$ there exists a broken Floer trajectory $(u^{\nu,\mu})_{\mu = 1,\ldots,M_{\nu}}$ from $\gamma^{\nu,\mu}$ to $\gamma^{\nu,\mu+1}$ with respect to $(H^\nu, J^\nu)$, such that
    \begin{IEEEeqnarray*}{rCls+x*}
        \gamma^{1,1}     & = & \gamma^-, \\
        \gamma^{N,M_{N}} & = & \gamma^+,
    \end{IEEEeqnarray*}
    and $(u^\nu_m)_m^{}$ converges to $(u^{\nu,\mu})_{\mu}$. More precisely, this means that there exist $L_\nu = 1,\ldots,N_\nu$ and numbers $s_m^{\nu,1} \leq \cdots \leq s_m^{\nu,M_\nu}$ such that 
    \begin{IEEEeqnarray*}{rCls+x*}
        \lim_{m \to +\infty} s_m^{\nu,L_\nu}                         & \in & \R, \\
        \lim_{m \to +\infty} (s_m^{\nu,\mu+1}  - s_m^{\nu,\mu})      & =   & + \infty, \\
        \lim_{m \to +\infty} u^{\nu}_m(\cdot + s^{\nu,\mu}_m, \cdot) & =   & u^{\nu,\mu}.
    \end{IEEEeqnarray*}
    Consider the list $(w^\nu, u^{\nu,\mu})_{\nu,\mu}$ ordered according to the dictionary order of the indices $\nu, \mu$. In this list, if two elements $(w^\nu, u^{\nu,\mu})$, $(w^{\nu'}, u^{\nu',\mu'})$ are equal then they must be adjacent. The list obtained from $(w^\nu, u^{\nu,\mu})_{\nu,\mu}$ by removing duplicate elements is the desired broken Floer trajectory.
\end{proof}

\section{Transversality for solutions of the parametrized Floer equation}

In this section, let $(\dot{\Sigma}, j, \sigma, \tau) = (\R \times S^1, j, \edv s, \edv t)$ be the cylinder from \cref{exa:sphere and cylinder} and $(X, \lambda)$ be a nondegenerate Liouville domain. Let $H \colon S^1 \times S \times \hat{X} \longrightarrow \R$ be a function such that the pullback $H \colon \R \times S^1 \times S \times \hat{X} \longrightarrow \R$ is as in \cref{def:admissible hamiltonian abstract}. Define $\mathcal{J}$ to be the set of almost complex structures $J \colon S^1 \times S \times \hat{X} \longrightarrow \End(T \hat{X})$ such that the pullback $J \colon \R \times S^1 \times S \times \hat{X} \longrightarrow \End(T \hat{X})$ is as in \cref{def:admissible acs abstract}. The set $\mathcal{J}$ admits the structure of a smooth Fréchet manifold, and therefore the tangent space $T_{J} \mathcal{J}$ at $J$ is a Fréchet space. Let $(z^{\pm}, \gamma^{\pm})$ be $1$-periodic orbits of $H$, i.e. $z^{\pm} \in S$ is a critical point of $\tilde{f}$ and $\gamma^{\pm}$ is a $1$-periodic orbit of $H_{z^{\pm}} \colon S^1 \times \hat{X} \longrightarrow \R$.

If $w \colon \R \longrightarrow S$ and $u \colon \R \times S^1 \longrightarrow \hat{X}$ are maps, we will denote by $\mathbf{u}$ the map
\begin{IEEEeqnarray*}{c+x*}
    \mathbf{u} \colon \R \times S^1 \longrightarrow S^1 \times S \times \hat{X}, \qquad \mathbf{u}(s,t) \coloneqq (t, w(s), u(s,t)).
\end{IEEEeqnarray*}
The pair $(w,u)$ is a solution of the parametrized Floer equation if
\begin{IEEEeqnarray*}{rCls+x*}
    \partial_s w - \nabla \tilde{f}(w)                              & = & 0, \\
    (\dv u - X_H(\mathbf{u}) \otimes \tau)^{0,1}_{J(\mathbf{u}), j} & = & 0.
\end{IEEEeqnarray*}

Define $[z^{\pm}, \gamma^{\pm}]$ to be the equivalence class
\begin{IEEEeqnarray*}{rCls+x*}
    [z^{\pm}, \gamma^{\pm}] 
    & \coloneqq & \{ t \cdot (z^{\pm}, \gamma^{\pm}) \mid t \in S^1 \}  \\
    & =         & \{ (t \cdot z^{\pm}, \gamma^{\pm}(\cdot + t)) \mid t \in S^1 \},
\end{IEEEeqnarray*}
and denote by $\hat{\mathcal{M}}(X,H,J,[z^+,\gamma^+],[z^-,\gamma^-])$ the moduli space of solutions $(w,u) \in C^{\infty}(\R, S) \times C^{\infty}(\R \times S^1, \hat{X})$ of the parametrized Floer equation such that
\begin{IEEEeqnarray*}{c+x*}
    \lim_{s \to \pm \infty} (w(s), u(s, \cdot)) \in [z^{\pm}, \gamma^{\pm}]. 
\end{IEEEeqnarray*}

Denote by $\mathcal{M}$ the moduli space of gradient flow lines $w \colon \R \longrightarrow S$ of $\tilde{f}$ such that
\begin{IEEEeqnarray*}{c+x*}
    \lim_{s \to \pm \infty} w(s) \in [z^{\pm}].
\end{IEEEeqnarray*}
By the assumptions on $(S, g^{S}, \tilde{f})$ explained in \cref{sec:floer trajectories} and \cite[Section 3.2]{austinMorseBottTheoryEquivariant1995}, the space $\mathcal{M}$ is a smooth finite dimensional manifold. Moreover,
\begin{IEEEeqnarray}{c+x*}
    \dim \mathcal{M} = \morse(z^+) + \morse(z^-) + 1. \plabel{eq:dimension of m}
\end{IEEEeqnarray}

Let $\varepsilon = (\varepsilon_{\ell})_{\ell \in \Z_{\geq 0}}$ be a sequence of positive numbers $\varepsilon_{\ell}$ such that $\lim_{\ell \to +\infty} \varepsilon_{\ell} = 0$. Define a function
\begin{IEEEeqnarray*}{rrCl}
    \| \cdot \|^{\varepsilon} \colon & T_{J_{\mathrm{ref}}} \mathcal{J} & \longrightarrow & [0, + \infty] \\
    & Y & \longmapsto & \sum_{\ell=0}^{+ \infty} \varepsilon_{\ell} \| Y \|_{C^{\ell}(S^1 \times S \times X)}, 
\end{IEEEeqnarray*}
where $\| \cdot \|_{C^{\ell}(S^1 \times S \times X)}$ is the $C^{\ell}$-norm which is determined by some finite covering of $T {X} \longrightarrow S^1 \times S \times X$ by coordinate charts and local trivializations. Define 
\begin{IEEEeqnarray*}{c+x*}
    T^{\varepsilon}_{J_{\mathrm{ref}}} \mathcal{J} \coloneqq \{ Y \in T_{J_{\mathrm{ref}}} \mathcal{J} \mid \| Y \|^{\varepsilon} < + \infty \}.
\end{IEEEeqnarray*}
By \cite[Lemma 5.1]{floerUnregularizedGradientFlow1988}, $(T^{\varepsilon}_{J_{\mathrm{ref}}} \mathcal{J}, \| \cdot \|^{\varepsilon})$ is a Banach space consisting of smooth sections and containing sections with support in
arbitrarily small sets. For every $Y \in T_{J_{\mathrm{ref}}}^{\varepsilon} \mathcal{J}$, define
\begin{IEEEeqnarray*}{c+x*}
    \exp_{J_{\mathrm{ref}}}(Y) \coloneqq J_{Y} \coloneqq \p{}{2}{1 + \frac{1}{2} J_{\mathrm{ref}} Y} J_{\mathrm{ref}} \p{}{2}{1 + \frac{1}{2} J_{\mathrm{ref}} Y}^{-1}.
\end{IEEEeqnarray*}
There exists a neighbourhood $\mathcal{O} \subset T_{J_{\mathrm{ref}}}^{\varepsilon} \mathcal{J}$ of $0$ such that $\exp_{J_{\mathrm{ref}}}^{} \colon \mathcal{O} \longrightarrow \mathcal{J}$ is injective. Define $\mathcal{J}^{\varepsilon} \coloneqq \exp_{J_{\mathrm{ref}}}^{}(\mathcal{O})$, which is automatically a Banach manifold with one global parametrization $\exp_{J_{\mathrm{ref}}}^{} \colon \mathcal{O} \longrightarrow \mathcal{J}^{\varepsilon}$. The tangent space of $\mathcal{J}^{\varepsilon}$ at $J_{\mathrm{ref}}$ is given by
\begin{IEEEeqnarray*}{c+x*}
    T_{J_{\mathrm{ref}}} \mathcal{J}^{\varepsilon} = T_{J_{\mathrm{ref}}}^{\varepsilon} \mathcal{J}.
\end{IEEEeqnarray*}
Notice that the definition of $\mathcal{J}^{\varepsilon}$ involved making several choices, namely the sequence $\varepsilon$, the choices necessary to define the $C^{\ell}$-norm, and a reference almost complex structure $J_{\mathrm{ref}}$.

\begin{definition}
    For $w \in \mathcal{M}$, let $\mathcal{F}_w$ be the Banach manifold of maps $u \colon \R \times S^1 \longrightarrow \hat{X}$ of the form
    \begin{IEEEeqnarray*}{c+x*}
        u(s,t) = \exp_{u_0(s,t)} \xi(s,t),
    \end{IEEEeqnarray*}
    where
    \begin{IEEEeqnarray*}{rCls+x*}
        u_0 & \in & C^{\infty}(\R \times S^1, \hat{X}) \text{ is such that } \lim_{s \to \pm \infty} (w(s), u_0(s, \cdot)) \in [z^{\pm}, \gamma^{\pm}], \\
        \xi & \in & W^{1,p}(\R \times S^1, u_0^* T \hat{X}).
    \end{IEEEeqnarray*}
\end{definition}

\begin{definition}
    For $J \in \mathcal{J}^{\varepsilon}$, we define a bundle $\pi^J \colon \mathcal{E}^J \longrightarrow \mathcal{B}$ as follows. The base, fibre and total space are given by
    \begin{IEEEeqnarray*}{rCls+x*}
        \mathcal{B}           & \coloneqq & \{ (w,u) \mid w \in \mathcal{M}, \, u \in \mathcal{F}_w \}, \\
        \mathcal{E}^J_{(w,u)} & \coloneqq & L^p(\Hom^{0,1}_{J(\mathbf{u}), j} (T \dot{\Sigma}, u^* T \hat{X})), \\
        \mathcal{E}^J         & \coloneqq & \{ (w,u,\xi) \mid (w,u) \in \mathcal{B}, \, \xi \in \mathcal{E}^J_{(w,u)} \}.
    \end{IEEEeqnarray*}
    The projection is given by $\pi^J(w,u,\xi) \coloneqq (w,u)$. The \textbf{Cauchy--Riemann operator} is the section $\delbar\vphantom{\partial}^J \colon \mathcal{B} \longrightarrow \mathcal{E}^J$ given by
    \begin{IEEEeqnarray*}{c+x*}
        \delbar\vphantom{\partial}^J(w,u) \coloneqq (\dv u - X_H(\mathbf{u}) \otimes \tau)^{0,1}_{J(\mathbf{u}),j} \in \mathcal{E}^J_{(w,u)}.
    \end{IEEEeqnarray*}
\end{definition}

With this definition, $(\delbar\vphantom{\partial}^J)^{-1}(0) = \hat{\mathcal{M}}(X,H,J,[z^+,\gamma^+],[z^-,\gamma^-])$.

\begin{definition}
    Define the universal bundle, $\pi \colon \mathcal{E} \longrightarrow \mathcal{B} \times \mathcal{J}^{\varepsilon}$, and the \textbf{universal Cauchy--Riemann operator}, $\delbar \colon \mathcal{B} \times \mathcal{J}^{\varepsilon} \longrightarrow \mathcal{E}$, by
    \begin{IEEEeqnarray*}{rCls+x*}
        \mathcal{E} & \coloneqq & \{ (w,u,J,\xi) \mid (w,u) \in \mathcal{B}, \, J \in \mathcal{J}^{\varepsilon}, \, \xi \in \mathcal{E}^{J}_{(w,u)} \}, \\
        \pi         & \colon    & \mathcal{E} \longrightarrow \mathcal{B} \times \mathcal{J}^{\varepsilon}, \qquad \pi(w,u,J,\xi) \coloneqq (w,u,J), \\
        \delbar     & \colon    & \mathcal{B} \times \mathcal{J}^{\varepsilon} \longrightarrow \mathcal{E}, \qquad \delbar(w,u,J) \coloneqq \delbar\vphantom{\partial}^J(w,u).
    \end{IEEEeqnarray*}
\end{definition}

For $(w,u,J)$ such that $\delbar(w,u,J) = 0$, choose a splitting $T_{(w,u)} \mathcal{B} = T_w \mathcal{M} \oplus T_u \mathcal{F}_w$. The sections $\delbar\vphantom{\partial}^J$ and $\delbar$ have corresponding linearized operators, which we denote by
\begin{IEEEeqnarray*}{rCls+x*}
    \mathbf{D}_{(w,u,J)} & \colon & T_w \mathcal{M} \oplus T_u \mathcal{F}_w \longrightarrow \mathcal{E}^J_{(w,u)}, \\
    \mathbf{L}_{(w,u,J)} & \colon & T_w \mathcal{M} \oplus T_u \mathcal{F}_w \oplus T_J \mathcal{J}^{\varepsilon} \longrightarrow \mathcal{E}^J_{(w,u)},
\end{IEEEeqnarray*}
respectively. We can write these operators with respect to the decompositions above as block matrices
\begin{IEEEeqnarray}{rCl}
    \mathbf{D}_{(w,u,J)}
    & = &
    \begin{bmatrix}
        \mathbf{D}^{\mathcal{M}}_{(w,u,J)} & \mathbf{D}^{\mathcal{F}}_{(w,u,J)}
    \end{bmatrix}, \plabel{eq:splitting linearized ops 1}
    \\
    \mathbf{L}_{(w,u,J)}
    & = &
    \begin{bmatrix}
        \mathbf{D}^{\mathcal{M}}_{(w,u,J)} & \mathbf{D}^{\mathcal{F}}_{(w,u,J)} & \mathbf{J}_{(w,u,J)}
    \end{bmatrix}. \plabel{eq:splitting linearized ops 2}
\end{IEEEeqnarray}

Let $\tau$ be a trivialization of $u^* T \hat{X}$ and denote also by $\tau$ the induced trivializations of $(\gamma^{\pm})^* T \hat{X}$. We can consider the Conley--Zehnder indices $\conleyzehnder^{\tau}(\gamma^{\pm})$ of $\gamma^{\pm}$ computed with respect to $\tau$. We denote $\ind^{\tau}(z^{\pm}, \gamma^{\pm}) \coloneqq \morse(z^\pm) + \conleyzehnder^{\tau}(\gamma^{\pm})$.

\begin{theorem}
    \phantomsection\label{thm:s1eft d is fredholm}
    The operators $\mathbf{D}^{\mathcal{F}}_{(w,u,J)}$ and $\mathbf{D}_{(w,u,J)}$ are Fredholm and
    \begin{IEEEeqnarray}{rCls+x*}
        \operatorname{ind} \mathbf{D}^{\mathcal{F}}_{(w,u,J)} & = & \conleyzehnder^{\tau}(\gamma^+) - \conleyzehnder^{\tau}(\gamma^-), \plabel{eq:s1eft fredholm ind 1} \\
        \operatorname{ind} \mathbf{D}_{(w,u,J)}               & = & \ind^{\tau}(z^+, \gamma^+) - \ind^{\tau}(z^-,\gamma^-) + 1. \plabel{eq:s1eft fredholm ind 2}
    \end{IEEEeqnarray}
\end{theorem}
\begin{proof}
    The operator $\mathbf{D}^{\mathcal{F}}_{(w,u,J)}$ is the linearized operator in Floer theory, which is Fredholm and has index given by Equation \eqref{eq:s1eft fredholm ind 1}. Therefore,
    \begin{IEEEeqnarray*}{c+x*}
        0 \oplus \mathbf{D}^{\mathcal{F}}_{(w,u,J)} \colon T_w \mathcal{M} \oplus T_u \mathcal{F}_w \longrightarrow \mathcal{E}^J_{(w,u)}
    \end{IEEEeqnarray*}
    is Fredholm and
    \begin{IEEEeqnarray}{c+x*}
        \operatorname{ind} (0 \oplus \mathbf{D}^{\mathcal{F}}_{(w,u,J)}) = \dim T_w \mathcal{M} + \operatorname{ind} \mathbf{D}^{\mathcal{F}}_{(w,u,J)}. \plabel{eq:index of operator floer}
    \end{IEEEeqnarray}
    Since $\mathbf{D}^{\mathcal{M}}_{(w,u,J)} \oplus 0 \colon T_w \mathcal{M} \oplus T_w \mathcal{F}_w \longrightarrow \mathcal{E}^J_{(w,u)}$ is compact, the operator
    \begin{IEEEeqnarray*}{c+x*}
        \mathbf{D}_{(w,u,J)} = \mathbf{D}^{\mathcal{M}}_{(w,u,J)} \oplus \mathbf{D}^{\mathcal{F}}_{(w,u,J)} = \mathbf{D}^{\mathcal{M}}_{(w,u,J)} \oplus 0 + 0 \oplus \mathbf{D}^{\mathcal{F}}_{(w,u,J)}
    \end{IEEEeqnarray*}
    is Fredholm and
    \begin{IEEEeqnarray*}{rCls+x*}
        \operatorname{ind} \mathbf{D}_{(w,u,J)}
        & = & \operatorname{ind} (\mathbf{D}^{\mathcal{M}}_{(w,u,J)} \oplus \mathbf{D}^{\mathcal{F}}_{(w,u,J)}) & \quad [\text{by Equation \eqref{eq:splitting linearized ops 1}}] \\
        & = & \operatorname{ind} (0 \oplus \mathbf{D}^{\mathcal{F}}_{(w,u,J)})                                  & \quad [\text{since $\mathbf{D}^{\mathcal{M}}_{(w,u,J)}$ is compact}] \\
        & = & \dim T_w \mathcal{M} + \operatorname{ind} \mathbf{D}^{\mathcal{F}}_{(w,u,J)}                      & \quad [\text{by Equation \eqref{eq:index of operator floer}}] \\
        & = & \ind^{\tau}(z^+, \gamma^+) - \ind^{\tau}(z^-,\gamma^-) + 1                                        & \quad [\text{by Equations \eqref{eq:dimension of m} and \eqref{eq:s1eft fredholm ind 1}}]. & \qedhere
    \end{IEEEeqnarray*}
\end{proof}

\begin{theorem}
    \label{thm:s1eft l is surjective}
    The operator $\mathbf{L}_{(w,u,J)}$ is surjective.
\end{theorem}
\begin{proof}
    It suffices to prove that 
    \begin{IEEEeqnarray*}{c+x*}
        \mathbf{L}^{\mathcal{F}}_{(w,u,J)} \coloneqq \mathbf{D}^{\mathcal{F}}_{(w,u,J)} \oplus \mathbf{J}_{(w,u,J)} \colon T_u \mathcal{F}_w \oplus T_J \mathcal{J}^{\varepsilon} \longrightarrow \mathcal{E}^{J}_{(w,u)}
    \end{IEEEeqnarray*}
    is surjective. Since $\mathbf{D}^{\mathcal{F}}_{(w,u,J)}$ is Fredholm (by \cref{thm:s1eft d is fredholm}), its image is closed and has finite codimension. This implies that $\img \mathbf{L}^{\mathcal{F}}_{(w,u,J)}$ is also of finite codimension and closed. So, it suffices to show that $\img \mathbf{L}^{\mathcal{F}}_{(w,u,J)}$ is dense, which is equivalent to showing that the annihilator $\Ann \img \mathbf{L}^{\mathcal{F}}_{(w,u,J)}$ is zero. Let $\eta \in \Ann \img \mathbf{L}^{\mathcal{F}}_{(w,u,J)}$, i.e. 
    \begin{IEEEeqnarray*}{c+x*}
        \eta \in L^q(\Hom^{0,1}_{J(\mathbf{u}), j} (T \dot{\Sigma}, u^* T \hat{X}))
    \end{IEEEeqnarray*}
    is such that
    \begin{IEEEeqnarray}{rClCsrCl}
        0 & = & \p{<}{}{\eta, \mathbf{D}^{\mathcal{F}}_{(w,u,J)}(\xi)}_{L^2} & \quad & \text{ for all } & \xi & \in & T_u \mathcal{F}_w,           \plabel{eq:element in annihilator 1} \\
        0 & = & \p{<}{}{\eta, \mathbf{J}     _{(w,u,J)}(Y  )}_{L^2} & \quad & \text{ for all } & Y   & \in & T_J \mathcal{J}^{\varepsilon}. \plabel{eq:element in annihilator 2}
    \end{IEEEeqnarray}
    By Equation \eqref{eq:element in annihilator 1}, $\eta$ satisfies the Cauchy--Riemann type equation $(\mathbf{D}^{\mathcal{F}}_{(w,u,J)})^{*} \eta = 0$, and therefore $\eta$ is smooth (by elliptic regularity) and satisfies unique continuation.

    We prove that $\eta = 0$ in the case where $w$ is constant. In this case, $w(s) \eqqcolon w_0$ for every $s$, we can view $\gamma^{\pm}$ as $1$-periodic orbits of $H_{w_0}$ (after a reparametrization) and $u$ is a solution of the Floer equation:
    \begin{IEEEeqnarray*}{c+x*}
        \pdv{u}{s}(s,t) + J_{w_0}(t,u(s,t)) \p{}{2}{ \pdv{u}{t}(s,t) - X_{H_{w_0}}(t,u(s,t)) } = 0.
    \end{IEEEeqnarray*}
    Let $R(u)$ be the set of regular points of $u$, i.e. points $z = (s,t)$ such that
    \begin{IEEEeqnarray}{c+x*}
        \plabel{eq:set of regular points of u}
        \pdv{u}{s}(s,t) \neq 0, \qquad u(s,t) \neq \gamma^{\pm}(t), \qquad u(s,t) \notin u(\R - \{s\}, t).
    \end{IEEEeqnarray}
    By \cite[Theorem 4.3]{floerTransversalityEllipticMorse1995}, $R(u)$ is open. By unique continuation, it is enough to show that $\eta$ vanishes in $R(u)$. Let $z_0 = (s_0,t_0) \in R(u)$ and assume by contradiction that $\eta(z_0) \neq 0$. By \cite[Lemma 3.2.2]{mcduffHolomorphicCurvesSymplectic2012}, there exists $Y \in T_J \mathcal{J}$ such that
    \begin{IEEEeqnarray}{c+x*}
        \plabel{eq:variation of acs before cut off}
        \p{<}{}{\eta(z_0), Y(\mathbf{u}(z_0)) \circ (\dv u(z_0) - X_H(\mathbf{u}(z_0)) \otimes \tau_{z_0}) \circ j_{z_0} } > 0.
    \end{IEEEeqnarray}
    Choose a neighbourhood $V = V_{\R} \times V_{S^1}$ of $z_0 = (s_0,t_0)$ in $\dot{\Sigma} = \R \times S^1$ such that 
    \begin{IEEEeqnarray}{c+x*}
        \plabel{eq:inner product bigger than 0 in v}
        \p{<}{}{\eta, Y(\mathbf{u}) \circ (\dv u - X_H(\mathbf{u}) \otimes \tau) \circ j }|_V > 0.
    \end{IEEEeqnarray}
    Since $z_0$ is as in \eqref{eq:set of regular points of u}, there exists a neighbourhood $U_{\hat{X}}$ of $u(z_0)$ in $\hat{X}$ such that
    \begin{IEEEeqnarray*}{c+x*}
        u(s,t) \in U_{\hat{X}} \Longrightarrow s \in V_{\R}.
    \end{IEEEeqnarray*}
    Choose a slice $A \subset S^1 \times S$ which contains $(t_0, w_0)$ and which is transverse to the action of $S^1$ on $S^1 \times S$. Define $U_{S^1 \times S} = S^1 \cdot A$. For $A$ chosen small enough,
    \begin{IEEEeqnarray*}{c+x*}
        (t, w_0) \in U_{S^1 \times S} \Longrightarrow t \in V_{S^1}.
    \end{IEEEeqnarray*}
    Then, defining $U \coloneqq U_{S^1 \times S} \times U_{\hat{X}}$ we have that $\mathbf{u}^{-1}(U) \subset V$. Choose an $S^1$-invariant function $\beta \colon S^1 \times S \times \hat{X} \longrightarrow [0,1]$ such that
    \begin{IEEEeqnarray}{c+x*}
        \plabel{eq:bump function for transversality}
        \supp \beta \subset U, \qquad \beta(\mathbf{u}(z_0)) = 1, \qquad \beta Y \in T_J \mathcal{J}^{\varepsilon}.
    \end{IEEEeqnarray}
    Here, we can achieve that $\beta Y$ is of class $C^{\varepsilon}$ by \cite[Theorem B.6]{wendlLecturesSymplecticField2016}. Since $\mathbf{u}^{-1}(U) \subset V$ and $\supp \beta \subset U$, we have that $\supp (\beta \circ \mathbf{u}) \subset V$. Then,
    \begin{IEEEeqnarray*}{rCls+x*}
        0
        & = & \p{<}{}{\eta, \mathbf{J}_{(w,u,J)}(\beta Y)}_{L^2}                        & \quad [\text{by Equation \eqref{eq:element in annihilator 2}}] \\
        & = & \p{<}{}{\eta, \beta(\mathbf{u}) \, \mathbf{J}_{(w,u,J)}(Y)}_{L^2}    & \quad [\text{since $\mathbf{J}_{(w,u,J)}$ is $C^\infty$-linear}] \\
        & = & \p{<}{}{\eta, \beta(\mathbf{u}) \, \mathbf{J}_{(w,u,J)}(Y)}_{L^2(V)} & \quad [\text{since $\supp (\beta \circ \mathbf{u}) \subset V$}] \\
        & > & 0                                                                         & \quad [\text{by Equation \eqref{eq:inner product bigger than 0 in v}}],
    \end{IEEEeqnarray*}
    which is the desired contradiction.

    We prove that $\eta = 0$ in the case where $w$ is not constant. Let $z_0 = (t_0, s_0) \in \R \times S^1$ and assume by contradiction that $\eta(z_0) \neq 0$. Choose $Y$ as in \eqref{eq:variation of acs before cut off} and $V$ as in \eqref{eq:inner product bigger than 0 in v}. Choose a slice $A \subset S^1 \times S$ which contains $(t_0, w(0))$ and which is transverse to the action of $S^1$ on $S^1 \times S$. Define $U_{S^1 \times S} = S^1 \cdot A$. Since $w$ is orthogonal to the infinitesimal action on $S$, for $A$ chosen small enough we have
    \begin{IEEEeqnarray*}{c+x*}
        (t, w(s)) \in U_{S^1 \times S} \Longrightarrow (s,t) \in V.
    \end{IEEEeqnarray*}
    Defining $U = U_{S^1 \times S} \times \hat{X}$, we have that $\mathbf{u}^{-1}(U) \subset V$. Choosing $\beta$ as in \eqref{eq:bump function for transversality}, we obtain a contradiction in the same way as in the previous case.
\end{proof}

\begin{remark}
    We recall some terminology related to the Baire category theorem (we use the terminology from \cite[Section 10.2]{roydenRealAnalysis2010}). Let $X$ be a complete metric space and $E \subset X$. Then, $E$ is \textbf{meagre} or of the \textbf{first category} if $E$ is a countable union of nowhere dense subsets of $X$. We say that $E$ is \textbf{nonmeagre} or of the \textbf{second category} if $E$ is not meagre. We say that $E$ is \textbf{comeagre} or \textbf{residual} if $X \setminus E$ is meagre. Hence, a countable intersection of comeagre sets is comeagre. With this terminology, the Baire category theorem (see \cite[Section 10.2]{roydenRealAnalysis2010}) says that if $E$ is comeagre then $E$ is dense. The Sard--Smale theorem (see \cite[Theorem 1.3]{smaleInfiniteDimensionalVersion1965}) says that if $f \colon M \longrightarrow N$ is a Fredholm map between separable connected Banach manifolds of class $C^q$, for some $q > \max \{0, \operatorname{ind} f \}$, then the set of regular values of $f$ is comeagre.
\end{remark}

\begin{theorem}
    \label{thm:transversality in s1eft}
    There exists a dense subset $\mathcal{J}_{\mathrm{reg}} \subset \mathcal{J}$ with the following property. Let $J \in \mathcal{J}_{\mathrm{reg}}$ be an almost complex structure, $[z^{\pm}, \gamma^{\pm}]$ be equivalence classes of $1$-periodic orbits of $H$, and $(w,u) \in \hat{\mathcal{M}}(X, H, J, [z^+, \gamma^+], [z^-, \gamma^-])$. Then, near $(w,u)$ the space $\hat{\mathcal{M}}(X, H, J, [z^+, \gamma^+], [z^-, \gamma^-])$ is a manifold of dimension%
    \begin{IEEEeqnarray*}{c+x*}
        \dim_{(w,u)} \hat{\mathcal{M}}(X, H, J, [z^+, \gamma^+], [z^-, \gamma^-]) = \ind^{\tau}(z^+, \gamma^+) - \ind^{\tau}(z^-, \gamma^-) + 1.
    \end{IEEEeqnarray*}
\end{theorem}
\begin{proof}
    Recall that the space $\mathcal{J}^{\varepsilon}$ is defined with respect to a reference almost complex structure $J_{\mathrm{ref}}$. We will now emphasize this fact using the notation $\mathcal{J}^{\varepsilon}(J_{\mathrm{ref}})$. As a first step, we show that for every $[z^{\pm}, \gamma^{\pm}]$ and every reference almost complex structure $J_{\mathrm{ref}}$ there exists a comeagre set $\mathcal{J}^{\varepsilon}_{\mathrm{reg}}(J_{\mathrm{ref}}^{}, [z^{\pm}, \gamma^{\pm}]) \subset \mathcal{J}^{\varepsilon}(J_{\mathrm{ref}})$ such that every $J \in \mathcal{J}^{\varepsilon}_{\mathrm{reg}}(J_{\mathrm{ref}}^{}, [z^{\pm}, \gamma^{\pm}])$ has the property in the statement of the theorem. For shortness, for every $J$ let $\hat{\mathcal{M}}(J,[z^{\pm}, \gamma^{\pm}]) \coloneqq \hat{\mathcal{M}}(X, H, J, [z^+, \gamma^+], [z^-, \gamma^-])$. By \cref{thm:s1eft l is surjective} and the implicit function theorem \cite[Theorem A.3.3]{mcduffHolomorphicCurvesSymplectic2012}, the universal moduli space
    \begin{IEEEeqnarray*}{c+x*}
        \hat{\mathcal{M}}([z^{\pm}, \gamma^{\pm}]) \coloneqq \{ (w,u,J) \mid J \in \mathcal{J}^{\varepsilon}(J_{\mathrm{ref}}), \, (w,u) \in \hat{\mathcal{M}}(J, [z^{\pm}, \gamma^{\pm}]) \}
    \end{IEEEeqnarray*}
    is a smooth Banach manifold. Consider the smooth map
    \begin{IEEEeqnarray*}{c}
        \pi \colon \hat{\mathcal{M}}([z^{\pm}, \gamma^{\pm}]) \longrightarrow \mathcal{J}^{\varepsilon}(J_{\mathrm{ref}}), \qquad \pi(w,u,J) = J.
    \end{IEEEeqnarray*}
    By \cite[Lemma A.3.6]{mcduffHolomorphicCurvesSymplectic2012},
    \begin{IEEEeqnarray}{rCr}
          \ker \dv \pi(w,u,J) & \cong &   \ker \mathbf{D}_{(w,u,J)} \plabel{eq:d pi and d u have isomorphic kernels}, \\
        \coker \dv \pi(w,u,J) & \cong & \coker \mathbf{D}_{(w,u,J)} \plabel{eq:d pi and d u have isomorphic cokernels}.
    \end{IEEEeqnarray}
    Therefore, $\dv \pi (w,u,J)$ is Fredholm and has the same index as $\mathbf{D}_{(w,u,J)}$. By the Sard--Smale theorem, the set $\mathcal{J}^{\varepsilon}_{\mathrm{reg}}(J_{\mathrm{ref}}^{}, [z^{\pm}, \gamma^{\pm}]) \subset \mathcal{J}^{\varepsilon}(J_{\mathrm{ref}})$ of regular values of $\pi$ is comeagre. By Equation \eqref{eq:d pi and d u have isomorphic cokernels}, $J \in \mathcal{J}^{\varepsilon}(J_{\mathrm{ref}})$ is a regular value of $\pi$ if and only if $\mathbf{D}_{(w,u,J)}$ is surjective for every $(w,u) \in (\delbar\vphantom{\partial}^{J})^{-1}(0)$. Therefore, by the implicit function theorem, for every $J \in \mathcal{J}^{\varepsilon}_{\mathrm{reg}}(J_{\mathrm{ref}}^{}, [z^{\pm}, \gamma^{\pm}])$ the set $\hat{\mathcal{M}}(J,[z^{\pm},\gamma^{\pm}]) = (\delbar\vphantom{\partial}^J)^{-1}(0) \subset \mathcal{B}$ is a manifold of dimension%
    \begin{IEEEeqnarray*}{rCls+x*}
        \IEEEeqnarraymulticol{3}{l}{\dim_{(w,u)} \hat{\mathcal{M}}(J,[z^{\pm},\gamma^{\pm}])}\\ \quad
        & = & \dim \ker \mathbf{D}_{(w,u,J)}                              & \quad [\text{by the implicit function theorem}] \\
        & = & \operatorname{ind} \mathbf{D}_{(w,u,J)}                     & \quad [\text{since $\mathbf{D}_{(w,u,J)}$ is surjective}] \\
        & = & \ind^{\tau}(z^+, \gamma^+) - \ind^{\tau}(z^-, \gamma^-) + 1 & \quad [\text{by \cref{thm:s1eft d is fredholm}}].
    \end{IEEEeqnarray*}

    As a second step, we show that we can switch the order of the quantifiers in the first step, i.e. that for every reference almost complex structure $J_{\mathrm{ref}}$ there exists a comeagre set $\mathcal{J}^{\varepsilon}_{\mathrm{reg}}(J_{\mathrm{ref}}^{}) \subset \mathcal{J}^{\varepsilon}(J_{\mathrm{ref}})$ such that for every $J \in \mathcal{J}^{\varepsilon}_{\mathrm{reg}}(J_{\mathrm{ref}}^{})$ and every $[z^{\pm}, \gamma^{\pm}]$, the property in the statement of the theorem statement holds. For this, define
    \begin{IEEEeqnarray*}{c+x*}
        \mathcal{J}^{\varepsilon}_{\mathrm{reg}}(J_{\mathrm{ref}}^{}) \coloneqq \bigcap_{[z^{\pm}, \gamma^{\pm}]} \mathcal{J}^{\varepsilon}_{\mathrm{reg}}(J_{\mathrm{ref}}^{}, [z^{\pm}, \gamma^{\pm}]).
    \end{IEEEeqnarray*}
    Since $H$ is nondegenerate, in the above expression we are taking an intersection over a finite set of data, and hence $\mathcal{J}^{\varepsilon}_{\mathrm{reg}}(J_{\mathrm{ref}}^{})$ is comeagre. This finishes the proof of the second step. By the Baire category theorem, $\mathcal{J}^{\varepsilon}_{\mathrm{reg}}(J_{\mathrm{ref}}^{}) \subset \mathcal{J}^{\varepsilon}(J_{\mathrm{ref}}^{})$ is dense. Finally, define%
    \begin{IEEEeqnarray*}{c+x*}
        \mathcal{J}_{\mathrm{reg}} \coloneqq \bigcup_{J_{\mathrm{ref}} \in \mathcal{J}} \mathcal{J}^{\varepsilon}_{\mathrm{reg}}(J_{\mathrm{ref}}^{}).
    \end{IEEEeqnarray*}
    Then $\mathcal{J}_{\mathrm{reg}}$ is the desired set of almost complex structures.
\end{proof}

\chapter{\texorpdfstring{$S^1$}{S1}-equivariant Floer homology}
\label{chp:floer}

\section{Categorical setup}

In this section, we define categories that will allow us to express the constructions of this chapter as functors. We will define a category of complexes (see \cref{def:category complexes,def:category of complexes up to homotopy}) and a category of modules (see \cref{def:category modules}). Associated to these, there is a Homology functor between the two categories (\cref{def:homology functor}).

\begin{remark}
    Recall that a \textbf{preorder} on a set $S$ is a binary relation $\leq$ which is reflexive and transitive. A preordered set $(S,\leq)$ can be seen as a category $S$ by declaring that objects of $S$ are elements of the set $S$ and that there exists a unique morphism from $a$ to $b$ if and only if $a \leq b$, for $a, b \in S$. Throughout this thesis, we will view $\R$ as a category in this sense.
\end{remark}

\begin{definition}
    Let $\mathbf{C}$ be a category. A \textbf{filtered object} in $\mathbf{C}$ is a functor $V \colon \R \longrightarrow \mathbf{C}$. A \textbf{morphism} of filtered objects from $V$ to $W$ is a natural transformation $\phi \colon V \longrightarrow W$. We denote by $\Hom(\R, \mathbf{C})$ the category of filtered objects in $\mathbf{C}$. In this case, we will use the following notation. If $a \in \R$, we denote by $V^a$ the corresponding object of $\mathbf{C}$. If $\mathbf{C}$ is abelian and $a \leq b \in \R$, we denote $V^{(a,b]} \coloneqq V^b / V^a \coloneqq \coker (\iota^{b,a} \colon V^a \longrightarrow V^b)$.
\end{definition}

\begin{definition}
    \label{def:category complexes}
    Denote by $\tensor[_\Q]{\mathbf{Mod}}{}$ the category of $\Q$-modules. We define a category $\komp$ as follows. An object of $\komp$ is a triple $(C,\del,U)$, where $C \in \Hom(\R, \tensor[_\Q]{\mathbf{Mod}}{})$ is a filtered $\Q$-module and $\partial, U \colon C \longrightarrow C$ are natural transformations such that 
    \begin{IEEEeqnarray*}{lCls+x*}
        \partial \circ \partial & = & 0, \\
        \partial \circ U        & = & U \circ \partial.
    \end{IEEEeqnarray*}
    A morphism in $\komp$ from $(C,\del^C,U^C)$ to $(D,\del^D,U^D)$ is a natural transformation $\phi \colon C \longrightarrow D$ for which there exists a natural transformation $T \colon C \longrightarrow D$ such that
    \begin{IEEEeqnarray*}{rCrCl}
        \partial^D & \circ \phi - \phi \circ & \partial^C & = & 0, \\
        U^D        & \circ \phi - \phi \circ & U^C        & = & \partial^D \circ T + T \circ \partial^C.
    \end{IEEEeqnarray*}
\end{definition}

\begin{definition}
    \phantomsection\label{def:category of complexes up to homotopy}
    Let $\phi, \psi \colon (C, \partial^C, U^C) \longrightarrow (D, \partial^D, U^D)$ be morphisms in $\komp$. A \textbf{chain homotopy} from $\phi$ to $\psi$ is a natural transformation $T \colon C \longrightarrow D$ such that
    \begin{IEEEeqnarray*}{c+x*}
        \psi - \phi = \partial^D \circ T + T \circ \partial^C.
    \end{IEEEeqnarray*}
    The notion of chain homotopy defines an equivalence relation $\sim$ on each set of morphisms in $\komp$. We denote the quotient category (see for example \cite[Theorem 0.4]{rotmanIntroductionAlgebraicTopology1988}) by 
    \begin{IEEEeqnarray*}{c+x*}
        \comp \coloneqq \komp / \sim.
    \end{IEEEeqnarray*}
\end{definition}

As we will see in \cref{sec:Floer homology}, the $S^1$-equivariant Floer chain complex of $X$ (with respect to a Hamiltonian $H$ and almost complex structure $J$) is an object 
\begin{IEEEeqnarray*}{c+x*}
    \homology{}{S^1}{}{F}{C}{}{}(X,H,J) \in \comp.
\end{IEEEeqnarray*}

\begin{definition}
    \label{def:category modules}
    We define a category $\modl$ as follows. An object of $\modl$ is a pair $(C,U)$, where $C \in \Hom(\R, \tensor[_\Q]{\mathbf{Mod}}{})$ is a filtered $\Q$-module and $U \colon C \longrightarrow C$ is a natural transformation. A morphism in $\modl$ from $(C,U^C)$ to $(D,U^D)$ is a natural transformation $\phi \colon C \longrightarrow D$ such that $\phi \circ U^C = U^D \circ \phi$.
\end{definition}

In \cref{sec:Floer homology}, we will show that the $S^1$-equivariant Floer homology of $X$ (with respect to a Hamiltonian $H$ and almost complex structure $J$) and the $S^1$-equivariant symplectic homology of $X$ are objects of $\modl$:
\begin{IEEEeqnarray*}{rCls+x*}
    \homology{}{S^1}{}{F}{H}{}{}(X,H,J) & \in & \modl, \\
    \homology{}{S^1}{}{S}{H}{}{}(X)     & \in & \modl.
\end{IEEEeqnarray*}

\begin{lemma}
    The category $\modl$ is abelian, complete and cocomplete.
\end{lemma}
\begin{proof}
    Recall the definition of (co)complete: a category $\mathbf{I}$ is small if the class of morphisms of $\mathbf{I}$ is a set. A category is (co)complete if for any $\mathbf{I}$ small and for any functor $F \colon \mathbf{I} \longrightarrow \modl$, the (co)limit of $F$ exists. By \cite[Theorem 3.4.12]{riehlCategoryTheoryContext2016}, it suffices to show that $\modl$ has products, coequalizers, coproducts and coequalizers. First, notice that $\tensor[_\Q]{\mathbf{Mod}}{}$ is abelian, complete and cocomplete. Therefore, the same is true for $\Hom(\R, \tensor[_\Q]{\mathbf{Mod}}{})$. Let $f \colon C \longrightarrow D$ be a morphism in $\modl$. Then $f$ has a kernel and a cokernel in $\Hom(\R, \tensor[_\Q]{\mathbf{Mod}}{})$. We need to show that the kernel and the cokernel are objects of $\modl$, i.e. that they come equipped with a $U$ map. The $U$ maps for $\ker f, \coker f$ are the unique maps (coming from the universal property of the (co)kernel) such that diagram
    \begin{IEEEeqnarray*}{c+x*}
        \begin{tikzcd}
            \ker f \ar[r] \ar[d, swap, dashed, "\exists ! U_{\ker f}"] & C \ar[d, "U_C"] \ar[r, "f"] & D \ar[d, "U_D"] \ar[r] & \coker f \ar[d, dashed, "\exists ! U_{\coker f}"] \\
            {\ker f} \ar[r]                                            & {C} \ar[r, "f"]             & {D} \ar[r]             & {\coker f}
        \end{tikzcd}
    \end{IEEEeqnarray*}
    commutes. Let $C_i$, for $i \in I$, be a family of objects in $\modl$. Then, the product $\prod_{i \in I}^{} C_i$ and the coproduct $\bigoplus_{i \in I}^{} C_i$ exist in $\Hom(\R, \tensor[_\Q]{\mathbf{Mod}}{})$. Again, we need to show that the product and coproduct come equipped with a $U$ map. The $U$ maps for the product and coproduct are the maps
    \begin{IEEEeqnarray*}{LCRRCRCL+x*}
        U_{\bigproduct_{i \in I}^{} C_i}   & = & \bigproduct_{i \in I}^{} U_{C_i} \colon   & \bigproduct_{i \in I}^{} C_i   & \longrightarrow & \bigproduct_{i \in I}^{} C_i, \\
        U_{\bigdirectsum_{i \in I}^{} C_i} & = & \bigdirectsum_{i \in I}^{} U_{C_i} \colon & \bigdirectsum_{i \in I}^{} C_i & \longrightarrow & \bigdirectsum_{i \in I}^{} C_i, 
    \end{IEEEeqnarray*}
    coming from the respective universal properties.
\end{proof}

\begin{definition}
    \label{def:homology functor}
    Let $(C,\partial,U) \in \comp$. The \textbf{homology} of $(C,\partial,U)$ is the object of $\modl$ given by $H(C, \partial, U) \coloneqq (H(C, \partial), H(U))$, where $H(C, \partial) = \ker \partial / \img \partial$ and $H(U)$ is the unique map such that the diagram
    \begin{IEEEeqnarray*}{c+x*}
        \begin{tikzcd}
            \img \partial \ar[r] \ar[d, swap, "U"] & \ker \partial \ar[r] \ar[d, "U"] & \ker \partial / \img \partial \ar[d, dashed, "\exists !"] \ar[r, equals] & H(C, \partial) \ar[d, "H(U)"] \\
            \img \partial \ar[r]                   & \ker \partial \ar[r]             & \ker \partial / \img \partial \ar[r, equals]                             & H(C, \partial)
        \end{tikzcd}
    \end{IEEEeqnarray*}
    commutes. If $\phi \colon (C, \partial^C, U^C) \longrightarrow (D, \partial^D, U^D)$ is a morphism in $\comp$, we define the induced morphism on homology, $H(\phi) \colon H(C, \partial^C) \longrightarrow H(D, \partial^D)$, to be the unique map such that the diagram
    \begin{IEEEeqnarray*}{c+x*}
        \begin{tikzcd}
            \img \partial^C \ar[r] \ar[d, swap, "\phi"] & \ker \partial^C \ar[r] \ar[d, "\phi"] & \ker \partial^C / \img \partial^C \ar[d, dashed, "\exists !"] \ar[r, equals] & H(C, \partial^C) \ar[d, "H(\phi)"] \\
            \img \partial^D \ar[r]                      & \ker \partial^D \ar[r]                & \ker \partial^D / \img \partial^D \ar[r, equals]                             & H(D, \partial^D)
        \end{tikzcd}
    \end{IEEEeqnarray*}
    commutes. With these definitions, homology is a functor $H \colon \comp \longrightarrow \modl$.
\end{definition}

\section{Action functional}
\label{sec:action functional}

Our goal in this section is to establish the definitions that we will need to later define the $S^1$-equivariant Floer Chain complex. We define suitable families of admissible Hamiltonians (\cref{def:hamiltonians}) and almost complex structures (\cref{def:acs}). The key points of this section are \cref{def:generators}, where we define the set of generators of the $S^1$-equivariant Floer chain complex, and \cref{def:flow lines}, where we define the trajectories that are counted in the differential of the $S^1$-equivariant Floer chain complex. We also define the action of a generator (\cref{def:action functional}), which will induce a filtration on the $S^1$-equivariant Floer chain complex. We will assume that $(X,\lambda)$ is a nondegenerate Liouville domain with completion $(\hat{X},\hat{\lambda})$. Let $\varepsilon \coloneqq \frac{1}{2} \operatorname{Spec}(\partial X,\lambda|_{\partial X})$.

We start by recalling some basic facts about $S^{2N+1}$ and $\C P^N$. For each $N \in \Z_{\geq 1}$ we denote%
\begin{IEEEeqnarray*}{c+x*}
    S^{2N + 1} \coloneqq \{ (z_0,\ldots,z_N) \in \C ^{N+1} \ | \ |z_0|^2 + \cdots + |z_N|^2 = 1 \}.
\end{IEEEeqnarray*}
There is an action $S^1 \times S^{2N + 1} \longrightarrow S^{2N + 1}$ given by $(t,z) \longmapsto e ^{2 \pi i t} z$. This action is free and proper, so we can consider the quotient manifold $S^{2N+1}/S^1$. The Riemannian metric of $\C ^{N+1} = \R ^{2(N+1)}$ pulls back to a Riemannian metric on $S^{2N + 1}$. The action of $S^1$ on $S^{2N + 1}$ is by isometries, so there exists a unique Riemannian metric on $S^{2N+1}/S^1$ such that the projection $S^{2N+1} \longrightarrow S^{2N+1}/S^1$ is a Riemannian submersion. The set $\C \setminus \{0\}$ is a group with respect to multiplication, and it acts on $\C ^{N+1} \setminus \{0\}$ by multiplication. This action is free and proper, so we can form the quotient 
\begin{IEEEeqnarray*}{c+x*}
    \C P^{N} \coloneqq (\C ^{N+1} \setminus \{0\})/(\C \setminus \{0\}).
\end{IEEEeqnarray*}
By the universal property of the quotient, there exists a unique map $S^{2N+1}/S^1 \longrightarrow \C P^N$ such that the following diagram commutes:
\begin{IEEEeqnarray*}{c+x*}
    \begin{tikzcd}
        S^{2N + 1} \ar[r, hook] \ar[d, two heads] & \C ^{N+1} \setminus \{0\} \ar[d, two heads] \\
        S^{2N + 1} / S^1 \ar[r, hook, two heads, dashed, swap, "\exists !"] & \C P^N
    \end{tikzcd}
\end{IEEEeqnarray*}
The map $S^{2N + 1} / S^1 \longrightarrow \C P^N$ is a diffeomorphism. Define the Fubini--Study metric on $\C P^N$ to be the unique Riemannian metric on $\C P^N$ such that $S^{2N + 1} / S^1 \longrightarrow \C P^N$ is an isometry.

We will now consider a special family of functions on $S^{2N+1}$ and $\C P^N$. Define a function%
\begin{IEEEeqnarray*}{rrCl}
    f_N \colon & \C P^N & \longrightarrow & \R \\
               & [w]    & \longmapsto     & \frac{\sum_{j=0}^{N} j|w_j|^2}{\sum_{j=0}^{N} |w_j|^2}.
\end{IEEEeqnarray*}
Define $\tilde{f}_N$ to be the pullback of $f_N$ to $S^{2N+1}$. Let $e_0,\ldots,e_N$ be the canonical basis of $\C ^{N+1}$ (as a vector space over $\C$). Then, 
\begin{IEEEeqnarray*}{rCls+x*}
    \critpt \tilde{f}_N & = & \{ e^{2 \pi i t} e_j \mid t \in S^1, j = 0,\ldots,N \}, \\
    \critpt f_N         & = & \{[e_0],\ldots,[e_N]\}.
\end{IEEEeqnarray*}
The function $f_N$ is Morse, while $\tilde{f}_N$ is Morse--Bott. The Morse indices are given by
\begin{IEEEeqnarray*}{rCll}
    \morse([e_j],f_N)     & = & 2j,               & \quad \text{for all } j=0,\ldots,N, \\
    \morse(z,\tilde{f}_N) & = & \morse([z], f_N), & \quad \text{for all } z \in \critpt f_N.
\end{IEEEeqnarray*}
We will use the notation $\morse(z) \coloneqq \morse(z,\tilde{f}_N) = \morse([z], f_N)$. 

We now study the relation between $\tilde{f}_{N^-}$ and $\tilde{f}_{N^+}$ for $N^- \geq N^+$. For every $k$ such that $0 \leq k \leq N^- - N^+$, define maps
\begin{IEEEeqnarray*}{rrCl}
    \inc^{N^-,N^+}_k \colon & S^{2N^++1}           & \longrightarrow & S^{2N^-+1} \\
                            & (z_0,\ldots,z_{N^+}) & \longmapsto     & (\underbrace{0,\ldots,0}_k,z_0,\ldots,z_{N^+},0,\ldots,0).
\end{IEEEeqnarray*}
Let $I_k \colon \R \longrightarrow \R$ be given by $I_k(x) = x + k$. Then, the following diagram commutes:
\begin{IEEEeqnarray*}{c+x*}
    \begin{tikzcd}[row sep=scriptsize, column sep={{{{6em,between origins}}}}]
                                                                                       & S^{2N^+ + 1} \arrow[dl, swap, "\inc_{k}^{N^-,N^+}"] \arrow[rr, "\tilde{f}_{N^+}"] \arrow[dd]                         &                                          & \R \arrow[dl, "I_k"] \arrow[dd, equals] \\
        S^{2N^- + 1} \arrow[rr, crossing over, near end, "\tilde{f}_{N^-}"] \arrow[dd] &                                                                                                                      & \R \\
                                                                                       & \C P^{N^+} \arrow[dl, dashed, swap, outer sep = -4pt, "\exists ! i_{k}^{N^-,N^+}"] \arrow[rr, near start, "f_{N^+}"] &                                          & \R \arrow[dl, "I_k"] \\
        \C P ^{N^-} \arrow[rr, swap, "f_{N^-}"]                                        &                                                                                                                      & \R \arrow[from=uu, crossing over, equals]
    \end{tikzcd}
\end{IEEEeqnarray*}
The vector fields $\nabla \tilde{f}_{N^+}$ and $\nabla \tilde{f}_{N^-}$ are $\inc_{k}^{N^-,N^+}$-related, and analogously the vector fields $\nabla {f}_{N^+}$ and $\nabla {f}_{N^-}$ are ${i}_{k}^{N^-,N^+}$-related. For $t \in \R$, denote by $\phi^t_{\tilde{f}_{N^-}}$ the time-$t$ gradient flow of $\tilde{f}_{N^-}$ and analogously for $\phi^t_{f_{N^+}}$. Then, the following diagram commutes:
\begin{IEEEeqnarray*}{c+x*}
    \begin{tikzcd}[row sep=scriptsize, column sep={{{{6em,between origins}}}}]
                                                                                                & S^{2N^+ + 1} \arrow[dl, swap, "{\inc_k^{N^-,N^+}}"] \arrow[rr, "\phi^t_{\tilde{f}_N}"] \arrow[dd] &                                          & S^{2N^+ + 1} \arrow[dl, near end, "\inc_k^{N^-,N^+}"] \arrow[dd] \\
        S^{2N^- + 1} \arrow[rr, crossing over, near end, "\phi^t_{\tilde{f}_{N^-}}"] \arrow[dd] &                                                                                                   & S^{2N^- + 1} \\
                                                                                                & \C P^{N^+} \arrow[dl, swap, "i_k^{N^-,N^+}"] \arrow[rr, near start, "\phi^t_{f_{N^+}}"]           &                                          & \C P^{N^+} \arrow[dl, "i_k^{N^-,N^+}"] \\
        \C P ^{N^-} \arrow[rr, swap, "\phi^t_{f_{N^-}}"]                                        &                                                                                                   & \C P^{N^-} \arrow[from=uu, crossing over]
    \end{tikzcd}
\end{IEEEeqnarray*}

\begin{definition}
    \label{def:hamiltonians}
    A parametrized Hamiltonian $H \colon S^1 \times S^{2N+1} \times \hat{X} \longrightarrow \R$ is \textbf{admissible} if it satisfies the conditions in \cref{item:invariant,item:profile,item:ndg,item:flow lines,item:pullbacks}. We denote the set of such $H$ by $\mathcal{H}(X,N)$.
    \begin{enumerate}
        \item \label{item:profile} There exist $D \in \R$, $C \in \R_{>0} \setminus \operatorname{Spec}(\del X, \lambda|_{\del X})$ and $\delta > 0$ such that:
            \begin{enumerate}[label=(\Roman*)]
                \item on $S^1 \times S^{2N+1} \times X$, we have that $- \varepsilon < H < 0$, $H$ is $S^1$-independent and $H$ is $C^2$-small (so that there are no nonconstant $1$-periodic orbits);
                \item on $S^1 \times S^{2N+1} \times [0,\delta] \times \del X$, we have that $-\varepsilon < H < \varepsilon$ and $H$ is $C^2$-close to $(t,z,r,x) \longmapsto h(e^r)$, where $h \colon [1,e ^{\delta}] \longrightarrow \R$ is increasing and strictly convex;
                \item[(S)] on $S^1 \times S^{2N+1} \times [\delta, + \infty) \times \del X$, we have that $H(t,z,r,x) = C e^r + D$.
            \end{enumerate}
        \item \label{item:invariant} Consider the action of $S^1$ on $S^1 \times S^{2N+1} \times \hat{X}$ given by $t' \cdot (t,z,x) = (t' + t, e ^{2 \pi i t'} z, x)$. Then $H$ is invariant under this action, i.e. $H(t'+ t, e ^{2 \pi i t'} z, x) = H(t,z,x)$.
        \item \label{item:ndg} If $z$ is a critical point of $\tilde{f}_N$ then $H_z$ is nondegenerate.
        \item \label{item:flow lines} For every $(t,z,x) \in S^1 \times S^{2N+1} \times \hat{X}$ we have $\p{<}{}{\nabla_{S^{2N+1}}H(t,z,x), \nabla \tilde{f}_N(z)} \leq 0$.
        \item \label{item:pullbacks} There exists $E \geq 0$ such that $(\inc^{N,N-1}_0)^* H = (\inc^{N,N-1}_1)^* H + E$.
    \end{enumerate}
\end{definition}

\begin{definition}
    \label{def:acs}
    A parametrized almost complex structure $J \colon S^1 \times S^{2N+1} \times \hat{X} \longrightarrow \End(T \hat{X})$ is \textbf{admissible} if it satisfies the conditions in \cref{def:acs 1,def:acs 2,def:acs 3,def:acs 4}. We denote the set of such $J$ by $\mathcal{J}(X,N)$.
    \begin{enumerate}
        \item \label{def:acs 1} $J$ is $S^1$-invariant, i.e. $J(t' + t, e ^{2 \pi i t'} z, x) = J(t, z, x)$ for every $t' \in S^1$ and $(t,z,x) \in S^1 \times S^{2N+1} \times \hat{X}$.
        \item \label{def:acs 2} $J$ is $\hat{\omega}$-compatible.
        \item \label{def:acs 3} The restriction of $J$ to $S^1 \times S^{2N+1} \times \R_{\geq 0} \times \del X$ is cylindrical.
        \item \label{def:acs 4} $(\inc_0^{N,N-1})^* J = (\inc_1^{N,N-1})^* J$.
    \end{enumerate}
\end{definition}

\begin{definition}
    Denote by $\admissible{X}$ the set of tuples 
    \begin{IEEEeqnarray*}{c+x*}
        (H,J) \in \bigcoproduct_{N \in \Z_{\geq 1}}^{} \mathcal{H}(X,N) \times \mathcal{J}(X,N)
    \end{IEEEeqnarray*}
    which are regular, where ``regular'' means that the moduli spaces of \cref{def:flow lines} are transversely cut out. Define a preorder $\leq$ on $\admissible{X}$ by 
    \begin{IEEEeqnarray*}{rCl}
        (H^+,J^+) \leq (H^-,J^-) & \mathrel{\mathop:}\Longleftrightarrow & N^+ \leq N^- \text{ and } H^+ \leq (i_0 ^{N^-,N^+})^* H^-.
    \end{IEEEeqnarray*}
\end{definition}

\begin{definition}
    \label{def:generators}
    Let $N \in \Z_{\geq 1}$ and $H \in \mathcal{H}(X,N)$. Define
    \begin{IEEEeqnarray*}{c+x*}
        \hat{\mathcal{P}}(H)
        \coloneqq
        \left\{
        (z, \gamma)
        \ \middle\vert
        \begin{array}{l}
            z \in S^{2N+1} \text{ is a critical point of } \tilde{f}_N, \\
            \gamma \in C^{\infty}(S^1, \hat{X}) \text{ is a $1$-periodic orbit of } H_z
        \end{array}
    \right\}.
    \end{IEEEeqnarray*}
    There is an action of $S^1$ on $\hat{\mathcal{P}}(H)$ given by $t \cdot (z,\gamma) \coloneqq (e ^{2 \pi i t'} z, \gamma(\cdot - t))$. Define the quotient
    \begin{IEEEeqnarray*}{c+x*}
        \mathcal{P}(H) \coloneqq \hat{\mathcal{P}}(H) / S^1.
    \end{IEEEeqnarray*}
\end{definition}

\begin{remark}
    \label{rmk:types of orbits}
    If $(z, \gamma) \in \hat{\mathcal{P}}(H)$, then either $\img \gamma$ is in region $\rmn{1}$ and $\gamma$ is constant or $\img \gamma$ is in region $\rmn{2}$ and $\gamma$ is nonconstant. In the slope region, i.e. region S, there are no $1$-periodic orbits of $H$ because $C$ is not in $\operatorname{Spec}(\del X, \lambda|_{\del X})$ and by \cref{cor:hamiltonian orbits are reeb orbits}.
\end{remark}

\begin{definition}
    \label{def:flow lines}
    Let $N \in \Z_{\geq 1}$, $H \in \mathcal{H}(X,N)$ and $J \in \mathcal{J}(X,N)$. A pair $(w,u)$, where $w \colon \R \longrightarrow S^{2N+1}$ and $u \colon \R \times S^1 \longrightarrow \hat{X}$ is a solution of the \textbf{parametrized Floer equation} if
    \begin{equation*}
        \left\{ \,
            \begin{IEEEeqnarraybox}[
                \IEEEeqnarraystrutmode
                \IEEEeqnarraystrutsizeadd{7pt}
                {7pt}][c]{rCl}
                \dot{w}(s) & = & \nabla \tilde{f}_N(w(s)) \\
                \pdv{u}{s}(s,t) & = & - J^t_{w(s)}(u(s,t)) \p{}{2}{ \pdv{u}{t}(s,t) - X_{H^t_{w(s)}} (u(s,t)) }.
            \end{IEEEeqnarraybox}
        \right.
    \end{equation*}
    For $[z^+,\gamma^+], [z^-,\gamma^-] \in \mathcal{P}(H)$, define $\hat{\mathcal{M}}(H,J,[z^+,\gamma^+],[z^-,\gamma^-])$ to be the moduli space of solutions $(w,u)$ of the parametrized Floer equation such that $(w(s),u(s,\cdot))$ converges as $s \to \pm \infty$ to an element in the equivalence class $[z^\pm,\gamma^\pm]$. We define the following two group actions.
    \begin{IEEEeqnarray*}{rsrsrCl}
        \R  & \quad \text{acts on} \quad & \hat{\mathcal{M}}(H,J,[z^+,\gamma^+],[z^-,\gamma^-]) & \quad \text{by} \quad & s \cdot (w,u) & \coloneqq & (w(\cdot - s), u(\cdot-s, \cdot)), \\
        S^1 & \quad \text{acts on} \quad & \hat{\mathcal{M}}(H,J,[z^+,\gamma^+],[z^-,\gamma^-]) & \quad \text{by} \quad & t \cdot (w,u) & \coloneqq & (e ^{2 \pi i t} w, u(\cdot, \cdot - t)).
    \end{IEEEeqnarray*}
    The actions of $\R$ and $S^1$ on $\hat{\mathcal{M}}(H,J,[z^+,\gamma^+],[z^-,\gamma^-])$ commute, so they define an action of $\R \times S^1$ on $\hat{\mathcal{M}}(H,J,[z^+,\gamma^+],[z^-,\gamma^-])$. Finally, let 
    \begin{IEEEeqnarray*}{c+x*}
        \mathcal{M}(H,J,[z^+,\gamma^+],[z^-,\gamma^-]) \coloneqq \hat{\mathcal{M}}(H,J,[z^+,\gamma^+],[z^-,\gamma^-]) / \R \times S^1.
    \end{IEEEeqnarray*}
\end{definition}

\begin{definition}
    \phantomsection\label{def:action functional}
    For $(z, \gamma) \in \hat{\mathcal{P}}(H)$, the \textbf{action} of $(z, \gamma)$, denoted $\mathcal{A}_H(z, \gamma)$, is given by%
    \begin{IEEEeqnarray*}{c+x*}
        \mathcal{A}_{H}(z,\gamma) \coloneqq \mathcal{A}_{H_z}(\gamma) = \int_{S^1}^{} \gamma^* \hat{\lambda} - \int_{S^1}^{} H(t,z,\gamma(t)) \edv t.
    \end{IEEEeqnarray*}
    The action functional is a map $\mathcal{A}_H \colon \hat{\mathcal{P}}(H) \longrightarrow \R$. Since $H$ is $S^1$-invariant, $\mathcal{A}_H$ is $S^1$-invariant as well, and therefore there is a corresponding map $\mathcal{A}_H$ whose domain is $\mathcal{P}(H)$.
\end{definition}

\begin{lemma}
    \label{lem:action admissible}
    The actions of $1$-periodic orbits of $H$ are ordered according to
    \begin{IEEEeqnarray*}{c+x*}
        0 < \mathcal{A}_H(\rmn{1}) < \varepsilon < \mathcal{A}_H(\rmn{2}).
    \end{IEEEeqnarray*}
\end{lemma}
\begin{proof}
    Consider \cref{fig:action ordering 1}. By \cref{lem:action in symplectization,def:hamiltonians}, we have that $\mathcal{A}_H$ is constant equal to $-H$ in regions $\rmn{1}$ and S and $\mathcal{A}_H$ is strictly increasing in region $\rmn{2}$. We remark that strictly speaking, the Hamiltonian plotted in the picture is not $H$ but instead a Hamiltonian which is $C^2$-close to $H$. However, it suffices to prove the statement for the Hamiltonian which approximates $H$. From this discussion, we conclude that $0 < \mathcal{A}_H(\rmn{1}) < \varepsilon$. We show that $\mathcal{A}_H(\rmn{2}) > \varepsilon$.
    \begin{IEEEeqnarray*}{rCls+x*}
        \mathcal{A}_H(\rmn{2})
        & =    & e^r T(r) - H(r)          & \quad [\text{by \cref{lem:action in symplectization}}]                                                                                                \\
        & \geq & 2 \varepsilon e^r - H(r) & \quad [\text{$2 \varepsilon = \min \operatorname{Spec}(\del X, \lambda|_{\del X})$ and $T(r) \in \operatorname{Spec}(\del X, \lambda|_{\del X})$}] \\
        & >    & \varepsilon (2 e^r - 1)  & \quad [\text{$H(r) < \varepsilon$}]                                                                                                                   \\
        & >    & \varepsilon              & \quad [\text{$r > 0$}].                                                                                                                                & \qedhere
    \end{IEEEeqnarray*}
\end{proof}

\begin{figure}[ht]
    \centering

    \begin{tikzpicture}
        [
        help lines/.style={thin, draw = black!50},
        Hamiltonian/.style={thick},
        action/.style={thick}
        ]
        \tikzmath{
            \a = 4;
            \b = 1;
            \c = 3;
            \d = 1;
            \h = 0.5;
            \sml = 0.05;
            \y = -0.3;
            \z = -0.1;
            \f = \c + \d;
            \m = - 12 * (-\y + \z) / (-1+exp(\d))^4;
            \n = 2 * (-1 + 3 * exp(\d)) * (-\y + \z) / (-1+exp(\d))^4;
            \o = ( -2 * exp(\d) * \y + 6 * exp(2 * \d) * \y - 4 * exp(3 * \d) * \y + exp(4 * \d) * \y + \z - 2 * exp(\d) * \z ) / (-1+exp(\d))^4;
            \u = -2 * (\y - \z) / (-1+exp(\d));
            \v = (2 * exp(\d) * \y - \z - exp(\d) * \z) / (-1+exp(\d));
            function h1 (\r) { return \y; };
            function h2 (\r) { return {\o + \n * \r + 1/2 * exp(\d) * \m * \r^2 + 1/6 * (-1 - exp(\d)) * \m * \r^3 + 1/12 * \m * \r^4 }; };
            function h2p(\r) { return {\n + 1/6 * \m * \r * (-3 * exp(\d) * (-2 + \r) + \r * (-3 + 2 * \r))}; };
            function hs (\r) { return { \u * \r + \v }; };
            function H1(\r) { return { \y }; };
            function H2(\r) { return { h2(exp(\r)) }; };
            function Hs(\r) { return { hs(exp(\r)) }; };
            function a1(\r) { return { -\y }; };
            function a2(\r) { return { exp(\r) * h2p(exp(\r)) - H2(\r) }; };
            function as(\r) { return { -\v }; };
            \e = ln((\a-\v)/\u) - \d;
            \g = \f + \e;
        }
        \draw[->] (0 ,  0)                   -- (\g, 0);
        \draw[->] (0 ,-\b)                   -- (0 ,\a) node[above] {$\R$};
        \draw[->] (\c,-\b) node[below] {$0$} -- (\c,\a) node[above] {$\R$};

        \draw[help lines] (0 , \h) node[left]  {$+\varepsilon$} -- (\g, \h);
        \draw[help lines] (0 ,-\h) node[left]  {$-\varepsilon$} -- (\g,-\h);
        \draw[help lines] (\f,-\b) node[below] {$\delta$}       -- (\f, \a);

        \draw[Hamiltonian, domain =  0:\c] plot (\x, {H1(\x - \c)});
        \draw[Hamiltonian, domain = \c:\f] plot (\x, {H2(\x - \c)});
        \draw[Hamiltonian, domain = \f:\g] plot (\x, {Hs(\x - \c)}) node[right] {$H$};

        \draw[action, domain =  0:\c] plot (\x, {a1(\x - \c)});
        \draw[action, domain = \c:\f] plot (\x, {a2(\x - \c)});
        \draw[action, domain = \f:\g] plot (\x, {as(\x - \c)}) node[right] {$\mathcal{A}_H$};

        \draw (\c/2       ,\a) node[below] {$\mathrm{I}$};
        \draw (\c + \d/2  ,\a) node[below] {$\mathrm{II}$};
        \draw (\c + 3*\d/2,\a) node[below] {$\mathrm{S}$};
        \draw[help lines, decoration = {brace, mirror, raise=5pt}, decorate] (0,-\b-.75) -- node[below=6pt] {$X$} (\c - \sml,-\b-.75);
        \draw[help lines, decoration = {brace, mirror, raise=5pt}, decorate] (\c + \sml,-\b-.75) -- node[below=6pt] {$\R_{\geq 0} \times \del X$} (\g,-\b-.75);
    \end{tikzpicture}

    \caption{Action of a $1$-periodic orbit of $H$}
    \label{fig:action ordering 1}
\end{figure}
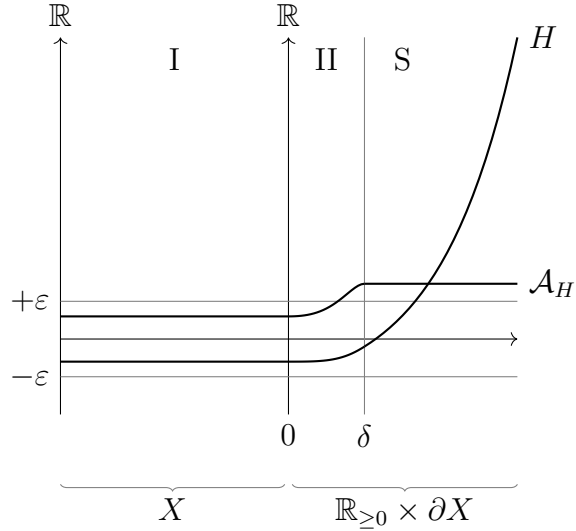

\begin{remark}
    Denote by $\critpt \mathcal{A}_{H} \subset S^{2N+1} \times C^\infty(S^1,\hat{X})$ the set of critical points of the action functional. Then, $\hat{\mathcal{P}}(H) = \critpt \mathcal{A}_{H}$, as is usual for various Floer theories. However, if $(w,u)$ is a path in $S^{2N+1} \times C^\infty(S^1,\hat{X})$, it is not true that $(w,u)$ is a gradient flow line of $\mathcal{A}_{H}$ if and only if $(w,u)$ is a solution of the parametrized Floer equations.
\end{remark}

\section{\texorpdfstring{$S^1$}{S1}-equivariant Floer homology}
\label{sec:Floer homology}

Let $(X,\lambda)$ be a nondegenerate Liouville domain. In this section, we define the $S^1$-equivariant Floer chain complex of $(X,\lambda)$ and other related invariants, namely the $S^1$-equivariant Floer homology, the positive $S^1$-equivariant Floer homology, the $S^1$-equivariant symplectic homology and the positive $S^1$-equivariant symplectic homology. The presentation we will give will be based on \cite{guttSymplecticCapacitiesPositive2018}. Other references discussing $S^1$-equivariant symplectic homology are \cite{guttMinimalNumberPeriodic2014,guttPositiveEquivariantSymplectic2017,bourgeoisGysinExactSequence2013,bourgeoisFredholmTheoryTransversality2010,bourgeoisEquivariantSymplecticHomology2016,seidelBiasedViewSymplectic2008}. The $S^1$-equivariant Floer complex of $X$ depends on the additional data of $(H,J) \in \admissible{X}$. More precisely, it can be encoded in a functor $\homology{}{S^1}{X}{F}{C}{}{} \colon \admissible{X}^{} \longrightarrow \comp$. We start by defining this functor on objects. For each $I = (H,J) \in \admissible{X}$, we need to say what is $\homology{}{S^1}{X}{F}{C}{}{}(H,J) \coloneqq \homology{}{S^1}{}{F}{C}{}{}(X,H,J) \in \comp$.

\begin{definition}
    We define $\homology{}{S^1}{}{F}{C}{}{}(X,H,J)$ to be the free $\Q$-module generated by the elements of $\mathcal{P}(H)$. Define $\homology{}{S^1}{}{F}{C}{a}{}(X,H,J)$ to be the subspace generated by the elements $[z,\gamma]$ of $\mathcal{P}(H)$ such that $\mathcal{A}_{H}(z,\gamma) \leq a$. These modules come equipped with inclusion maps
    \begin{IEEEeqnarray*}{rCls+x*}
        \iota^{a}   \colon \homology{}{S^1}{}{F}{C}{a}{}(X,H,J) & \longrightarrow & \homology{}{S^1}{}{F}{C}{}{}(X,H,J),  & \quad for $a \in \R$, \\
        \iota^{b,a} \colon \homology{}{S^1}{}{F}{C}{a}{}(X,H,J) & \longrightarrow & \homology{}{S^1}{}{F}{C}{b}{}(X,H,J), & \quad for $a \leq b$.
    \end{IEEEeqnarray*}
\end{definition}

For $[z^\pm,\gamma^\pm] \in \mathcal{P}(H)$, consider the moduli space $\mathcal{M}(H,J,[z^+,\gamma^+],[z^-,\gamma^-])$. Near a point $(w,u) \in \mathcal{M}(H,J,[z^+,\gamma^+],[z^-,\gamma^-])$, this space is a manifold (see \cref{thm:transversality in s1eft}) of dimension
\begin{IEEEeqnarray}{c+x*}
    \plabel{eq:dimension for ms}
    \dim_{(w,u)} \mathcal{M}(H,J,[z^+,\gamma^+],[z^-,\gamma^-])
    =
    \ind^{\tau^+}(z^+,\gamma^+) - \ind^{\tau^-}(z^-,\gamma^-) - 1,
\end{IEEEeqnarray}
where
\begin{IEEEeqnarray*}{c+x*}
    \ind^{\tau^\pm}(z^\pm,\gamma^\pm) \coloneqq \morse(z^{\pm}) + \conleyzehnder^{\tau^{\pm}}(\gamma^{\pm})
\end{IEEEeqnarray*}
and $\tau^{\pm}$ are symplectic trivializations of $(\gamma^{\pm})^* T \hat{X}$ which extend to a symplectic trivialization $\tau$ of $u^* T \hat{X}$. With $\tau^{\pm}$ chosen like this, even though each individual term on the right-hand side of Equation \eqref{eq:dimension for ms} depends on $\tau^{\pm}$, the right-hand side is independent of the choice of $\tau$. Throughout this chapter, if $\mathcal{M}$ is a moduli space of solutions of the parametrized Floer equation, we will denote by $\# \mathcal{M}$ the signed count of points $(w,u)$ in $\mathcal{M}$ such that $\dim_{(w,u)} \mathcal{M} = 0$. 

\begin{definition}
    \label{def:differential}
    We define $\del \colon \homology{}{S^1}{}{F}{C}{}{}(X,H,J) \longrightarrow \homology{}{S^1}{}{F}{C}{}{}(X,H,J)$ by
    \begin{IEEEeqnarray*}{c+x*}
        \del ([z^+,\gamma^+]) \coloneqq \sum_{[z^-,\gamma^-] \in \mathcal{P}(H)}^{} \# \mathcal{M}_{\vphantom{0}}(H,J,[z^+,\gamma^+],[z^-,\gamma^-]) \cdot [z^-,\gamma^-],
    \end{IEEEeqnarray*}
    for each $[z^+,\gamma^+] \in \mathcal{P}(H)$.
\end{definition}

By \cref{lem:action energy for floer trajectories}, the differential respects the action filtration, i.e. the differential $\del$ maps $\homology{}{S^1}{}{F}{C}{a}{}(X,H,J)$ to itself. By \cite[Proposition 2.2]{bourgeoisEquivariantSymplecticHomology2016}, $\partial \circ \partial = 0$.

\begin{definition}
    \phantomsection\label{def:U map}
    We define a map $U \colon \homology{}{S^1}{}{F}{C}{}{}(X,H,J) \longrightarrow \homology{}{S^1}{}{F}{C}{}{}(X,H,J)$ as follows. First, recall that a critical point $z$ of $\tilde{f}_N$ is of the form $z = e^{2 \pi i t} e_j$, for $t \in S^1$ and $j = 0, \ldots, N$. If $j \geq 1$, let $\shf(e^{2 \pi i t} e_j) \coloneqq e^{2 \pi i t} e_{j-1}$. Finally, define
    \begin{IEEEeqnarray*}{c+x*}
        U ([z,\gamma]) \coloneqq
        \begin{cases}
            [\shf(z),\gamma] & \text{if } \morse(z) \geq 2, \\
            0                & \text{if } \morse(z) = 0,
        \end{cases}
    \end{IEEEeqnarray*}
    for $[z,\gamma] \in \mathcal{P}(H)$.
\end{definition}

The definition of $U$ is well-posed because by \cref{def:hamiltonians} \ref{item:pullbacks}, the Hamiltonians $H_{e_j}$ and $H_{e_{j-1}}$ differ by a constant. Therefore, if $\gamma$ is a $1$-periodic orbit of $H_{e_j}$ then it is also a $1$-periodic orbit of $H_{e_{j-1}}$. By \cite[Section 6.3]{guttSymplecticCapacitiesPositive2018}, $U$ is a chain map, i.e. $U \circ \partial = \partial \circ U$.

\begin{lemma}
    The map $U \colon \homology{}{S^1}{}{F}{C}{}{}(X,H,J) \longrightarrow \homology{}{S^1}{}{F}{C}{}{}(X,H,J)$ respects the filtration.
\end{lemma}
\begin{proof}
    Let $[z,\gamma] \in \mathcal{P}(H)$ be such that $\morse(z) \geq 2$ and $\mathcal{A}_{H}(z,\gamma) \leq a$. We wish to show that $\mathcal{A}_{H}(\shf(z),\gamma) \leq \mathcal{A}_{H}(z,\gamma) \leq a$. Assumption \ref{item:pullbacks} of \cref{def:hamiltonians} implies that $H_{\shf(z)} = H_z + E$, where $E \geq 0$. Then,
    \begin{IEEEeqnarray*}{rCls+x*}
        \mathcal{A}_{H}(\shf(z),\gamma)
        & =    & \int_{S^1}^{} \gamma^* \hat{\lambda} - \int_{0}^{1} H(t,\shf(z),\gamma(t)) \edv t & \quad [\text{by definition of $\mathcal{A}_{H}$}] \\
        & =    & \int_{S^1}^{} \gamma^* \hat{\lambda} - \int_{0}^{1} H(t,z,\gamma(t)) \edv t - E   & \quad [\text{since $H_{\shf(z)} = H_z + E$}]      \\
        & =    & \mathcal{A}_{H}(z,\gamma) - E                                                     & \quad [\text{by definition of $\mathcal{A}_{H}$}] \\
        & \leq & \mathcal{A}_{H}(z,\gamma)                                                         & \quad [\text{since $E \geq 0$}]                   \\
        & \leq & a                                                                                 & \quad [\text{by assumption on $[z,\gamma]$}].  & \qedhere
    \end{IEEEeqnarray*}
\end{proof}

We will now define the continuation maps. For $(H^+,J^+) \leq (H^-,J^-) \in \admissible{X}$, we want to define a morphism $\phi^{-,+} \colon \homology{}{S^1}{}{F}{C}{}{}(X,H^+,J^+) \longrightarrow \homology{}{S^1}{}{F}{C}{}{}(X,H^-,J^-)$. Consider the map
\begin{IEEEeqnarray*}{rrCl}
    \inc^{N^-,N^+}_k \colon & \hat{\mathcal{P}}((\inc_k ^{N^-,N^+})^* H^-) & \longrightarrow & \hat{\mathcal{P}}(H^-) \\
                            & (z,\gamma)                                   & \longmapsto     & (\inc^{N^-,N^+}_k(z),\gamma).
\end{IEEEeqnarray*}
This map fits into the commutative diagram
\begin{IEEEeqnarray*}{c+x*}
    \begin{tikzcd}[row sep=scriptsize, column sep={{{{6em,between origins}}}}]
                                                                              & \hat{\mathcal{P}}((\inc_k^{N^-,N^+})^* H^-) \arrow[dl, "\inc^{N^-,N^+}_k"] \arrow[rr] \arrow[dd] &                                                      & \critpt (\tilde{f}_{N^+}) \arrow[dl, "\inc^{N^-,N^+}_k"] \arrow[dd] \\
        \hat{\mathcal{P}}(H^-) \arrow[rr, crossing over, near end] \arrow[dd] &                                                                                                   & \critpt (\tilde{f}_{N^-})                           & \\
                                                                              & \mathcal{P}((\inc_k^{N^-,N^+})^* H^-)  \arrow[dl, dashed, "\exists ! i^{N^-,N^+}_k"] \arrow[rr]  &                                                      & \critpt (f_{N^+}) \arrow[dl, "i^{N^-,N^+}_k"]                       \\
        \mathcal{P}(H^-) \arrow[rr] \ar[uu, leftarrow, crossing over]         &                                                                                                   & \critpt (f_{N^-}) \ar[uu, leftarrow, crossing over] &
    \end{tikzcd}
\end{IEEEeqnarray*}

\begin{definition}
    An \textbf{admissible} homotopy of parametrized Hamiltonians from $H^-$ to $H^+$ is a map $H \colon \R \times S^1 \times S^{2N^+ +1} \times \hat{X} \longrightarrow \R$ which satisfies the conditions in \cref{item:homotopy h 1,item:homotopy h 2,item:homotopy h 3}, where $H_s(t,z,x) = H(s,t,z,x)$. We denote the set of such $H$ by $\mathcal{H}(H^+,H^-)$.
    \begin{enumerate}
        \item \label{item:homotopy h 3} For every $s \in \R$, we have that $H_s$ satisfies all the assumptions in \cref{def:hamiltonians}, with the exceptions that $C_s$ may be in $\operatorname{Spec}(\del X,\lambda|_{\del X})$, and it is not necessarily true that $z \in \critpt \tilde{f}_N$ implies that $H_{s,z}$ is nondegenerate.
        \item \label{item:homotopy h 1} There exists $s_0 > 0$ such that if $\pm s > s_0$ then $H_s = (\inc^{N^\pm,N^+}_0)^* H^\pm$.
        \item \label{item:homotopy h 2} For every $(s,t,z,x) \in \R \times S^1 \times S^{2N^+ + 1} \times \hat{X}$ we have that $\del_s H(s,t,x,z) \leq 0$.
    \end{enumerate}
\end{definition}

\begin{definition}
    An \textbf{admissible} homotopy of parametrized almost complex structures from $J^-$ to $J^+$ is a map $J \colon \R \times S^1 \times S^{2N^+ +1} \times \hat{X} \longrightarrow \End(T \hat{X})$ which satisfies the conditions in \cref{item:homotopy j 1,item:homotopy j 3}, where $J_s(t,z,x) = J(s,t,z,x)$. We denote the set of such $J$ by $\mathcal{J}(J^+,J^-)$.
    \begin{enumerate}
        \item \label{item:homotopy j 3} For every $s \in \R$, we have that $J_s$ satisfies all the assumptions in \cref{def:acs}.
        \item \label{item:homotopy j 1} There exists $s_0 > 0$ such that if $\pm s > s_0$ then $J_s = (\inc^{N^\pm,N^+}_0)^* J^\pm$.
    \end{enumerate}
\end{definition}

\begin{definition}
    Let $[z^\pm,\gamma^\pm] \in \mathcal{P}((\inc^{N^\pm,N^+}_0)^* H^\pm)$ and $(H,J)$ be a homotopy from $(H^-,J^-)$ to $(H^+,J^+)$. A pair $(w,u)$, where $w \colon \R \longrightarrow S^{2N^+ +1}$ and $u \colon \R \times S^1 \longrightarrow \hat{X}$ is a solution of the \textbf{parametrized Floer equation} (with respect to $(H, J)$) if
    \begin{equation*}
        \left\{ \,
            \begin{IEEEeqnarraybox}[
                \IEEEeqnarraystrutmode
                \IEEEeqnarraystrutsizeadd{7pt}
                {7pt}][c]{rCl}
                \dot{w}(s) & = & \nabla \tilde{f}_N(w(s)) \\
                \pdv{u}{s}(s,t) & = & - J^t_{s,w(s)}(u(s,t)) \p{}{2}{ \pdv{u}{t}(s,t) - X_{H^t_{s,w(s)}} (u(s,t)) }.
            \end{IEEEeqnarraybox}
        \right.
    \end{equation*}
    Define $\hat{\mathcal{M}}(H,J,[z^+,\gamma^+],[z^-,\gamma^-])$ to be the moduli space of solutions $(w,u)$ of the pa\-ra\-me\-trized Floer equation such that $(w(s),u(s,\cdot))$ converges as $s \to \pm \infty$ to an element in the equivalence class $[z^\pm,\gamma^\pm]$. Define an action of $S^1$ on $\hat{\mathcal{M}}(H,J,[z^+,\gamma^+],[z^-,\gamma^-])$ by 
    \begin{IEEEeqnarray*}{c+x*}
        t \cdot (w,u) = (e ^{2 \pi i t} w, u(\cdot, \cdot - t)).
    \end{IEEEeqnarray*}
    Finally, let $\mathcal{M}(H,J,[z^+,\gamma^+],[z^-,\gamma^-]) \coloneqq \hat{\mathcal{M}}(H,J,[z^+,\gamma^+],[z^-,\gamma^-])/S^1$.
\end{definition}

\begin{definition}
    \label{def:continuation map}
    The \textbf{continuation map} is the map 
    \begin{IEEEeqnarray*}{c+x*}
        \phi^{-,+} \colon \homology{}{S^1}{}{F}{C}{}{}(X,H^+,J^+) \longrightarrow \homology{}{S^1}{}{F}{C}{}{}(X,H^-,J^-)
    \end{IEEEeqnarray*}
    given as follows. Choose a regular homotopy $(H, J)$ from $(H^-,J^-)$ to $(H^+,J^+)$. Then, for every $[z^+, \gamma^+] \in \mathcal{P}(H^+)$,
    \begin{IEEEeqnarray*}{c}
        \phi^{-,+}([z^+,\gamma^+]) \coloneqq \sum_{[z^-,\gamma^-] \in \mathcal{P}((\inc_0 ^{N^-,N^+})^* H^-)} \# \mathcal{M}_{\vphantom{0}}(H,J,[z^+,\gamma^+],[z^-,\gamma^-]) \cdot [\inc^{N^-,N^+}_0 (z^-),\gamma^-].
    \end{IEEEeqnarray*}
\end{definition}

\begin{lemma}
    The map $\phi^{-,+}$ respects the action filtrations.
\end{lemma}
\begin{proof}
    Assume that $[z^\pm,\gamma^\pm] \in \mathcal{P}((\inc_0 ^{N^\pm,N^+})^* H^\pm)$ is such that $\mathcal{A}_{H^+}(z^+,\gamma^+) \leq a$ and $\mathcal{M}(H,J,[z^+,\gamma^+],[z^-,\gamma^-])$ is nonempty. We wish to show that $\mathcal{A}_{H^-}(\inc^{N^-,N^+}_0(z^-),\gamma^-) \leq a$. The proof is the following computation.
    \begin{IEEEeqnarray*}{rCls+x*}
        \IEEEeqnarraymulticol{3}{l}{\mathcal{A}_{H^-}(\inc^{N^-,N^+}_0(z^-),\gamma^-)}\\
        \quad & =    & \int_{S^1}^{} \gamma^* \hat{\lambda} - \int_{0}^{1} H^-(t, \inc^{N^-,N^+}_0(z^-),\gamma^-(t)) \edv t       & \quad [\text{definition of action functional}]  \\
              & =    & \int_{S^1}^{} \gamma^* \hat{\lambda} - \int_{0}^{1} ((\inc_0 ^{N^-,N^+})^* H^-)(t, z^-,\gamma^-(t)) \edv t & \quad [\text{definition of $\inc^{N^-,N^+}_0$}] \\
              & =    & \mathcal{A}_{(\inc_0 ^{N^-,N^+})^* H^-}(z^-,\gamma^-)                                                      & \quad [\text{definition of action functional}]  \\
              & \leq & \mathcal{A}_{H^+}(z^+,\gamma^+)                                                                            & \quad [\text{by \cref{lem:action energy for floer trajectories}}]    \\
              & \leq & a                                                                                                          & \quad [\text{by assumption}].                    & \qedhere
    \end{IEEEeqnarray*}
\end{proof}

By \cite[Section 2.4]{bourgeoisEquivariantSymplecticHomology2016}, the $U$ maps and the continuation maps commute. Moreover, by the usual arguments in Floer theory, we have (see also \cite[Section 5.3]{guttSymplecticCapacitiesPositive2018}):
\begin{enumerate}
    \item The continuation map $\phi^{-,+}$ is a chain map, i.e. $\phi^{-,+} \circ \del^+ = \del^- \circ \phi^{-,+}$.
    \item The continuation map $\phi^{-,+}$ is independent (up to chain homotopy, i.e. as a morphism in $\comp$) on the choice of regular homotopy $(H, J)$.
    \item The continuation maps are functorial, i.e. if $(H^0,J^0) \leq (H^1,J^1) \leq (H^2,J^2) \in \admissible{X}$ then $\phi^{2,1} \circ \phi^{1,0} = \phi^{2,0}$.
\end{enumerate}

\begin{remark}
    \label{rmk:grading for s1esh}
    By the determinant property of \cref{thm:properties of cz}, the parity of the Conley--Zehnder index of a Hamiltonian $1$-periodic orbit is independent of the choice of trivialization. Therefore, $\homology{}{S^1}{}{F}{C}{}{}(X,H,J)$ has a $\Z_{2}$-grading given by
    \begin{IEEEeqnarray}{c}
        \deg([z,\gamma]) \coloneqq \mu([z,\gamma]) \coloneqq \morse(z) + \conleyzehnder(\gamma). \plabel{eq:grading s1esh}
    \end{IEEEeqnarray}
    If $\pi_1(X) = 0$ and $c_1(TX)|_{\pi_2(X)} = 0$, then by \cref{lem:cz of hamiltonian is independent of triv over filling disk} we have well-defined Conley--Zehnder indices in $\Z$. Therefore, Equation \eqref{eq:grading s1esh} defines a $\Z$-grading on $\homology{}{S^1}{}{F}{C}{}{}(X,H,J)$. With respect to this grading,
    \begin{IEEEeqnarray*}{rCls+x*}
        \deg(\partial)   & = & -1, \\
        \deg(U)          & = & -2, \\
        \deg(\phi^{-,+}) & = & 0.
    \end{IEEEeqnarray*}
\end{remark}

\begin{definition}
    If $(X,\lambda)$ is a nondegenerate Liouville domain, the \textbf{$S^1$-equivariant Floer chain complex} of $X$ is the functor
    \begin{IEEEeqnarray*}{rrCl}
        \homology{}{S^1}{X}{F}{C}{}{} \colon & \admissible{X} & \longrightarrow & \comp                        \\
                                             & (H^+,J^+)    & \longmapsto     & (\homology{}{S^1}{}{F}{C}{}{}(X,H^+,J^+), \del^+, U^+)  \\
                                             & \downarrow   & \longmapsto     & \downarrow \phi^{-,+}        \\
                                             & (H^-,J^-)    & \longmapsto     & (\homology{}{S^1}{}{F}{C}{}{}(X,H^-,J^-), \del^-, U^-),
    \end{IEEEeqnarray*}
    The \textbf{$S^1$-equivariant Floer homology} of $X$ is the functor $\homology{}{S^1}{X}{F}{H}{}{} = H \circ \homology{}{S^1}{X}{F}{C}{}{}$. The \textbf{positive $S^1$-equivariant Floer homology} of $X$ is the functor $\homology{}{S^1}{X}{F}{H}{+}{}$ given by
    \begin{IEEEeqnarray*}{rCls+x*}
        \homology{}{S^1}{X}{F}{H}{+}{}(H,J) 
        & \coloneqq & \homology{}{S^1}{}{F}{H}{(\varepsilon, +\infty)}{}(X,H,J) \\
        & =         & \homology{}{S^1}{}{F}{H}{}{}(X,H,J) / \homology{}{S^1}{}{F}{H}{\varepsilon}{}(X,H,J).
    \end{IEEEeqnarray*}
\end{definition}

\begin{definition}
    For $(X,\lambda)$ is a nondegenerate Liouville domain, the \textbf{$S^1$-equivariant symplectic homology} of $X$ is the object in $\modl$ given by $\homology{}{S^1}{}{S}{H}{}{}(X,\lambda) \coloneqq \colim \homology{}{S^1}{X}{F}{H}{}{}$. The \textbf{positive $S^1$-equivariant symplectic homology} of $X$ is given by 
    \begin{IEEEeqnarray*}{rCls+x*}
        \homology{}{S^1}{}{S}{H}{+}{}(X,\lambda)
        & \coloneqq & \colim \homology{}{S^1}{X}{F}{H}{+}{} \\
        & =         & \homology{}{S^1}{}{S}{H}{(\varepsilon, +\infty)}{}(X, \lambda) \\
        & =         & \homology{}{S^1}{}{S}{H}{}{}(X, \lambda) / \homology{}{S^1}{}{S}{H}{\varepsilon}{}(X, \lambda).
    \end{IEEEeqnarray*}
\end{definition}

\section{Viterbo transfer map of a Liouville embedding}
\label{sec:viterbo transfer map of liouville embedding}

Our goal is to prove that $\homology{}{S^1}{}{S}{H}{}{}$ is a contravariant functor from a suitable category of Liouville domains onto $\modl$. More specifically, suppose that $(V,\lambda_V)$ and $(W,\lambda_W)$ are nondegenerate Liouville domains and $\varphi \colon (V,\lambda_V) \longrightarrow (W,\lambda_W)$ is a $0$-codimensional strict generalized Liouville embedding. We will define a \textbf{Viterbo transfer map}
\begin{IEEEeqnarray*}{rrCl}
    \varphi_! \colon & \homology{}{S^1}{}{S}{H}{}{}(W,\lambda_W)  & \longrightarrow & \homology{}{S^1}{}{S}{H}{}{}(V,\lambda_V), \\
    \varphi_! \colon & \homology{}{S^1}{}{S}{H}{+}{}(W,\lambda_W) & \longrightarrow & \homology{}{S^1}{}{S}{H}{+}{}(V,\lambda_V),
\end{IEEEeqnarray*}
which is a morphism in $\modl$. We will start by definition the Viterbo transfer map in the case where $\varphi$ is a Liouville embedding instead of just a generalized Liouville embedding. Consider the completions $\hat{V}$ and $\hat{W}$ of $V$ and $W$ respectively, as well as the induced map $\hat{\varphi} \colon \hat{V} \longrightarrow \hat{W}$. Choose $R$ so small that $\hat{\varphi}(V \union ([0,R] \times \del V)) \subset W$. We define
\begin{IEEEeqnarray*}{rCls+x*}
    \varepsilon_V & \coloneqq & \frac{1}{2} \min \operatorname{Spec}(\del V, \lambda_V), \\
    \varepsilon_W & \coloneqq & \frac{1}{2} \min \operatorname{Spec}(\del W, \lambda_W), \\
    \varepsilon   & \coloneqq & \min \{ \varepsilon_V, \varepsilon_W \}.
\end{IEEEeqnarray*}


\begin{definition}
    \label{def:stair hamiltonians}
    A \textbf{stair} parametrized Hamiltonian is a map $\overline{H} \colon S^1 \times S^{2N+1} \times \hat{W} \longrightarrow \R$ such that $\overline{H}$ satisfies the conditions in \cref{item:invariant,item:flow lines,item:pullbacks,item:ndg} from \cref{def:hamiltonians} as well as the conditions in the Items below. We denote the set of such $\overline{H}$ by $\mathcal{H}(W,V,N)$.
    \begin{enumerate}[label=(\Roman*)]
        \item \label{item:stair 1} On $S^1 \times S^{2N+1} \times V$, we have that $\hat{\varphi}^* \overline{H}$ has values in $(0, \varepsilon)$, is $S^1$-independent and is $C^2$-close to a constant.
        \item \label{item:stair 2} On $S^1 \times S^{2N+1} \times [0, \delta_V] \times \del V$, we have that $-\varepsilon < \hat{\varphi}^* \overline{H} < \varepsilon$ and $\hat{\varphi}^* \overline{H}$ is $C^2$-close to $(t,z,r,x) \longmapsto h_{\rmn{2}}(e^r)$, where $h_{\rmn{2}} \colon [1,e^{\delta_V}] \longrightarrow \R$ is increasing and strictly convex.
            \myitem[($\mathrm{S}_{V}$)] \plabel{item:stair v} On $S^1 \times S^{2N+1} \times [\delta_V, R - \delta_V] \times \del V$, we have that $\hat{\varphi}^* \overline{H}(t,z,r,x) = C_V e^r + D_V$, for $D_V \in \R$ and $C_V \in \R_{>0} \setminus \operatorname{Spec}(\del V, \lambda_V|_{\del V}) \union \operatorname{Spec}(\del W, \lambda_W|_{\del W})$.
        \item \label{item:stair 3} On $S^1 \times S^{2N+1} \times [R - \delta_V, R] \times \del V$, we have that $\hat{\varphi}^* \overline{H}$ is $C^2$-close to the function $(t,z,r,x) \longmapsto h_{\rmn{3}}(e^r)$, where $h_{\rmn{3}} \colon [e^{R - \delta_V},e^{R}] \longrightarrow \R$ is increasing and strictly concave.
        \item \label{item:stair 4} On $S^1 \times S^{2N+1} \times W \setminus \hat{\varphi}(V \union [0, R] \times \del V)$, the function $\overline{H}$ is $C^2$-close to a constant.
        \item \label{item:stair 5} On $S^1 \times S^{2N+1} \times [0, \delta_W] \times \del W$, we have that $\overline{H}$ is $C^2$-close to $(t,z,r,x) \longmapsto h_{\rmn{5}}(e^r)$, where $h \colon [1,e^{\delta_W}] \longrightarrow \R$ is increasing and strictly convex.
            \myitem[($\mathrm{S}_{W}$)] \plabel{item:stair w} On $S^1 \times S^{2N+1} \times [\delta_W, +\infty) \times \del W$, we have that $\overline{H}(t,z,r,x) = C_W e^r + D_W$, for $D_W \in \R$ and $C_W \in \R_{>0} \setminus \operatorname{Spec}(\del V, \lambda_V|_{\del V}) \union \operatorname{Spec}(\del W, \lambda_W|_{\del W})$ such that $C_W < e^{-\delta_W}(C_V e^{R - \delta_V} + D_V)$.
    \end{enumerate}
\end{definition}

\begin{remark}
    If $(z, \gamma) \in \hat{\mathcal{P}}(H)$, then either $\gamma$ is nonconstant and $\img \gamma$ is in region $\rmn{2}$, $\rmn{3}$ or $\rmn{5}$, or $\gamma$ is constant and $\img \gamma$ is in region $\rmn{1}$ or $\rmn{4}$. There are no $1$-periodic orbits in the slope regions $\mathrm{S}_{V}$ and $\mathrm{S}_{W}$.
\end{remark}

\begin{lemma}
    \label{lem:action stair}
    The actions of $1$-periodic orbits of $\overline{H}$ are ordered according to
    \begin{IEEEeqnarray*}{c+x*}
        \mathcal{A}_{\overline{H}}(\rmn{4}) < \mathcal{A}_{\overline{H}}(\rmn{5}) < 0 < \mathcal{A}_{\overline{H}}(\rmn{1}) < \varepsilon < \mathcal{A}_{\overline{H}}(\rmn{2}).
    \end{IEEEeqnarray*}
\end{lemma}
\begin{proof}
    Consider \cref{fig:action stair}. By \cref{lem:action in symplectization,def:stair hamiltonians}, we have that $\mathcal{A}_{\overline{H}}$ is constant in regions $\rmn{1}$, $\mathrm{S}_{V}$, $\rmn{4}$ and $\mathrm{S}_{W}$, $\mathcal{A}_{\overline{H}}$ is strictly increasing in regions $\rmn{2}$ and $\rmn{5}$, and $\mathcal{A}_{\overline{H}}$ is strictly decreasing in region $\rmn{3}$. From this reasoning, we conclude that $\mathcal{A}_{\overline{H}}(\rmn{4}) < \mathcal{A}_{\overline{H}}(\rmn{5})$ and $0 < \mathcal{A}_{\overline{H}}(\rmn{1}) < \varepsilon$. By the same argument as in the proof of \cref{lem:action admissible}, we conclude that $\varepsilon < \mathcal{A}_{\overline{H}}(\rmn{2})$. We show that $\mathcal{A}_{\overline{H}}(\rmn{5}) < 0$.
    \begin{IEEEeqnarray*}{rCls+x*}
        \IEEEeqnarraymulticol{3}{l}{\mathcal{A}_{\overline{H}}(\rmn{5})}\\
        \quad & = & e^{r_W} T(r_W) - H(r_W)                       & \quad [\text{by \cref{lem:action in symplectization}}]                                \\
        \quad & < & e^{r_W} C_W - H(r_W)                          & \quad [\text{$T(\delta_W) = C_W$ and $T' = \exp \cdot h_{\rmn{5}}'' \circ \exp > 0$}] \\
        \quad & < & e^{r_W} C_W - (C_V e^{R-\delta_V} + D_V)      & \quad [\text{$H(r_W) > H(R - \delta_V) = C_V e^{R-\delta_V} + D_V$}]            \\
        \quad & < & e^{\delta_W} C_W - (C_V e^{R-\delta_V} + D_V) & \quad [\text{since $r_W < \delta_W$}]                                                       \\
        \quad & < & 0                                             & \quad [\text{since $C_W < e^{-\delta_W}(C_V e^{R - \delta_V} + D_V)$}].                      & \qedhere
    \end{IEEEeqnarray*}
\end{proof}

\begin{figure}[ht]
    \centering

    \begin{tikzpicture}
        [
        help lines/.style={thin, draw = black!50},
        Hamiltonian/.style={thick},
        action/.style={thick},
        axisv/.style={},
        axisw/.style={}
        ]
        \tikzmath{
            \a = 4;
            \b = 3;
            \c = 3;
            \d = 0.5;
            \e = 3;
            \f = 3;
            \g = 1;
            \h = 0.4;
            \sml = 0.05;
            \dOne = -0.3;
            \dFour = 2.5;
            \vFive = 2.6;
            \mTwo   = -(12 * (-\dOne + \dFour) * exp(\d))/((-1 + exp(\d))^3 * (1 + exp(\d)) * (-exp(\d) + exp(\e)));
            \n      = (2 * (-\dOne + \dFour) * exp(\d) * (-1 + 3 * exp(\d)))/((-1 + exp(\d))^3 * (1 + exp(\d)) * (-exp(\d) + exp(\e)));
            \o      = (\dFour * exp(1)^\d - 2 * \dFour * exp(2 * \d) + 2 * \dOne * exp(4 * \d) - \dOne * exp(5 * \d) - \dOne * exp(\e) + 2 * \dOne * exp(\d + \e) - 2 * \dOne * exp(3 * \d + \e) + \dOne * exp(4 * \d + \e))/((-1 + exp(\d))^3 * (1 + exp(\d)) * (-exp(\d) + exp(\e)));
            \uv     = (2 * (-\dOne + \dFour) * exp(\d))/((1 + exp(\d)) * (-exp(\d) + exp(\e))) ;
            \vv     = (\dFour * exp(\d) - \dOne * exp(\e))/(exp(\d) - exp(\e)) ;
            \mThree = -(12 * (-\dOne + \dFour) * exp(4 * \d + \e))/((-1 + exp(\d))^3 * (1 + exp(\d)) * (exp(\d) - exp(\e)));
            \q      = - (2 * (-\dOne + \dFour) * exp(3 * \d + \e) * (-3 + exp(\d)))/((-1 + exp(\d))^3 * (1 + exp(\d)) * (exp(\d) - exp(\e)));
            \s      = (-\dFour * exp(\d) + 2 * \dFour * exp(2 * \d) - 2 * \dFour * exp(4 * \d) + \dFour * exp(5 * \d) + \dFour * exp(\e) - 2 * \dFour * exp(\d + \e) + 2 * \dOne * exp(3 * \d + \e) - \dOne * exp(4 * \d + \e))/((-1 + exp(\d))^3 * (1 + exp(\d)) * (exp(\d) - exp(\e)));
            \uw = -2 * (\dFour - \vFive) / (-1+exp(\g));
            \vw = (2 * exp(\g) * \dFour - \vFive - exp(\g) * \vFive) / (-1+exp(\g));
            \jj = - 12 * (-\dFour + \vFive) / (-1+exp(\g))^4;
            \kk = 2 * (-1 + 3 * exp(\g)) * (-\dFour + \vFive) / (-1+exp(\g))^4;
            \la = ( -2 * exp(\g) * \dFour + 6 * exp(2 * \g) * \dFour - 4 * exp(3 * \g) * \dFour + exp(4 * \g) * \dFour + \vFive - 2 * exp(\g) * \vFive ) / (-1+exp(\g))^4;
            function h2  (\r) { return {\o + \n * \r + 1/2 * exp(\d) * \mTwo * \r^2 + 1/6 * (-1 - exp(\d)) * \mTwo * \r^3 + (\mTwo * \r^4)/12}; };
            function dh2 (\r) { return {\n + 1/6 * \mTwo * \r * (-3 * exp(\d) * (-2 + \r) + \r * (-3 + 2 * \r))}; };
            function h3  (\r) { return {\s + \q * \r - (1/6) * exp(-\d) * \mThree * (-3 + \r) * \r^2 + 1/12 * \mThree * (-2 + \r) * \r^3}; };
            function dh3 (\r) { return {\q + (1/6) * exp(-\d) * \mThree * \r * (6 - 3 * (1 + exp(\d)) * \r + 2 * exp(\d) *  \r^2) }; };
            function h5  (\r) { return {\la + \kk * \r + 1/2 * exp(\g) * \jj * \r^2 + 1/6 * (-1 - exp(\g)) * \jj * \r^3 + 1/12 * \jj * \r^4 }; };
            function dh5 (\r) { return {\kk + 1/6 * \jj * \r * (-3 * exp(\g) * (-2 + \r) + \r * (-3 + 2 * \r))}; };
            function hsv (\r) { return {\uv * \r + \vv}; };
            function hsw (\r) { return {\uw * \r + \vw}; };
            function H2  (\r) { return {h2 (exp(\r))}; };
            function H3  (\r) { return {h3 (exp(\r))}; };
            function H5  (\r) { return {h5 (exp(\r))}; };
            function Hsv (\r) { return {hsv(exp(\r))}; };
            function Hsw (\r) { return {hsw(exp(\r))}; };
            function a2  (\r) { return { exp(\r) * dh2(exp(\r)) - H2(\r) }; };
            function a3  (\r) { return { exp(\r) * dh3(exp(\r)) - H3(\r) }; };
            function a5  (\r) { return { exp(\r) * dh5(exp(\r)) - H5(\r) }; };
            \i = ln((\a-\vw)/\uw) - \g;
            \test = -\uw + exp(-\g) * (\uv * exp(\e-\d) + \vv);
        }
        \draw[Hamiltonian, domain = 0          :\c            ] plot (\x, {\dOne});
        \draw[Hamiltonian, domain = \c         :\c+\d         ] plot (\x, {H2(\x - \c)});
        \draw[Hamiltonian, domain = \c+\d      :\c+\e-\d      ] plot (\x, {Hsv(\x - \c)});
        \draw[Hamiltonian, domain = \c+\e-\d   :\c+\e         ] plot (\x, {H3(\x - \c - \e)});
        \draw[Hamiltonian, domain = \c+\e      :\c+\e+\f      ] plot (\x, {\dFour});
        \draw[Hamiltonian, domain = \c+\e+\f   :\c+\e+\f+\g   ] plot (\x, {H5(\x - \c - \e - \f)});
        \draw[Hamiltonian, domain = \c+\e+\f+\g:\c+\e+\f+\g+\i] plot (\x, {Hsw(\x - \c - \e - \f)}) node[right] {$\overline{H}$};

        \draw[action, domain = 0          :\c            ] plot (\x, {-\dOne});
        \draw[action, domain = \c         :\c+\d         ] plot (\x, {a2(\x - \c)});
        \draw[action, domain = \c+\d      :\c+\e-\d      ] plot (\x, {-\vv});
        \draw[action, domain = \c+\e-\d   :\c+\e         ] plot (\x, {a3(\x - \c - \e)});
        \draw[action, domain = \c+\e      :\c+\e+\f      ] plot (\x, {-\dFour});
        \draw[action, domain = \c+\e+\f   :\c+\e+\f+\g   ] plot (\x, {a5(\x - \c - \e - \f)});
        \draw[action, domain = \c+\e+\f+\g:\c+\e+\f+\g+\i] plot (\x, {-\vw}) node[right] {$\mathcal{A}_{\overline{H}}$};

        \draw[help lines] (0,\h) node[left] {$+\varepsilon$} -- (\c+\e+\f+\g+\i,\h);
        \draw[help lines] (0,-\h) node[left] {$-\varepsilon$} -- (\c+\e+\f+\g+\i,-\h);
        \draw[help lines] (\c+\d,-\b) node[below, axisv] {$\delta_V$} -- (\c+\d,\a);
        \draw[help lines] (\c+\e-\d,-\b) node[below, axisv] {$R-\delta_V\hspace{1.5em}$} -- (\c+\e-\d,\a);
        \draw[help lines] (\c+\e,-\b) node[below, axisv] {$\hspace{0.5em}R$} -- (\c+\e,\a);
        \draw[help lines] (\c+\e+\f+\g,-\b) node[below, axisw] {$\delta_W$} -- (\c+\e+\f+\g,\a);
        \draw[->] (0,-\b) -- (0,\a) node[above] {$\R$};
        \draw (0,0) -- (\c,0);
        \draw[->, axisw] (\c+\e+\f,0) -- (\c+\e+\f+\g+\i,0);
        \draw[->, axisw] (\c+\e+\f,-\b) node[below] {$0$} -- (\c+\e+\f,\a) node[above] {$\R$};
        \draw[->, axisv] (\c,0) -- (\c+\e+\f,0);
        \draw[->, axisv] (\c,-\b) node[below] {$0$} -- (\c,\a) node[above] {$\R$};

        \draw (\c/2,\a) node[below] {$\mathrm{I}$};
        \draw (\c+\d/2,\a) node[below] {$\mathrm{II}$};
        \draw (\c+\e/2,\a) node[below] {$\mathrm{S}_{{V}}$};
        \draw (\c+\e-\d/2,\a) node[below] {$\mathrm{III}$};
        \draw (\c+\e+\f/2,\a) node[below] {$\mathrm{IV}$};
        \draw (\c+\e+\f+\g/2,\a) node[below] {$\mathrm{V}$};
        \draw (\c+\e+\f+\g+1,\a) node[below] {$\mathrm{S}_{{W}}$};

        \draw[help lines, decoration = {brace, mirror, raise=5pt}, decorate] (0,-\b-.75) -- node[below=6pt] {\scriptsize $V$} (\c - \sml,-\b-.75);
        \draw[help lines, decoration = {brace, mirror, raise=5pt}, decorate] (\c+\sml,-\b-.75) -- node[below=6pt] {\scriptsize $[0,R] \times \del V$} (\c + \e - \sml,-\b-.75);
        \draw[help lines, decoration = {brace, mirror, raise=5pt}, decorate] (\c+\e+\sml,-\b-.75) -- node[below=6pt] {\scriptsize ${W \setminus \hat{\varphi} (V \union [0,R] \times \del V)}$} (\c + \e + \f - \sml,-\b-.75);
        \draw[help lines, decoration = {brace, mirror, raise=5pt}, decorate] (\c+\e+\f+\sml,-\b-.75) -- node[below=6pt] {\scriptsize $\R_{\geq 0} \times \del W$} (\c+\e+\f+\g+\i,-\b-.75);
    \end{tikzpicture}

    \caption{Action of a $1$-periodic orbit of $\overline{H}$}
    \label{fig:action stair}
\end{figure}

\begin{definition}
    \phantomsection\label{def:stair acs}
    A \textbf{stair} parametrized almost complex structure is a map $\overline{J} \colon S^1 \times S^{2N+1} \times \hat{W} \longrightarrow \End(T \hat{W})$ satisfying the conditions in \cref{def:stair acs 1,def:stair acs 2,def:stair acs 3,def:stair acs 4} below. We denote the set of such $\overline{J}$ by $\mathcal{J}(W,V,N)$.
    \begin{enumerate}
        \item \label{def:stair acs 1} $\overline{J}$ is $S^1$-invariant.
        \item \label{def:stair acs 2} $\overline{J}$ is $\hat{\omega}$-compatible.
        \item \label{def:stair acs 3} $\overline{J}$ is cylindrical on $S^1 \times S^{2N+1} \times [0, \delta] \times \del V$ and on $S^1 \times S^{2N+1} \times \R_{\geq 0} \times \del W$.
        \item \label{def:stair acs 4} $(\tilde{\iota}_0^{N,N-1})^* \overline{J} = (\tilde{\iota}_1^{N,N-1})^* \overline{J}$.
    \end{enumerate}
\end{definition}

\begin{definition}
    Define sets
    \begin{IEEEeqnarray*}{rCls+x*}
        \stair{W,V} & \coloneqq & 
        \left\{
            (\overline{H}, \overline{J})
            \ \middle\vert
            \begin{array}{l}
                \overline{H} \in \mathcal{H}(W,V,N) \text{ and } \overline{J} \in \mathcal{J}(W,V,N) \text{ for some }N, \\
                (\overline{H}, \overline{J}) \text{ is regular}
            \end{array}
        \right\},
        \\
        \admstair{W,V} & \coloneqq & 
        \left\{
            (H,J,\overline{H}, \overline{J})
            \ \middle\vert
            \begin{array}{l}
                H \in \mathcal{H}(W,N), J \in \mathcal{J}(W,N), \\
                \overline{H} \in \mathcal{H}(W,V,N) \text{ and } \overline{J} \in \mathcal{J}(W,V,N) \text{ for some }N, \\
                H \leq \overline{H}, \text{ and } (H,J) \text{ and } (\overline{H}, \overline{J}) \text{ are regular}
            \end{array}
        \right\}.
    \end{IEEEeqnarray*}
    Define preorders on $\stair{W,V}$ and $\admstair{W,V}$ by
    \begin{IEEEeqnarray*}{rCls+x*}
        (\overline{H}^+,\overline{J}^+) \leq (\overline{H}^-,\overline{J}^-) 
        & \mathrel{\mathop:}\Longleftrightarrow & 
        \left\{
            \begin{array}{l}
                N^+ \leq N^-, \\
                \overline{H}^+ \leq (\inc_0 ^{N^-,N^+})^* \overline{H}^-,
            \end{array}
        \right. \\
        (H^+,J^+,\overline{H}^+,\overline{J}^+) \leq (H^-,J^-,\overline{H}^-,\overline{J}^-)
        & \mathrel{\mathop:}\Longleftrightarrow & 
        \left\{
            \begin{array}{l}
                N^+ \leq N^-, \\
                H^+ \leq (\inc_0 ^{N^-,N^+})^* H^-, \\
                \overline{H}^+ \leq (\inc_0 ^{N^-,N^+})^* \overline{H}^-.
            \end{array}
        \right.
    \end{IEEEeqnarray*}
\end{definition}

\begin{definition}
    Define a function $\pi^{\mathcal{H}}_{W,V,N} \colon \mathcal{H}(W,V,N) \longrightarrow \mathcal{H}(V,N)$ by $\pi_{W,V,N}^{\mathcal{H}}(\overline{H}) = \overline{H}_V$, where
    \begin{IEEEeqnarray*}{c+x*}
        \overline{H}_V(t,z,x) \coloneqq
        \begin{cases}
            \overline{H}(t,z,\hat{\varphi}(x)) & \text{if } x \in V \union ([0,R] \times \del V), \\
            C_V e^r + D_V                      & \text{if } x = (r,y) \in [R, +\infty) \times \del V.
        \end{cases}
    \end{IEEEeqnarray*}
    Define a function $\pi^{\mathcal{J}}_{W,V,N} \colon \mathcal{J}(W,V,N) \longrightarrow \mathcal{J}(V,N)$ by $\pi_{W,V,N}^{\mathcal{J}}(\overline{J}) = \overline{J}_V$, where
    \begin{IEEEeqnarray*}{c+x*}
        \overline{J}_V(t,z,x) \coloneqq
        \begin{cases}
            \dv \hat{\varphi}^{-1}(\hat{\varphi}(x)) \circ \overline{J}(t,z,\hat{\varphi}(x)) \circ \dv \hat{\varphi}(x)       & \text{if } x \in V \union ([0,R] \times \del V), \\
            \dv \hat{\varphi}^{-1}(\hat{\varphi}(0,y)) \circ \overline{J}(t,z,\hat{\varphi}(0,y)) \circ \dv \hat{\varphi}(0,y) & \text{if } x = (r,y) \in [0, +\infty) \times \del V.
        \end{cases}
    \end{IEEEeqnarray*}
\end{definition}

\begin{definition}
    Define the functors
    \begin{IEEEeqnarray*}{rrClCl}
        \pi_W \colon                                      & \admstair{W,V} & \longrightarrow & \admissible{W}, & \text{ given by } & \pi_W(H,J,\overline{H},\overline{J}) \coloneqq (H,J),                        \\
        \pi_{W,V} \colon                                  & \admstair{W,V} & \longrightarrow & \stair{W,V},    & \text{ given by } & \pi_W(H,J,\overline{H},\overline{J}) \coloneqq (\overline{H}, \overline{J}), \\
        \pi_{W,V}^{\mathcal{H} \times \mathcal{J}} \colon & \stair{W,V}    & \longrightarrow & \admissible{V}, & \text{ given by } & \pi_{W,V}^{\mathcal{H} \times \mathcal{J}}(\overline{H},\overline{J}) \coloneqq (\pi^{\mathcal{H}}_{W,V,N}(\overline{H}),\pi^{\mathcal{J}}_{W,V,N}(\overline{J})) = (\overline{H}_V, \overline{J}_V),
    \end{IEEEeqnarray*}
    for $(\overline{H}, \overline{J}) \in \mathcal{H}(W,V,N) \times \mathcal{J}(W,V,N)$. Let $\pi_V^{} \coloneqq \pi_{W,V}^{\mathcal{H} \times \mathcal{J}} \circ \pi_{W,V}^{} \colon \admstair{W,V}^{} \longrightarrow \admissible{V}^{}$.
\end{definition}

\begin{definition}
    \phantomsection\label{def:homotopy stair to admissible hamiltonian}
    Let $H^+ \in \mathcal{H}(W,N^+)$ be an admissible parametrized Hamiltonian and $H^- \in \mathcal{H}(W,V,N^-)$ be a stair parametrized Hamiltonian. Assume that $N^+ \leq N^-$ and $(\tilde{i}_0^{N^-,N^+}) H^+ \leq H^-$. An \textbf{admissible} homotopy of parametrized Hamiltonians from $H^-$ to $H^+$ is a map $H \colon \R \times S^1 \times S^{2 N^+ + 1} \times \hat{W} \longrightarrow \R$ which satisfies the conditions in \cref{item:homotopy stair to admissible hamiltonian 1,item:homotopy stair to admissible hamiltonian 2,item:homotopy stair to admissible hamiltonian 3} for some $s_0 > 0$, where $H_s(t,z,x) = H(s,t,z,x)$. We denote the set of such $H$ by $\mathcal{H}(H^+,H^-)$.
    \begin{enumerate}
        \item \label{item:homotopy stair to admissible hamiltonian 1} For every $s \in (-s_0, s_0)$, we have that $H_s$ satisfies all the conditions in \cref{def:stair hamiltonians} with the exceptions that $C_{W,s}$ and $C_{V,s}$ are possibly in $\operatorname{Spec}(\del W, \lambda_W|_{\del W}) \union \operatorname{Spec}(\del V, \lambda_V|_{\del V})$ and $H_{s,z}$ is not necessarily nondegenerate for $z \in \critpt \tilde{f}_{N^+}$.
        \item \label{item:homotopy stair to admissible hamiltonian 2} For every $s$, if $\pm s \geq s_0$ then $H_s = (\tilde{i}_0^{N^\pm, N^+})^* H^\pm$.
        \item \label{item:homotopy stair to admissible hamiltonian 3} For every $(s,t,z,x) \in \R \times S^1 \times S^{2 N^+ + 1} \times \hat{W}$ we have $\del_s H(s,t,x,z) \leq 0$.
    \end{enumerate}
\end{definition}

\begin{remark}
    In \cref{def:homotopy stair to admissible hamiltonian}, the parameters of $H_s$ depend on $s$. In particular, the ``constant'' value that $H_s$ takes in regions $\rmn{1}$ and $\rmn{4}$ is dependent on $s$. However, the parameter $R$ does not depend on $s$.
\end{remark}

\begin{definition}
    \label{def:homotopy stair to admissible acs}
    Let $J^+ \in \mathcal{J}(W,N^+)$ be an admissible parametrized almost complex structure and $J^- \in \mathcal{J}(W,V,N^-)$ be a stair parametrized almost complex structure. An \textbf{admissible} homotopy of parametrized almost complex structures from $J^-$ to $J^+$ is a map $J \colon \R \times S^1 \times S^{2 N^+ + 1} \times \hat{W} \longrightarrow \End(T \hat{W})$ which satisfies the conditions in \cref{item:homotopy stair to admissible acs 1,item:homotopy stair to admissible acs 2} for some $s_0 > 0$, where $J_s(t,z,x) = J(s,t,z,x)$. We denote the set of such $J$ by $\mathcal{J}(J^+,J^-)$.
    \begin{enumerate}
        \item \label{item:homotopy stair to admissible acs 1} For every $s \in (-s_0, s_0)$, we have that $J_s$ satisfies all the conditions in \cref{def:stair acs}.
        \item \label{item:homotopy stair to admissible acs 2} For every $s$, if $\pm s \geq s_0$ then $J_s = (\tilde{i}_0^{N^\pm, N^+})^* J^\pm$.
    \end{enumerate}
\end{definition}

\begin{remark}
    \label{rmk:floer complex wrt stair}
    Let $(H,J,\overline{H},\overline{J}) \in \admstair{W,V}$ and consider $\pi_W(K) = (H,J) \in \admissible{W}$ and $\pi_{W,V}(K) = (\overline{H},\overline{J}) \in \stair{W,V}$. In \cref{sec:Floer homology} we defined $\homology{}{S^1}{}{F}{C}{}{}(W,H,J)$, the Floer chain complex of $W$ with respect to the auxiliary data $(H,J)$, for every $(H,J) \in \admissible{W}$. Despite the fact that $(\overline{H}, \overline{J})$ is not an element of $\admissible{W}$, the Floer Chain complex $\homology{}{S^1}{}{F}{C}{}{}(W,\overline{H}, \overline{J})$ of $W$ with respect to the auxiliary data $(\overline{H}, \overline{J})$ is well-defined. More precisely, it is possible to replicate the results of \cref{sec:Floer homology} but with the category $\stair{W,V}$ instead of $\admissible{W}$. Then, we can define a functor
    \begin{IEEEeqnarray*}{rrCl}
        \homology{\mathrm{I-V}}{S^1}{W}{F}{C}{}{} \colon & \stair{W,V}             & \longrightarrow & \comp \\
                                                         & (\overline{H}, \overline{J}) & \longmapsto     & \homology{\mathrm{I-V}}{S^1}{W}{F}{C}{}{}(\overline{H},\overline{J}) \coloneqq \homology{}{S^1}{}{F}{C}{}{}(W,\overline{H}, \overline{J}).
    \end{IEEEeqnarray*}
    For every $(H^+, J^+, H^-, J^-) \in \admstair{W,V}$, we have that $H^+ \leq {H}^-$, and therefore we can define a continuation map $\phi^{-,+} \colon \homology{}{S^1}{}{F}{C}{}{}(W,H^+,J^+) \longrightarrow \homology{}{S^1}{}{F}{C}{}{}(W,H^-,J^-)$ which is given by counting solutions of the Floer equation with respect to $H \in \mathcal{H}(H^+,H^-)$ and $J \in \mathcal{J}(J^+,J^-)$. These continuation maps assemble into a natural transformation 
    \begin{IEEEeqnarray*}{c+x*}
        \phi \colon \homology{}{S^1}{W}{F}{C}{}{} \circ \pi_W^{} \longrightarrow \homology{\mathrm{I-V}}{S^1}{W}{F}{C}{}{} \circ \pi_{W,V}^{}.
    \end{IEEEeqnarray*}
\end{remark}

\begin{definition}
    \label{def:subcomplex}
    We define a functor $\homology{\mathrm{III,IV,V}}{S^1}{W}{F}{C}{}{} \colon \stair{W,V}^{} \longrightarrow \comp$ as follows. If $(\overline{H},\overline{J}) \in \stair{W,V}$, then the module $\homology{\mathrm{III,IV,V}}{S^1}{W}{F}{C}{}{}(\overline{H}, \overline{J}) \coloneqq \homology{\mathrm{III,IV,V}}{S^1}{}{F}{C}{}{}(W,\overline{H},\overline{J})$ is the submodule of $\homology{\mathrm{I-V}}{S^1}{}{F}{C}{}{}(W,\overline{H},\overline{J})$ which is generated by (equivalence classes of) $1$-periodic orbits $[z, \gamma]$ of $\overline{H}$ such that $\img \gamma$ is in region $\rmn{3}$, $\rmn{4}$ or $\rmn{5}$. The maps
    \begin{IEEEeqnarray*}{rrCl}
        \del \colon       & \homology{\mathrm{III,IV,V}}{S^1}{}{F}{C}{}{}(W,\overline{H},\overline{J})     & \longrightarrow & \homology{\mathrm{III,IV,V}}{S^1}{}{F}{C}{}{}(W,\overline{H},\overline{J}),   \\
        U \colon          & \homology{\mathrm{III,IV,V}}{S^1}{}{F}{C}{}{}(W,\overline{H},\overline{J})     & \longrightarrow & \homology{\mathrm{III,IV,V}}{S^1}{}{F}{C}{}{}(W,\overline{H},\overline{J}),   \\
        \phi^{-,+} \colon & \homology{\mathrm{III,IV,V}}{S^1}{}{F}{C}{}{}(W,\overline{H}^+,\overline{J}^+) & \longrightarrow & \homology{\mathrm{III,IV,V}}{S^1}{}{F}{C}{}{}(W,\overline{H}^-,\overline{J}^-).
    \end{IEEEeqnarray*}
    are the restrictions (see \cref{lem:maps restrict to subcomplex}) of the maps
    \begin{IEEEeqnarray*}{rrCl}
        \del \colon       & \homology{\mathrm{I-V}}{S^1}{}{F}{C}{}{}(W,\overline{H},\overline{J})     & \longrightarrow & \homology{\mathrm{I-V}}{S^1}{}{F}{C}{}{}(W,\overline{H},\overline{J}),   \\
        U \colon          & \homology{\mathrm{I-V}}{S^1}{}{F}{C}{}{}(W,\overline{H},\overline{J})     & \longrightarrow & \homology{\mathrm{I-V}}{S^1}{}{F}{C}{}{}(W,\overline{H},\overline{J}),   \\
        \phi^{-,+} \colon & \homology{\mathrm{I-V}}{S^1}{}{F}{C}{}{}(W,\overline{H}^+,\overline{J}^+) & \longrightarrow & \homology{\mathrm{I-V}}{S^1}{}{F}{C}{}{}(W,\overline{H}^-,\overline{J}^-),
    \end{IEEEeqnarray*}
    This completes the definition of $\homology{\mathrm{III,IV,V}}{S^1}{W}{F}{C}{}{}$. Since $\homology{\mathrm{III,IV,V}}{S^1}{}{F}{C}{}{}(W,\overline{H},\overline{J})$ is a subcomplex of $\homology{\mathrm{I-V}}{S^1}{}{F}{C}{}{}(W,\overline{H},\overline{J})$, we have an inclusion natural transformation $\iota \colon \homology{\mathrm{III,IV,V}}{S^1}{W}{F}{C}{}{} \longrightarrow \homology{\mathrm{I-V}}{S^1}{W}{F}{C}{}{}$.
\end{definition}

\begin{lemma}
    \label{lem:maps restrict to subcomplex}
    In \cref{def:subcomplex}, the maps $\del, U$ and $\phi^{-,+}$ restrict to maps on $\homology{\mathrm{III,IV,V}}{S^1}{W}{F}{C}{}{}$.
\end{lemma}
\begin{proof}
    To show that $U$ restricts to a map on $\homology{\mathrm{III,IV,V}}{S^1}{W}{F}{C}{}{}$, we simply note that by definition $U$ affects only $z$ and not $\gamma$.

    We show that $\del$ restricts to a map on $\homology{\mathrm{III,IV,V}}{S^1}{W}{F}{C}{}{}$. For this, let $[z^{\pm}, \gamma^{\pm}] \in \mathcal{P}(\overline{H})$ be such that $\img \gamma^+$ is in region $\rmn{3}$, $\rmn{4}$ or $\rmn{5}$ and assume that there exists a Floer trajectory from $[z^+, \gamma^+]$ to $[z^-, \gamma^-]$ with respect to $(\overline{H}, \overline{J})$. We need to show that $\img \gamma^-$ is in region $\rmn{3}$, $\rmn{4}$ or $\rmn{5}$. Assume by contradiction that $\img \gamma^-$ is in region $\rmn{1}$ or $\rmn{2}$. In the case where $\img \gamma^+$ is in region $\rmn{4}$ or $\rmn{5}$, the computation
    \begin{IEEEeqnarray*}{rCls+x*}
        0
        & <    & \mathcal{A}_{\overline{H}}(z^-,\gamma^-) & \quad [\text{by \cref{lem:action stair}}]    \\
        & \leq & \mathcal{A}_{\overline{H}}(z^+,\gamma^+) & \quad [\text{by \cref{lem:action energy for floer trajectories}}] \\
        & <    & 0                                        & \quad [\text{by \cref{lem:action stair}}]
    \end{IEEEeqnarray*}
    gives a contradiction. It remains to derive a contradiction in the case where $\img \gamma^+$ is in region $\rmn{3}$. By \cref{cor:hamiltonian orbits are reeb orbits}, $\gamma^+$ is (approximately) of the form $\gamma^+(t) = (r^+, \rho^+(t))$ for some Reeb orbit $\rho^+$ in $(\del V, \lambda_V|_{\del V})$. The ``no escape'' lemma (\cref{lem:no escape}) implies that the Floer trajectory is inside $\hat{\varphi}(V \union [0, r^+] \times \del V)$, while the ``asymptotic behaviour'' lemma (\cref{lem:asymptotic behaviour}) implies that the Floer trajectory must leave $\hat{\varphi}(V \union [0, r^+] \times \del V)$. This completes the proof that $\del$ restricts to a map on $\homology{\mathrm{III,IV,V}}{S^1}{W}{F}{C}{}{}$.

    To show that $\phi^{-,+}$ restricts to a map on $\homology{\mathrm{III,IV,V}}{S^1}{W}{F}{C}{}{}$, we would use a proof analogous to that of $\del$. The key difference is that now the Floer trajectory would be defined with respect to homotopies of Hamiltonians and almost complex structures. This does not affect the proof because \cref{lem:action energy for floer trajectories,lem:asymptotic behaviour,lem:no escape} also apply to homotopies.
\end{proof}

\begin{definition}
    \label{def:quotient complex}
    Define a functor $\homology{\mathrm{I,II}}{S^1}{W}{F}{C}{}{} \colon \stair{W,V}^{} \longrightarrow \comp$ as follows. For $(\overline{H},\overline{J}) \in \stair{W,V}$, the module $\homology{\mathrm{I,II}}{S^1}{W}{F}{C}{}{}(\overline{H}, \overline{J}) \coloneqq \homology{\mathrm{I,II}}{S^1}{}{F}{C}{}{}(W,\overline{H}, \overline{J})$ is given by the quotient
    \begin{IEEEeqnarray*}{rCls+x*}
        \homology{\mathrm{I,II}}{S^1}{}{F}{C}{}{}(W,\overline{H},\overline{J}) & \coloneqq & \homology{\mathrm{I-V}}{S^1}{}{F}{C}{}{}(W,\overline{H}, \overline{J}) / \homology{\mathrm{III,IV,V}}{S^1}{}{F}{C}{}{}(W,\overline{H},\overline{J}).
    \end{IEEEeqnarray*}
    For $(\overline{H}^+,\overline{J}^+) \leq (\overline{H}^{-},\overline{J}^-) \in \stair{W,V}$, the continuation map $\phi^{-,+} \colon \homology{\mathrm{I,II}}{S^1}{}{F}{C}{}{}(W,\overline{H}^+,\overline{J}^+) \longrightarrow \homology{\mathrm{I,II}}{S^1}{}{F}{C}{}{}(W,\overline{H}^-,\overline{J}^-)$ is the unique map such that the diagram
    \begin{IEEEeqnarray*}{c+x*}
        \begin{tikzcd}
            \homology{\mathrm{III,IV,V}}{S^1}{}{F}{C}{}{}(W,\overline{H}^+,\overline{J}^+) \ar[r, hookrightarrow, "\iota^{+}"] \ar[d, swap, "\phi^{-,+}"] & \homology{\mathrm{I-V}}{S^1}{}{F}{C}{}{}(W,\overline{H}^+,\overline{J}^+) \ar[d, "\phi^{-,+}"] \ar[r, two heads, "\pi^{+}"] & \homology{\mathrm{I,II}}{S^1}{}{F}{C}{}{}(W,\overline{H}^+,\overline{J}^+) \ar[d, dashed, "\exists ! \phi^{-,+}"]\\
            \homology{\mathrm{III,IV,V}}{S^1}{}{F}{C}{}{}(W,\overline{H}^-,\overline{J}^-) \ar[r, hookrightarrow, swap, "\iota^{-}"]                      & \homology{\mathrm{I-V}}{S^1}{}{F}{C}{}{}(W,\overline{H}^-,\overline{J}^-) \ar[r, two heads, swap, "\pi^{-}"]                & \homology{\mathrm{I,II}}{S^1}{}{F}{C}{}{}(W,\overline{H}^-,\overline{J}^-)
        \end{tikzcd}
    \end{IEEEeqnarray*}
    commutes. There is a projection natural transformation $\pi \colon \homology{\mathrm{I-V}}{S^1}{W}{F}{C}{}{} \longrightarrow \homology{\mathrm{I,II}}{S^1}{W}{F}{C}{}{}$.
\end{definition}

\begin{definition}
    \label{def:v with respect to stair nt}
    We define a natural transformation $\eta \colon \homology{}{S^1}{V}{F}{C}{}{} \circ \pi^{\mathcal{H} \times \mathcal{J}}_{W,V} \longrightarrow \homology{\mathrm{I,II}}{S^1}{W}{F}{C}{}{}$ as follows. For $(\overline{H},\overline{J}) \in \stair{W,V}$, the map $\eta^{\overline{H},\overline{J}} \colon \homology{}{S^1}{}{F}{C}{}{}(V,\overline{H}_V, \overline{J}_V) \longrightarrow \homology{\mathrm{I,II}}{S^1}{}{F}{C}{}{}(W,\overline{H}, \overline{J})$ is given by $\eta^{\overline{H},\overline{J}}([z,\gamma]) \coloneqq [z, \hat{\varphi} \circ \gamma]$.
\end{definition}

\begin{lemma}
    \cref{def:v with respect to stair nt} is well posed, i.e.:
    \begin{enumerate}
        \item \label{lem:v with respect to stair nt 1} $\eta^{\overline{H},\overline{J}}$ is well-defined and it is a morphism of filtered modules.
        \item \label{lem:v with respect to stair nt 2} $\eta^{\overline{H},\overline{J}}$ commutes with the $U$ map.
        \item \label{lem:v with respect to stair nt 3} $\eta^{\overline{H},\overline{J}}$ is a chain map.
        \item \label{lem:v with respect to stair nt 4} The maps $\eta^{\overline{H},\overline{J}}$ assemble into a natural transformation.
    \end{enumerate}
\end{lemma}
\begin{proof}
    \ref{lem:v with respect to stair nt 1}: Since $\hat{\varphi}$ is a Liouville embedding, if $[z,\gamma] \in \mathcal{P}(\overline{H}_V)$ then $[z,\hat{\varphi} \circ \gamma] \in \mathcal{P}(\overline{H})$ and $\mathcal{A}_{\overline{H}}(z,\hat{\varphi} \circ \gamma) = \mathcal{A}_{\overline{H}_V}(z,\gamma)$.

    \ref{lem:v with respect to stair nt 2}: We need to show that $U^{}_W \circ \eta^{\overline{H},\overline{J}}([z,\gamma]) = \eta^{\overline{H},\overline{J}} \circ U ^{}_V ([z,\gamma])$, for $[z,\gamma] \in \mathcal{P}(\overline{H}_V)$. If $\morse(z) = 0$, then both sides of the equation are $0$. If $\morse(z) > 0$, then
    \begin{IEEEeqnarray*}{rCls+x*}
        U^{}_W \circ \eta^{\overline{H},\overline{J}}([z,\gamma])
        & = & U^{}_W ([z,\hat{\varphi} \circ \gamma])           & \quad [\text{by definition of $\eta$}] \\
        & = & [\shf(z),\hat{\varphi} \circ \gamma]                          & \quad [\text{by definition of $U$}]    \\
        & = & \eta^{\overline{H},\overline{J}} [\shf(z),\gamma]                         & \quad [\text{by definition of $\eta$}] \\
        & = & \eta^{\overline{H},\overline{J}} \circ U ^{}_V ([z,\gamma]) & \quad [\text{by definition of $U$}].
    \end{IEEEeqnarray*}

    \ref{lem:v with respect to stair nt 3}: We need to show that $\eta^{\overline{H},\overline{J}} \circ \del ^{}_V([z^+,\gamma^+]) = \del ^{}_W \circ \eta^{\overline{H},\overline{J}}([z^+,\gamma^+])$, for every $[z^+,\gamma^+] \in \mathcal{P}(\overline{H}_V)$. By the ``no escape'' lemma (\cref{lem:no escape}), if $[z^-,\gamma^-] \in \mathcal{P}(\overline{H}_V)$ then the map
    \begin{IEEEeqnarray*}{rrCl}
        & \mathcal{M}_{\vphantom{0}}(\overline{H}_V,\overline{J}_V,[z^+,\gamma^+],[z^-,\gamma^-]) & \longrightarrow & \mathcal{M}_{\vphantom{0}}(\overline{H},\overline{J},[z^+,\hat{\varphi} \circ \gamma^+],[z^-,\hat{\varphi} \circ \gamma^-]) \\
        & [w,u]                                                                    & \longmapsto     & [w,\hat{\varphi} \circ u]
    \end{IEEEeqnarray*}
    is an orientation preserving diffeomorphism. Then, we compute
    \begin{IEEEeqnarray*}{rCls+x*}
        \IEEEeqnarraymulticol{3}{l}{\eta^{\overline{H},\overline{J}} \circ \del ^{}_V([z^+,\gamma^+])}\\
        \quad & = & \sum_{[z^-,\gamma^-]   \in \mathcal{P}(\overline{H}_V)                              } \# \mathcal{M}_{\vphantom{0}}(\overline{H}_V, \overline{J}_V, [z^+,\gamma^+]                    , [z^-,\gamma^-]                    ) \cdot \eta^{\overline{H},\overline{J}} ([z^-,\gamma^-]) \\
        \quad & = & \sum_{[z^-,\gamma^-]   \in \mathcal{P}(\overline{H}_V)                              } \# \mathcal{M}_{\vphantom{0}}(\overline{H}_V, \overline{J}_V, [z^+,\gamma^+]                    , [z^-,\gamma^-]                    ) \cdot [z^-,\hat{\varphi} \circ \gamma^-] \\
        \quad & = & \sum_{[z^-,\gamma^-]   \in \mathcal{P}(\overline{H}_V)                              } \# \mathcal{M}_{\vphantom{0}}(\overline{H}  , \overline{J}  , [z^+,\hat{\varphi} \circ \gamma^+], [z^-,\hat{\varphi} \circ \gamma^-]) \cdot [z^-,\hat{\varphi} \circ \gamma^-] \\
        \quad & = & \sum_{[z^-,\gamma^-_W] \in \mathcal{P}^{\mathrm{I,II}}(\overline{H})} \# \mathcal{M}_{\vphantom{0}}(\overline{H}  , \overline{J}  , [z^-,\gamma^-_W]                  , [z^+,\gamma^+_W])                   \cdot [z^-,\gamma^-_W] \\
        \quad & = & \sum_{[z^-,\gamma^-_W] \in \mathcal{P}(\overline{H})                                } \# \mathcal{M}_{\vphantom{0}}(\overline{H}  , \overline{J}  , [z^-,\gamma^-_W]                  , [z^+,\gamma^+_W])                   \cdot [z^-,\gamma^-_W] \\
        \quad & = & \del ^{}_W ([z^+,\hat{\varphi} \circ \gamma^+]) \\
        \quad & = & \del ^{}_W \circ \eta^{\overline{H},\overline{J}}([z^+,\gamma^+]).
    \end{IEEEeqnarray*}
    In this computation, in the third equality we used the orientation preserving diffeomorphism defined above, in the fourth equality we performed the variable change $[z^-,\gamma^-_W] \coloneqq [z^-,\hat{\varphi} \circ \gamma^-] \in \mathcal{P}^{\mathrm{I,II}}(\overline{H})$ and in the fifth equality we used the fact that if $[z^-,\gamma^-_W] \in \mathcal{P}^{\mathrm{III,IV,V}}(\overline{H})$ then $[z^-,\gamma^-_W] = 0$ as an element of $\homology{\mathrm{I,II}}{S^1}{}{F}{C}{}{}(W,\overline{H},\overline{J})$.

    \ref{lem:v with respect to stair nt 4}: This proof is analogous to that of \ref{lem:v with respect to stair nt 3}.
\end{proof}

\begin{proposition}
    The map $\eta \colon \homology{}{S^1}{V}{F}{C}{}{} \circ \pi^{\mathcal{H} \times \mathcal{J}}_{W,V} \longrightarrow \homology{\mathrm{I,II}}{S^1}{W}{F}{C}{}{}$ is a natural isomorphism.
\end{proposition}
\begin{proof}
    It suffices to show that $\eta^{\overline{H},\overline{J}} \colon \homology{}{S^1}{}{F}{C}{}{}(V,\overline{H}_V,\overline{J}_V) \longrightarrow \homology{\mathrm{I,II}}{S^1}{}{F}{C}{}{}(W,\overline{H},\overline{J})$ admits an inverse as a map of $\Q$-modules. Define $\nu^{\overline{H},\overline{J}} \colon \homology{\mathrm{I-V}}{S^1}{}{F}{C}{}{}(W,\overline{H},\overline{J}) \longrightarrow \homology{}{S^1}{}{F}{C}{}{}(V,\overline{H}_V,\overline{J}_V)$ by%
    \begin{IEEEeqnarray*}{c+x*}
        \nu^{\overline{H},\overline{J}}([z,\gamma]) =
        \begin{cases}
            [z,\hat{\varphi}^{-1} \circ \gamma] & \text{if } [z,\gamma] \in \mathcal{P}^{\mathrm{I,II}}(\overline{H}), \\
            0                                   & \text{if } [z,\gamma] \in \mathcal{P}^{\mathrm{III,IV,V}}(\overline{H}).
        \end{cases}
    \end{IEEEeqnarray*}
    Then, by the universal property of the quotient of $\Q$-modules, $\nu^{\overline{H},\overline{J}}$ descends to a map $\nu^{\overline{H},\overline{J}} \colon \homology{\mathrm{I,II}}{S^1}{}{F}{C}{}{}(W,\overline{H},\overline{J}) \longrightarrow \homology{}{S^1}{}{F}{C}{}{}(V,\overline{H}_V,\overline{J}_V)$, which is the inverse of $\eta^{\overline{H},\overline{J}}$.
\end{proof}

\begin{definition}
    \label{def:viterbo transfer map}
    The \textbf{Viterbo transfer map}, $\varphi_! \colon \homology{}{S^1}{}{S}{H}{}{}(W, \lambda_W) \longrightarrow \homology{}{S^1}{}{S}{H}{}{}(V, \lambda_V)$, is given as follows. Consider the following diagram in the category of functors from $\admstair{W,V}$ to $\comp$:
    \begin{IEEEeqnarray}{c+x*}
        \plabel{eq:viterbo transfer map diagram}
        \begin{tikzcd}
            \homology{\mathrm{III,IV,V}}{S^1}{W}{F}{C}{}{} \circ \pi_{W,V}^{} \ar[r, hook, "\iota \circ \pi_{W,V}"] &
            \homology{\mathrm{I-V}}{S^1}{W}{F}{C}{}{} \circ \pi_{W,V}^{} \ar[r, hook, "\pi \circ \pi_{W,V}"]        &
            \homology{\mathrm{I,II}}{S^1}{W}{F}{C}{}{} \circ \pi_{W,V}^{} \\                                        &
            \homology{}{S^1}{W}{F}{C}{}{} \circ \pi_{W}^{} \ar[u, "\phi"] \ar[r, dashed, swap, "\exists ! \varphi"] &
            \homology{}{S^1}{V}{F}{C}{}{} \circ \pi_{V}^{} \ar[u, swap, two heads, hook, "\eta \circ \pi_{W,V}"]
        \end{tikzcd}
    \end{IEEEeqnarray}
    Passing to homology, we get a natural transformation $H \varphi \colon \homology{}{S^1}{W}{F}{H}{}{} \circ \pi_{W}^{} \longrightarrow \homology{}{S^1}{V}{F}{H}{}{} \circ \pi_{V}^{}$. Then, $\varphi_!$ is the unique map such that the following diagram commutes:
    \begin{IEEEeqnarray}{c+x*}
        \plabel{eq:viterbo transfer map}
        \begin{tikzcd}
            \homology{}{S^1}{W}{F}{H}{}{} \circ \pi_W^{} \ar[d, "H \varphi"] \ar[r] & \colim \homology{}{S^1}{W}{F}{H}{}{} \circ \pi_W^{} \ar[r, equal] \ar[d, dashed, "\exists ! \varphi_! = \colim H \varphi"] & \homology{}{S^1}{}{S}{H}{}{}(W,\lambda_W) \ar[d, dashed, "\exists ! \varphi_!"] \\
            \homology{}{S^1}{V}{F}{H}{}{} \circ \pi_V^{} \ar[r] & \colim \homology{}{S^1}{V}{F}{H}{}{} \circ \pi_V^{} \ar[r, equal] & \homology{}{S^1}{}{S}{H}{}{}(V,\lambda_V)
        \end{tikzcd}
    \end{IEEEeqnarray}

    We define the \textbf{Viterbo transfer map} on positive $S^1$-equivariant symplectic homology by declaring it to be the unique map such that the following diagram commutes:
    \begin{IEEEeqnarray*}{c+x*}
        \begin{tikzcd}
            \homology{}{S^1}{}{S}{H}{\varepsilon}{}(W,\lambda_W) \ar[r] \ar[d, swap, "\varphi^\varepsilon_!"] & \homology{}{S^1}{}{S}{H}{}{}(W,\lambda_W) \ar[r] \ar[d, "\varphi_!"] & \homology{}{S^1}{}{S}{H}{+}{}(W,\lambda_W) \ar[d, dashed, "\exists ! \varphi^+_!"] \\
            \homology{}{S^1}{}{S}{H}{\varepsilon}{}(W,\lambda_W) \ar[r] & \homology{}{S^1}{}{S}{H}{}{}(W,\lambda_W) \ar[r] & \homology{}{S^1}{}{S}{H}{+}{}(W,\lambda_W)
        \end{tikzcd}
    \end{IEEEeqnarray*}
\end{definition}

\begin{remark}
    \label{rmk:viterbo transfer map def}
    We have the following observations about \cref{def:viterbo transfer map}.
    \begin{enumerate}
        \item In diagram \eqref{eq:viterbo transfer map}, we view $\colim \homology{}{S^1}{W}{F}{H}{}{} \circ \pi_W$ and $\colim \homology{}{S^1}{V}{F}{H}{}{} \circ \pi_V$ as constant functors, and we view $\varphi_! \colon \colim \homology{}{S^1}{W}{F}{H}{}{} \circ \pi_W \longrightarrow \colim \homology{}{S^1}{V}{F}{H}{}{} \circ \pi_V$ as a constant natural transformation, which is just a map. Existence and uniqueness of $\varphi$ comes from the universal property of colimits.
        \item Since $\pi_W ( \admstair{W,V} )$ is a cofinal subset of $\admissible{W}$, we have $\homology{}{S^1}{}{S}{H}{}{}(W,\lambda_W) = \colim \homology{}{S^1}{W}{F}{H}{}{} = \colim \homology{}{S^1}{W}{F}{H}{}{} \circ \pi_W$, and analogously for $V$.
        \item We are also using the fact that
            \begin{IEEEeqnarray*}{rCls+x*}
                \homology{}{S^1}{}{S}{H}{+}{}(W,\lambda_W)
                & = & \homology{}{S^1}{}{S}{H}{}{}(W,\lambda_W) / \homology{}{S^1}{}{S}{H}{\varepsilon_W}{} (W,\lambda_W) \\
                & = & \homology{}{S^1}{}{S}{H}{}{}(W,\lambda_W) / \homology{}{S^1}{}{S}{H}{\varepsilon}{}(W,\lambda_W).
            \end{IEEEeqnarray*}
            This is true because $\homology{}{S^1}{}{S}{H}{}{}$ is obtained as a direct limit of Floer homologies for increasing Hamiltonians, and for $(H,J) \in \admissible{W}$ with $H$ big enough we have that $H$ restricted to the interior of $W$ takes values in $(-\varepsilon,0) \subset (-\varepsilon_W,0)$ (and analogously for $V$).
    \end{enumerate}
\end{remark}

Let $\liouvle$ be the category whose objects are nondegenerate Liouville domains and whose morphisms are $0$-codimensional Liouville embeddings which are either strict or diffeomorphisms.

\begin{theorem}[{\cite[Theorem 3.1.16]{guttMinimalNumberPeriodic2014}}]
    \label{thm:sh is functor not generalized}
    The following are contravariant functors:
    \begin{IEEEeqnarray*}{rrClCrrCl}
        \homology{}{S^1}{}{S}{H}{}{} \colon & \liouvle           & \longrightarrow & \modl                                      & \qquad & \homology{}{S^1}{}{S}{H}{+}{} \colon & \liouvle           & \longrightarrow & \modl                                  \\
                                            & (V,\lambda_V)      & \longmapsto     & \homology{}{S^1}{}{S}{H}{}{}(V,\lambda_V)  & \qquad &                                      & (V,\lambda_V)      & \longmapsto     & \homology{}{S^1}{}{S}{H}{+}{}(V,\lambda_V) \\
                                            & \varphi \downarrow & \longmapsto     & \uparrow \varphi_!                         & \qquad &                                      & \varphi \downarrow & \longmapsto     & \uparrow \varphi_!^+                   \\
                                            & (W,\lambda_W)      & \longmapsto     & \homology{}{S^1}{}{S}{H}{}{}(W,\lambda_W), & \qquad &                                      & (W,\lambda_W)      & \longmapsto     & \homology{}{S^1}{}{S}{H}{+}{}(W,\lambda_W).
    \end{IEEEeqnarray*}
\end{theorem}

\section{Viterbo transfer map of a generalized Liouville embedding}
\label{sec:viterbo transfer map of exact symplectic embedding}

We now define the Viterbo transfer map in the case where $\varphi \colon (V,\lambda_V) \longrightarrow (W,\lambda_W)$ is a generalized Liouville embedding, i.e. $\varphi^* \edv \lambda_W = \edv \lambda_V$ and $(\varphi^* \lambda_W - \lambda_V)|_{\partial V}$ is exact.

\begin{lemma}[{\cite[Lemma 7.5]{guttSymplecticCapacitiesPositive2018}}]
    \label{lem:exists deformed form}
    If $\phi \colon (V,\lambda_V) \longrightarrow (W, \lambda_W)$ is a $0$-codimensional strict generalized Liouville embedding, then there exists a $1$-form $\lambda'_W$ on $W$ such that $\edv \lambda'_W = \edv \lambda_W^{}$, $\lambda'_W = \lambda_W^{}$ near $\partial W$ and $\phi^* \lambda'_W = \lambda_V^{}$.
\end{lemma}

\begin{lemma}
    \phantomsection\label{lem:sh indep of potential}
    Let $(X,\lambda_X)$ and $(Y,\lambda_Y)$ be nondegenerate Liouville domains and assume that $\phi \colon (X,\lambda_X) \longrightarrow (Y, \lambda_Y)$ is a $0$-codimensional strict Liouville embedding. Suppose that $\lambda'_X \in \Omega^1(X)$ and $\lambda'_Y \in \Omega^1(Y)$ are $1$-forms such that
    \begin{IEEEeqnarray*}{rClCrCl}
        \edv \lambda'_X   & = & \edv \lambda_X^{}, & \quad & \lambda'_X & = & \lambda_X^{} \text{ near } \partial X, \\
        \edv \lambda'_Y   & = & \edv \lambda_Y^{}, & \quad & \lambda'_Y & = & \lambda_Y^{} \text{ near } \partial Y, \\
        \phi^* \lambda'_Y & = & \lambda'_X.
    \end{IEEEeqnarray*}
    Then,
    \begin{IEEEeqnarray*}{rClCl}
        \homology{}{S^1}{}{S}{H}{}{}(X,\lambda_X)  & = & \homology{}{S^1}{}{S}{H}{}{}(X,\lambda'_X), \\
        \homology{}{S^1}{}{S}{H}{+}{}(X,\lambda_X) & = & \homology{}{S^1}{}{S}{H}{+}{}(X,\lambda'_X),
    \end{IEEEeqnarray*}
    and the diagrams
    \begin{IEEEeqnarray}{c+x*}
        \plabel{eq:viterbo transfer map indep potential}
        \begin{tikzcd}
            \homology{}{S^1}{}{S}{H}{}{}(Y,\lambda_Y) \ar[r, equal] \ar[d, swap, "\phi_!"] & \homology{}{S^1}{}{S}{H}{}{}(Y,\lambda'_Y) \ar[d, "\phi'_!"] \\
            \homology{}{S^1}{}{S}{H}{}{}(X,\lambda_X) \ar[r, equal]                        & \homology{}{S^1}{}{S}{H}{}{}(X,\lambda'_X)
        \end{tikzcd} \quad
        \begin{tikzcd}
            \homology{}{S^1}{}{S}{H}{+}{}(Y,\lambda_Y) \ar[r, equal] \ar[d, swap, "\phi_!^+"] & \homology{}{S^1}{}{S}{H}{+}{}(Y,\lambda'_Y) \ar[d, "{\phi'}_!^+"] \\
            \homology{}{S^1}{}{S}{H}{+}{}(X,\lambda_X) \ar[r, equal]                          & \homology{}{S^1}{}{S}{H}{+}{}(X,\lambda'_X)
        \end{tikzcd}
    \end{IEEEeqnarray}
    commute.
\end{lemma}
\begin{proof}
    We note that the following concepts only depend on $\edv \lambda_X$ and on $\lambda_X$ near $\del X$: the set of admissible Hamiltonians and admissible almost complex structures, the Hamiltonian vector field, action, the module which underlies the Floer complex (by all the previous statements), the Floer equation and the notion of Floer trajectories (also by the previous statements), the $U$ map, the differential and the continuation maps. All the statements follow immediately from the definitions given in \cref{sec:Floer homology}, except the fact that the action actually only depends on $\edv \lambda_X$ and on $\lambda_X|_{\partial X}$. To prove this, it is enough to show that 
    \begin{IEEEeqnarray}{c+x*}
        \phantomsection\label{eq:action indep form}
        \int_{S^1}^{} \gamma^* (\hat{\lambda}_X^{} - \hat{\lambda}'_X) = 0.
    \end{IEEEeqnarray}
    Since $\hat{\lambda}_X^{} - \hat{\lambda}'_X$ is closed, it defines a cohomology class $[\hat{\lambda}_X^{} - \hat{\lambda}'_X] \in H^1_{\mathrm{dR}}(\hat{X})$. The orbit $\gamma$ also defines a homology class $[\gamma] \coloneqq \gamma_* [S^1] \in H_1(\hat{X};\Z)$. Equation \eqref{eq:action indep form} can be restated as
    \begin{IEEEeqnarray}{c+x*}
        \phantomsection\label{eq:action indep form topology}
        [\hat{\lambda}_X^{} - \hat{\lambda}'_X]([\gamma]) = 0.
    \end{IEEEeqnarray}
    If $\gamma$ is contractible, then Equation \eqref{eq:action indep form topology} holds. If $\gamma$ is noncontractible, $\gamma$ must have an associated Reeb orbit $\rho \in C^{\infty}(S^1, \partial X)$. Denote by $\iota \colon \partial X \longrightarrow \hat{X}$ the inclusion.
    \begin{IEEEeqnarray*}{rCls+x*}
        [\hat{\lambda}_X^{} - \hat{\lambda}'_X]([\gamma]) 
        & = & [\hat{\lambda}_X^{} - \hat{\lambda}'_X](\iota_* [\rho])  & \quad [\text{since $\gamma$ and $\iota \circ \rho$ are homotopic}] \\
        & = & (\iota^*[\hat{\lambda}_X^{} - \hat{\lambda}'_X])([\rho]) & \quad [\text{by definition of pullback}] \\
        & = & 0                                                     & \quad [\text{since $\lambda'_X = \lambda_X^{}$ near $\partial X$}].
    \end{IEEEeqnarray*}
    Since the functors and natural transformations in diagram \eqref{eq:viterbo transfer map diagram} only depend on $\edv \lambda_X, \edv \lambda_Y$ and on $\lambda_X, \lambda_Y$ near the boundaries, the diagrams \eqref{eq:viterbo transfer map indep potential} commute.
\end{proof}

\begin{definition}[{\cite[Definition 7.6]{guttSymplecticCapacitiesPositive2018}}]
    \phantomsection\label{def:viterbo transfer generalized}
    If $\varphi \colon (V,\lambda_V) \longrightarrow (W,\lambda_W)$ is a strict generalized Liouville embedding of codimension $0$, then the \textbf{Viterbo transfer map} of $\varphi$ is defined as follows. Choose $\lambda'_W \in \Omega^1(W)$ as in \cref{lem:exists deformed form}. Denote by $\varphi' \colon (V,\lambda_V) \longrightarrow (W,\lambda'_W)$ the Liouville embedding which as a map of sets coincides with $\varphi$. Then, define
    \begin{IEEEeqnarray*}{rRCRCl}
        \varphi_! \colon   & \homology{}{S^1}{}{S}{H}{}{}(W,\lambda_W)  & = & \homology{}{S^1}{}{S}{H}{}{}(W,\lambda'_W)  & \xrightarrow{\varphi'_!} & \homology{}{S^1}{}{S}{H}{}{}(V,\lambda_V), \\
        \varphi^+_! \colon & \homology{}{S^1}{}{S}{H}{+}{}(W,\lambda_W) & = & \homology{}{S^1}{}{S}{H}{+}{}(W,\lambda'_W) & \xrightarrow{\varphi'_!} & \homology{}{S^1}{}{S}{H}{+}{}(V,\lambda_V),
    \end{IEEEeqnarray*}
    where the equality was explained in \cref{lem:sh indep of potential} and the arrows are the Viterbo transfer maps of a Liouville embedding as in \cref{def:viterbo transfer map}.
\end{definition}

\begin{lemma}
    In \cref{def:viterbo transfer generalized}, $\varphi_!$ and $\varphi_!^+$ are independent of the choice of $\lambda'_W$.
\end{lemma}
\begin{proof}
    Let $\lambda'_W$ and $\lambda''_W$ be $1$-forms as in \cref{lem:exists deformed form}, and denote the corresponding Liouville embeddings by $\varphi' \colon (W,\lambda'_W) \longrightarrow (V,\lambda_V)$ and $\varphi'' \colon (W,\lambda''_W) \longrightarrow (V,\lambda_V)$ (note that as set theoretic maps, $\varphi' = \varphi'' = \varphi$). Then, by \cref{lem:sh indep of potential}, the following diagram commutes:
    \begin{IEEEeqnarray*}{c+x*}
        \begin{tikzcd}
            \homology{}{S^1}{}{S}{H}{}{}(W,\lambda_W) \ar[r, equals] \ar[d, equals] & \homology{}{S^1}{}{S}{H}{}{}(W,\lambda'_W) \ar[d, equals] \ar[r, "\varphi'_!"] & \homology{}{S^1}{}{S}{H}{}{}(V,\lambda_V) \ar[d, equals] \\
            \homology{}{S^1}{}{S}{H}{}{}(W,\lambda_W) \ar[r, equals] & \homology{}{S^1}{}{S}{H}{}{}(W,\lambda''_W) \ar[r, "\varphi''_!"] & \homology{}{S^1}{}{S}{H}{}{}(V,\lambda_V)
        \end{tikzcd}
    \end{IEEEeqnarray*}
    In this diagram, the top arrow is the Viterbo transfer map defined with respect to $\lambda'_W$ and the bottom arrow is the Viterbo transfer map defined with respect to $\lambda''_W$.
\end{proof}

Let $\liouvndg$ be the ``category'' whose objects are nondegenerate Liouville domains and whose morphisms are $0$-codimensional generalized Liouville embeddings which are either strict or diffeomorphisms. Strictly speaking, since composition of generalized Liouville embeddings is not in general a generalized Liouville embedding, this is not a category. However, $\liouvndg$ does fit into the notion of \textbf{categroid} (see \cref{def:categroid}), which is an object like a category with only partially defined compositions. One can then talk about functors between categroids.

\begin{theorem}
    The assignments
    \begin{IEEEeqnarray*}{rrClCrrCl}
        \homology{}{S^1}{}{S}{H}{}{} \colon & \liouvndg               & \longrightarrow & \modl         & \qquad & \homology{}{S^1}{}{S}{H}{+}{} \colon & \liouvndg               & \longrightarrow & \modl           \\
                  & (V,\lambda_V)      & \longmapsto     & \homology{}{S^1}{}{S}{H}{}{}(V,\lambda_V)    & \qquad &             & (V,\lambda_V)      & \longmapsto     & \homology{}{S^1}{}{S}{H}{+}{}(V,\lambda_V)    \\
                  & \varphi \downarrow & \longmapsto     & \uparrow \varphi_! & \qquad &             & \varphi \downarrow & \longmapsto     & \uparrow \varphi_!^+ \\
                  & (W,\lambda_W)      & \longmapsto     & \homology{}{S^1}{}{S}{H}{}{}(W,\lambda_W),   & \qquad &             & (W,\lambda_W)      & \longmapsto     & \homology{}{S^1}{}{S}{H}{+}{}(W,\lambda_W)
    \end{IEEEeqnarray*}
    are contravariant functors.
\end{theorem}
\begin{proof}
    We prove the result only for $\homology{}{S^1}{}{S}{H}{}{}$, since the proof for $\homology{}{S^1}{}{S}{H}{+}{}$ is analogous. It suffices to assume that $\varphi \colon (V, \lambda_V) \longrightarrow (W, \lambda_W)$ and $\psi \colon (W, \lambda_W) \longrightarrow (Z, \lambda_Z)$ are composable strict, generalized Liouville embeddings of codimension 0 and to prove that $(\psi \circ \varphi)_! = \varphi_! \circ \psi_!$. Here, ``composable'' means that the composition $\psi \circ \varphi$ is also a generalized Liouville embedding. We start by choosing
    \begin{IEEEeqnarray*}{rClCrClrCllCrCl}
        \lambda'_W  & \in & \Omega^1(W) & \quad\text{such that}\quad & \edv \lambda'_W  & = & \edv \lambda_W^{},\quad & \lambda'_W  & = & \lambda_W^{} & \text{ near } \partial W, & \quad\text{and}\quad & \varphi^* \lambda'_W & = & \lambda_V^{}, \\
        \lambda'_Z  & \in & \Omega^1(Z) & \quad\text{such that}\quad & \edv \lambda'_Z  & = & \edv \lambda_Z^{},\quad & \lambda'_Z  & = & \lambda_Z^{} & \text{ near } \partial Z, & \quad\text{and}\quad & \psi^* \lambda'_Z    & = & \lambda_W^{}, \\
        \lambda''_Z & \in & \Omega^1(Z) & \quad\text{such that}\quad & \edv \lambda''_Z & = & \edv \lambda'_Z,  \quad & \lambda''_Z & = & \lambda'_Z   & \text{ near } \partial Z, & \quad\text{and}\quad & \psi^* \lambda''_Z   & = & \lambda'_W.
    \end{IEEEeqnarray*}
    Therefore, we have Liouville embeddings
    \begin{IEEEeqnarray*}{rCrCl}
        \varphi' & \colon & (V,\lambda_V^{}) & \longrightarrow & (W, \lambda'_W), \\
        \psi'    & \colon & (W,\lambda_W^{}) & \longrightarrow & (Z, \lambda'_Z), \\
        \psi''   & \colon & (W,\lambda'_W)   & \longrightarrow & (Z, \lambda''_Z).
    \end{IEEEeqnarray*}
    We can define the Viterbo transfer maps
    \begin{IEEEeqnarray*}{rLCLCl}
        \varphi_!              \colon & \homology{}{S^1}{}{S}{H}{}{}(W,\lambda_W) & = & \homology{}{S^1}{}{S}{H}{}{}(W,\lambda'_W)  & \xrightarrow{\varphi'_!}                & \homology{}{S^1}{}{S}{H}{}{}(V,\lambda_V), \\
        \psi_!                 \colon & \homology{}{S^1}{}{S}{H}{}{}(Z,\lambda_Z) & = & \homology{}{S^1}{}{S}{H}{}{}(Z,\lambda'_Z)  & \xrightarrow{\psi'_!}                   & \homology{}{S^1}{}{S}{H}{}{}(W,\lambda_W), \\
        (\varphi \circ \psi)_! \colon & \homology{}{S^1}{}{S}{H}{}{}(Z,\lambda_Z) & = & \homology{}{S^1}{}{S}{H}{}{}(Z,\lambda''_Z) & \xrightarrow{(\psi'' \circ \varphi')_!} & \homology{}{S^1}{}{S}{H}{}{}(V,\lambda_V).
    \end{IEEEeqnarray*}
    Consider the following commutative diagram:
    \begin{IEEEeqnarray*}{c+x*}
        \begin{tikzcd}
            \homology{}{S^1}{}{S}{H}{}{}(Z,\lambda_Z) \ar[r, equals] \ar[dr, dashed, swap, "\psi_!"] \ar[drdr, dashed, bend right, swap, "(\psi \circ \varphi)_!"] & \homology{}{S^1}{}{S}{H}{}{}(Z,\lambda'_Z) \ar[d, "\psi'_!"] \ar[r, equals]                 & \homology{}{S^1}{}{S}{H}{}{}(Z,\lambda''_Z) \ar[d, "\psi''_!"] \ar[dd, bend left=90, "(\psi'' \circ \varphi')_!"] \\
                                                                                                                                                                   & \homology{}{S^1}{}{S}{H}{}{}(W,\lambda_W) \ar[r, equals] \ar[dr, swap, dashed, "\varphi_!"] & \homology{}{S^1}{}{S}{H}{}{}(W,\lambda'_W) \ar[d, "\varphi'_!"]                                                \\
                                                                                                                                                                   &                                                                                             & \homology{}{S^1}{}{S}{H}{}{}(V,\lambda_V)
        \end{tikzcd}
    \end{IEEEeqnarray*}
    Here, the two small triangles and the outside arrows commute by definition of the Viterbo transfer map of a generalized Liouville embedding, the square commutes by \cref{lem:sh indep of potential}, and $(\psi'' \circ \varphi')_! = \varphi'_! \circ \psi''_!$ by \cref{thm:sh is functor not generalized}. Therefore, $(\psi \circ \varphi)_! = \varphi_! \circ \psi_!$.
\end{proof}

\section{\texorpdfstring{$\delta$}{Delta} map}
\label{sec:delta map}

Let $(X,\lambda)$ be a nondegenerate Liouville domain. Our goal in this section is to define a map $\delta \colon \homology{}{S^1}{}{S}{H}{+}{}(X) \longrightarrow H_\bullet(BS^1;\Q) \otimes H_\bullet(X,\partial X; \Q)$. As we will see, $\delta = \alpha \circ \delta_0$, where $\delta_0 \colon \homology{}{S^1}{}{S}{H}{+}{}(X) \longrightarrow \homology{}{S^1}{}{S}{H}{\varepsilon}{}(X)$ is the continuation map associated to a long exact sequence in homology (see \cref{def:delta map}) and $\alpha \colon \homology{}{S^1}{}{S}{H}{\varepsilon}{}(X) \longrightarrow H_\bullet(BS^1;\Q) \otimes H_\bullet(X,\partial X; \Q)$ is an isomorphism which we define in several steps (see \cref{lem:iso floer and alt floer,lem:iso from floer to morse,lem:iso from floer to singular,lem:iso from symplectic to singular}). For every $(H,J) \in \admissible{X}$, define
\begin{IEEEeqnarray*}{rCrCrCls+x*}
    H' & \coloneqq & H_{e_0} & \colon & S^1 \times \hat{X} & \longrightarrow & \R, \\
    J' & \coloneqq & J_{e_0} & \colon & S^1 \times \hat{X} & \longrightarrow & \End(T \hat{X}),
\end{IEEEeqnarray*}
where $e_0 \in S^{2N+1} \subset \C^{N+1}$ is the first vector in the canonical basis of $\C^{N+1}$. We start by giving an alternative definition of the $S^1$-equivariant Floer chain complex.

\begin{definition}[{\cite[Remark 5.15]{guttSymplecticCapacitiesPositive2018}}]
    We define a chain complex $\homology{}{S^1}{}{F}{C}{}{}(X,H,J)_{\mathrm{alt}}$ as follows. Let $u$ be a formal variable of degree $2$ and consider $\Q \{1,\ldots,u^N\}$, the $\Q$-module of polynomials in $u$ of degree less or equal to $2N$. As a $\Q$-module,
    \begin{IEEEeqnarray*}{c+x*}
        \homology{}{S^1}{}{F}{C}{}{}(X,H,J)_{\mathrm{alt}} \coloneqq \Q \{1,\ldots,u^N\} \otimes \homology{}{}{}{F}{C}{}{}(X,H',J'),
    \end{IEEEeqnarray*}
    where $\homology{}{}{}{F}{C}{}{}(X,H',J')$ is the Floer chain complex (not $S^1$-equivariant) of $X$ with respect to $(H',J')$, with $\Q$ coefficients. We will now define a differential $\partial_{\mathrm{alt}}$ on $\homology{}{S^1}{}{F}{C}{}{}(X,H,J)_{\mathrm{alt}}$. For every $j = 0,\ldots,N$, define a map $\varphi_j \colon \homology{}{}{}{F}{C}{}{}(X,H',J') \longrightarrow \homology{}{}{}{F}{C}{}{}(X,H',J')$ by
    \begin{IEEEeqnarray*}{c+x*}
        \varphi_j(\gamma^+) \coloneqq \sum_{\gamma^- \in \mathcal{P}(H')} \# \mathcal{M}_{\vphantom{0}}(H,J,[e_j,\gamma^+],[e_0,\gamma^-]) \cdot \gamma^-,
    \end{IEEEeqnarray*}
    for every $\gamma^+ \in \mathcal{P}(H')$. Note that $\varphi_0 \colon \homology{}{}{}{F}{C}{}{}(X,H',J') \longrightarrow \homology{}{}{}{F}{C}{}{}(X,H',J')$ is the usual differential of the Floer chain complex. Finally, we define
    \begin{IEEEeqnarray*}{rrCl}
        \del_{\mathrm{alt}} \colon & \Q \{1,\ldots,u^N\} \tensorpr \homology{}{}{}{F}{C}{}{}(X,H',J') & \longrightarrow & \Q \{1,\ldots,u^N\} \tensorpr \homology{}{}{}{F}{C}{}{}(X,H',J') \\
                                   & u^k \tensorpr \gamma                                             & \longmapsto     & \sum_{j=0}^{k} u ^{k-j} \tensorpr \varphi_j(\gamma).
    \end{IEEEeqnarray*}
\end{definition}

\begin{lemma}[{\cite[Section 2.3]{bourgeoisEquivariantSymplecticHomology2016}}]
    \label{lem:iso floer and alt floer}
    The map
    \begin{IEEEeqnarray*}{rCl}
        \homology{}{S^1}{}{F}{C}{}{}(X,H,J) & \longrightarrow & \homology{}{S^1}{}{F}{C}{}{}(X,H,J)_{\mathrm{alt}} \\
        {[e_j, \gamma]}                     & \longmapsto     & u^j \otimes \gamma
    \end{IEEEeqnarray*}
    is an isomorphism of chain complexes.
\end{lemma}

Recall that in $X$, the Hamiltonian $H$ is assumed to be $C^2$-small and $S^1$-independent. Therefore, if $\gamma \colon S^1 \longrightarrow \hat{X}$ is a $1$-periodic orbit of $H'$ and $\img \gamma \subset X$, then $\gamma$ is constant with value $x \in X$, where $x$ is a critical point of $H'$. We will now assume that the Hamiltonian $H$ is chosen such that if $x^{\pm}$ are critical points of $H'$, then
\begin{IEEEeqnarray}{c+x*}
    \plabel{eq:self indexing}
    H'(x^+) \leq H'(x^-) \Longrightarrow \morse(x^+,H') \geq \morse(x^-,H').
\end{IEEEeqnarray}
We will denote by $(MC(X,H'), \partial^M)$ the Morse complex of $X$ with respect to $H'$, defined with the following conventions. As a vector space, $MC(X,H')$ is the vector space over $\Q$ generated by the critical points of $H'$. If $x^\pm$ are critical points of $H'$, the coefficient $\p{<}{}{\partial^{M} (x^+), x^-}$ is the count of gradient flow lines of $H'$ from $x^-$ to $x^+$. Finally, the degree of a critical point $x$ is the Morse index of $x$.

\begin{lemma}
    \label{lem:iso from floer to morse}
    There is a canonical isomorphism of chain complexes
    \begin{IEEEeqnarray*}{c+x*}
        (\homology{}{S^1}{}{F}{C}{\varepsilon}{}(X,H,J), \partial_{\mathrm{alt}}) = (\Q \{1,\ldots,u^N\} \otimes MC(X,H'), \id \otimes \partial^M).
    \end{IEEEeqnarray*}
\end{lemma}
\begin{proof}
    By \cref{rmk:types of orbits,lem:action admissible,lem:iso floer and alt floer}, there is a canonical isomorphism of $\Q$-modules
    \begin{IEEEeqnarray*}{c+x*}
        \homology{}{S^1}{}{F}{C}{\varepsilon}{}(X,H,J) = \Q \{1,\ldots,u^N\} \otimes MC(X,H').
    \end{IEEEeqnarray*}
    We show that this isomorphism is a chain map. We claim that if $j \geq 1$ and $x^+, x^-$ are critical points of $H'$, then $\dim_{(w,u)} \mathcal{M}(H,J,[e_j,x^+],[e_0,x^-]) \geq 1$. To see this, we compute
    \begin{IEEEeqnarray*}{rCls+x*}
        \dim_{(w,u)} \mathcal{M}(H,J,[e_j,x^+],[e_0,x^-])
        & =    & \ind(e_j, x^+) - \ind(e_0, x^-) - 1 \\
        & =    & \morse(e_j) - \morse(e_0) + \morse(x^+,H') - \morse(x^-,H') - 1 \\
        & =    & 2 j + \morse(x^+,H') - \morse(x^-,H') - 1 \\
        & \geq & 2 j - 1 \\
        & \geq & 1,
    \end{IEEEeqnarray*}
    where in the fourth line we used \cref{lem:action energy for floer trajectories} and Equation \eqref{eq:self indexing}. Therefore, if $j \geq 1$ and $x^+$ is a critical point of $H'$ then $\varphi_j(x^+) = 0$. This implies that
    \begin{IEEEeqnarray*}{c+x*}
        \partial_{\mathrm{alt}}(u^k \otimes x^+) = u^k \otimes \varphi_0(x^+),
    \end{IEEEeqnarray*}
    where $\varphi_0(x^+) = \partial^M(x^+)$ is the Morse theory differential applied to $x^+$.
\end{proof}

\begin{lemma}
    \label{lem:iso from floer to singular}
    There is a canonical isomorphism 
    \begin{IEEEeqnarray*}{c+x*}
        \homology{}{S^1}{}{F}{H}{\varepsilon}{}(X,H,J) = \Q \{1,\ldots,u^N\} \otimes H_\bullet(X, \partial X; \Q).
    \end{IEEEeqnarray*}
\end{lemma}
\begin{proof}
    \begin{IEEEeqnarray*}{rCls+x*}
        \homology{}{S^1}{}{F}{H}{\varepsilon}{}(X,H,J)
        & = & H(\Q \{1,\ldots,u^N\} \otimes MC(X,H')) \\     
        & = & \Q \{1,\ldots,u^N\} \otimes MH_\bullet(X,H') \\            
        & = & \Q \{1,\ldots,u^N\} \otimes H_{\bullet}(X, \partial X; \Q),
    \end{IEEEeqnarray*}
    where in the first equality we used \cref{lem:iso from floer to morse}, in the second equality we used the definition of the differential of $\Q \{1,\ldots,u^N\} \otimes MC(X,H')$, and in the third equality we used the isomorphism between Morse homology and singular homology.
\end{proof}

\begin{lemma}
    \label{lem:iso from symplectic to singular}
    There is a canonical isomorphism 
    \begin{IEEEeqnarray*}{c+x*}
        \alpha \colon \homology{}{S^1}{}{S}{H}{\varepsilon}{}(X) \longrightarrow H_\bullet(BS^1;\Q) \otimes H_\bullet(X,\partial X; \Q).
    \end{IEEEeqnarray*}
\end{lemma}
\begin{proof}
    \begin{IEEEeqnarray*}{rCls+x*}
        \homology{}{S^1}{}{S}{H}{\varepsilon}{}(X)
        & = & \varinjlim_{N,H,J} \homology{}{S^1}{}{F}{H}{\varepsilon}{}(X,H,J) \\
        & = & \varinjlim_{N,H,J} \Q \{1,\ldots,u^N\} \otimes H_\bullet(X, \partial X; \Q) \\
        & = & \Q[u] \otimes H_\bullet(X, \partial X; \Q) \\
        & = & H_\bullet(BS^1; \Q) \otimes H_\bullet(X, \partial X; \Q),
    \end{IEEEeqnarray*}
    where in the first equality we used the definition of $S^1$-equivariant symplectic homology and in the second equality we used \cref{lem:iso from floer to singular}.
\end{proof}

\begin{definition}
    \phantomsection\label{def:delta map}
    We define a map $\delta \colon \homology{}{S^1}{}{S}{H}{+}{}(X) \longrightarrow H_\bullet(BS^1;\Q) \otimes H_\bullet(X,\partial X; \Q)$ as follows. For every $(H,J) \in \admissible{X}$, consider the short exact sequence of complexes
    \begin{IEEEeqnarray*}{c+x*}
        \begin{tikzcd}
            0 \ar[r] & \homology{}{S^1}{}{F}{C}{\varepsilon}{}(X,H,J) \ar[r] & \homology{}{S^1}{}{F}{C}{}{}(X,H,J) \ar[r] & \homology{}{S^1}{}{F}{C}{+}{}(X,H,J) \ar[r] & 0
        \end{tikzcd}
    \end{IEEEeqnarray*}
    There is an associated long exact sequence in homology
    \begin{IEEEeqnarray*}{c+x*}
        \begin{tikzcd}
            \cdots \ar[r] & \homology{}{S^1}{}{F}{H}{}{}(X,H,J) \ar[r] & \homology{}{S^1}{}{F}{H}{+}{}(X,H,J) \ar[r, "\delta^{H,J}"] & \homology{}{S^1}{}{F}{H}{\varepsilon}{}(X,H,J) \ar[r] & \cdots
        \end{tikzcd}
    \end{IEEEeqnarray*}
    Passing to the colimit, we obtain a sequence
    \begin{IEEEeqnarray*}{c+x*}
        \begin{tikzcd}
            \cdots \ar[r] & \homology{}{S^1}{}{S}{H}{}{}(X) \ar[r] & \homology{}{S^1}{}{S}{H}{+}{}(X) \ar[r, "\delta_0"] & \homology{}{S^1}{}{S}{H}{\varepsilon}{}(X) \ar[r] & \cdots
        \end{tikzcd}
    \end{IEEEeqnarray*}
    Finally, define $\delta \coloneqq \alpha \circ \delta_0 \colon \homology{}{S^1}{}{S}{H}{+}{}(X) \longrightarrow H_\bullet(BS^1;\Q) \otimes H_\bullet(X,\partial X; \Q)$, where $\alpha$ is the isomorphism from \cref{lem:iso from symplectic to singular}.
\end{definition}

Let $\varphi \colon (X,\lambda_X) \longrightarrow (Y, \lambda_Y)$ be a $0$-codimensional strict generalized Liouville embedding. Define $\rho \colon H_\bullet(Y,\partial Y; \Q) \longrightarrow H_\bullet(X,\partial X; \Q)$ to be the unique map such that the diagram
\begin{IEEEeqnarray*}{c+x*}
    \begin{tikzcd}
        H_\bullet(X,\del X; \Q) \ar[r, hook, two heads, "\varphi_*"]    & H_\bullet(\varphi(X),\varphi(\del X); \Q) \ar[d, hook, two heads] \\
        H_\bullet(Y,\del Y; \Q) \ar[r] \ar[u, dashed, "\exists ! \rho"] & H_\bullet(Y, Y \setminus \varphi(\itr X); \Q)
    \end{tikzcd}
\end{IEEEeqnarray*}
commutes, where $\varphi_*$ is an isomorphism by functoriality of homology and the vertical arrow on the right is an isomorphism by excision. The map $\rho$ is such that $\rho([Y]) = [X]$.

\begin{proposition}[{\cite[Proposition 3.3]{guttSymplecticCapacitiesPositive2018}}]
    The diagram
    \begin{IEEEeqnarray*}{c+x*}
        \begin{tikzcd}
            \homology{}{S^1}{}{S}{H}{+}{}(Y) \ar[r, "\delta_Y"] \ar[d, swap, "\varphi_!"] & H_\bullet(BS^1;\Q) \otimes H_\bullet(Y,\partial Y; \Q) \ar[d, "\id \otimes \rho"] \\
            \homology{}{S^1}{}{S}{H}{+}{}(X) \ar[r, swap, "\delta_X"]                     & H_\bullet(BS^1;\Q) \otimes H_\bullet(X,\partial X; \Q)
        \end{tikzcd}
    \end{IEEEeqnarray*}
    commutes.
\end{proposition}

\chapter{Symplectic capacities}
\label{chp:symplectic capacities}

\section{Symplectic capacities}
\label{sec:symplectic capacities}

In this section we define the notion of symplectic capacity (\cref{def:symplectic capacity}). A capacity is a function $c$ which assigns to every symplectic manifold $X$ (in a restricted subclass) a number $c(X) \in [0,+\infty]$, and which is functorial with respect to symplectic embeddings (in a restricted subclass). In the remaining sections of this chapter, we will define various capacities, namely the Lagrangian capacity (\cref{def:lagrangian capacity}), the Gutt--Hutchings capacities (\cref{def:gutt hutchings capacities}) and the McDuff--Siegel capacities (\cref{def:g tilde}). In this section we also deal with two small technicalities:
\begin{enumerate}
    \item Most of the capacities we will deal with in this thesis are functorial with respect to generalized Liouville embeddings, which do not form a category. However, they form an object which is like a category but has only partially defined composition of morphisms. We will use the nomenclature of \cite{andersenTQFTQuantumTeichmuller2014} and call such an object a categroid (\cref{def:categroid}).
    \item As we will see, some capacities we will consider are defined on the class of nondegenerate Liouville domains. In the last part of this section, we will see how such a capacity can be extended uniquely to a capacity of Liouville domains.
\end{enumerate}

\begin{definition}[{\cite[Definition 22]{andersenTQFTQuantumTeichmuller2014}}]
    \label{def:categroid}
    A \textbf{categroid} $\mathbf{C}$ consists of a family of objects $\operatorname{Obj}(\mathbf{C})$ and for any pair of objects $A,B \in \mathbf{C}$ a set $\Hom_{\mathbf{C}}(A,B)$ such that the following holds.
    \begin{enumerate}
        \item For any three objects $A$, $B$, $C$ there is a subset $\operatorname{Comp}_{\mathbf{C}}(A,B,C) \subset \Hom_{\mathbf{C}}(B,C) \times \Hom_{\mathbf{C}}(A,B)$ of \textbf{composable morphisms} and an associated \textbf{composition map}
            \begin{IEEEeqnarray*}{c+x*}
                \circ \colon \operatorname{Comp}_{\mathbf{C}}(A,B,C) \longrightarrow \Hom_{\mathbf{C}}(A,C)
            \end{IEEEeqnarray*}
            such that composition of composable morphisms is associative.
        \item For any object $A$ there exists an \textbf{identity morphism} $\id_A \in \Hom_{\mathbf{C}}(A,A)$ which is composable with any morphism $f \in \Hom_{\mathbf{C}}(A,B)$ or $g \in \Hom_{\mathbf{C}}(B,A)$ and satisfies%
            \begin{IEEEeqnarray*}{rCls+x*}
                f \circ \id_A & = & f, \\
                \id_A \circ g & = & g.
            \end{IEEEeqnarray*}
    \end{enumerate}
\end{definition}

In this context, one has obvious definitions of subcategroids and also of functors between categroids. Denote by $\symp$ the category of symplectic manifolds, where morphisms are $0$-codimensional symplectic embeddings.

\begin{definition}
    \label{def:symplectic categroid}
    A \textbf{symplectic categroid} is a subcategroid $\mathbf{C}$ of $\symp$ such that $(X,\omega) \in \mathbf{C}$ implies $(X,\alpha \omega) \in \mathbf{C}$ for all $\alpha > 0$.
\end{definition}

\begin{definition}
    \label{def:symplectic capacity}
    Let $\mathbf{C}$ be a symplectic categroid. A \textbf{symplectic capacity} is a functor $c \colon \mathbf{C} \longrightarrow [0,+\infty]$ satisfying
    \begin{description}
        \item[(Monotonicity)] If $(X,\omega_X) \longrightarrow (Y, \omega_Y)$ is a morphism in $\mathbf{C}$ then $c(X,\omega_X) \leq c(Y,\omega_Y)$;
        \item[(Conformality)] If $\alpha > 0$ then $c(X,\alpha \omega) = \alpha \, c(X, \omega)$.
    \end{description}
\end{definition}

Notice that the monotonicity property is just a restatement of the fact that $c$ is a functor.

\begin{definition}
    \label{def:nontrivial}
    Let $c \colon \mathbf{C} \longrightarrow [0, +\infty]$ be a symplectic capacity with the property that $B^{2n}(1), Z^{2n}(1) \in \mathbf{C}$ for every $n$. We say that $c$ is \textbf{nontrivial} or \textbf{normalized} if it satisfies%
    \begin{description}
        \item[(Nontriviality)] $0 < c(B^{2n}(1)) \leq  c(Z^{2n}(1)) < + \infty$;
        \item[(Normalization)] $0 < c(B^{2n}(1)) = 1 = c(Z^{2n}(1)) < + \infty$.
    \end{description}
\end{definition}

\begin{example}
    Let $(X, \omega)$ be a $2n$-dimensional symplectic manifold. Recall that the \textbf{symplectic volume} of $X$ is given by
    \begin{IEEEeqnarray*}{c+x*}
        \operatorname{vol}(X) \coloneqq \int_{X}^{} \frac{\omega^n}{n!}.
    \end{IEEEeqnarray*}
    The \textbf{volume capacity} of $X$ is given by
    \begin{IEEEeqnarray*}{c+x*}
        c_{\mathrm{vol}}(X) \coloneqq \p{}{2}{\frac{\operatorname{vol}(X)}{\operatorname{vol}(B)}}^{1/n},
    \end{IEEEeqnarray*}
    where $B \coloneqq B^{2n}(1) \coloneqq \{z \in \C^{n} \mid \pi |z|^2 \leq 1 \}$.
\end{example}

\begin{example}
    Let $(Y,\Omega)$ be a symplectic manifold. We define the \textbf{embedding capacities}, denoted by $c_{(Y,\Omega)}$ and $c^{(Y,\Omega)}$, by
    \begin{IEEEeqnarray*}{rCll}
        c_{(Y,\Omega)}(X, \omega) & \coloneqq & \sup & \{ a > 0 \mid \text{there exists a symplectic embedding } (Y, a \Omega) \longrightarrow (X, \omega) \}, \\
        c^{(Y,\Omega)}(X, \omega) & \coloneqq & \inf & \{ a > 0 \mid \text{there exists a symplectic embedding } (X, \omega) \longrightarrow (Y, a \Omega) \},
    \end{IEEEeqnarray*}
    for any symplectic manifold $(X, \omega)$. Let $\omega_0$ denote the canonical symplectic structure of $\C^n$. In the case where $(Y, \Omega) = (B^{2n}(1), \omega_0)$ or $(Y, \Omega) = (P^{2n}(1), \omega_0)$, we denote
    \begin{IEEEeqnarray*}{lClCl}
        c_B(X,\omega) & \coloneqq & c_{(B^{2n}(1), \omega)}(X, \omega) & = & \sup \{ a \ | \ \text{$\exists$ symplectic embedding } B^{2n}(a) \longrightarrow X \}, \\
        c_P(X,\omega) & \coloneqq & c_{(P^{2n}(1), \omega)}(X, \omega) & = & \sup \{ a \ | \ \text{$\exists$ symplectic embedding } P^{2n}(a) \longrightarrow X \}.
    \end{IEEEeqnarray*}
    Embedding capacities tend to be hard to compute, since they are defined as a restatement of a hard embedding problem. For example, a restatement of Gromov's nonsqueezing theorem \cite{gromovPseudoHolomorphicCurves1985} is that $c_B$ is a normalized symplectic capacity. The capacity $c_B$ is also called \textbf{Gromov width}.
\end{example}

\begin{definition}[{\cite[Section 4.2]{guttSymplecticCapacitiesPositive2018}}]
    \phantomsection\label{def:perturbation of liouville domain}
    If $(X,\lambda)$ is a Liouville domain and $f \colon \partial X \longrightarrow \R$ is a smooth function, we define a new Liouville domain $(X_f,\lambda_f)$ as follows. Consider the completion $\hat{X}$, which has as subsets $X \subset \hat{X}$ and $\R \times \partial X \subset \hat{X}$. Then,
    \begin{IEEEeqnarray*}{c+x*}
        X_f \coloneqq \hat{X} \setminus \{ (\rho,y) \in \R \times \partial X \mid \rho > f(y) \}
    \end{IEEEeqnarray*}
    and $\lambda_f$ is the restriction of $\hat{\lambda}$ to $X_f$. Define $\mathcal{F}_{X}^{\pm}$ to be the set of $f^{\pm} \colon \partial X \longrightarrow \R^\pm$ such that $(X_{f^\pm}, \lambda_{f^\pm})$ is nondegenerate.
\end{definition}

\begin{definition}
    \label{def:liouville categroid}
    A \textbf{Liouville categroid} is a subcategroid $\mathbf{L}$ of $\symp$ such that
    \begin{enumerate}
        \item Every object of $\mathbf{L}$ is a Liouville domain.
        \item If $X \in \mathbf{L}$ and $f^{+} \in \mathcal{F}^{+}_X$ then $X_{f^{+}} \in \mathbf{L}$ and the inclusion $X \longrightarrow X_{f^+}$ is a morphism in $\mathbf{L}$ which is composable with any other morphisms $Y \longrightarrow X$ or $X_{f^+} \longrightarrow Z$ in $\mathbf{L}$.
        \item If $X \in \mathbf{L}$ and $f^{-} \in \mathcal{F}^{-}_X$ then $X_{f^{-}} \in \mathbf{L}$ and the inclusion $X_{f^-} \longrightarrow X$ is a morphism in $\mathbf{L}$ which is composable with any other morphisms $Y \longrightarrow X_{f^-}$ or $X \longrightarrow Z$ in $\mathbf{L}$.
    \end{enumerate}
\end{definition}

\begin{example}
    Let $\liouvgle$ be the categroid whose objects are Liouville domains and whose morphisms are $0$-codimensional generalized Liouville embeddings. Then $\liouvgle$ is a Liouville categroid.
\end{example}
 
\begin{lemma}
    \label{lem:c is the unique extension to lvds}
    Let $\mathbf{L}$ be a Liouville categroid. Let $\mathbf{L}_{\mathrm{ndg}}$ be the full subcategroid of $\mathbf{L}$ of nondegenerate Liouville domains (i.e., if $X, Y \in \mathbf{L}_{\mathrm{ndg}}$ then $\Hom_{\mathbf{L}_{\mathrm{ndg}}}(X,Y) = \Hom_{\mathbf{L}}(X,Y)$). If $c \colon \mathbf{L}_{\mathrm{ndg}} \longrightarrow [0, +\infty]$ is a symplectic capacity, then there exists a unique symplectic capacity $\overline{c} \colon \mathbf{L} \longrightarrow [0, + \infty]$ such that the following diagram commutes:
    \begin{IEEEeqnarray}{c+x*}
        \plabel{eq:diagram extend cap liouv}
        \begin{tikzcd}
            \mathbf{L}_{\mathrm{ndg}} \ar[d] \ar[dr, "c"] & \\
            \mathbf{L} \ar[r, swap, "\overline{c}"] & {[0,+\infty]}
        \end{tikzcd}
    \end{IEEEeqnarray}
\end{lemma}
\begin{proof}
    This proof is based on \cite[Section 4.2]{guttSymplecticCapacitiesPositive2018}. We claim that if $\varepsilon > 0$ and $(X, \lambda)$ is a nondegenerate Liouville domain in $\mathbf{L}_{\mathrm{ndg}}$, then $(X_{\varepsilon}, \lambda_{\varepsilon})$ is nondegenerate and 
    \begin{IEEEeqnarray}{c+x*}
        \plabel{eq:capacity of deformed domain}
        c(X_\varepsilon, \lambda_\varepsilon) = e^{\varepsilon} c (X, \lambda).
    \end{IEEEeqnarray}
    To see this, notice that the time $\varepsilon$ flow of the Liouville vector field $Z$ of $\hat{X}$ restricts to a Liouville embedding $\phi \colon (X, e^{\varepsilon} \lambda) \longrightarrow (X_\varepsilon, \lambda_\varepsilon)$ and also to a contactomorphism $\phi \colon (\partial X, e^{\varepsilon} \lambda|_{\partial X}) \longrightarrow (\partial X_\varepsilon, \partial \lambda_\varepsilon|_{\partial X_\varepsilon})$. This shows that $(X_\varepsilon, \lambda_\varepsilon)$ is nondegenerate. In particular, $(X_\varepsilon, \lambda_\varepsilon) \in \mathbf{L}_{\mathrm{ndg}}$. Finally,%
    \begin{IEEEeqnarray*}{rCls+x*}
        c(X_\varepsilon, \lambda_\varepsilon)
        & = & c(X, e^{\varepsilon} \lambda) & \quad [\text{by functoriality of $c$}] \\
        & = & e^{\varepsilon} c(X,\lambda)  & \quad [\text{by conformality}]. &
    \end{IEEEeqnarray*}
    This finishes the proof of Equation \eqref{eq:capacity of deformed domain}. Define functions $c^{\pm} \colon \mathbf{L} \longrightarrow [0,+\infty]$ by
    \begin{IEEEeqnarray*}{rCls+x*}
        c^+(X) & \coloneqq & \inf_{f^+ \in \mathcal{F}^+_X} c(X_{f^+}), \\
        c^-(X) & \coloneqq & \sup_{f^- \in \mathcal{F}^-_X} c(X_{f^-}).
    \end{IEEEeqnarray*}
    We claim that if $(X, \lambda) \in \mathbf{L}$ is a Liouville domain then 
    \begin{IEEEeqnarray}{c+x*}
        \plabel{eq:c minus equals c plus}
        c^-(X) = c^+(X).
    \end{IEEEeqnarray}
    Monotonicity of $c$ implies $c^-(X) \leq c^+(X)$. To show the reverse inequality, it is enough to show that $c^+(X) \leq e^{\varepsilon} c^-(X)$ for every $\varepsilon > 0$. For this, choose $f^- \in \mathcal{F}^{-}_X$ such that $\img f^- \subset (- \varepsilon, 0)$ and define $f^+ = f^- + \varepsilon$. By the previous discussion, $(X_{f^+}, \lambda_{f^+})$ is nondegenerate and $f^+ \in \mathcal{F}^+_X$. Then,
    \begin{IEEEeqnarray*}{rCls+x*}
        c^+(X)
        & =    & \inf_{g^+ \in \mathcal{F}^+_X} c(X_{g^+})                 & \quad [\text{by definition of $c^+$}] \\
        & \leq & c(X_{f^+})                                                & \quad [\text{since $f^+ \in \mathcal{F}^+_X$}] \\
        & =    & e^{\varepsilon} c(X_{f^-})                                & \quad [\text{by Equation \eqref{eq:capacity of deformed domain}}] \\
        & \leq & e^{\varepsilon} \sup_{g^- \in \mathcal{F}^-_X} c(X_{g^-}) & \quad [\text{since $f^- \in \mathcal{F}^-_X$}] \\
        & =    & e^{\varepsilon} c^-(X)                                    & \quad [\text{by definition of $c^-$}],
    \end{IEEEeqnarray*}
    which finishes the proof of Equation \eqref{eq:c minus equals c plus}. Moreover, if $(X, \lambda) \in \mathbf{L}_{\mathrm{ndg}}$ is nondegenerate, then $c^-(X) \leq c(X) \leq c^+(X) = c^-(X)$, which implies
    \begin{IEEEeqnarray*}{c+x*}
        c^-(X) = c(X) = c^+(X).
    \end{IEEEeqnarray*}
    We now show that $c^{\pm}$ are symplectic capacities. The conformality property is immediate. To prove monotonicity, let $X \longrightarrow Y$ be a morphism in $\mathbf{L}$.
    \begin{IEEEeqnarray*}{rCls+x*}
        c^-(X)
        & =    & \sup_{f^- \in \mathcal{F}^-_X} c(X_{f^-}) & \quad [\text{by definition of $c^-$}] \\
        & \leq & \inf_{g^+ \in \mathcal{F}^+_Y} c(Y_{g^+}) & \quad [\text{since $X_{f^-} \subset X \longrightarrow Y \subset Y_{g^+}$ and by monotonicity of $c$}] \\
        & =    & c^+(Y)                                    & \quad [\text{by definition of $c^+$}].
    \end{IEEEeqnarray*}
    The result follows from Equation \eqref{eq:c minus equals c plus}. To prove existence, simply notice that by the above discussion, the function $\overline{c} \coloneqq c^- = c^+ \colon \mathbf{L} \longrightarrow [0, +\infty]$ has all the desired properties.

    To prove uniqueness, let $\overline{c}$ be any function as in the statement of the lemma. We wish to show that $\overline{c} \coloneqq c^- = c^+$. We start by showing that $c^-(X) \leq \overline{c}(X)$.
    \begin{IEEEeqnarray*}{rCls+x*}
        c^-(X)
        & =    & \sup_{f^- \in \mathcal{F}^-_X} c(X_{f^-})            & \quad [\text{by definition of $c^-$}] \\
        & =    & \sup_{f^- \in \mathcal{F}^-_X} \overline{c}(X_{f^-}) & \quad [\text{by assumption on $\overline{c}$}] \\
        & \leq & \sup_{f^- \in \mathcal{F}^-_X} \overline{c}(X)       & \quad [\text{by monotonicity of $\overline{c}$}] \\
        & =    & \overline{c}(X).
    \end{IEEEeqnarray*}
    Analogously, we can show that $c^+(X) \geq \overline{c}(X)$, which concludes the proof.
\end{proof}

\begin{lemma}
    \label{lem:can prove ineqs for ndg}
    For $i = 0,1$, let $c_i \colon \mathbf{L}_{\mathrm{ndg}} \rightarrow [0, +\infty]$ be symplectic capacities with extensions $\overline{c}_i \colon \mathbf{L} \rightarrow [0, +\infty]$ as in \cref{lem:c is the unique extension to lvds}. If $c_0(Y) \leq c_1(Y)$ for every nondegenerate Liouville domain $Y \in \mathbf{L}_{\mathrm{ndg}}$ then $\overline{c}_0(X) \leq \overline{c}_1(X)$ for every Liouville domain $X \in \mathbf{L}$.
\end{lemma}
\begin{proof}
    \begin{IEEEeqnarray*}{rCls+x*}
        \overline{c}_0(X)
        & =    & \sup_{f^- \in \mathcal{F}^-_X} c_0(X_{f^-}) & \quad [\text{by the definition of $\overline{c}_0$ in \cref{lem:c is the unique extension to lvds}}] \\
        & \leq & \sup_{f^- \in \mathcal{F}^-_X} c_1(X_{f^-}) & \quad [\text{by assumption on $c_0$ and $c_1$}] \\
        & =    & \overline{c}_1(X)                           & \quad [\text{by the definition of $\overline{c}_1$ in \cref{lem:c is the unique extension to lvds}}].  & \qedhere
    \end{IEEEeqnarray*}
\end{proof}

By the exposition above, if $c$ is a capacity of nondegenerate Liouville domains then it can be extended to a capacity of Liouville domains. In particular, $c(X)$ is defined for any star-shaped domain $X$. However, it will be useful to us to compute capacities of the cube $P(r)$ and of the nondisjoint union of cylinders $N(r)$. These spaces are not quite star-shaped domains, because they have corners and $N(r)$ is noncompact. So we will consider a further extension of the capacity $c$. Let $\mathbf{Star}$ be the category of star-shaped domains, where there is a unique morphism $X \longrightarrow Y$ if and only if $X \subset Y$. Denote by $\mathbf{Star}_{\mathrm{ncp}}$ the category of ``star-shaped domains'' which are possibly noncompact or possibly have corners, with the same notion of morphisms. 

\begin{lemma}
    \label{lem:c is the smallest extension to ss}
    Let $c \colon \mathbf{Star} \longrightarrow [0, +\infty]$ be a symplectic capacity. Define a symplectic capacity $\overline{c} \colon \mathbf{Star}_{\mathrm{ncp}} \longrightarrow [0, +\infty]$ by
    \begin{IEEEeqnarray*}{c+x*}
        \overline{c}(X) = \sup_{Y \subset X} c(Y),
    \end{IEEEeqnarray*}
    where the supremum is taken over star-shaped domains $Y \subset X$ which are compact and have smooth boundary. Then, the diagram
    \begin{IEEEeqnarray*}{c+x*}
        \begin{tikzcd}
            \mathbf{Star} \ar[dr, "c"] \ar[d] \\
            \mathbf{Star}_{\mathrm{ncp}} \ar[r, swap, "\overline{c}"] & {[0, + \infty]}
        \end{tikzcd}
    \end{IEEEeqnarray*}
    commutes. Moreover, $\overline{c}$ is the smallest capacity making this diagram commute.
\end{lemma}
\begin{proof}
    It is immediate that $\overline{c}$ is a symplectic capacity. We show that the diagram commutes. If $X$ is a compact star-shaped domain with smooth boundary, then
    \begin{IEEEeqnarray*}{rCls+x*}
        c(X) 
        & \leq & \sup_{Y \subset X} c(Y) & \quad [\text{since $X$ is compact and has smooth boundary}] \\
        & \leq & c(X)                    & \quad [\text{by monotonicity}].
    \end{IEEEeqnarray*}
    If $\tilde{c} \colon \mathbf{Star}_{\mathrm{ncp}} \longrightarrow [0, +\infty]$ is another capacity making the diagram commute, then
    \begin{IEEEeqnarray*}{rCls+x*}
        \overline{c}(X)
        & =    & \sup_{Y \subset X} c(Y)         & \quad [\text{by definition of $\overline{c}$}] \\
        & =    & \sup_{Y \subset X} \tilde{c}(Y) & \quad [\text{since $\tilde{c}$ makes the diagram commute}] \\
        & \leq & \tilde{c}(X)                    & \quad [\text{by monotonicity of $\tilde{c}$}].               & \qedhere
    \end{IEEEeqnarray*}
\end{proof}

\begin{remark}
    We will always assume that every capacity of nondegenerate Liouville domains that we define is extended as in \cref{lem:c is the unique extension to lvds,lem:c is the smallest extension to ss} to possibly degenerate Liouville domains and to ``star-shaped domains'' which are possibly noncompact or possibly have corners.
\end{remark}

\section{Lagrangian capacity}

Here, we define the Lagrangian capacity (\cref{def:lagrangian capacity}) and state its properties (\cref{prop:properties of cL}). One of the main goals of this thesis is to study whether the Lagrangian capacity can be computed in some cases, for example for toric domains. In the end of the section, we state some easy inequalities concerning the Lagrangian capacity (\cref{lem:c square leq c lag,lem:c square geq delta}), known computations (\cref{prp:cl of ball,prp:cl of cylinder}) and finally the main conjecture of this thesis (\cref{conj:the conjecture}), which is inspired by all the previous results. The Lagrangian capacity is defined in terms of the minimal area of Lagrangian submanifolds, which we now define.

\begin{definition}
    Let $(X,\omega)$ be a symplectic manifold. If $L$ is a Lagrangian submanifold of $X$, then we define the \textbf{minimal symplectic area of} $L$, denoted $A_{\mathrm{min}}(L)$, by%
    \begin{IEEEeqnarray*}{c+x*}
        A_{\mathrm{min}}(L) \coloneqq \inf \{ \omega(\sigma) \mid \sigma \in \pi_2(X,L), \, \omega(\sigma) > 0 \}.
    \end{IEEEeqnarray*}
\end{definition}

\begin{lemma}
    \label{lem:properties of minimal area}
    Let $\iota \colon (X,\omega) \longrightarrow (X',\omega')$ be a symplectic embedding, $L \subset X$ be an embedded Lagrangian submanifold and $L' = \iota(L)$. In this case,
    \begin{enumerate}
        \item \label{lem:properties of minimal area 1} $A_{\mathrm{min}}(L) \geq A_{\mathrm{min}}(L')$;
        \item \label{lem:properties of minimal area 2} $A_{\mathrm{min}}(L) = A_{\mathrm{min}}(L')$, provided that $\pi_2(X',\iota(X)) = 0$.
    \end{enumerate}
\end{lemma}
\begin{proof}
    \ref{lem:properties of minimal area 1}: By definition of minimal area and since the diagram
    \begin{IEEEeqnarray}{c+x*}
        \plabel{eq:diag minimal area}
        \begin{tikzcd}[ampersand replacement = \&]
            \pi_2(X,L) \ar[d, swap, "\iota_*"] \ar[dr, "\omega"] \\
            \pi_2(X',L') \ar[r, swap, "\omega'"] \& \R
        \end{tikzcd}
    \end{IEEEeqnarray}
    commutes.

    \ref{lem:properties of minimal area 2}: Considering the long exact sequence of the triple $(X',\iota(X),L')$,
    \begin{IEEEeqnarray*}{c+x*}
        \begin{tikzcd}[ampersand replacement = \&]
            \cdots \ar[r] \& \pi_2(\iota(X),L') \ar[r] \& \pi_2(X',L') \ar[r] \& \pi_2(X',\iota(X)) = 0
        \end{tikzcd}
    \end{IEEEeqnarray*}
    we conclude that $\iota_{*} \colon \pi_2(X,L) \longrightarrow \pi_2(X',L')$ is surjective. Again, the result follows by the definition of minimal area and diagram \eqref{eq:diag minimal area}.
\end{proof}

\begin{lemma}
    \label{lem:a min with exact symplectic manifold}
    Let $(X,\lambda)$ be an exact symplectic manifold and $L \subset X$ be a Lagrangian submanifold. If $\pi_1(X) = 0$, then
    \begin{IEEEeqnarray*}{c+x*}
        A _{\mathrm{min}}(L) = \inf \left\{ \lambda(\rho) \ | \ \rho \in \pi_1(L), \ \lambda(\rho) > 0 \right\}.
    \end{IEEEeqnarray*}
\end{lemma}
\begin{proof}
    The diagram
    \begin{IEEEeqnarray*}{c+x*}
        \begin{tikzcd}[ampersand replacement = \&]
            \pi_2(L) \ar[d, swap, "0"] \ar[r] \& \pi_2(X) \ar[d, "\omega"] \ar[r] \& \pi_2(X,L) \ar[d, "\omega"] \ar[r, two heads,"\del"] \& \pi_1(L) \ar[d, "\lambda"] \ar[r, "0"] \& \pi_1(X) \ar[d, "\lambda"] \\
            \R \ar[r, equals] \& \R \ar[r, equals] \& \R \ar[r, equals] \& \R \ar[r, equals] \& \R
        \end{tikzcd}
    \end{IEEEeqnarray*}
    commutes, where $\del([\sigma]) = [\sigma|_{S^1}]$, and the top row is exact.
\end{proof}

\begin{definition}[{\cite[Section 1.2]{cieliebakPuncturedHolomorphicCurves2018}}]
    \phantomsection\label{def:lagrangian capacity}
    Let $(X,\omega)$ be a symplectic manifold. We define the \textbf{Lagrangian capacity} of $(X,\omega)$, denoted $c_L(X,\omega)$, by%
    \begin{IEEEeqnarray*}{c}
        c_L(X,\omega) \coloneqq \sup \{ A_{\mathrm{min}}(L) \mid L \subset X \text{ is an embedded Lagrangian torus}\}.
    \end{IEEEeqnarray*}
\end{definition}

\begin{proposition}[{\cite[Section 1.2]{cieliebakPuncturedHolomorphicCurves2018}}]
    \label{prop:properties of cL}
    The Lagrangian capacity $c_L$ satisfies:
    \begin{description}
        \item[(Monotonicity)] If $(X,\omega) \longrightarrow (X',\omega')$ is a symplectic embedding with $\pi_2(X',\iota(X)) = 0$, then $c_L(X,\omega) \leq c_L(X',\omega')$.
        \item[(Conformality)] If $\alpha \neq 0$, then $c_L(X,\alpha \omega) = |\alpha| \, c_L(X,\omega)$.
    \end{description}
\end{proposition}
\begin{proof}
    We prove monotonicity. 
    \begin{IEEEeqnarray*}{rCls+x*}
        c_L(X,\omega)  
        & =    & \sup _{L \subset X} A _{\min}(L)    & \quad [\text{by definition of $c_L$}] \\
        & \leq & \sup _{L' \subset X'} A _{\min}(L') & \quad [\text{by \cref{lem:properties of minimal area}}] \\
        & =    & c_L(X',\omega')                     & \quad [\text{by definition of $c_L$}].
    \end{IEEEeqnarray*}

    We prove conformality. Note that a submanifold $L \subset X$ is Lagrangian with respect to $\omega$ if and only if it is Lagrangian with respect to $\alpha \omega$.
    \begin{IEEEeqnarray*}{rCls+x*}
        c_L(X,\alpha \omega)
        & = & \sup _{L \subset (X,\alpha \omega)} A _{\mathrm{min}}(L,\alpha \omega)    & \quad [\text{by definition of $c_L$}] \\
        & = & \sup _{L \subset (X,\omega)       } |\alpha| A _{\mathrm{min}}(L, \omega) & \quad [\text{by definition of minimal area}] \\
        & = & |\alpha| \, c_L(X,\omega)                                                 & \quad [\text{by definition of $c_L$}].         & \qedhere
    \end{IEEEeqnarray*}
\end{proof}

\begin{lemma}
    \label{lem:c square leq c lag}
    If $X$ is a star-shaped domain, then $c_L(X) \geq c_P(X)$.
\end{lemma}
\begin{proof}
    Let $\iota \colon P(a) \longrightarrow X$ be a symplectic embedding, for some $a > 0$. We want to show that $c_L(X) \geq a$. Define $T = \{ z \in \C^n \mid |z_1|^2 = a/\pi, \ldots, |z_n|^2 = a/ \pi \} \subset \partial P(a)$ and $L = \iota(T)$. Then,
    \begin{IEEEeqnarray*}{rCls+x*}
        c_L(X)
        & \geq & A_{\mathrm{min}}(L) & \quad [\text{by definition of $c_L$}] \\
        & =    & A_{\mathrm{min}}(T) & \quad [\text{by \cref{lem:properties of minimal area}}] \\
        & =    & a                   & \quad [\text{by \cref{lem:a min with exact symplectic manifold}}]. & \qedhere
    \end{IEEEeqnarray*}
\end{proof}

Recall that if $X_{\Omega}$ is a toric domain, its diagonal is given by $\delta_{\Omega} \coloneqq \sup \{ a \mid (a, \ldots, a) \in \Omega \}$ (see \cref{def:moment map}).

\begin{lemma}
    \label{lem:c square geq delta}
    If $X_{\Omega}$ is a convex or concave toric domain, then $c_P(X_{\Omega}) \geq \delta_\Omega$.
\end{lemma}
\begin{proof}
    Since $X_{\Omega}$ is a convex or concave toric domain, we have that $P(\delta_\Omega) \subset X_{\Omega}$. The result follows by definition of $c_P$.
\end{proof}

Actually, Gutt--Hutchings show that $c_P(X_{\Omega}) = \delta_\Omega$ for any convex or concave toric domain $X_{\Omega}$ (\cite[Theorem 1.18]{guttSymplecticCapacitiesPositive2018}). However, for our purposes we will only need the inequality in \cref{lem:c square geq delta}. We now consider the results by Cieliebak--Mohnke for the Lagrangian capacity of the ball and the cylinder.

\begin{proposition}[{\cite[Corollary 1.3]{cieliebakPuncturedHolomorphicCurves2018}}]
    \phantomsection\label{prp:cl of ball}
    The Lagrangian capacity of the ball is
    \begin{IEEEeqnarray*}{c+x*}
        c_L(B^{2n}(1)) = \frac{1}{n}.   
    \end{IEEEeqnarray*}
\end{proposition}

\begin{proposition}[{\cite[p.~215-216]{cieliebakPuncturedHolomorphicCurves2018}}]
    \label{prp:cl of cylinder}
    The Lagrangian capacity of the cylinder is
    \begin{IEEEeqnarray*}{c+x*}
        c_L(Z^{2n}(1)) = 1.
    \end{IEEEeqnarray*}
\end{proposition}

By \cref{lem:c square leq c lag,lem:c square geq delta}, if $X_{\Omega}$ is a convex or concave toric domain then $c_L(X_\Omega) \geq \delta_\Omega$. But as we have seen in \cref{prp:cl of ball,prp:cl of cylinder}, if $X_\Omega$ is the ball or the cylinder then $c_L(X_\Omega) = \delta_\Omega$. This motivates \cref{conj:cl of ellipsoid} below for the Lagrangian capacity of an ellipsoid, and more generally \cref{conj:the conjecture} below for the Lagrangian capacity of any convex or concave toric domain.

\begin{conjecture}[{\cite[Conjecture 1.5]{cieliebakPuncturedHolomorphicCurves2018}}]
    \label{conj:cl of ellipsoid}
    The Lagrangian capacity of the ellipsoid is%
    \begin{IEEEeqnarray*}{c+x*}
        c_L(E(a_1,\ldots,a_n)) = \p{}{2}{\frac{1}{a_1} + \cdots + \frac{1}{a_n}}^{-1}.
    \end{IEEEeqnarray*}
\end{conjecture}

\begin{conjecture}
    \label{conj:the conjecture}
    If $X_{\Omega}$ is a convex or concave toric domain then 
    \begin{IEEEeqnarray*}{c+x*}
        c_L(X_{\Omega}) = \delta_\Omega.
    \end{IEEEeqnarray*}
\end{conjecture}

In \cref{lem:computation of cl,thm:my main theorem} we present our results concerning \cref{conj:the conjecture}.

\section{Gutt--Hutchings capacities}
\label{sec:equivariant capacities}


In this section we will define the Gutt--Hutchings capacities (\cref{def:gutt hutchings capacities}) and the $S^1$-equivariant symplectic homology capacities (\cref{def:s1esh capacities}), and list their properties (\cref{thm:properties of gutt-hutchings capacities,prp:properties of s1esh capacities} respectively). We will also compare the two capacities (\cref{thm:ghc and s1eshc}). The definition of these capacities relies on $S^1$-equivariant symplectic homology. In the commutative diagram below, we display the modules and maps which will play a role in this section, for a nondegenerate Liouville domain $X$.

\begin{IEEEeqnarray}{c+x*}
    \plabel{eq:diagram for s1esh capacities}
    \begin{tikzcd}
        \homology{}{S^1}{}{S}{H}{(\varepsilon,a]}{}(X) \ar[r, "\delta^a_0"] \ar[d, swap, "\iota^a"] & \homology{}{S^1}{}{S}{H}{\varepsilon}{}(X) \ar[d, two heads, hook, "\alpha"] \ar[r, "\iota^{a,\varepsilon}"] & \homology{}{S^1}{}{S}{H}{a}{}(X) \\
        \homology{}{S^1}{}{S}{H}{+}{}(X) \ar[ur, "\delta_0"] \ar[r, swap, "\delta"]                 & H_\bullet(BS^1;\Q) \otimes H_\bullet(X, \partial X;\Q)
    \end{tikzcd}
\end{IEEEeqnarray}

Here, $\iota^a$ and $\iota^{a, \varepsilon}$ are the maps induced by the action filtration, $\delta_0$ and $\delta$ are the maps from \cref{def:delta map} and $\alpha$ is the isomorphism from \cref{lem:iso from symplectic to singular}. We point out that every vertex in the above diagram has a $U$ map and every map in the diagram commutes with this $U$ map. Specifically, all the $S^1$-equivariant symplectic homologies have the $U$ map given as in \cref{def:U map} and $H_\bullet(BS^1;\Q) \otimes H_\bullet(X, \partial X;\Q) \cong \Q[u] \otimes H_\bullet(X, \partial X;\Q)$ has the map $U \coloneqq u^{-1} \otimes \id$. We will also make use of a version of diagram \eqref{eq:diagram for s1esh capacities} in the case where $X$ is star-shaped, namely diagram \eqref{eq:diagram for s1esh capacities case ss} below. In this case, the modules in the diagram admit gradings and every map is considered to be a map in a specific degree. By \cite[Proposition 3.1]{guttSymplecticCapacitiesPositive2018}, $\delta$ and $\delta_0$ are isomorphisms.

\begin{IEEEeqnarray}{c+x*}
    \plabel{eq:diagram for s1esh capacities case ss}
    \begin{tikzcd}
        \homology{}{S^1}{}{S}{H}{(\varepsilon,a]}{n - 1 + 2k}(X) \ar[r, "\delta^a_0"] \ar[d, swap, "\iota^a"]                   & \homology{}{S^1}{}{S}{H}{\varepsilon}{n - 2 + 2k}(X) \ar[d, two heads, hook, "\alpha"] \ar[r, "\iota^{a,\varepsilon}"] & \homology{}{S^1}{}{S}{H}{a}{n - 2 + 2k}(X) \\
        \homology{}{S^1}{}{S}{H}{+}{n - 1 + 2k}(X) \ar[ur, two heads, hook, "\delta_0"] \ar[r, swap, two heads, hook, "\delta"] & H_{2k-2}(BS^1;\Q) \otimes H_{2n}(X, \partial X;\Q)
    \end{tikzcd}
\end{IEEEeqnarray}


\begin{definition}[{\cite[Definition 4.1]{guttSymplecticCapacitiesPositive2018}}]
    \label{def:gutt hutchings capacities}
    If $k \in \Z_{\geq 1}$ and $(X,\lambda)$ is a nondegenerate Liouville domain, the \textbf{Gutt--Hutchings capacities} of $X$, denoted $\cgh{k}(X)$, are defined as follows. Consider the map
    \begin{IEEEeqnarray*}{c+x*}
        \delta \circ U^{k-1} \circ \iota^a \colon \homology{}{S^1}{}{S}{H}{(\varepsilon,a]}{}(X) \longrightarrow H_\bullet(BS^1;\Q) \otimes H_\bullet(X, \partial X;\Q)
    \end{IEEEeqnarray*}
    from diagram \eqref{eq:diagram for s1esh capacities}. Then, we define
    \begin{IEEEeqnarray*}{c+x*}
        \cgh{k}(X) \coloneqq \inf \{ a > 0 \mid [\mathrm{pt}] \otimes [X] \in \img (\delta \circ U^{k-1} \circ \iota^a) \}.
    \end{IEEEeqnarray*}
\end{definition}

\begin{theorem}[{\cite[Theorem 1.24]{guttSymplecticCapacitiesPositive2018}}]
    \label{thm:properties of gutt-hutchings capacities}
    The functions $\cgh{k}$ of Liouville domains satisfy the following axioms, for all equidimensional Liouville domains $(X,\lambda_X)$ and $(Y,\lambda_Y)$:
    \begin{description}
        \item[(Monotonicity)] If $X \longrightarrow Y$ is a generalized Liouville embedding then $\cgh{k}(X) \leq \cgh{k}(Y)$.
        \item[(Conformality)] If $\alpha > 0$ then $\cgh{k}(X, \alpha \lambda_X) = \alpha \, \cgh{k}(X, \lambda_X)$.
        \item[(Nondecreasing)] $\cgh{1}(X) \leq \cgh{2}(X) \leq \cdots \leq +\infty$. 
        \item[(Reeb orbits)] If $\cgh{k}(X) < + \infty$, then $\cgh{k}(X) = \mathcal{A}(\gamma)$ for some Reeb orbit $\gamma$ which is contractible in $X$.
    \end{description}
\end{theorem}

The following lemma provides an alternative definition of $\cgh{k}$, in the spirit of \cite{floerApplicationsSymplecticHomology1994}.

\begin{lemma}
    \label{def:ck alternative}
    Let $(X,\lambda)$ be a nondegenerate Liouville domain such that $\pi_1(X) = 0$ and $c_1(TX)|_{\pi_2(X)} = 0$. Let $E \subset \C^n$ be a nondegenerate star-shaped domain and suppose that $\phi \colon E \longrightarrow X$ is a symplectic embedding. Consider the map
    \begin{IEEEeqnarray*}{c+x*}
        \begin{tikzcd}
            \homology{}{S^1}{}{S}{H}{(\varepsilon,a]}{n - 1 + 2k}(X) \ar[r, "\iota^a"] & \homology{}{S^1}{}{S}{H}{+}{n - 1 + 2k}(X) \ar[r, "\phi_!"] & \homology{}{S^1}{}{S}{H}{+}{n - 1 + 2k}(E)
        \end{tikzcd}
    \end{IEEEeqnarray*}
    Then, $\cgh{k}(X) = \inf \{ a > 0 \mid \phi_! \circ \iota^a \text{ is nonzero} \}$.
\end{lemma}
\begin{proof}
    For every $a \in \R$ consider the following commutative diagram:
    \begin{IEEEeqnarray*}{c+x*}
        \begin{tikzcd}
            \homology{}{S^1}{}{S}{H}{(\varepsilon, a]}{n - 1 + 2k}(X) \ar[r, "\iota^a_X"] \ar[d, swap, "\phi_!^a"] & \homology{}{S^1}{}{S}{H}{+}{n - 1 + 2k}(X) \ar[r, "U ^{k-1}_X"] \ar[d, "\phi_!"]       & \homology{}{S^1}{}{S}{H}{+}{n+1}(X) \ar[r, "\delta_X"] \ar[d, "\phi_!"]       & H_0(BS^1) \tensorpr H_{2n}(X,\del X) \ar[d, hook, two heads, "\id \tensorpr \rho"] \\
            \homology{}{S^1}{}{S}{H}{(\varepsilon, a]}{n - 1 + 2k}(E) \ar[r, swap, "\iota^a_E"]                    & \homology{}{S^1}{}{S}{H}{+}{n - 1 + 2k}(E) \ar[r, swap, hook, two heads, "U ^{k-1}_E"] & \homology{}{S^1}{}{S}{H}{+}{n+1}(E) \ar[r, swap, hook, two heads, "\delta_E"] & H_0(BS^1) \tensorpr H_{2n}(E,\del E)
        \end{tikzcd}
    \end{IEEEeqnarray*}
    By \cite[Proposition 3.1]{guttSymplecticCapacitiesPositive2018} and since $E$ is star-shaped, the maps $U_E$ and $\delta_E$ are isomorphisms. Since $\rho([X]) = [E]$, the map $\rho$ is an isomorphism. By definition, $\cgh{k}$ is the infimum over $a$ such that the top arrow is surjective. This condition is equivalent to $\phi_! \circ \iota^a_X$ being nonzero.
\end{proof}

The following computation will be useful to us in the proofs of \cref{lem:computation of cl,thm:my main theorem}.

\begin{lemma}[{\cite[Lemma 1.19]{guttSymplecticCapacitiesPositive2018}}]
    \label{lem:cgh of nondisjoint union of cylinders}
    $\cgh{k}(N^{2n}(\delta)) = \delta \, (k + n - 1)$.
\end{lemma}


We now consider other capacities which can be defined using $S^1$-equivariant symplectic homology.

\begin{definition}[{\cite[Section 2.5]{irieSymplecticHomologyFiberwise2021}}]
    \label{def:s1esh capacities}
    If $k \in \Z_{\geq 1}$ and $(X,\lambda)$ is a nondegenerate Liouville domain, the \textbf{$S^1$-equivariant symplectic homology capacities} of $X$, denoted $\csh{k}(X)$, are defined as follows. Consider the map
    \begin{IEEEeqnarray*}{c+x*}
        \iota^{a,\varepsilon} \circ \alpha^{-1} \colon H_\bullet(BS^1;\Q) \otimes H_\bullet(X, \partial X;\Q) \longrightarrow \homology{}{S^1}{}{S}{H}{a}{}(X)
    \end{IEEEeqnarray*}
    from diagram \eqref{eq:diagram for s1esh capacities}. Then, we define
    \begin{IEEEeqnarray*}{c+x*}
        \csh{k}(X) \coloneqq \inf \{ a > 0 \mid \iota^{a,\varepsilon} \circ \alpha^{-1}([\C P^{k-1}] \otimes [X]) = 0 \}.
    \end{IEEEeqnarray*}
\end{definition}

\begin{theorem}
    \label{prp:properties of s1esh capacities}
    The functions $\csh{k}$ of Liouville domains satisfy the following axioms, for all Liouville domains $(X,\lambda_X)$ and $(Y,\lambda_Y)$ of the same dimension:
    \begin{description}
        \item[(Monotonicity)] If $X \longrightarrow Y$ is a generalized Liouville embedding then $\csh{k}(X) \leq \csh{k}(Y)$.
        \item[(Conformality)] If $\mu > 0$ then $\csh{k}(X, \mu \lambda_X) = \mu \, \csh{k}(X, \lambda_X)$.
        \item[(Nondecreasing)] $\csh{1}(X) \leq \csh{2}(X) \leq \cdots \leq +\infty$. 
    \end{description}
\end{theorem}
\begin{proof}
    We prove monotonicity. Consider the following commutative diagram:
    \begin{IEEEeqnarray}{c+x*}
        \plabel{eq:s1eshc diagram}
        \begin{tikzcd}
            H_\bullet(BS^1;\Q) \otimes H_\bullet(Y, \partial Y;\Q) \ar[d, swap, "\id \otimes \rho"] & \homology{}{S^1}{}{S}{H}{\varepsilon}{}(Y) \ar[l, swap, hook', two heads, "\alpha_Y"] \ar[r, "\iota^{a, \varepsilon}_Y"] \ar[d, "\phi_!^\varepsilon"] & \homology{}{S^1}{}{S}{H}{a}{}(Y) \ar[d, "\phi^a_!"] \\
            H_\bullet(BS^1;\Q) \otimes H_\bullet(X, \partial X;\Q)                                  & \homology{}{S^1}{}{S}{H}{\varepsilon}{}(X) \ar[l, hook', two heads, "\alpha_X"] \ar[r, swap, "\iota^{a, \varepsilon}_X"]                              & \homology{}{S^1}{}{S}{H}{a}{}(X)
        \end{tikzcd}
    \end{IEEEeqnarray}
    If $\iota_Y^{a,\varepsilon} \circ \alpha_Y^{-1}([\C P^{k-1}] \otimes [Y]) = 0$, then
    \begin{IEEEeqnarray*}{rCls+x*}
        \IEEEeqnarraymulticol{3}{l}{\iota_X^{a,\varepsilon} \circ \alpha_X^{-1}([\C P^{k-1}] \otimes [X])} \\
        \quad & = & \iota_X^{a,\varepsilon} \circ \alpha_X^{-1} \circ (\id \otimes \rho)([\C P^{k-1}] \otimes [Y]) & \quad [\text{since $\rho([Y]) = [X]$}] \\
              & = & \phi_! \circ \iota_Y^{a,\varepsilon} \circ \alpha_{Y}^{-1} ([\C P^{k-1}] \otimes [Y])          & \quad [\text{by diagram \eqref{eq:s1eshc diagram}}] \\
              & = & 0                                                                                              & \quad [\text{by assumption}].
    \end{IEEEeqnarray*}

    To prove conformality, choose $\varepsilon > 0$ such that $\varepsilon, \mu \varepsilon < \min \operatorname{Spec}(\partial X, \lambda|_{\partial X})$. Since the diagram
    \begin{IEEEeqnarray*}{c+x*}
        \begin{tikzcd}
            H_\bullet(BS^1;\Q) \otimes H_\bullet(X, \partial X;\Q) \ar[d, equals] & \homology{}{S^1}{}{S}{H}{\varepsilon}{}(X, \lambda) \ar[d, equals] \ar[l, swap, hook', two heads, "\alpha_{\lambda}"] \ar[r, "\iota^{a, \varepsilon}_\lambda"]            & \homology{}{S^1}{}{S}{H}{a}{}(X, \lambda) \ar[d, equals] \\
            H_\bullet(BS^1;\Q) \otimes H_\bullet(X, \partial X;\Q)                & \homology{}{S^1}{}{S}{H}{\mu \varepsilon}{}(X, \mu \lambda) \ar[l, hook', two heads, "\alpha_{\mu \lambda}"] \ar[r, swap, "\iota^{\mu a, \mu \varepsilon}_{\mu \lambda}"] & \homology{}{S^1}{}{S}{H}{\mu a}{}(X, \mu \lambda)
        \end{tikzcd}
    \end{IEEEeqnarray*}
    commutes (by \cite[Proposition 3.1]{guttSymplecticCapacitiesPositive2018}), the result follows.

    To prove the nondecreasing property, note that if $\iota^{a,\varepsilon} \circ \alpha^{-1}([\C P ^{k}] \otimes [X]) = 0$, then
    \begin{IEEEeqnarray*}{rCls+x*}
        \IEEEeqnarraymulticol{3}{l}{\iota^{a,\varepsilon} \circ \alpha^{-1}([\C P ^{k-1}] \otimes [X])}\\
        \quad & = & \iota^{a,\varepsilon} \circ \alpha^{-1} \circ U ([\C P ^{k}] \otimes [X])     & \quad [\text{since $U([\C P^k] \otimes [X]) = [\C P^{k-1}] \otimes [X]$}] \\
              & = & U^{a} \circ \iota^{a,\varepsilon} \circ \alpha^{-1} ([\C P ^{k}] \otimes [X]) & \quad [\text{since $\iota^{a,\varepsilon}$ and $\alpha$ commute with $U$}] \\
              & = & 0                                                                             & \quad [\text{by assumption}].                                        & \qedhere
    \end{IEEEeqnarray*}
\end{proof}

\begin{theorem}
    \label{thm:ghc and s1eshc}
    If $(X, \lambda)$ is a Liouville domain, then
    \begin{enumerate}
        \item \label{thm:comparison cgh csh 1} $\cgh{k}(X) \leq \csh{k}(X)$;
        \item \label{thm:comparison cgh csh 2} $\cgh{k}(X) = \csh{k}(X)$ provided that $X$ is star-shaped.
    \end{enumerate}
\end{theorem}
\begin{proof}
    By \cref{lem:can prove ineqs for ndg}, we may assume that $X$ is nondegenerate. Since
    \begin{IEEEeqnarray*}{rCls+x*}
        \IEEEeqnarraymulticol{3}{l}{\iota^{a,\varepsilon} \circ \alpha^{-1}([\C P ^{k-1}] \otimes [X]) = 0}\\
        \quad & \Longleftrightarrow & \alpha^{-1}([\C P ^{k-1}] \otimes [X]) \in \ker \iota^{a,\varepsilon}   & \quad [\text{by definition of kernel}] \\
        \quad & \Longleftrightarrow & \alpha^{-1}([\C P ^{k-1}] \otimes [X]) \in \img \delta^a_0              & \quad [\text{since the top row of \eqref{eq:diagram for s1esh capacities} is exact}] \\
        \quad & \Longleftrightarrow & [\C P ^{k-1}] \otimes [X] \in \img (\alpha \circ \delta^a_0)            & \quad [\text{by definition of image}] \\
        \quad & \Longleftrightarrow & [\C P ^{k-1}] \otimes [X] \in \img (\delta \circ \iota^a)               & \quad [\text{since diagram \eqref{eq:diagram for s1esh capacities} commutes}] \\
        \quad & \Longrightarrow     & [\mathrm{pt}] \otimes [X] \in \img (U^{k-1} \circ \delta \circ \iota^a) & \quad [\text{since $U^{k-1}([\C P ^{k-1}] \otimes [X]) = [\mathrm{pt}] \otimes [X]$}] \\
        \quad & \Longleftrightarrow & [\mathrm{pt}] \otimes [X] \in \img (\delta \circ U^{k-1} \circ \iota^a) & \quad [\text{since $\delta$ and $U$ commute}],
    \end{IEEEeqnarray*}
    we have that $\cgh{k}(X) \leq \csh{k}(X)$. If $X$ is a star-shaped domain, we can view the maps of the computation above as being the maps in diagram \eqref{eq:diagram for s1esh capacities case ss}, i.e. they are defined in a specific degree. In this case, $U^{k-1} \colon H_{2k-2}(BS^1) \otimes H_{2n}(X, \partial X) \longrightarrow H_{0}(BS^1) \otimes H_{2n}(X, \partial X)$ is an isomorphism, and therefore the implication in the previous computation is actually an equivalence.
\end{proof}

\begin{remark}
    The capacities $\cgh{k}$ and $\csh{k}$ are defined in terms of a certain homology class being in the kernel or in the image of a map with domain or target the $S^1$-equivariant symplectic homology. Other authors have constructed capacities in an analogous manner, for example Viterbo \cite[Definition 2.1]{viterboSymplecticTopologyGeometry1992} and \cite[Section 5.3]{viterboFunctorsComputationsFloer1999}, Schwarz \cite[Definition 2.6]{schwarzActionSpectrumClosed2000} and Ginzburg--Shon \cite[Section 3.1]{ginzburgFilteredSymplecticHomology2018}.
\end{remark}

\section{McDuff--Siegel capacities}

We now define the McDuff--Siegel capacities. These will assist us in our goal of proving \cref{conj:the conjecture} (at least in particular cases) because they can be compared with the Lagrangian capacity (\cref{thm:lagrangian vs g tilde}) and with the Gutt--Hutchings capacities (\cref{prp:g tilde and cgh}).

\begin{definition}[{\cite[Definition 3.3.1]{mcduffSymplecticCapacitiesUnperturbed2022}}]
    \label{def:g tilde}
    Let $(X,\lambda)$ be a nondegenerate Liouville domain. For $\ell, k \in \Z_{\geq 1}$, we define the \textbf{McDuff--Siegel capacities} of $X$, denoted $\tilde{\mathfrak{g}}^{\leq \ell}_k(X)$, as follows. Choose $x \in \itr X$ and $D$ a symplectic divisor at $x$. Then,
    \begin{IEEEeqnarray*}{c+x*}
        \tilde{\mathfrak{g}}^{\leq \ell}_k(X) \coloneqq \sup_{J \in \mathcal{J}(X,D)} \mathop{\inf\vphantom{\sup}}_{\Gamma_1, \ldots, \Gamma_p} \sum_{i=1}^{p} \mathcal{A}(\Gamma_i),
    \end{IEEEeqnarray*}
    where the infimum is over tuples of Reeb orbits $\Gamma_1, \ldots, \Gamma_p$ such that there exist integers $k_1, \ldots, k_p \geq 1$ with
    \begin{IEEEeqnarray}{c+x*}
        \phantomsection\label{eq:g tilde two definitions conditions}
        \sum_{i=1}^{p} \# \Gamma_i \leq \ell, \qquad \sum_{i=1}^{p} k_i \geq k, \qquad \bigproduct_{i=1}^{p} \mathcal{M}_X^J(\Gamma_i)\p{<}{}{\mathcal{T}^{(k_i)}x} \neq \varnothing.
    \end{IEEEeqnarray}
\end{definition}

The following theorem shows that the definition of $\tilde{\mathfrak{g}}^{\leq \ell}_k$ we give in \cref{def:g tilde} and the one given in \cite[Definition 3.3.1]{mcduffSymplecticCapacitiesUnperturbed2022} are equal.

\begin{theorem}[{\cite[Remark 3.1.2]{mcduffSymplecticCapacitiesUnperturbed2022}}]
    \label{thm:g tilde two definitions}
    If $(X, \lambda)$ is a nondegenerate Liouville domain, $\ell, k \in \Z_{\geq 1}$, $x \in \itr X$ and $D$ is a symplectic divisor through $x$, then
    \begin{IEEEeqnarray*}{c+x*}
        \tilde{\mathfrak{g}}^{\leq \ell}_k(X) = \sup_{J \in \mathcal{J}(X,D)} \mathop{\inf\vphantom{\sup}}_{\Gamma} \mathcal{A}(\Gamma),
    \end{IEEEeqnarray*}
    where the infimum is taken over tuples of Reeb orbits $\Gamma = (\gamma_1, \ldots, \gamma_p)$ such that $p \leq \ell$ and $\overline{\mathcal{M}}^{J}_{X}(\Gamma)\p{<}{}{\mathcal{T}^{(k)}x} \neq \varnothing$.
\end{theorem}
\begin{proof}
    $(\geq)$: Let $\Gamma_1, \ldots, \Gamma_p$ and $k_1, \ldots, k_p$ be as in \eqref{eq:g tilde two definitions conditions}. We wish to show that there exists a tuple of Reeb orbits $\Gamma$ such that
    \begin{IEEEeqnarray*}{c+x*}
        \# \Gamma \leq \ell, \qquad \mathcal{A}(\Gamma) \leq \sum_{i=1}^{p} \mathcal{A}(\Gamma_i), \qquad \overline{\mathcal{M}}_X^J(\Gamma)\p{<}{}{\mathcal{T}^{(k)}x} \neq \varnothing.
    \end{IEEEeqnarray*}
    By \cref{rmk:compactifications with tangency}, the tuple $\Gamma = \Gamma_1 \cup \cdots \cup \Gamma_p$ is as desired.

    $(\leq)$: Let $\Gamma^+$ be a tuple of Reeb orbits such that $\# \Gamma^+ \leq \ell$ and $\overline{\mathcal{M}}^{J}_{X}(\Gamma^+)\p{<}{}{\mathcal{T}^{(k)}x} \neq \varnothing$. We wish to show that there exist tuples of Reeb orbits $\Gamma^-_1, \ldots, \Gamma^-_p$ and numbers $k_1, \ldots, k_p$ satisfying \eqref{eq:g tilde two definitions conditions} and
    \begin{IEEEeqnarray*}{c+x*}
        \sum_{i=1}^{p} \mathcal{A}(\Gamma_i) \leq \mathcal{A}(\Gamma).
    \end{IEEEeqnarray*}
    Choose $F = (F^1, \ldots, F^N) \in \overline{\mathcal{M}}^J_X(\Gamma^+)\p{<}{}{\mathcal{T}^{(k)}x}$ and let $C$ be the component of $F$ which inherits the constraint $\p{<}{}{\mathcal{T}^{(k)}x}$.

    We prove the result in the case where $C$ is nonconstant. In this case, $C \in \mathcal{M}^J_X(\Gamma^-)\p{<}{}{\mathcal{T}^{(k)}x}$ for some tuple of Reeb orbits $\Gamma^-$. By \cref{lem:action energy for holomorphic}, $\mathcal{A}(\Gamma^-) \leq \mathcal{A}(\Gamma^+)$. We show that $\# \Gamma^- \leq \# \Gamma^+ \leq \ell$. Let $\mathbf{n}$ be the set of nodal points of $C$. Since the graph of $F$ is a tree, for every $\gamma \in \Gamma^+$ there exists a unique $f(\gamma) \in \Gamma^- \cup \mathbf{n}$ such that the subtree of $F$ emanating from $C$ at $f(\gamma)$ is positively asymptotic to $\gamma$. By the maximum principle (\cref{thm:maximum principle holomorphic}), $f \colon \Gamma^+ \longrightarrow \Gamma^- \cup \mathbf{n}$ is surjective, and therefore $\# \Gamma^- \leq \# \Gamma^+ \leq \ell$.

    We prove the result in the case where $C$ is constant. Let $C_1, \ldots, C_p$ be the nonconstant components near $C$ as in \cref{rmk:compactifications with tangency}. There exist tuples of Reeb orbits $\Gamma_1^-, \ldots, \Gamma_p^-$ and $k_1, \ldots, k_p \in \Z_{\geq 1}$ such that
    \begin{IEEEeqnarray*}{c+x*}
        \sum_{i=1}^{p} \mathcal{A}(\Gamma_i^-) \leq \mathcal{A}(\Gamma^+), \qquad \sum_{i=1}^{p} k_i \geq k, \qquad C_i \in \mathcal{M}^J_X(\Gamma_i^-)\p{<}{}{\mathcal{T}^{(k_i)}x} \neq \varnothing.
    \end{IEEEeqnarray*}
    By a reasoning similar to the previous case, $\sum_{i=1}^{p} \# \Gamma_i^- \leq \# \Gamma^+ \leq \ell$.
\end{proof}

\begin{remark}
    \phantomsection\label{cor:g tilde 1}
    If $(X, \lambda)$ is a nondegenerate Liouville domain, $k \in \Z_{\geq 1}$, $x \in \itr X$ and $D$ is a symplectic divisor through $x$, then
    \begin{IEEEeqnarray*}{c+x*}
        \tilde{\mathfrak{g}}^{\leq 1}_k(X) = \sup_{J \in \mathcal{J}(X,D)} \mathop{\inf\vphantom{\sup}}_{\gamma} \mathcal{A}(\gamma),
    \end{IEEEeqnarray*}
    where the infimum is over Reeb orbits $\gamma$ such that $\mathcal{M}^J_X(\gamma)\p{<}{}{\mathcal{T}^{(k)}x} \neq \varnothing$.
\end{remark}

\begin{theorem}[{\cite[Theorem 3.3.2]{mcduffSymplecticCapacitiesUnperturbed2022}}]
    \label{thm:properties of g tilde}
    The functions $\tilde{\mathfrak{g}}^{\leq \ell}_k$ are independent of the choices of $x$ and $D$ and satisfy the following properties, for all nondegenerate Liouville domains $(X,\lambda_X)$ and $(Y,\lambda_Y)$ of the same dimension:
    \begin{description}
        \item[(Monotonicity)] If $X \longrightarrow Y$ is a generalized Liouville embedding then $\tilde{\mathfrak{g}}^{\leq \ell}_k(X) \leq \tilde{\mathfrak{g}}^{\leq \ell}_k(Y)$.
        \item[(Conformality)] If $\alpha > 0$ then $\tilde{\mathfrak{g}}^{\leq \ell}_k(X, \alpha \lambda_X) = \alpha \, \tilde{\mathfrak{g}}^{\leq \ell}_k(X, \lambda_X)$.
        \item[(Nondecreasing)] $\tilde{\mathfrak{g}}^{\leq \ell}_1(X) \leq \tilde{\mathfrak{g}}^{\leq \ell}_{2}(X) \leq \cdots \leq +\infty$.
    \end{description}
\end{theorem}

We now state a result comparing the McDuff--Siegel capacities and the Gutt--Hutchings capacities. We will later apply this result to show that $c_L(X_{\Omega}) = \delta_\Omega$ for every $4$-dimensional convex toric domain $X_{\Omega}$ (\cref{lem:computation of cl}).

\begin{proposition}[{\cite[Proposition 5.6.1]{mcduffSymplecticCapacitiesUnperturbed2022}}]
    \label{prp:g tilde and cgh}
    If $X_{\Omega}$ is a $4$-dimensional convex toric domain then
    \begin{IEEEeqnarray*}{c+x*}
        \tilde{\mathfrak{g}}^{\leq 1}_k(X_\Omega) = \cgh{k}(X_\Omega).
    \end{IEEEeqnarray*}
\end{proposition}

Finally, we state two stabilization results which we will use in \cref{sec:augmentation map of an ellipsoid}.

\begin{lemma}[{\cite[Lemma 3.6.2]{mcduffSymplecticCapacitiesUnperturbed2022}}]
    \label{lem:stabilization 1}
    Let $(X, \lambda)$ be a Liouville domain. For any $c, \varepsilon \in \R_{> 0}$, there is a subdomain with smooth boundary $\tilde{X} \subset X \times B^2(c)$ such that:
    \begin{enumerate}
        \item The Liouville vector field $Z_{\tilde{X}} = Z_{X} + Z_{B^2(c)}$ is outwardly transverse along $\partial \tilde{X}$.
        \item $X \times \{0\} \subset \tilde{X}$ and the Reeb vector field of $\partial \tilde{X}$ is tangent to $\partial X \times \{0\}$.
        \item Any Reeb orbit of the contact form $(\lambda + \lambda_0)|_{\partial \tilde{X}}$ (where $\lambda_0 = 1/2 (x \edv y - y \edv x)$) with action less than $c - \varepsilon$ is entirely contained in $\partial X \times \{0\}$ and has normal Conley--Zehnder index equal to $1$.
    \end{enumerate}
\end{lemma}

\begin{lemma}[{\cite[Lemma 3.6.3]{mcduffSymplecticCapacitiesUnperturbed2022}}]
    \label{lem:stabilization 2}
    Let $X$ be a Liouville domain, and let $\tilde{X}$ be a smoothing of $X \times B^2(c)$ as in \cref{lem:stabilization 1}.
    \begin{enumerate}
        \item Let $J \in \mathcal{J}(\tilde{X})$ be a cylindrical almost complex structure on the completion of $\tilde{X}$ for which $\hat{X} \times \{0\}$ is $J$-holomorphic. Let $C$ be an asymptotically cylindrical $J$-holomorphic curve in $\hat{X}$, all of whose asymptotic Reeb orbits are nondegenerate and lie in $\partial X \times \{0\}$ with normal Conley--Zehnder index $1$. Then $C$ is either disjoint from the slice $\hat{X} \times \{0\}$ or entirely contained in it.
        \item Let $J \in \mathcal{J}(\partial \tilde{X})$ be a cylindrical almost complex structure on the symplectization of $\partial \tilde{X}$ for which $\R \times \partial X \times \{0\}$ is $J$-holomorphic. Let $C$ be an asymptotically cylindrical $J$-holomorphic curve in $\R \times \partial \tilde{X}$, all of whose asymptotic Reeb orbits are nondegenerate and lie in $\partial X \times \{0\}$ with normal Conley--Zehnder index $1$. Then $C$ is either disjoint from the slice $\R \times \partial X \times \{0\}$ or entirely contained in it. Moreover, only the latter is possible if $C$ has at least one negative puncture. 
    \end{enumerate}
\end{lemma}

\section{Computations not requiring contact homology}


We now state and prove one of our main theorems, which is going to be a key step in proving that $c_L(X_{\Omega}) = \delta_{\Omega}$. The proof uses techniques similar to those used in the proof of \cite[Theorem 1.1]{cieliebakPuncturedHolomorphicCurves2018}.

\begin{theorem}
    \label{thm:lagrangian vs g tilde}
    If $(X, \lambda)$ is a Liouville domain then 
    \begin{IEEEeqnarray*}{c+x*}
        c_L(X) \leq \inf_k^{} \frac{\tilde{\mathfrak{g}}_k^{\leq 1}(X)}{k}.
    \end{IEEEeqnarray*}
\end{theorem}
\begin{proof}
    By \cref{lem:can prove ineqs for ndg}, we may assume that $X$ is nondegenerate. Let $k \in \Z_{\geq 1}$ and $L \subset \itr X$ be an embedded Lagrangian torus. We wish to show that for every $\varepsilon > 0$ there exists $\sigma \in \pi_2(X,L)$ such that $0 < \omega(\sigma) \leq \tilde{\mathfrak{g}}_k^{\leq 1}(X) / k + \varepsilon$. Define
    \begin{IEEEeqnarray*}{rCls+x*}
        a      & \coloneqq & \tilde{\mathfrak{g}}_k^{\leq 1}(X), \\
        K_1    & \coloneqq & \ln(2), \\
        K_2    & \coloneqq & \ln(1 + a / \varepsilon k), \\
        K      & \coloneqq & \max \{K_1, K_2\}, \\
        \delta & \coloneqq & e^{-K}, \\
        \ell_0 & \coloneqq & a / \delta.
    \end{IEEEeqnarray*}

    By \cref{lem:geodesics lemma CM abs} and the Lagrangian neighbourhood theorem, there exists a Riemannian metric $g$ on $L$ and a symplectic embedding $\phi \colon D^*L \longrightarrow X$ such that $\phi(D^*L) \subset \itr X$, $\phi|_L = \id_L$ and such that if $\gamma$ is a closed geodesic in $L$ with length $\ell(\gamma) \leq \ell_0$ then $\gamma$ is noncontractible, nondegenerate and satisfies $0 \leq \morse(\gamma) \leq n - 1$.

    Let $D^*_{\delta} L$ be the codisk bundle of radius $\delta$. Notice that $\delta$ has been chosen in such a way that the symplectic embedding $\phi \colon D^* L \longrightarrow X$ can be seen as an embedding like that of \cref{lem:energy wrt different forms}. We will now use the notation of \cref{sec:sft compactness}. Define symplectic cobordisms%
    \begin{IEEEeqnarray*}{rCl}
        (X^+, \omega^+) & \coloneqq & (X \setminus \phi(D^*_{\delta} L), \omega), \\
        (X^-, \omega^-) & \coloneqq & (D^*_{\delta} L, \edv \lambda_{T^* L}),
    \end{IEEEeqnarray*}
    which have the common contact boundary 
    \begin{IEEEeqnarray*}{c+x*}
        (M, \alpha) \coloneqq (S^*_{\delta} L, \lambda_{T^* L}).
    \end{IEEEeqnarray*}
    Here, it is implicit that we are considering the restriction of the form $\lambda_{T^*L}$ on $T^* L$ to $D^*_{\delta} L$ or $S^*_{\delta} L$. Then, $(X,\omega) = (X^-, \omega^-) \circledcirc (X^+, \omega^+)$. 
    Recall that there are piecewise smooth $2$-forms $\tilde{\omega} \in \Omega^2(\hat{X})$ and $\tilde{\omega}^{\pm} \in \Omega^2(\hat{X}^{\pm})$ which are given as in \cref{def:energy of a asy cylindrical holomorphic curve}. Choose $x \in \itr \phi(D^*_{\delta} L)$ and let $D \subset \phi(D^*_{\delta} L)$ be a symplectic divisor through $x$. Choose also generic almost complex structures
    \begin{IEEEeqnarray*}{rCls+x*}
        J_M & \in & \mathcal{J}(M), \\
        J^+ & \in & \mathcal{J}_{J_M}(X^+), \\
        J^- & \in & \mathcal{J}^{J_M}(X^-, D),
    \end{IEEEeqnarray*}
    and denote by $J_{\partial X} \in \mathcal{J}(\partial X)$ the ``restriction'' of $J^+$ to $\R \times \partial X$. Let $(J_t)_{t} \subset \mathcal{J}(X, D)$ be the corresponding neck stretching family of almost complex structures. Since $a = \tilde{\mathfrak{g}}_k^{\leq 1}(X)$ and by \cref{cor:g tilde 1}, for every $t$ there exists a Reeb orbit $\gamma_t$ in $\partial X = \partial^+ X^+$ and a $J_t$-holomorphic curve $u_t \in \mathcal{M}_X^{J_t}(\gamma_t)\p{<}{}{\mathcal{T}^{(k)}x}$ such that $\mathcal{A}(\gamma_t) \leq a$. Since $\partial X$ has nondegenerate Reeb orbits, there are only finitely many Reeb orbits in $\partial X$ with action less than $a$. Therefore, possibly after passing to a subsequence, we may assume that $\gamma_t \eqqcolon \gamma_0$ is independent of $t$.

    The curves $u_t$ satisfy the energy bound $E_{\tilde{\omega}}(u_t) \leq a$. By the SFT compactness theorem, the sequence $(u_t)_{t}$ converges to a holomorphic building 
    \begin{IEEEeqnarray*}{c+x*}
        F = (F^1, \ldots, F^{L_0-1}, F^{L_0}, F^{{L_0}+1}, \ldots, F^N) \in \overline{\mathcal{M}}_X^{(J_t)_{t}}(\gamma_0)\p{<}{}{\mathcal{T}^{(k)}x},
    \end{IEEEeqnarray*}
    where
    \begin{IEEEeqnarray*}{rCls+x*}
        (X^{\nu}, \omega^\nu, \tilde{\omega}^{\nu}, J^{\nu})
        & \coloneqq & 
        \begin{cases}
            (T^* L               , \edv \lambda_{T^* L}             , \tilde{\omega}^-           , J^-)            & \text{if } \nu = 1    , \\
            (\R \times M         , \edv(e^r \alpha)                 , \edv \alpha                , J_M)            & \text{if } \nu = 2    , \ldots, {L_0} - 1, \\
            (\hat{X} \setminus L , \hat{\omega}                     , \tilde{\omega}^+           , J^+)            & \text{if } \nu = {L_0}    , \\
            (\R \times \partial X, \edv (e^r \lambda|_{\partial X}) , \edv \lambda|_{\partial X} , J_{\partial X}) & \text{if } \nu = {L_0} + 1, \ldots, N    , \\
        \end{cases} \\
        (X^*, \omega^*, \tilde{\omega}^*, J^*) & \coloneqq & \bigcoproduct_{\nu = 1}^N (X^{\nu}, \omega^\nu, \tilde{\omega}^{\nu}, J^{\nu}),
    \end{IEEEeqnarray*}
    and $F^{\nu}$ is a $J^\nu$-holomorphic curve in $X^{\nu}$ with asymptotic Reeb orbits $\Gamma^{\pm}_{\nu}$ (see \cref{fig:holomorphic building in the proof}). The holomorphic building $F$ satisfies the energy bound
    \begin{IEEEeqnarray}{c+x*}
        \plabel{eq:energy of holo building in proof}
        E_{\tilde{\omega}^*}(F) \coloneqq \sum_{\nu = 1}^{N} E_{\tilde{\omega}^{\nu}}(F^{\nu}) \leq a.
    \end{IEEEeqnarray}

    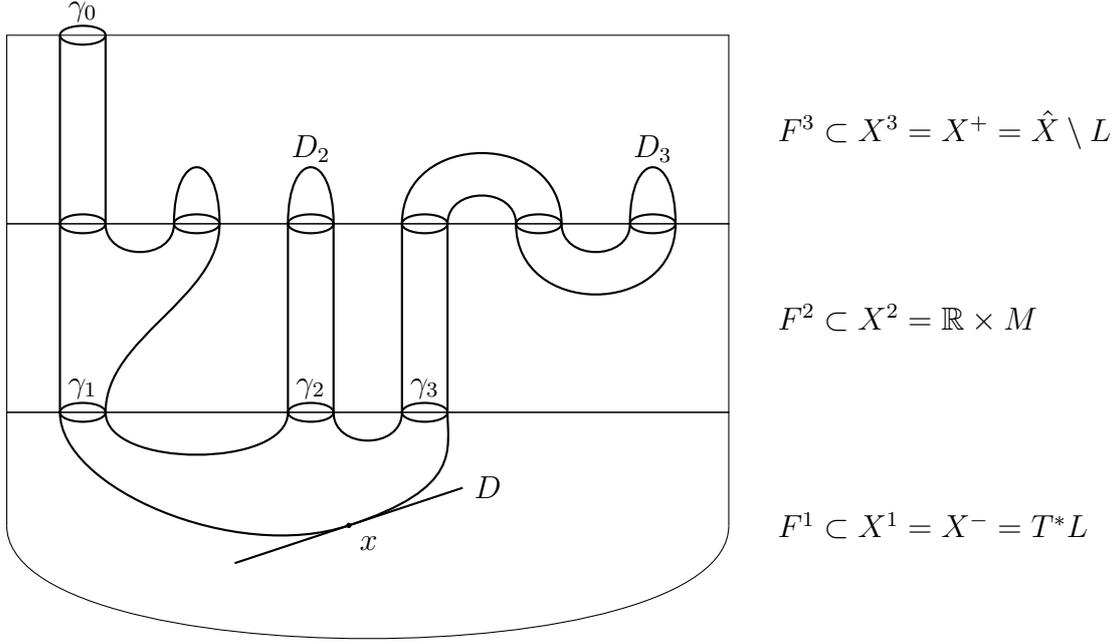
\begin{figure}[ht]
        \centering
        
        \begin{tikzpicture}
            [
                scale = 0.5,
                help/.style = {very thin, draw = black!50},
                curve/.style = {thick}
            ]
    
            \tikzmath{
                \rx = 0.6;
                \ry = 0.25;
            }
    
            \node[anchor=west] at (20, 13.5) {$F^3 \subset X^3 = X^+ = \hat{X} \setminus L$};
            \draw (0,6) rectangle (19,11);
    
            \node[anchor=west] at (20, 8.5) {$F^2 \subset X^2 = \R \times M$};
            \draw (0,11) rectangle (19,16);
    
            \node[anchor=west] at (20, 3) {$F^1 \subset X^1 = X^- = T^* L$};
            \draw (0,3) -- (0,6) -- (19,6) -- (19,3);
            \draw (0,3) .. controls (0,-1) and (19,-1) .. (19,3);
            
            \coordinate (G0) at ( 2,16);
            \coordinate (G1) at ( 2, 6);
            \coordinate (G2) at ( 8, 6);
            \coordinate (G3) at (11, 6);
            \coordinate (F1) at ( 2,11);
            \coordinate (F2) at ( 8,11);
            \coordinate (F3) at (11,11);
            \coordinate (F4) at ( 5,11);
            \coordinate (F5) at (14,11);
            \coordinate (F6) at (17,11);
            
            \coordinate (L) at (-\rx,0);
            \coordinate (R) at (+\rx,0);
    
            \coordinate (G0L) at ($ (G0) + (L) $);
            \coordinate (G1L) at ($ (G1) + (L) $);
            \coordinate (G2L) at ($ (G2) + (L) $);
            \coordinate (G3L) at ($ (G3) + (L) $);
            \coordinate (F1L) at ($ (F1) + (L) $);
            \coordinate (F2L) at ($ (F2) + (L) $);
            \coordinate (F3L) at ($ (F3) + (L) $);
            \coordinate (F4L) at ($ (F4) + (L) $);
            \coordinate (F5L) at ($ (F5) + (L) $);
            \coordinate (F6L) at ($ (F6) + (L) $);
    
            \coordinate (G0R) at ($ (G0) + (R) $);
            \coordinate (G1R) at ($ (G1) + (R) $);
            \coordinate (G2R) at ($ (G2) + (R) $);
            \coordinate (G3R) at ($ (G3) + (R) $);
            \coordinate (F1R) at ($ (F1) + (R) $);
            \coordinate (F2R) at ($ (F2) + (R) $);
            \coordinate (F3R) at ($ (F3) + (R) $);
            \coordinate (F4R) at ($ (F4) + (R) $);
            \coordinate (F5R) at ($ (F5) + (R) $);
            \coordinate (F6R) at ($ (F6) + (R) $);
    
            \coordinate (P) at (9,3);
            \coordinate (D) at (3,1);
    
            \draw[curve] (G0) ellipse [x radius = \rx, y radius = \ry] node[above = 1] {$\gamma_0$};
            \draw[curve] (G1) ellipse [x radius = \rx, y radius = \ry] node[above = 1] {$\gamma_1$};
            \draw[curve] (G2) ellipse [x radius = \rx, y radius = \ry] node[above = 1] {$\gamma_2$};
            \draw[curve] (G3) ellipse [x radius = \rx, y radius = \ry] node[above = 1] {$\gamma_3$};
            \draw[curve] (F1) ellipse [x radius = \rx, y radius = \ry];
            \draw[curve] (F2) ellipse [x radius = \rx, y radius = \ry];
            \draw[curve] (F3) ellipse [x radius = \rx, y radius = \ry];
            \draw[curve] (F4) ellipse [x radius = \rx, y radius = \ry];
            \draw[curve] (F5) ellipse [x radius = \rx, y radius = \ry];
            \draw[curve] (F6) ellipse [x radius = \rx, y radius = \ry];
    
            \fill (P) circle (2pt) node[anchor = north west] {$x$};
            \draw[curve] ($ (P) - (D) $) -- ( $ (P) + (D) $ ) node[anchor = west] {$D$};
    
            \draw[curve] (G1L) -- (G0L);
            \draw[curve] (F1R) -- (G0R);
            \draw[curve] (G2L) -- (F2L);
            \draw[curve] (G2R) -- (F2R);
            \draw[curve] (G3L) -- (F3L);
            \draw[curve] (G3R) -- (F3R);
    
            \draw[curve] (F4L) .. controls ($ (F4L) + (0,2) $) and ($ (F4R) + (0,2) $) .. (F4R);
            \draw[curve] (F2L) .. controls ($ (F2L) + (0,2) $) and ($ (F2R) + (0,2) $) .. (F2R);
            \draw[curve] (F6L) .. controls ($ (F6L) + (0,2) $) and ($ (F6R) + (0,2) $) .. (F6R);
            
            \draw[curve] (F3R) .. controls ($ (F3R) + (0,1) $) and ($ (F5L) + (0,1) $) .. (F5L);
            \draw[curve] (F5R) .. controls ($ (F5R) - (0,1) $) and ($ (F6L) - (0,1) $) .. (F6L);
            
            \draw[curve] (F3L) .. controls ($ (F3L) + (0,2.5) $) and ($ (F5R) + (0,2.5) $) .. (F5R);
            \draw[curve] (F5L) .. controls ($ (F5L) - (0,2.5) $) and ($ (F6R) - (0,2.5) $) .. (F6R);
            
            \draw[curve] (F1R) .. controls ($ (F1R) - (0,1) $) and ($ (F4L) - (0,1) $) .. (F4L);
            \draw[curve] (G1R) .. controls ($ (G1R) + (0,2) $) and ($ (F4R) - (0,2) $) .. (F4R);
    
            \draw[curve] (G1R) .. controls ($ (G1R) - (0,1.5) $) and ($ (G2L) - (0,1.5) $) .. (G2L);
            \draw[curve] (G2R) .. controls ($ (G2R) - (0,1) $) and ($ (G3L) - (0,1) $) .. (G3L);
    
            \draw[curve] (G1L) .. controls ($ (G1L) - (0,2) $) and ($ (P) - (D) $) .. (P);
            \draw[curve] (G3R) .. controls ($ (G3R) - (0,1) $) and ($ (P) + (D) $) .. (P);
    
            \node at ($ (F2) + (0,2) $) {$D_2$}; 
            \node at ($ (F6) + (0,2) $) {$D_3$}; 
        \end{tikzpicture}
    
        \caption{The holomorphic building $F = (F^1, \ldots, F^N)$ in the case ${L_0} = N = p = 3$}
        \label{fig:holomorphic building in the proof}
    \end{figure}

    Moreover, by \cref{lem:no nodes}, $F$ has no nodes. Let $C$ be the component of $F$ in $X^-$ which carries the tangency constraint $\p{<}{}{\mathcal{T}^{(k)}x}$. Then, $C$ is positively asymptotic to Reeb orbits $(\gamma_1, \ldots, \gamma_p)$ of $M$. For $\mu = 1, \ldots, p$, let $C_\mu$ be the subtree emanating from $C$ at $\gamma_\mu$. For exactly one $\mu = 1, \ldots, p$, the top level of the subtree $C_\mu$ is positively asymptotic to $\gamma_0$, and we may assume without loss of generality that this is true for $\mu = 1$. By the maximum principle, $C_\mu$ has a component in $X^{L_0} = \hat{X} \setminus L$ for every $\mu = 2, \ldots, p$. Also by the maximum principle, there do not exist components of $C_\mu$ in $X^{L_0} = \hat{X} \setminus L$ which intersect $\R_{\geq 0} \times \partial X$ or components of $C_\mu$ in the top symplectization layers $X^{{L_0}+1}, \ldots, X^N$, for every $\mu = 2, \ldots, p$.
    
    We claim that if $\gamma$ is a Reeb orbit in $M$ which is an asymptote of $F^\nu$ for some $\nu = 2,\ldots,{L_0}-1$, then $\mathcal{A}(\gamma) \leq a$. To see this, notice that
    \begin{IEEEeqnarray*}{rCls+x*}
        a
        & \geq & E_{\tilde{\omega}^*}(F)           & \quad [\text{by Equation \eqref{eq:energy of holo building in proof}}] \\
        & \geq & E_{\tilde{\omega}^N}(F^N)         & \quad [\text{by monotonicity of $E$}] \\
        & \geq & (e^K - 1) \mathcal{A}(\Gamma^-_N) & \quad [\text{by \cref{lem:energy wrt different forms}}] \\
        & \geq & \mathcal{A}(\Gamma^-_N)           & \quad [\text{since $K \geq K_1$}] \\
        & \geq & \mathcal{A}(\Gamma^-_\nu)         & \quad [\text{by \cref{lem:action energy for holomorphic}}]
    \end{IEEEeqnarray*}
    for every $\nu = 2, \ldots, {L_0}-1$. Every such $\gamma$ has a corresponding geodesic in $L$ (which by abuse of notation we denote also by $\gamma$) such that $\ell(\gamma) = \mathcal{A}(\gamma)/\delta \leq a / \delta = \ell_0$. Hence, by our choice of Riemannian metric, the geodesic $\gamma$ is noncontractible, nondegenerate and such that $\morse(\gamma) \leq n - 1$. Therefore, the Reeb orbit $\gamma$ is noncontractible, nondegenerate and such that $\conleyzehnder(\gamma) \leq n - 1$.

    We claim that if $D$ is a component of $C_\mu$ for some $\mu = 2,\ldots,p$ and $D$ is a plane, then $D$ is in $X^{L_0} = \hat{X} \setminus L$. Assume by contradiction otherwise. Notice that since $D$ is a plane, $D$ is asymptotic to a unique Reeb orbit $\gamma$ in $M = S^*_{\delta} L$ with corresponding noncontractible geodesic $\gamma$ in $L$. We will derive a contradiction by defining a filling disk for $\gamma$. If $D$ is in a symplectization layer $\R \times S^*_\delta L$, then the map $\pi \circ D$, where $\pi \colon \R \times S^*_{\delta} L \longrightarrow L$ is the projection, is a filling disk for the geodesic $\gamma$. If $D$ is in the bottom level, i.e. $X^1 = T^*L$, then the map $\pi \circ D$, where $\pi \colon T^*L \longrightarrow L$ is the projection, is also a filling disk. This proves the claim.
    
    So, summarizing our previous results, we know that for every $\mu = 2,\ldots,p$ there is a holomorphic plane $D_\mu$ in $X^{L_0} \setminus (\R_{\geq 0} \times \partial X) = X \setminus L$. For each plane $D_\mu$ there is a corresponding disk in $X$ with boundary on $L$, which we denote also by $D_\mu$. It is enough to show that $E_{\omega}(D_{\mu_0}) \leq a/k + \varepsilon$ for some $\mu_0 = 2,\ldots,p$. By \cref{lem:punctures and tangency}, $p \geq k + 1 \geq 2$. By definition of average, there exists $\mu_0 = 2,\ldots,p$ such that
    \begin{IEEEeqnarray*}{rCls+x*}
        E_{\omega}(D_{\mu_0})
        & \leq & \frac{1}{p-1} \sum_{\mu=2}^{p} E_{\omega}(D_{\mu})                           & \quad [\text{by definition of average}] \\
        & =    & \frac{E_{\omega}(D_2 \cup \cdots \cup D_p)}{p-1}                             & \quad [\text{since energy is additive}] \\
        & \leq & \frac{e^K}{e^K - 1} \frac{E_{\tilde{\omega}}(D_2 \cup \cdots \cup D_p)}{p-1} & \quad [\text{by \cref{lem:energy wrt different forms}}] \\
        & \leq & \frac{e^K}{e^K - 1} \frac{a}{p-1}                                            & \quad [\text{by Equation \eqref{eq:energy of holo building in proof}}] \\
        & \leq & \frac{e^K}{e^K - 1} \frac{a}{k}                                              & \quad [\text{since $p \geq k + 1$}] \\
        & \leq & \frac{a}{k} + \varepsilon                                                    & \quad [\text{since $K \geq K_2$}].                                  & \qedhere
    \end{IEEEeqnarray*}
\end{proof}


\begin{theorem}
    \label{lem:computation of cl}
    If $X_{\Omega}$ is a $4$-dimensional convex toric domain then
    \begin{IEEEeqnarray*}{c+x*}
        c_L(X_{\Omega}) = \delta_\Omega.
    \end{IEEEeqnarray*}
\end{theorem}
\begin{proof}
    For every $k \in \Z_{\geq 1}$,
    \begin{IEEEeqnarray*}{rCls+x*}
        \delta_\Omega
        & \leq & c_P(X_{\Omega})                                         & \quad [\text{by \cref{lem:c square geq delta}}] \\
        & \leq & c_L(X_{\Omega})                                         & \quad [\text{by \cref{lem:c square leq c lag}}] \\
        & \leq & \frac{\tilde{\mathfrak{g}}^{\leq 1}_{k}(X_{\Omega})}{k} & \quad [\text{by \cref{thm:lagrangian vs g tilde}}] \\
        & =    & \frac{\cgh{k}(X_{\Omega})}{k}                           & \quad [\text{by \cref{prp:g tilde and cgh}}] \\
        & \leq & \frac{\cgh{k}(N(\delta_\Omega))}{k}                     & \quad [\text{$X_{\Omega}$ is convex, hence $X_{\Omega} \subset N(\delta_\Omega)$}] \\
        & =    & \frac{\delta_\Omega(k+1)}{k}                            & \quad [\text{by \cref{lem:cgh of nondisjoint union of cylinders}}].
    \end{IEEEeqnarray*}
    The result follows by taking the infimum over $k$.
\end{proof}


The proof of \cref{lem:computation of cl} suggests the following conjecture. Notice that \cref{thm:main theorem} implies \cref{conj:the conjecture}.

\begin{conjecture}
    \label{thm:main theorem}
    If $X$ is a Liouville domain, $\pi_1(X) = 0$ and $c_1(TX)|_{\pi_2(X)} = 0$, then%
    \begin{IEEEeqnarray*}{c+x*}
        c_L(X,\lambda) \leq \inf_k \frac{\cgh{k}(X,\lambda)}{k}.
    \end{IEEEeqnarray*}
\end{conjecture}
\begin{proof}[Proof attempt]
    By \cref{lem:can prove ineqs for ndg}, we may assume that $X$ is nondegenerate. Let $k \in \Z_{\geq 1}$ and $L \subset \itr X$ be an embedded Lagrangian torus. Let also $a > \cgh{k}(X)$. We wish to show that for every $\varepsilon > 0$ there exists $\sigma \in \pi_2(X,L)$ such that $0 < \omega(\sigma) \leq a / k + \varepsilon$. Start by replicating word by word the proof of \cref{thm:lagrangian vs g tilde} until the point where we choose $x \in \phi(D^*_{\delta} L)$. Instead of choosing $x$, choose a nondegenerate star-shaped domain $E \subset \C^n$ and an exact symplectic embedding $\varphi \colon E \longrightarrow X$ such that $\varphi(E) \subset \itr \phi(D^*_{\delta} L)$. Since $a > \cgh{k}(X)$ and by \cref{def:ck alternative}, the map
    \begin{IEEEeqnarray}{c+x*}
        \plabel{eq:nonzero map in proof of cl leq cgh}
        \begin{tikzcd}
            \homology{}{S^1}{}{S}{H}{(\varepsilon,a]}{n - 1 + 2k}(X) \ar[r, "\iota^a"] & \homology{}{S^1}{}{S}{H}{+}{n - 1 + 2k}(X) \ar[r, "\varphi_!"] & \homology{}{S^1}{}{S}{H}{+}{n - 1 + 2k}(E)
        \end{tikzcd}
    \end{IEEEeqnarray}
    is nonzero. Choose Hamiltonians
    \begin{IEEEeqnarray*}{rClCrClCs}
        H^+ \colon S^1 \times S^{2N+1} \times \hat{X}         & \longrightarrow & \R, & \quad & H^+ & \in & \mathcal{H}(X,N),      & \quad & (see \cref{def:hamiltonians}), \\
        H^- \colon S^1 \times S^{2N+1} \times \hat{X}         & \longrightarrow & \R, & \quad & H^- & \in & \mathcal{H}(X,E,N),    & \quad & (see \cref{def:stair hamiltonians}), \\
        H \colon \R \times S^1 \times S^{2N+1} \times \hat{X} & \longrightarrow & \R, & \quad & H   & \in & \mathcal{H}(H^+, H^-), & \quad & (see \cref{def:homotopy stair to admissible hamiltonian}).
    \end{IEEEeqnarray*}
    Choose also an almost complex structure
    \begin{IEEEeqnarray*}{rClCrClCs}
        J \colon S^1 \times S^{2N+1} \times \hat{X} & \longrightarrow & \End(T \hat{X}), & \quad & J & \in & \mathcal{J}(X, E, N), & \quad & (see \cref{def:stair acs}).
    \end{IEEEeqnarray*}
    The almost complex structure $J$ defines a neck stretching family of almost complex structures
    \begin{IEEEeqnarray*}{rClCrClCs}
        J_m \colon S^1 \times S^{2N+1} \times \hat{X} & \longrightarrow & \End(T \hat{X}), & \quad & J_m & \in & \mathcal{J}(X, E, N),
    \end{IEEEeqnarray*}
    for $m \in \Z_{\geq 1}$. Since the map \eqref{eq:nonzero map in proof of cl leq cgh} is nonzero and by definition of the Viterbo transfer map, if $N, H^\pm, H$ are chosen big enough (in the sense of the partial orders defined in \cref{sec:Floer homology,sec:viterbo transfer map of liouville embedding}) then for every $m$ there exist $(z^{\pm}_m, \gamma^{\pm}_m) \in \hat{\mathcal{P}}(H^{\pm})$ and a Floer trajectory $(w_m, u_m)$ with respect to $H, J_m$ from $(z^-_m, \gamma^-_m)$ to $(z^+_m, \gamma^+_m)$, such that
    \begin{enumerate}
        \item $\img \gamma^+_m$ is near $\partial X$ and $\mathcal{A}_{H^+}(z^+_m, \gamma^+_m) \leq a$;
        \item $\img \gamma^-_m$ is near \parbox{\widthof{$\partial X$}}{$\partial E$} and $\ind (z^-_m, \gamma^-_m) \geq n - 1 + 2k$.
    \end{enumerate}
    By \cref{lem:action energy for floer trajectories}, we have the energy bound $E(w_m, u_m) \leq a$. Possibly after passing to a subsequence, we may assume that $(z^{\pm}_m, \gamma^{\pm}_m)$ converges to $(z_0^{\pm}, \gamma^{\pm}_0) \in \hat{\mathcal{P}}(H^{\pm})$.

    Now we come to the first challenge of the proof. We would like to use an adaptation of the SFT compactness theorem to take the limit of the sequence $(w_m, u_m)_m$. We will assume that such a theorem can be proven, and that we get a resulting limit $F = (F^1, \ldots, F^N)$ as in the proof of \cref{thm:lagrangian vs g tilde}, but where each $F^{\nu} = (w^\nu, u^\nu) \colon \dot{\Sigma}^\nu \longrightarrow S^{2 N + 1} \times X^{\nu}$ is a solution of the parametrized Floer equation (\cref{def:floer trajectory abstract}). Let $C$ be the component of $F$ in $X^-$ which is negatively asymptotic to $(z_0^-, \gamma_0^-)$.
    
    Notice that near $X \setminus \phi(D^*_{\delta} L)$, the Hamiltonian $H$ is independent of $\hat{X}$. Therefore, in the intermediate symplectization levels (i.e. for $\nu = 2,\ldots,L-1$) the map $u^{\nu} \colon \dot{\Sigma}^{\nu} \longrightarrow X^{\nu}$ is $J^{\nu}_{w^{\nu}}$-holomorphic, where $J^{\nu}_{w^{\nu}} \colon \dot{\Sigma}^{\nu} \times X^{\nu} \longrightarrow \End(T X^{\nu})$ is a domain dependent almost complex structure obtained from composing an almost complex structure $J^{\nu} \colon \dot{\Sigma}^{\nu} \times S^{2 N + 1} \times X^{\nu} \longrightarrow \End(T X^{\nu})$ with $w^\nu$. Hence, as in the proof of \cref{thm:lagrangian vs g tilde}, the component $C$ has $p$ positive punctures asymptotic to Reeb orbits $(\gamma_1, \ldots, \gamma_p)$ and for every $\mu = 2, \ldots, p$ there is a disk $D_{\mu}$ in $X$ with boundary on $L$. 
    
    At this point, we need to show that $p \geq k + 1$, which brings us to the main difficulty in the proof. In the proof of \cref{thm:lagrangian vs g tilde}, we chose a generic almost complex structure so that $C$ would be regular. Then, the index formula for $C$ implied that $p \geq k + 1$ (see \cref{thm:transversality with tangency,lem:punctures and tangency simple,lem:punctures and tangency}). In line with this reasoning, we wish to show that $p \geq k + 1$ using the following computation:
    \begin{IEEEeqnarray*}{rCls+x*}
        0
        & \leq & \operatorname{ind}(C)                                                          \\
        & =    & (n - 3)(1 - p) + \sum_{\mu=1}^{p} \conleyzehnder(\gamma_\mu) - \ind(z^-_0, \gamma^-_0) \\
        & \leq & (n - 3)(1 - p) + \sum_{\mu=1}^{p} (n - 1) - (n - 1 + 2k)                         \\
        & =    & 2 (p - k - 1),
    \end{IEEEeqnarray*}
    where in the first line we would need to use a transversality theorem which applies to $C$, and in the second line we would need to use a Fredholm theory theorem which gives us the desired index formula for $C$. We point out a few difficulties that arise with this approach.
    \begin{enumerate}
        \item Because of the domain dependence of the almost complex structures and Hamiltonians, it is not clear how to choose the initial almost complex structure $J \colon S^1 \times S^{2N+1} \times \hat{X} \longrightarrow \End(T \hat{X})$ in such a way that the resulting almost complex structure $J^1 \colon \dot{\Sigma}^1 \times S^{2N+1} \times X^1 \longrightarrow \End(T X^1)$ is regular.
        \item We are working under the assumption that the analogue of the SFT compactness theorem which applies to solutions of the parametrized Floer equation produces a building $F$ whose symplectization levels are asymptotic to Reeb orbits. More specifically, this means that the gradient flow line in $S^{2N+1}$ corresponding to $C$ is not asymptotic at the punctures to critical points of $\tilde{f}_N$. Therefore, in this case the linearized operator corresponding to the gradient flow line equation on $S^{2N+1}$ will not be Fredholm.
        \item However, the assumption in the previous item could be wrong. Another reasonable possibility is that the analogue of the SFT compactness theorem which applies to solutions of the parametrized Floer equation produces a building $F$ whose bottom component is positively asymptotic to pairs $(z_\mu, \gamma_\mu)$, where $z_{\mu} \in S^{2N+1}$ is a critical point of $\tilde{f}_N$ and $\gamma_\mu$ is a Reeb orbit. In this case, one would expect that the relevant operator is Fredholm. However, the Morse index of the critical points $z_{\mu}$ would appear in the index formula, and the previous computation would no longer imply that $p \geq k + 1$.
    \end{enumerate}
    Finally, we point out that if $p \geq k + 1$, then by the same computation as in the proof of \cref{thm:lagrangian vs g tilde}, we have the desired energy bound
    \begin{IEEEeqnarray*}{c+x*}
        E_{\omega}(D_{\mu_0}) \leq \frac{a}{k} + \varepsilon
    \end{IEEEeqnarray*}
    for some $\mu_0 = 2, \ldots, p$. This finishes the proof attempt.
\end{proof}

\chapter{Contact homology}
\label{chp:contact homology}

\section{Assumptions on virtual perturbation scheme}
\label{sec:assumptions of virtual perturbation scheme}

In this chapter, we wish to use techniques from contact homology to prove \cref{conj:the conjecture}. Consider the proof of \cref{lem:computation of cl}: to prove the inequality $c_L(X_{\Omega}) \leq \delta_\Omega$, we needed to use the fact that $\tilde{\mathfrak{g}}^{\leq 1}_k(X_{\Omega}) \leq \cgh{k}(X_{\Omega})$ (which is true if $X_{\Omega}$ is convex and $4$-dimensional). Our approach here will be to consider the capacities $\mathfrak{g}^{\leq \ell}_{k}$ from \cite{siegelHigherSymplecticCapacities2020}, which satisfy $\tilde{\mathfrak{g}}^{\leq 1}_k(X) \leq {\mathfrak{g}}^{\leq 1}_k(X) = \cgh{k}(X)$. As we will see, $\mathfrak{g}^{\leq \ell}_{k}(X)$ is defined using the linearized contact homology of $X$, where $X$ is any nondegenerate Liouville domain. 

Very briefly, the linearized contact homology chain complex, denoted $CC(X)$, is generated by the good Reeb orbits of $\partial X$, and therefore maps whose domain is $CC(X)$ should count holomorphic curves which are asymptotic to Reeb orbits. The ``naive'' way to define such counts of holomorphic curves would be to show that they are the elements of a moduli space which is a compact, $0$-dimensional orbifold. However, there is the possibility that a curve is multiply covered. This means that in general it is no longer possible to show that the moduli spaces are transversely cut out, and therefore we do not have access to counts of moduli spaces of holomorphic curves (or at least not in the usual sense of the notion of signed count). In the case where the Liouville domain is $4$-dimensional, there exists the possibility of using automatic transversality techniques to show that the moduli spaces are regular. This is the approach taken by Wendl \cite{wendlAutomaticTransversalityOrbifolds2010}. Nelson \cite{nelsonAutomaticTransversalityContact2015}, Hutchings--Nelson \cite{hutchingsCylindricalContactHomology2016} and Bao--Honda \cite{baoDefinitionCylindricalContact2018} use automatic transversality to define cylindrical contact homology.

In order to define contact homology in more general contexts, one needs to replace the notion of count by a suitable notion of virtual count, which is obtained through a virtual perturbation scheme. This was done by Pardon \cite{pardonAlgebraicApproachVirtual2016,pardonContactHomologyVirtual2019} to define contact homology in greater generality. The theory of polyfolds by Hofer--Wysocki--Zehnder \cite{hoferPolyfoldFredholmTheory2021} can also be used to define virtual moduli counts. Alternative approaches using Kuranishi structures have been given by Ishikawa \cite{ishikawaConstructionGeneralSymplectic2018} and Bao--Honda \cite{baoSemiglobalKuranishiCharts2021}.

Unfortunately, linearized contact homology is not yet defined in the generality we need.
\begin{enumerate}
    \item In order to prove \cref{conj:the conjecture}, we only need the capacities $\mathfrak{g}^{\leq \ell}_k$ for $\ell = 1$. These are defined using the linearized contact homology (as a chain complex) and an augmentation map which counts curves satisfying a tangency constraint. As far as we know, the current work on defining virtual moduli counts does not yet deal with moduli spaces of curves satisfying tangency constraints.
    \item In addition to \cref{conj:the conjecture}, in this chapter we will also prove some properties of the capacities $\mathfrak{g}^{\leq \ell}_k$ for $\ell > 1$. The definition of these capacities for $\ell > 1$ requires the structure of an $\mathcal{L}_{\infty}$-algebra on the linearized contact homology as well as an $\mathcal{L}_{\infty}$-augmentation map counting curves which satisfy a tangency constraint.
\end{enumerate}
So, during this chapter, we will work under assumption that it is possible to define a virtual perturbation scheme which makes the invariants and maps described above well-defined (this is expected to be the case).

\begin{assumption}
    \label{assumption}
    We assume the existence of a virtual perturbation scheme which to every compactified moduli space $\overline{\mathcal{M}}$ of asymptotically cylindrical holomorphic curves (in a symplectization or in a Liouville cobordism, possibly satisfying a tangency constraint) assigns a virtual count $\#^{\mathrm{vir}} \overline{\mathcal{M}}$. We will assume in addition that the virtual perturbation scheme has the following properties.
    \begin{enumerate}
        \item If $\#^{\mathrm{vir}} \overline{\mathcal{M}} \neq 0$ then $\operatorname{virdim} \overline{\mathcal{M}} = 0$;
        \item If $\overline{\mathcal{M}}$ is transversely cut out then $\#^{\mathrm{vir}} \overline{\mathcal{M}} = \# \overline{\mathcal{M}}$. In particular, if $\overline{\mathcal{M}}$ is empty then $\#^{\mathrm{vir}} \overline{\mathcal{M}} = 0$;
        \item The virtual count of the boundary of a moduli space (defined as a sum of virtual counts of the moduli spaces that constitute the codimension one boundary strata) is zero. In particular, the expected algebraic identities ($\partial^2 = 0$ for differentials, $\varepsilon \circ \partial = 0$ for augmentations) hold, as well as independence of auxiliary choices of almost complex structure and symplectic divisor.
    \end{enumerate}
\end{assumption}

\section{\texorpdfstring{$\mathcal{L}_{\infty}$-}{L infinity }algebras}

In this section, we give a brief review of the algebraic definitions which will play a role. Our main reference is \cite[Section 2]{siegelHigherSymplecticCapacities2020}. The key definitions are that of $\mathcal{L}_{\infty}$-algebra (\cref{def:l infinity algebra}) and its associated bar complex (\cref{def:bar complex}). We start by defining the suspension of a graded vector space. The purpose of this definition is to define $\mathcal{L}_{\infty}$-algebras in such a way that the $\mathcal{L}_{\infty}$-relations do not have extra signs (these extra signs are ``absorbed'' by the degree shift in the suspension).

\begin{definition}
    Let $V = \bigoplus_{k \in \Z} V^k$ be a graded vector space over a field $K$. The \textbf{suspension} of $V$ is the graded vector space $V[+1] = \bigoplus_{k \in \Z} (V[+1])^k$ given by $(V[+1])^k = V^{k+1}$. Define $s \colon V \longrightarrow V[+1]$ to be the linear map of degree $-1$ given by $s(v) = v$.
\end{definition}

\begin{remark}
    We use the Koszul sign convention, i.e. if $f,g \colon V \longrightarrow V$ are linear maps and $x, y \in V$ then $(f \otimes g)(x \otimes y) = (-1)^{\deg(x) \deg(g)} f(x) \otimes g(y)$.
\end{remark}

\begin{definition}
    Let $k \in \Z_{\geq 1}$ and denote by $\operatorname{Sym}(k)$ the symmetric group on $k$ elements. Let $V$ be a vector field over a field $K$. We define an action of $\operatorname{Sym}(k)$ on $\bigotimes_{j=1}^{k} V$ as follows. For $\sigma \in \operatorname{Sym}(k)$ and $v_1, \ldots, v_k \in V$, let
    \begin{IEEEeqnarray*}{rCls+x*}
        \operatorname{sign}(\sigma, v_1, \ldots, v_k) & \coloneqq & (-1)^{\operatorname{sum} \{ \deg(v_i) \deg(v_j) \, \mid \, 1 \leq i < j \leq k , \sigma(i) > \sigma(j) \} }, \\
        \sigma \cdot (v_1 \otimes \cdots \otimes v_k) & \coloneqq & \operatorname{sign}(\sigma, v_1, \ldots, v_k) \, v_{\sigma(1)} \otimes \cdots \otimes v_{\sigma(k)}.
    \end{IEEEeqnarray*}
    Define $\bigodot_{j=1}^k V \coloneqq \bigotimes_{j=1}^{k} V / \operatorname{Sym}(k)$ and denote by $v_1 \odot \cdots \odot v_k$ the equivalence class of $v_1 \otimes \cdots \otimes v_k$.
\end{definition}

We come to the main definition of this section, which encodes the algebraic structure of linearized contact homology (see \cref{def:lch l infinity}).

\begin{definition}
    \label{def:l infinity algebra}
    An \textbf{$\mathcal{L}_{\infty}$-algebra} is a graded vector space $V = \bigoplus_{k \in \Z} V^k$ together with a family $\ell = (\ell^k)_{k \in \Z_{\geq 1}}$ of maps $\ell^k \colon \bigodot_{j=1}^{k} V[+1] \longrightarrow V[+1]$ of degree $1$, satisfying the \textbf{$\mathcal{L}_{\infty}$-relations}, i.e.
    \begin{IEEEeqnarray*}{l}
        0 = \sum_{k=1}^{n} \sum_{\sigma \in \operatorname{Sh}(k,n-k)} \operatorname{sign}(\sigma, s v_1, \ldots, s v_n) \\
        \hphantom{0 = \sum_{k=1}^{n} \sum_{\sigma \in \operatorname{Sh}(k,n-k)} \quad} \ell^{n-k+1} ( \ell^k ( s v_{\sigma(1)} \odot \cdots \odot s v_{\sigma(k)} ) \odot s v_{\sigma(k+1)} \odot \cdots \odot s v_{\sigma(n)} )
    \end{IEEEeqnarray*}
    for every $v_1,\ldots,v_n \in V$. Here, $\operatorname{Sh}(k,n-k) \subset \operatorname{Sym}(n)$ is the subgroup of permutations $\sigma$ such that $\sigma(1) < \cdots < \sigma(k)$ and $\sigma(k+1) < \cdots < \sigma(n)$.
\end{definition}

The definition of $\mathcal{L}_{\infty}$-algebra can be expressed more compactly via the notion of bar complex. Indeed, the family of maps $(\ell^k)_{k \in \Z_{\geq 1}}$ satisfies the $\mathcal{L}_{\infty}$-relations if and only if the map $\hat{\ell}$ defined below is a differential, i.e. $\hat{\ell} \circ \hat{\ell} = 0$.

\begin{definition}
    \label{def:bar complex}
    Let $(V,\ell)$ be an $\mathcal{L}_{\infty}$-algebra. The \textbf{bar complex} of $(V,\ell)$ is the vector space $\mathcal{B} V = \bigoplus_{k = 1}^{+\infty} \bigodot_{j=1}^k V[+1]$ together with the degree $1$ differential $\hat{\ell} \colon \mathcal{B} V \longrightarrow \mathcal{B} V$ given by
    \begin{IEEEeqnarray*}{rCl}
        \IEEEeqnarraymulticol{3}{l}{\hat{\ell}(v_1 \odot \cdots \odot v_n)}\\
        \quad & = & \sum_{k=1}^{n} \sum_{\sigma \in \operatorname{Sh}(k,n-k)} \operatorname{sign}(\sigma, v_1, \ldots, v_n) \, \ell^k ( v_{\sigma(1)} \odot \cdots \odot v_{\sigma(k)} ) \odot v_{\sigma(k+1)} \odot \cdots \odot v_{\sigma(n)}. 
    \end{IEEEeqnarray*}
\end{definition}


\begin{definition}
    Let $(V,\ell)$ be an $\mathcal{L}_{\infty}$-algebra. A \textbf{filtration} on $V$ is a family $(\mathcal{F}^{\leq a} V)_{a \in \R}$ of subspaces $\mathcal{F}^{\leq a} V \subset V$, satisfying the following properties:
    \begin{enumerate}
        \item if $a \leq b$ then $\mathcal{F}^{\leq a} V \subset \mathcal{F}^{\leq b} V$;
        \item $\bigcup_{a \in \R} \mathcal{F}^{\leq a} V = V$;
        \item $\ell^k( \mathcal{F}^{\leq a_1} V[+1] \odot \cdots \odot \mathcal{F}^{\leq a_k} V[+1] ) \subset \mathcal{F}^{\leq a_1 + \cdots + a_k} V[+1]$.
    \end{enumerate}
\end{definition}

\begin{definition}
    Let $(V, \ell)$ be an $\mathcal{L}_{\infty}$-algebra together with a filtration $(\mathcal{F}^{\leq a} V)_{a \in \R}$. The \textbf{induced filtration} on the bar complex is the family of complexes $(\mathcal{F}^{\leq a} \mathcal{B} V, \hat{\ell})_{a \in \R}$, where%
    \begin{IEEEeqnarray*}{c+x*}
        \mathcal{F}^{\leq a} \mathcal{B} V \coloneqq \bigoplus_{k=1}^{+\infty} \, \bigcup_{a_1 + \cdots + a_k \leq a} \, \bigodot_{j=1}^{k} \mathcal{F}^{\leq a_j} V[+1]
    \end{IEEEeqnarray*}
    and $\hat{\ell} \colon \mathcal{F}^{\leq a} \mathcal{B} V \longrightarrow \mathcal{F}^{\leq a} \mathcal{B} V$ is the restriction of $\hat{\ell} \colon \mathcal{B} V \longrightarrow \mathcal{B} V$.
\end{definition}

The linearized contact homology will have a filtration induced by the action of the Reeb orbits (see \cref{def:action filtration lch}). Also, the bar complex of any $\mathcal{L}_{\infty}$-algebra has a filtration by word length, which is defined below.

\begin{definition}
    \phantomsection\label{def:word length filtration}
    Let $(V, \ell)$ be an $\mathcal{L}_{\infty}$-algebra and consider its bar complex $(\mathcal{B}V, \hat{\ell})$. The \textbf{word length filtration} of $(\mathcal{B}V, \hat{\ell})$ is the family of complexes $(\mathcal{B}^{\leq m} V, \hat{\ell})_{m \in \Z_{\geq 1}}$, where $\mathcal{B}^{\leq m} V \coloneqq \bigoplus_{k=1}^{m} \bigodot_{j=1}^{k} V[+1]$ and $\hat{\ell} \colon \mathcal{B}^{\leq m} V \longrightarrow \mathcal{B}^{\leq m} V$ is the restriction of $\hat{\ell} \colon \mathcal{B}V \longrightarrow \mathcal{B}V$.
\end{definition}

\section{Contact homology}

In this section, we define the linearized contact homology of a nondegenerate Liouville domain $X$. This is the homology of a chain complex $CC(X)$, which is described in \cref{def:linearized contact homology}. This complex has additional structure, namely it is also an $\mathcal{L}_{\infty}$-algebra (\cref{def:lch l infinity}) and it admits a filtration by action (\cref{def:action filtration lch}). We also define an augmentation map (\cref{def:augmentation map}), which is necessary to define the capacities $\mathfrak{g}^{\leq \ell}_k$.

\begin{definition}
    Let $(M,\alpha)$ be a contact manifold and $\gamma$ be a Reeb orbit in $M$. We say that $\gamma$ is \textbf{bad} if $\conleyzehnder(\gamma) - \conleyzehnder(\gamma_0)$ is odd, where $\gamma_0$ is the simple Reeb orbit that corresponds to $\gamma$. We say that $\gamma$ is \textbf{good} if it is not bad.
\end{definition}

Since the parity of the Conley--Zehnder index of a Reeb orbit is independent of the choice of trivialization, the definition above is well posed.

\begin{definition}
    \label{def:linearized contact homology}
    If $(X,\lambda)$ is a nondegenerate Liouville domain, the \textbf{linearized contact homology chain complex} of $X$, denoted $CC(X)$, is a chain complex given as follows. First, let $CC(X)$ be the vector space over $\Q$ generated by the set of good Reeb orbits of $(\partial X, \lambda|_{\partial X})$. The differential of $CC(X)$, denoted $\partial$, is given as follows. Choose $J \in \mathcal{J}(X)$. If $\gamma$ is a good Reeb orbit of $\partial X$, we define
    \begin{IEEEeqnarray*}{c+x*}
        \partial \gamma = \sum_{\eta} \p{<}{}{\partial \gamma, \eta} \, \eta,
    \end{IEEEeqnarray*}
    where $\p{<}{}{\partial \gamma, \eta}$ is the virtual count (with combinatorial weights) of holomorphic curves in $\R \times \partial X$ with one positive asymptote $\gamma$, one negative asymptote $\eta$, and $k \geq 0$ extra negative asymptotes $\alpha_1,\ldots,\alpha_k$ (called \textbf{anchors}), each weighted by the count of holomorphic planes in $\hat{X}$ asymptotic to $\alpha_j$ (see \cref{fig:differential of lch}).
\end{definition}

\begin{figure}[htp]
    \centering
    
    \begin{tikzpicture}
        [
            scale = 0.5,
            help/.style = {very thin, draw = black!50},
            curve/.style = {thick}
        ]

        \tikzmath{
            \rx = 0.75;
            \ry = 0.25;
        }
        
        \node[anchor=west] at (13,9) {$\R \times \partial X$};
        \draw (0,6) rectangle (12,12); 

        \node[anchor=west] at (13,3) {$\hat{X}$};
        \draw (0,3) -- (0,6) -- (12,6) -- (12,3);
        \draw (0,3) .. controls (0,-1) and (12,-1) .. (12,3);
        
        \coordinate (G) at ( 2,12);
        \coordinate (E) at ( 2, 6);
        \coordinate (A) at ( 6, 6);
        \coordinate (B) at (10, 6);

        \coordinate (L) at (-\rx,0);
        \coordinate (R) at (+\rx,0);

        \coordinate (GL) at ($ (G) + (L) $);
        \coordinate (EL) at ($ (E) + (L) $);
        \coordinate (AL) at ($ (A) + (L) $);
        \coordinate (BL) at ($ (B) + (L) $);

        \coordinate (GR) at ($ (G) + (R) $);
        \coordinate (ER) at ($ (E) + (R) $);
        \coordinate (AR) at ($ (A) + (R) $);
        \coordinate (BR) at ($ (B) + (R) $);

        \draw[curve] (G) ellipse [x radius = \rx, y radius = \ry] node[above = 1] {$\gamma$};
        \draw[curve] (E) ellipse [x radius = \rx, y radius = \ry] node[above = 1] {$\eta$};
        \draw[curve] (A) ellipse [x radius = \rx, y radius = \ry] node[above = 1] {$\alpha_1$};
        \draw[curve] (B) ellipse [x radius = \rx, y radius = \ry] node[above = 1] {$\alpha_2$};

        \draw[curve] (ER) .. controls ($ (ER) + (0,2) $) and ($ (AL) + (0,2) $) .. (AL);
        \draw[curve] (AR) .. controls ($ (AR) + (0,2) $) and ($ (BL) + (0,2) $) .. (BL);

        \draw[curve] (AL) .. controls ($ (AL) - (0,2) $) and ($ (AR) - (0,2) $) .. (AR);
        \draw[curve] (BL) .. controls ($ (BL) - (0,2) $) and ($ (BR) - (0,2) $) .. (BR);

        \draw[curve] (GR) .. controls ($ (GR) - (0,5) $) and ($ (BR) + (0,5) $) .. (BR);
        
        \coordinate (C) at ($ (E) + (0,3) $);
        \draw[curve] (EL) .. controls ($ (EL) + (0,1) $) and ($ (C) - (0,1) $) .. (C);
        \draw[curve] (GL) .. controls ($ (GL) - (0,1) $) and ($ (C) + (0,1) $) .. (C);
    \end{tikzpicture}

    \caption{A holomorphic curve with anchors contributing to the coefficient $\p{<}{}{\partial \gamma, \eta}$}
    \label{fig:differential of lch}
\end{figure}
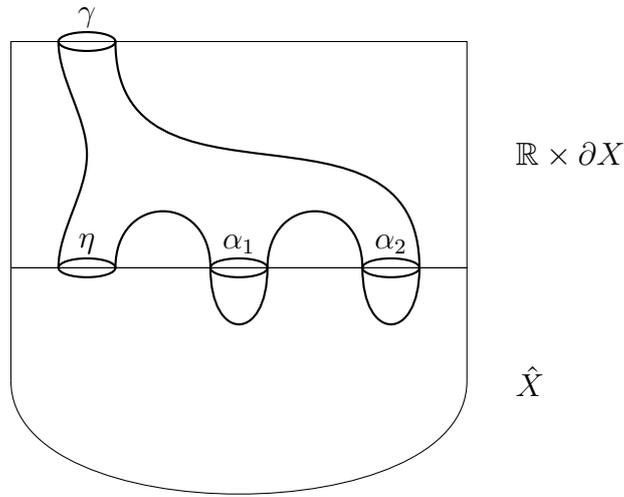

By assumption on the virtual perturbation scheme, $\partial \circ \partial = 0$ and $CC(X)$ is independent (up to chain homotopy equivalence) of the choice of almost complex structure $J$. In general, $CC(X)$ is not $\Z$-graded but only $\Z_2$-graded (see \cref{rmk:grading for lch}). We wish to define a structure of $\mathcal{L}_{\infty}$-algebra on $CC(X)[-1]$. Notice that the definition of $\mathcal{L}_{\infty}$-structure on a vector space (\cref{def:l infinity algebra}) also makes sense when the vector space is only $\Z_2$-graded.

\begin{definition}
    \label{def:lch l infinity}
    We define a structure of $\mathcal{L}_{\infty}$-algebra on $CC(X)[-1]$, given by maps $\ell^k \colon \bigodot^k CC(X) \longrightarrow CC(X)$, as follows. Choose an almost complex structure $J \in \mathcal{J}(X)$. If $\Gamma = (\gamma_1,\ldots,\gamma_k)$ is a tuple of good Reeb orbits, we define
    \begin{IEEEeqnarray*}{c+x*}
        \ell^{k} (\gamma_1 \odot \cdots \odot \gamma_{k}) = \sum_{\eta} \p{<}{}{\ell^{k} (\gamma_1 \odot \cdots \odot \gamma_{k}) , \eta} \, \eta, 
    \end{IEEEeqnarray*}
    where $\p{<}{}{\ell^{k} (\gamma_1 \odot \cdots \odot \gamma_{k}) , \eta}$ is the virtual count of holomorphic curves in $\R \times \partial X$ with positive asymptotes $\gamma_1, \ldots, \gamma_k$, one negative asymptote $\eta$, and a number of extra negative asymptotes with anchors in $\hat{X}$, such that exactly one of the components in the symplectization level is nontrivial (see \cref{fig:l infinity ops of lch}).
\end{definition}

\begin{figure}[htp]
    \centering
    
    \begin{tikzpicture}
        [
            scale = 0.5,
            help/.style = {very thin, draw = black!50},
            curve/.style = {thick}
        ]

        \tikzmath{
            \rx = 0.75;
            \ry = 0.25;
        }
        
        \node[anchor=west] at (17,9) {$\R \times \partial X$};
        \draw (0,6) rectangle (16,12); 

        \node[anchor=west] at (17,3) {$\hat{X}$};
        \draw (0,3) -- (0,6) -- (16,6) -- (16,3);
        \draw (0,3) .. controls (0,-1) and (16,-1) .. (16,3);
        
        \coordinate (G1) at ( 3,12);
        \coordinate (G2) at ( 7,12);
        \coordinate (G3) at (11,12);
        \coordinate (G4) at (14,12);
        \coordinate (F3) at (11, 6);
        \coordinate (F4) at (14, 6);
        \coordinate (E0) at ( 2, 6);
        \coordinate (A1) at ( 5, 6);
        \coordinate (A2) at ( 8, 6);

        \coordinate (L) at (-\rx,0);
        \coordinate (R) at (+\rx,0);

        \coordinate (G1L) at ($ (G1) + (L) $);
        \coordinate (G2L) at ($ (G2) + (L) $);
        \coordinate (G3L) at ($ (G3) + (L) $);
        \coordinate (G4L) at ($ (G4) + (L) $);
        \coordinate (F3L) at ($ (F3) + (L) $);
        \coordinate (F4L) at ($ (F4) + (L) $);
        \coordinate (E0L) at ($ (E0) + (L) $);
        \coordinate (A1L) at ($ (A1) + (L) $);
        \coordinate (A2L) at ($ (A2) + (L) $);

        \coordinate (G1R) at ($ (G1) + (R) $);
        \coordinate (G2R) at ($ (G2) + (R) $);
        \coordinate (G3R) at ($ (G3) + (R) $);
        \coordinate (G4R) at ($ (G4) + (R) $);
        \coordinate (F3R) at ($ (F3) + (R) $);
        \coordinate (F4R) at ($ (F4) + (R) $);
        \coordinate (E0R) at ($ (E0) + (R) $);
        \coordinate (A1R) at ($ (A1) + (R) $);
        \coordinate (A2R) at ($ (A2) + (R) $);

        \draw[curve] (G1) ellipse [x radius = \rx, y radius = \ry] node[above = 1] {$\gamma_1$};
        \draw[curve] (G2) ellipse [x radius = \rx, y radius = \ry] node[above = 1] {$\gamma_2$};
        \draw[curve] (G3) ellipse [x radius = \rx, y radius = \ry] node[above = 1] {$\gamma_3$};
        \draw[curve] (G4) ellipse [x radius = \rx, y radius = \ry] node[above = 1] {$\gamma_4$};
        \draw[curve] (F3) ellipse [x radius = \rx, y radius = \ry] node[above = 1] {$\gamma_3$};
        \draw[curve] (F4) ellipse [x radius = \rx, y radius = \ry] node[above = 1] {$\gamma_4$};
        \draw[curve] (E0) ellipse [x radius = \rx, y radius = \ry] node[above = 1] {$\eta$};
        \draw[curve] (A1) ellipse [x radius = \rx, y radius = \ry] node[above = 1] {$\alpha_1$};
        \draw[curve] (A2) ellipse [x radius = \rx, y radius = \ry] node[above = 1] {$\alpha_2$};

        \draw[curve] (G1R) .. controls ($ (G1R) - (0,2) $) and ($ (G2L) - (0,2) $) .. (G2L);

        \draw[curve] (E0R) .. controls ($ (E0R) + (0,2) $) and ($ (A1L) + (0,2) $) .. (A1L);
        \draw[curve] (A1R) .. controls ($ (A1R) + (0,2) $) and ($ (A2L) + (0,2) $) .. (A2L);

        \draw[curve] (A1L) .. controls ($ (A1L) - (0,3) $) and ($ (A1R) - (0,3) $) .. (A1R);
        \draw[curve] (A2L) .. controls ($ (A2L) - (0,3) $) and ($ (F4R) - (0,3) $) .. (F4R);

        \draw[curve] (A2R) .. controls ($ (A2R) - (0,1) $) and ($ (F3L) - (0,1) $) .. (F3L); 
        \draw[curve] (F3R) .. controls ($ (F3R) - (0,1) $) and ($ (F4L) - (0,1) $) .. (F4L);

        \draw[curve] (E0L) .. controls ($ (E0L) + (0,2) $) and ($ (G1L) - (0,2) $) .. (G1L);
        \draw[curve] (A2R) .. controls ($ (A2R) + (0,2) $) and ($ (G2R) - (0,2) $) .. (G2R);

        \draw[curve] (F3L) -- (G3L);
        \draw[curve] (F3R) -- (G3R);
        \draw[curve] (F4L) -- (G4L);
        \draw[curve] (F4R) -- (G4R);

        \node[rotate = 90] at ($ (F3) + (0,3) $) {trivial};
        \node[rotate = 90] at ($ (F4) + (0,3) $) {trivial};
    \end{tikzpicture}

    \caption{A holomorphic building contributing to the coefficient $\p{<}{}{ \ell^4 (\gamma_1 \odot \cdots \odot \gamma_4), \eta}$}
    \label{fig:l infinity ops of lch}
\end{figure}
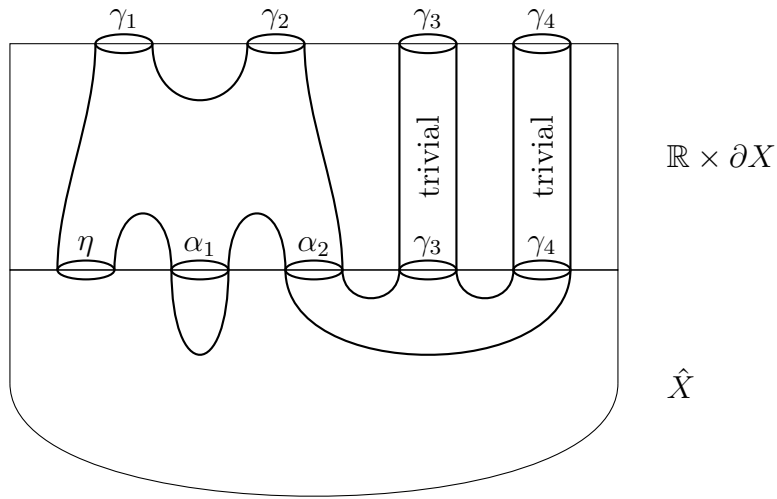

By the assumptions on the virtual perturbation scheme, the maps $\ell^k$ satisfy the $\mathcal{L}_{\infty}$-relations and $CC(X)$ is independent (as an $\mathcal{L}_{\infty}$-algebra, up to $\mathcal{L}_{\infty}$-homotopy equivalence) of the choice of $J$. We point out that the first $\mathcal{L}_{\infty}$-operation is equal to the differential of linearized contact homology, i.e. $\ell^1 = \partial$. 

\begin{remark}
    \label{rmk:grading for lch}
    In general, the Conley--Zehnder index of a Reeb orbit is well-defined as an element in $\Z_2$. Therefore, the complex $CC(X)$ has a $\Z_{2}$-grading given by $\deg(\gamma) \coloneqq n - 3 - \conleyzehnder(\gamma)$, and with respect to this definition of degree every $\mathcal{L}_{\infty}$-operation $\ell^k$ has degree $1$. If $\pi_1(X) = 0$ and $2 c_1(TX) = 0$, then by \cref{lem:cz of reeb is independent of triv over filling disk} we have well-defined Conley--Zehnder indices in $\Z$, which means that $CC(X)$ is $\Z$-graded. For some purposes, it will be enough to consider only the chain complex structure on $CC(X)$ and not the $\mathcal{L}_{\infty}$-algebra structure (namely, when we consider only the capacity $\mathfrak{g}^{\leq 1}_{k}$ instead of the higher capacities $\mathfrak{g}^{\leq \ell}_{k}$). In this case, to make comparisons with $S^1$-equivariant symplectic homology simpler, we define the grading instead by $\deg(\gamma) \coloneqq \conleyzehnder(\gamma)$, which implies that $\partial$ has degree $-1$.
\end{remark}

\begin{definition}
    \label{def:action filtration lch}
    For every $a \in \R$, we denote by $\mathcal{A}^{\leq a} CC(X)[-1]$ the submodule of $CC(X)[-1]$ generated by the good Reeb orbits $\gamma$ with action $\mathcal{A}(\gamma) \leq a$. We call this filtration the \textbf{action filtration} of $CC[-1]$. 
\end{definition}

In the next lemma, we check that this filtration is compatible with the $\mathcal{L}_{\infty}$-structure.

\begin{lemma}
    \label{lem:action filtration of lch}
    $\ell^k ( \mathcal{A}^{\leq a_1} CC(X) \odot \cdots \odot \mathcal{A}^{\leq a_k} CC(X) ) \subset \mathcal{A}^{\leq a_1 + \cdots + a_k} CC(X)$.
\end{lemma}
\begin{proof}
    Let $\gamma_1^+, \ldots, \gamma_k^+, \eta$ be good Reeb orbits such that 
    \begin{IEEEeqnarray*}{rCls+x*}
        \mathcal{A}(\gamma_i^+)                                         & \leq & a_i, \\
        \p{<}{}{\ell^k(\gamma_1^+ \odot \cdots \odot \gamma^+_k), \eta} & \neq & 0.
    \end{IEEEeqnarray*}
    We wish to show that $\mathcal{A}(\eta) \leq a_1 + \cdots + a_k$. Since $\p{<}{}{\ell^k(\gamma_1^+ \odot \cdots \odot \gamma^+_k), \eta} \neq 0$ and by assumption on the virtual perturbation scheme, there exists a tuple of Reeb orbits $\Gamma^-$ and a (nontrivial) punctured $J$-holomorphic sphere in $\R \times \partial X$ with asymptotes $\Gamma^\pm$, such that $\eta \in \Gamma^-$ and $\Gamma^+ \subset (\gamma^+_1,\ldots,\gamma^+_k)$. Then, 
    \begin{IEEEeqnarray*}{rCls+x*}
        \mathcal{A}(\eta)
        & \leq & \mathcal{A}(\Gamma^-)                       & \quad [\text{since $\eta \in \Gamma^-$}] \\
        & \leq & \mathcal{A}(\Gamma^+)                       & \quad [\text{by \cref{lem:action energy for holomorphic}}] \\
        & \leq & \mathcal{A}(\gamma^+_1, \ldots, \gamma^+_k) & \quad [\text{since $\Gamma^+ \subset (\gamma^+_1,\ldots,\gamma^+_k)$}] \\
        & \leq & a_1 + \cdots + a_k.                         & \quad [\text{by definition of action of a tuple}]. & \qedhere
    \end{IEEEeqnarray*}
\end{proof}

\begin{definition}
    \label{def:augmentation map}
    Consider the bar complex $(\mathcal{B}(CC(X)[-1]), \hat{\ell})$. For each $k \in \Z_{\geq 1}$, we define an augmentation ${\epsilon}_k \colon \mathcal{B}(CC(X)[-1]) \longrightarrow \Q$ as follows. Choose $x \in \itr X$, a symplectic divisor $D$ at $x$, and an almost complex structure $J \in \mathcal{J}(X,D)$. Then, for every tuple of good Reeb orbits $\Gamma = (\gamma_1, \ldots, \gamma_p)$ define ${\epsilon}_k (\gamma_1 \odot \cdots \odot \gamma_p)$ to be the virtual count of $J$-holomorphic planes in $\hat{X}$ which are positively asymptotic to $\Gamma$ and have contact order $k$ to $D$ at $x$ (see \cref{fig:augmentation of lch}).
\end{definition}

\begin{figure}[htp]
    \centering
    
    \begin{tikzpicture}
        [
            scale = 0.5,
            help/.style = {very thin, draw = black!50},
            curve/.style = {thick}
        ]

        \tikzmath{
            \rx = 0.75;
            \ry = 0.25;
        }

        \node[anchor=west] at (13,3) {$\hat{X}$};
        \draw (0,3) -- (0,6) -- (12,6) -- (12,3);
        \draw (0,3) .. controls (0,-1) and (12,-1) .. (12,3);
        
        \coordinate (G1) at (4,6); 
        \coordinate (G2) at (8,6);
        
        \coordinate (L) at (-\rx,0);
        \coordinate (R) at (+\rx,0);

        \coordinate (G1L) at ($ (G1) + (L) $);
        \coordinate (G2L) at ($ (G2) + (L) $);     
        
        \coordinate (G1R) at ($ (G1) + (R) $);
        \coordinate (G2R) at ($ (G2) + (R) $);

        \coordinate (P) at (7,3);
        \coordinate (D) at (2,1);

        \draw[curve] (G1) ellipse [x radius = \rx, y radius = \ry] node[above = 1] {$\gamma_1$};
        \draw[curve] (G2) ellipse [x radius = \rx, y radius = \ry] node[above = 1] {$\gamma_2$};

        \fill (P) circle (2pt) node[anchor = north west] {$x$};
        \draw[curve] ($ (P) - (D) $) -- ( $ (P) + (D) $ ) node[anchor = west] {$D$};

        \draw[curve] (G1R) .. controls ($ (G1R) - (0,2) $) and ($ (G2L) - (0,2) $) .. (G2L);
        \draw[curve] (G1L) .. controls ($ (G1L) - (0,2) $) and ($ (P) - (D) $) .. (P);
        \draw[curve] (G2R) .. controls ($ (G2R) - (0,2) $) and ($ (P) + (D) $) .. (P);
    \end{tikzpicture}

    \caption{A holomorphic curve contributing to the count $\epsilon_k(\gamma_1 \odot \gamma_2)$}
    \label{fig:augmentation of lch}
\end{figure}
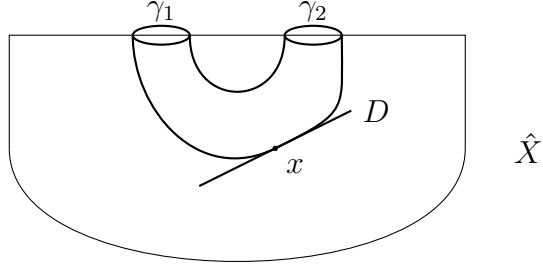

By assumption on the virtual perturbation scheme, ${\epsilon}_k$ is an augmentation, i.e. ${\epsilon}_k \circ \hat{\ell} = 0$. In addition, ${\epsilon}_k$ is independent (up to chain homotopy) of the choices of $x, D, J$.

\section{Higher symplectic capacities}


Here we define the symplectic capacities $\mathfrak{g}^{\leq \ell}_k$ from \cite{siegelHigherSymplecticCapacities2020}. We will prove the usual properties of symplectic capacities (see \cref{thm:properties of hsc}), namely monotonicity and conformality. In addition, we prove that the value of the capacities $\mathfrak{g}^{\leq \ell}_k$ can be represented by the action of a tuple of Reeb orbits. In \cref{rmk:computations using reeb orbits property} we show how this property could in principle be combined with results from \cite{guttSymplecticCapacitiesPositive2018} to compare the capacities $\mathfrak{g}^{\leq 1}_k(X_{\Omega})$ and $\cgh{k}(X_{\Omega})$ when $X_{\Omega}$ is a convex or concave toric domain.

\begin{definition}[{\cite[Section 6.1]{siegelHigherSymplecticCapacities2020}}]
    \label{def:capacities glk}
    Let $k, \ell \in \Z_{\geq 1}$ and $(X,\lambda)$ be a nondegenerate Liouville domain. The \textbf{higher symplectic capacities} of $X$ are given by
    \begin{IEEEeqnarray*}{c+x*}
        \mathfrak{g}^{\leq \ell}_k(X) \coloneqq \inf \{ a > 0 \mid \epsilon_k \colon H(\mathcal{A}^{\leq a} \mathcal{B}^{\leq \ell}(CC(X)[-1])) \longrightarrow \Q \text{ is nonzero} \}.
    \end{IEEEeqnarray*}
\end{definition}

The capacities $\mathfrak{g}^{\leq \ell}_{k}$ will be useful to us because they have similarities with the McDuff--Siegel capacities $\tilde{\mathfrak{g}}^{\leq \ell}_k$, but also with the Gutt--Hutchings capacities $\cgh{k}$ (for $\ell = 1$). More specifically:
\begin{enumerate}
    \item Both $\mathfrak{g}^{\leq \ell}_{k}$ and $\tilde{\mathfrak{g}}^{\leq \ell}_k$ are related to the energy of holomorphic curves in $X$ which are asymptotic to a word of $p \leq \ell$ Reeb orbits and satisfy a tangency constraint. In \cref{thm:g tilde vs g hat}, we will actually show that $\tilde{\mathfrak{g}}^{\leq \ell}_k(X) \leq {\mathfrak{g}}^{\leq \ell}_k(X)$. The capacities $\mathfrak{g}^{\leq \ell}_k$ can be thought of as the SFT counterparts of $\tilde{\mathfrak{g}}^{\leq \ell}_k$, or alternatively the capacities $\tilde{\mathfrak{g}}^{\leq \ell}_k$ can be thought of as the counterparts of $\mathfrak{g}^{\leq \ell}_k$ whose definition does not require the holomorphic curves to be regular.
    \item Both $\mathfrak{g}^{\leq 1}_{k}$ and $\cgh{k}$ are defined in terms of a map in homology being nonzero. In the case of $\mathfrak{g}^{\leq 1}_{k}$, we consider the linearized contact homology, and in the case of $\cgh{k}$ the invariant in question is $S^1$-equivariant symplectic homology. Taking into consideration the Bourgeois--Oancea isomorphism (see \cite{bourgeoisEquivariantSymplecticHomology2016}) between linearized contact homology and positive $S^1$-equivariant symplectic homology, one can think of $\mathfrak{g}^{\leq 1}_{k}$ and $\cgh{k}$ as restatements of one another under this isomorphism. This is the idea behind the proof of \cref{thm:g hat vs gh}, where we show that $\mathfrak{g}^{\leq 1}_{k}(X) = \cgh{k}(X)$.
\end{enumerate}

\begin{remark}
    \label{rmk:novikov coefficients}
    In the case where $X$ is only an exact symplectic manifold instead of a Liouville domain, the proof of \cref{lem:action filtration of lch} does not work. In this case, we do not have access to an action filtration on $CC(X)$. However, it is possible to define linearized contact homology with coefficients in a Novikov ring $\Lambda_{\geq 0}$, in which case a coefficient in $\Lambda_{\geq 0}$ encodes the energy of a holomorphic curve. This is the approach taken in \cite{siegelHigherSymplecticCapacities2020} to define the capacities $\mathfrak{g}^{\leq \ell}_{k}$. It is not obvious that the definition of $\mathfrak{g}^{\leq \ell}_k$ we give and the one in \cite{siegelHigherSymplecticCapacities2020} are equivalent. However, \cref{def:capacities glk} seems to be the natural analogue when we have access to an action filtration, and in addition the definition we provide will be enough for our purposes.
\end{remark}

\begin{theorem}
    \label{thm:properties of hsc}
    The functions ${\mathfrak{g}}^{\leq \ell}_k$ satisfy the following properties, for all nondegenerate Liouville domains $(X,\lambda_X)$ and $(Y,\lambda_Y)$ of the same dimension:
    \begin{description}
        \item[(Monotonicity)] If $X \longrightarrow Y$ is an exact symplectic embedding then $\mathfrak{g}^{\leq \ell}_k(X) \leq \mathfrak{g}^{\leq \ell}_k(Y)$.
        \item[(Conformality)] If $\mu > 0$ then ${\mathfrak{g}}^{\leq \ell}_k(X, \mu \lambda_X) = \mu \, {\mathfrak{g}}^{\leq \ell}_k(X, \lambda_X)$.
        \item[(Reeb orbits)] If $\pi_1(X) = 0$, $2 c_1(TX) = 0$ and ${\mathfrak{g}}^{\leq \ell}_k(X) < + \infty$, then there exists a tuple $\Gamma = (\gamma_1, \ldots, \gamma_p)$ of Reeb orbits such that 
        \begin{enumerate}
            \item ${\mathfrak{g}}^{\leq \ell}_k(X) = \mathcal{A}(\Gamma)$;
            \item $\conleyzehnder(\Gamma) = p (n - 3) + 2 (k + 1)$;
            \item $1 \leq p \leq \ell$.
        \end{enumerate}
    \end{description}
\end{theorem}
\begin{proof}
    We prove monotonicity. If $(X, \lambda^X) \longrightarrow (Y, \lambda^Y)$ is an exact symplectic embedding, then it is possible to define a Viterbo transfer map $H(\mathcal{B}(CC(Y)[-1])) \longrightarrow H(\mathcal{B}(CC(X)[-1]))$. This map respects the action filtration as well as the augmentation maps, i.e. the diagram
    \begin{IEEEeqnarray*}{c+x*}
        \begin{tikzcd}
            H(\mathcal{A}^{\leq a} \mathcal{B}^{\leq \ell} (CC(Y)[-1])) \ar[d] \ar[r] & H(\mathcal{B} (CC(Y)[-1])) \ar[d] \ar[r, "{\epsilon}_{k}^Y"] & \Q \ar[d, equals] \\
            H(\mathcal{A}^{\leq a} \mathcal{B}^{\leq \ell} (CC(X)[-1])) \ar[r]        & H(\mathcal{B} (CC(X)[-1])) \ar[r, swap, "{\epsilon}_{k}^X"]  & \Q
        \end{tikzcd}
    \end{IEEEeqnarray*}
    commutes. The result then follows by definition of $\tilde{\mathfrak{g}}^{\leq \ell}_k$.

    We prove conformality. If $\gamma$ is a Reeb orbit of $(\partial X, \lambda|_{\partial X})$ of action $\mathcal{A}_{\lambda}(\gamma)$ then $\gamma$ is a Reeb orbit of $(\partial X, \mu \lambda|_{\partial X})$ of action $\mathcal{A}_{\mu \lambda}(\gamma) = \mu \mathcal{A}_{\lambda}(\gamma)$. Therefore, there is a diagram
    \begin{IEEEeqnarray*}{c+x*}
        \begin{tikzcd}
            H(\mathcal{A}^{\leq     a} \mathcal{B}^{\leq \ell} (CC(X,     \lambda)[-1])) \ar[d, equals] \ar[r] & H(\mathcal{B} (CC(X,     \lambda)[-1])) \ar[d, equals] \ar[r, "{\epsilon}_{k}^{\lambda}"] & \Q \ar[d, equals] \\
            H(\mathcal{A}^{\leq \mu a} \mathcal{B}^{\leq \ell} (CC(X, \mu \lambda)[-1])) \ar[r]                & H(\mathcal{B} (CC(X, \mu \lambda)[-1])) \ar[r, swap, "{\epsilon}_{k}^{\mu \lambda}"]     & \Q
        \end{tikzcd}
    \end{IEEEeqnarray*}
    Again, the result follows by definition of $\mathfrak{g}^{\leq \ell}_{k}$.

    We prove the Reeb orbits property. Choose a point $x \in \itr X$, a symplectic divisor $D$ through $x$ and an almost complex structure $J \in \mathcal{J}(X,D)$. Consider the bar complex $\mathcal{B}^{\leq \ell} (CC(X)[-1])$, computed with respect to $J$. By assumption and definition of $\mathfrak{g}^{\leq \ell}_{k}$,
    \begin{IEEEeqnarray*}{rCls+x*}
        + \infty
        & > & {\mathfrak{g}}^{\leq \ell}_k(X) \\
        & = & \inf \{ a > 0 \mid \epsilon_k \colon H(\mathcal{A}^{\leq a} \mathcal{B}^{\leq \ell}(CC(X)[-1])) \longrightarrow \Q \text{ is nonzero} \} \\
        & = & \inf \{ a > 0 \mid \text{there exists } \beta \in H(\mathcal{A}^{\leq a} \mathcal{B}^{\leq \ell}(CC(X)[-1])) \text{ such that } {\epsilon}_k (\beta) \neq 0 \} \\
        & = & \inf \{ \mathcal{A}(\beta) \mid \beta \in H(\mathcal{B}^{\leq \ell}(CC(X)[-1])) \text{ such that } {\epsilon}_k (\beta) \neq 0 \},
    \end{IEEEeqnarray*}
    where $\mathcal{A}(\beta)$ is given as in \cref{rmk:notation for tuples of orbits}. Since the action spectrum of $(\partial X, \lambda|_{\partial X})$ is a discrete subset of $\R$, we conclude that in the above expression the infimum is a minimum. More precisely, there exists $\beta \in H(\mathcal{B}^{\leq \ell}(CC(X)[-1]))$ such that $\epsilon_k(\beta) \neq 0$ and ${\mathfrak{g}}^{\leq \ell}_k(X) = \mathcal{A}(\beta)$. The element $\beta$ can be written as a finite linear combination of words of Reeb orbits $\Gamma = (\gamma_1, \ldots, \gamma_p)$, where every word has length $p \leq \ell$ and Conley--Zehnder index equal to $p(n-3) + 2(k+1)$. Here, the statement about the Conley--Zehnder index follows from the computation
    \begin{IEEEeqnarray*}{rCls+x*}
        0
        & = & \operatorname{virdim} \overline{\mathcal{M}}^J_X(\Gamma)\p{<}{}{\mathcal{T}^{(k)}x} \\
        & = & (n-3)(2 - p) + \conleyzehnder(\Gamma) - 2n - 2k + 4 \\
        & = & \conleyzehnder(\Gamma) - p(n-3) - 2(k+1).
    \end{IEEEeqnarray*}
    One of the words in this linear combination is such that $\mathcal{A}(\Gamma) = \mathcal{A}(\beta) = {\mathfrak{g}}^{\leq \ell}_k(X)$.
\end{proof}

\begin{remark}
    \label{rmk:computations using reeb orbits property}
    In \cite[Theorem 1.6]{guttSymplecticCapacitiesPositive2018} (respectively \cite[Theorem 1.14]{guttSymplecticCapacitiesPositive2018}) Gutt--Hutchings give formulas for $\cgh{k}$ of a convex (respectively concave) toric domain. However, the given proofs only depend on specific properties of the Gutt--Hutchings capacity and not on the definition of the capacity itself. These properties are monotonicity, conformality, a Reeb orbits property similar to the one of \cref{thm:properties of hsc}, and finally that the capacity be finite on star-shaped domains. If we showed that $\mathfrak{g}^{\leq 1}_{k}$ is finite on star-shaped domains, we would conclude that $\mathfrak{g}^{\leq 1}_{k} = \cgh{k}$ on convex or concave toric domains, because in this case both capacities would be given by the formulas in the previously mentioned theorems. Showing that $\mathfrak{g}^{\leq 1}_{k}$ is finite boils down to showing that the augmentation map is nonzero, which we will do in \cref{sec:augmentation map of an ellipsoid}. However, in \cref{thm:g hat vs gh} we will use this information in combination with the Bourgeois--Oancea isomorphism to conclude that $\mathfrak{g}^{\leq 1}_{k}(X) = \cgh{k}(X)$ for any nondegenerate Liouville domain $X$. Therefore, the proof suggested above will not be necessary, although it is a proof of $\mathfrak{g}^{\leq 1}_{k}(X) = \cgh{k}(X)$ alternative to that of \cref{thm:g hat vs gh} when $X$ is a convex or concave toric domain.
\end{remark}

\section{Cauchy--Riemann operators on bundles}
\label{sec:cr operators}

In order to show that $\mathfrak{g}^{\leq 1}_{k}(X) = \cgh{k}(X)$, we will need to show that the augmentation map of a small ellipsoid in $X$ is nonzero (see the proof of \cref{thm:g hat vs gh}). Recall that the augmentation map counts holomorphic curves satisfying a tangency constraint. In \cref{sec:augmentation map of an ellipsoid}, we will explicitly compute how many such holomorphic curves there are. However, a count obtained by explicit methods will not necessarily agree with the virtual count that appears in the definition of the augmentation map. By assumption on the virtual perturbation scheme, it does agree if the relevant moduli space is transversely cut out. 

Therefore, in this section and the next we will describe the framework that allows us to show that this moduli space is transversely cut out. This section deals with the theory of real linear Cauchy--Riemann operators on line bundles, and our main reference is \cite{wendlAutomaticTransversalityOrbifolds2010}. The outline is as follows. First, we review the basic definitions about real linear Cauchy--Riemann operators (\cref{def:real linear cauchy riemann operator}). By the Riemann-Roch theorem (\cref{thm:riemann roch with punctures}), these operators are Fredholm and their index can be computed from a number of topological quantities associated to them. We will make special use of a criterion by Wendl (\cref{prp:wen D surjective injective criterion}) which guarantees that a real linear Cauchy--Riemann operator defined on a complex line bundle is surjective. For our purposes, we will also need an adaptation of this result to the case where the operator is accompanied by an evaluation map, which we state in \cref{lem:D plus E is surjective}. We now state the assumptions for the rest of this section.

Let $(\Sigma, j)$ be a compact Riemann surface without boundary, of genus $g$, with sets of positive and negative punctures $\mathbf{z}^{\pm} = \{z^{\pm}_1,\ldots,z^{\pm}_{p^{\pm}}\}$. Denote $\mathbf{z} = \mathbf{z}^{+} \cup \mathbf{z}^{-}$ and $\dot{\Sigma} = \Sigma \setminus \mathbf{z}$. Choose cylindrical coordinates $(s,t)$ near each puncture $z \in \mathbf{z}$ and denote $\mathcal{U}_z \subset \dot{\Sigma}$ the domain of the coordinates $(s,t)$.

\begin{definition}
    \label{def:asymptotically hermitian vector bundle}
    An \textbf{asymptotically Hermitian vector bundle} over $\dot{\Sigma}$ is given by a complex vector bundle $(E, J) \longrightarrow \dot{\Sigma}$ and for each $z \in \mathbf{z}$ a Hermitian vector bundle $(E_z, J_z, \omega_z) \longrightarrow S^1$ together with a complex vector bundle isomorphism $\Phi_z^{} \colon \pi^*_z E_z^{} \longrightarrow \iota_z^* E$, where $\iota_z \colon \mathcal{U}_z \longrightarrow \dot{\Sigma}$ is the inclusion and $\pi_{z} \colon \mathcal{U}_z \longrightarrow S^1$ is given by $\pi_{z}(w) = t(w)$:
    \begin{IEEEeqnarray*}{c+x*}
        \begin{tikzcd}
            E_z^{} \ar[d] & \pi_z^* E_z^{} \ar[r, "\Phi_z"] \ar[d] \ar[l] & \iota_z^* E \ar[r] \ar[d]             & E \ar[d] \\
            S^1           & \mathcal{U}_z \ar[r, equals] \ar[l, "\pi_z"]  & \mathcal{U}_z \ar[r, swap, "\iota_z"] & \dot{\Sigma}
        \end{tikzcd}
    \end{IEEEeqnarray*}
\end{definition}

From now until the end of this section, we will assume that $E$ is an asymptotically Hermitian vector bundle over $\dot{\Sigma}$ of complex rank $n$.

\begin{definition}
    \label{def:asymptotic trivialization}
    An \textbf{asymptotic trivialization} of an asymptotically Hermitian vector bundle $E$ is a family $\tau = (\tau_z)_{z \in \mathbf{z}}$ of unitary trivializations $\tau_z$ of $(E_z, J_z, \omega_z)$. By \cref{def:asymptotically hermitian vector bundle}, every such $\tau_z$ defines a complex trivialization of $\iota^*_z E$. If $\tau$ is an asymptotic trivialization, we will typically denote each $\tau_z$ also by $\tau$.
\end{definition}

\begin{definition}
    \label{def:sobolev spaces}
    Let $E$ be an asymptotically Hermitian vector bundle over $\dot{\Sigma}$, together with an asymptotic trivialization $\tau$. If $\eta$ is a section of $E$ and $z$ is a puncture, denote by $\eta_z \colon Z^{\pm} \longrightarrow \R^{2n}$ the map $\eta$ written with respect to the trivialization $\tau$ and cylindrical coordinates near $z$. The \textbf{Sobolev space} of sections of $E$ is
    \begin{IEEEeqnarray*}{c+x*}
        W^{k,p}(E) \coloneqq \{ \eta \in W^{k,p}_{\mathrm{loc}}(E) \mid \eta_z \in W^{k,p}(Z^{\pm}, \R^{2n}) \text{ for every } z \in \mathbf{z}^{\pm} \}. 
    \end{IEEEeqnarray*}
    If $\delta > 0$, the \textbf{weighted Sobolev space} of sections of $E$ is
    \begin{IEEEeqnarray*}{c+x*}
        W^{k,p,\delta}(E) \coloneqq \{ \eta \in W^{k,p}_{\mathrm{loc}}(E) \mid e^{\pm \delta s} \eta_z \in W^{k,p}(Z^{\pm}, \R^{2n}) \text{ for every } z \in \mathbf{z}^{\pm} \}.
    \end{IEEEeqnarray*}
\end{definition}

\begin{definition}
    \label{def:real linear cauchy riemann operator}
    A \textbf{real linear Cauchy--Riemann operator} is a map
    \begin{IEEEeqnarray*}{c+x*}
        \mathbf{D} \colon W^{1,p}(\dot{\Sigma}, E) \longrightarrow L^p(\dot{\Sigma}, \Hom^{0,1}(T \dot{\Sigma}, E))
    \end{IEEEeqnarray*}
    such that $\mathbf{D}$ is linear as a map of vector spaces over $\R$ and $\mathbf{D}$ satisfies the Leibniz rule, i.e. if $v \in W^{1,p}(\dot{\Sigma}, E)$ and $f \in C^{\infty}(\dot{\Sigma}, \R)$ then $\mathbf{D}(f v) = f \mathbf{D} v + v \otimes \overline{\partial} f$.
\end{definition}

We now consider the asymptotic operators of $\mathbf{D}$. Their relevance comes from the fact that the Fredholm index of $\mathbf{D}$ is determined by the asymptotic operators at the punctures.

\begin{definition}
    An \textbf{asymptotic operator} at $z \in \mathbf{z}$ is a bounded linear operator $\mathbf{A} \colon H^1(E_z) \longrightarrow L^2(E_z)$ such that when written with respect to a unitary trivialization of $E_z$, $\mathbf{A}$ takes the form
    \begin{IEEEeqnarray*}{rrCl}
        & H^1(S^1,\R^{2n}) & \longrightarrow & L^2(S^1,\R^{2n}) \\
        & \eta & \longmapsto & - J_0 \dot{\eta} - S \eta,
    \end{IEEEeqnarray*}
    where $S \colon S^1 \longrightarrow \End(\R^{2n})$ is a loop of symmetric $2n \times 2n$ matrices. We say that $\mathbf{A}$ is nondegenerate if its spectrum does not contain $0$.
\end{definition}

\begin{definition}
    Let $\mathbf{D}$ be a real linear Cauchy--Riemann operator and $\mathbf{A}$ be an asymptotic operator at $z \in \mathbf{z}$. We say that $\mathbf{D}$ is \textbf{asymptotic} to $\mathbf{A}$ at $z$ if the expressions for $\mathbf{D}$ and $\mathbf{A}$ with respect to an asymptotic trivialization near $z$ are of the form
    \begin{IEEEeqnarray*}{rCls+x*}
        (\mathbf{D} \xi)(s,t) & = & \partial_s \xi (s,t) + J_0 \partial_t \xi (s,t) + S(s,t) \xi(s,t) \\
        (\mathbf{A} \eta)(t)  & = & - J_0 \partial_t \eta (t) - S(t) \eta(t),
    \end{IEEEeqnarray*}
    where $S(s,t)$ converges to $S(t)$ uniformly as $s \to \pm \infty$.
\end{definition}

\begin{remark}
    Suppose that $E$ splits as a direct sum of complex vector bundles $E = E_1 \oplus E_2$. In this case, there are canonical inclusions
    \begin{IEEEeqnarray*}{rCls+x*}
        W^{1,p}(\dot{\Sigma}, E_i)                         & \subset & W^{1,p}(\dot{\Sigma}, E), \\
        L^p(\dot{\Sigma}, \Hom^{0,1}(T \dot{\Sigma}, E_i)) & \subset & L^p(\dot{\Sigma}, \Hom^{0,1}(T \dot{\Sigma}, E))
    \end{IEEEeqnarray*}
    for $i = 1,2$, and we have the following decompositions:%
    \begin{IEEEeqnarray*}{rCls+x*}
        W^{1,p}(\dot{\Sigma}, E)                         & = & W^{1,p}(\dot{\Sigma}, E_1) \oplus W^{1,p}(\dot{\Sigma}, E_2), \\
        L^p(\dot{\Sigma}, \Hom^{0,1}(T \dot{\Sigma}, E)) & = & L^p(\dot{\Sigma}, \Hom^{0,1}(T \dot{\Sigma}, E_1)) \oplus L^p(\dot{\Sigma}, \Hom^{0,1}(T \dot{\Sigma}, E_2))
    \end{IEEEeqnarray*}
    We can write $\mathbf{D}$ with respect to these decompositions as a block matrix:
    \begin{IEEEeqnarray*}{c+x*}
        \mathbf{D}
        =
        \begin{bmatrix}
            \mathbf{D}_{11} & \mathbf{D}_{12} \\
            \mathbf{D}_{21} & \mathbf{D}_{22}
        \end{bmatrix}.
    \end{IEEEeqnarray*}
    By \cite[Exercise 7.8]{wendlLecturesSymplecticField2016}, the diagonal terms $\mathbf{D}_{11}$ and $\mathbf{D}_{22}$ are real linear Cauchy--Riemann operators, while the off diagonal terms $\mathbf{D}_{12}$ and $\mathbf{D}_{21}$ are tensorial.
\end{remark}

Let $\mathbf{D}$ be a real linear Cauchy--Riemann operator and for every puncture $z \in \mathbf{z}$ let $\mathbf{A}_z$ be a nondegenerate asymptotic operator at $z$. By the Riemann-Roch theorem with punctures (\cref{thm:riemann roch with punctures}), $\mathbf{D}$ is a Fredholm operator. We now explain how to compute the Fredholm index of $\mathbf{D}$. Choose an asymptotic trivialization $\tau$ as in \cref{def:asymptotic trivialization}. First, recall that the \textbf{Euler characteristic} of $\dot{\Sigma}$ is given by $\chi(\dot{\Sigma}) = 2 - 2 g - \# \mathbf{z}$, where $g$ is the genus of $\Sigma$.

\begin{definition}[{\cite[Definition 5.1]{wendlLecturesSymplecticField2016}}]
    \label{def:relative first chern number}
    Let $S$ be a compact oriented surface with boundary and $(E,J)$ be a complex vector bundle over $S$. Let $\tau$ be a complex trivialization of $E|_{\partial S}$. The \textbf{relative first Chern number} of $E$ with respect to $\tau$, denoted $c_1^{\tau}(E) \in \Z$, is defined by the following properties.
    \begin{enumerate}
        \item If $E$ has complex rank $1$, then $c_1^{\tau}(E)$ is the signed count of zeros of a generic smooth section $\eta \colon S \longrightarrow E$ such that $\tau \circ \eta|_{\partial S} \colon \partial S \longrightarrow \C$ is constant.
        \item If $E_1$ and $E_2$ are complex vector bundles over $S$ with trivializations $\tau_1$ and $\tau_2$ over $\partial S$, then $c_1^{\tau_1 \oplus \tau_2}(E_1 \oplus E_2) = c_1^{\tau}(E_1) + c_1^{\tau}(E_2)$.
    \end{enumerate}
\end{definition}

The definition of relative first Chern number extends to the class of asymptotically Hermitian vector bundles over punctured surfaces.

\begin{definition}
    The \textbf{Conley--Zehnder} index of an asymptotic operator $\mathbf{A}_z$ is given as follows. Let $(\mathbf{A}_z \eta)(t) = -J_0 \partial_t \eta(t) - S(t) \eta(t)$ be the expression of $\mathbf{A}_z$ with respect to $\tau$. Let $\Psi \colon [0,1] \longrightarrow \operatorname{Sp}(2n)$ be the unique path of symplectic matrices such that
    \begin{IEEEeqnarray*}{rCls+x*}
        \Psi(0)       & = & \id_{\R^{2n}}, \\
        \dot{\Psi}(t) & = & J_0 S(t) \Psi(t).
    \end{IEEEeqnarray*}
    Since $\mathbf{A}_z$ is nondegenerate, $\Psi$ is an element of $\operatorname{SP}(n)$. Finally, define $\conleyzehnder^{\tau}(\mathbf{A}_z) \coloneqq \conleyzehnder(\Psi)$.
\end{definition}

\begin{theorem}[Riemann-Roch, {\cite[Theorem 5.4]{wendlLecturesSymplecticField2016}}]
    \label{thm:riemann roch with punctures}
    The operator $\mathbf{D}$ is Fredholm and its (real) Fredholm index is given by
    \begin{IEEEeqnarray*}{c+x*}
        \operatorname{ind} \mathbf{D} = n \chi (\dot{\Sigma}) + 2 c_1^{\tau}(E) + \sum_{z \in \mathbf{z}^+} \conleyzehnder^{\tau}(\mathbf{A}_z) - \sum_{z \in \mathbf{z}^-} \conleyzehnder^{\tau}(\mathbf{A}_z).
    \end{IEEEeqnarray*}
\end{theorem}

For the rest of this section, we restrict ourselves to the case where $n = \operatorname{rank}_{\C} E = 1$. We retain the assumption that $\mathbf{D}$ is a real linear Cauchy--Riemann operator and $\mathbf{A}_{z}$ is a nondegenerate asymptotic operator for every puncture $z \in \mathbf{z}$. Our goal is to state a criterion that guarantees surjectivity of $\mathbf{D}$. This criterion depends on other topological quantities which we now define.

For every $\lambda$ in the spectrum of $\mathbf{A}_z$, let $w^{\tau}(\lambda)$ be the winding number of any nontrivial section in the $\lambda$-eigenspace of $\mathbf{A}_z$ (computed with respect to the trivialization $\tau$). Define the \textbf{winding numbers}
\begin{IEEEeqnarray*}{rClls+x*}
    \alpha_-^{\tau}(\mathbf{A}_z) & \coloneqq & \max & \{ w^{\tau}(\lambda) \mid \lambda < 0 \text{ is in the spectrum of }\mathbf{A}_z \}, \\
    \alpha_+^{\tau}(\mathbf{A}_z) & \coloneqq & \min & \{ w^{\tau}(\lambda) \mid \lambda > 0 \text{ is in the spectrum of }\mathbf{A}_z \}.
\end{IEEEeqnarray*}
The \textbf{parity} (the reason for this name is Equation \eqref{eq:cz winding parity} below) and associated sets of even and odd punctures are given by
\begin{IEEEeqnarray*}{rCls+x*}
    p(\mathbf{A}_{z}) & \coloneqq & \alpha_{+}^{\tau}(\mathbf{A}_z) - \alpha^{\tau}_{-}(\mathbf{A}_z) \in \{0,1\}, \\
    \mathbf{z}_0      & \coloneqq & \{ z \in \mathbf{z} \mid p(\mathbf{A}_z) = 0 \}, \\
    \mathbf{z}_1      & \coloneqq & \{ z \in \mathbf{z} \mid p(\mathbf{A}_z) = 1 \}.
\end{IEEEeqnarray*}
Finally, the \textbf{adjusted first Chern number} is given by
\begin{IEEEeqnarray*}{c+x*}
    c_1(E,\mathbf{A}_{\mathbf{z}}) = c_1^{\tau}(E) + \sum_{z \in \mathbf{z}^+} \alpha_-^{\tau}(\mathbf{A}_z) - \sum_{z \in \mathbf{z}^-} \alpha_-^{\tau}(\mathbf{A}_z).
\end{IEEEeqnarray*}
These quantities satisfy the following equations.
\begin{IEEEeqnarray}{rCls+x*}
    \conleyzehnder^{\tau}(\mathbf{A}_z) & = & 2 \alpha_{-}^{\tau}(\mathbf{A_z}) + p(\mathbf{A}_z) = 2 \alpha_{+}^{\tau}(\mathbf{A_z}) - p(\mathbf{A}_z), \plabel{eq:cz winding parity} \\
    2 c_1 (E,\mathbf{A}_{\mathbf{z}})   & = & \operatorname{ind} \mathbf{D} - 2 - 2g + \# \mathbf{z}_0. \plabel{eq:chern and index}
\end{IEEEeqnarray}

\begin{proposition}[{\cite[Proposition 2.2]{wendlAutomaticTransversalityOrbifolds2010}}]
    \phantomsection\label{prp:wen D surjective injective criterion}
    \begin{enumerate}
        \item[] 
        \item If $\operatorname{ind} \mathbf{D} \leq 0$ and $c_1(E, \mathbf{A}_{\mathbf{z}}) < 0$ then $\mathbf{D}$ is injective.
        \item If $\operatorname{ind} \mathbf{D} \geq 0$ and $c_1(E, \mathbf{A}_{\mathbf{z}}) < \operatorname{ind} \mathbf{D}$ then $\mathbf{D}$ is surjective.
    \end{enumerate}
\end{proposition}

We will apply the proposition above to moduli spaces of punctured spheres which have no even punctures. The following lemma is just a restatement of the previous proposition in this simpler case.

\begin{lemma}
    \label{lem:conditions for D surjective genus zero}
    Assume that $g = 0$ and $\# \mathbf{z}_0 = 0$. Then,
    \begin{enumerate}
        \item If $\operatorname{ind} \mathbf{D} \leq 0$ then $\mathbf{D}$ is injective.
        \item If $\operatorname{ind} \mathbf{D} \geq 0$ then $\mathbf{D}$ is surjective.
    \end{enumerate}
\end{lemma}
\begin{proof}
    By \cref{prp:wen D surjective injective criterion} and Equation \eqref{eq:chern and index}.
\end{proof}

We now wish to deal with the case where $\mathbf{D}$ is taken together with an evaluation map (see \cref{lem:D plus E is surjective} below). The tools we need to prove this result are explained in the following remark.

\begin{remark}
    \label{rmk:formulas for xi in ker nonzero}
    Suppose that $\ker \mathbf{D} \neq \{0\}$. If $\xi \in \ker \mathbf{D} \setminus \{0\}$, it is possible to show that $\xi$ has only a finite number of zeros, all of positive order, i.e. if $w$ is a zero of $\xi$ then $\operatorname{ord}(\xi;w) > 0$. For every $z \in \mathbf{z}$, there is an \textbf{asymptotic winding number} $\operatorname{wind}_z^{\tau}(\xi) \in \Z$, which has the properties
    \begin{IEEEeqnarray*}{rCls+x*}
        z \in \mathbf{z}^+ & \Longrightarrow & \operatorname{wind}_z^{\tau}(\xi) \leq \alpha_-^{\tau}(\mathbf{A}_z), \\
        z \in \mathbf{z}^- & \Longrightarrow & \operatorname{wind}_z^{\tau}(\xi) \geq \alpha_+^{\tau}(\mathbf{A}_z).
    \end{IEEEeqnarray*}
    Define the \textbf{asymptotic vanishing} of $\xi$, denoted $Z_{\infty}(\xi)$, and the \textbf{count of zeros}, denoted $Z(\xi)$, by
    \begin{IEEEeqnarray*}{rCls+x*}
        Z_{\infty}(\xi) & \coloneqq & \sum_{z \in \mathbf{z}^+} \p{}{1}{\alpha_-^{\tau}(\mathbf{A}_z) - \operatorname{wind}_z^{\tau}(\xi)} + \sum_{z \in \mathbf{z}^-} \p{}{1}{\operatorname{wind}_z^{\tau}(\xi) - \alpha_+^{\tau}(\mathbf{A}_z)} \in \Z_{\geq 0}, \\
        Z(\xi)          & \coloneqq & \sum_{w \in \xi^{-1}(0)} \operatorname{ord}(\xi;w) \in \Z_{\geq 0}.
    \end{IEEEeqnarray*}
    In this case, we have the formula (see \cite[Equation 2.7]{wendlAutomaticTransversalityOrbifolds2010})
    \begin{IEEEeqnarray}{c}
        \plabel{eq:c1 and asy vanishing}
        c_1(E,\mathbf{A}_{\mathbf{z}}) = Z(\xi) + Z_{\infty}(\xi).
    \end{IEEEeqnarray}
\end{remark}

\begin{lemma}
    \label{lem:D plus E is surjective}
    Let $w \in \dot{\Sigma}$ be a point and $\mathbf{E} \colon W^{1,p}(\dot{\Sigma}, E) \longrightarrow E_w$ be the evaluation map at $w$, i.e. $\mathbf{E}(\xi) = \xi_w$. Assume that $g = 0$ and $\# \mathbf{z}_0 = 0$. If $\operatorname{ind} \mathbf{D} = 2$ then $\mathbf{D} \oplus \mathbf{E} \colon W^{1,p}(\dot{\Sigma}, E) \longrightarrow L^p(\dot{\Sigma}, \Hom^{0,1}(T \dot{\Sigma}, E)) \oplus E_w$ is surjective.
\end{lemma}
\begin{proof}
    It is enough to show that the maps
    \begin{IEEEeqnarray*}{rCls+x*}
        \mathbf{D} \colon W^{1,p}(\dot{\Sigma}, E)           & \longrightarrow & L^p(\dot{\Sigma}, \Hom^{0,1}(T \dot{\Sigma}, E)), \\
        \mathbf{E}|_{\ker \mathbf{D}} \colon \ker \mathbf{D} & \longrightarrow & E_w
    \end{IEEEeqnarray*}
    are surjective. By \cref{lem:conditions for D surjective genus zero}, $\mathbf{D}$ is surjective. Since $\dim \ker \mathbf{D} = \operatorname{ind} \mathbf{D} = 2$ and $\dim_{\R} E_w = 2$, the map $\mathbf{E}|_{\ker \mathbf{D}}$ is surjective if and only if it is injective. So, we show that $\ker(E|_{\ker \mathbf{D}}) = \ker \mathbf{E} \cap \ker \mathbf{D} = \{0\}$. For this, let $\xi \in \ker \mathbf{E} \cap \ker \mathbf{D}$ and assume by contradiction that $\xi \neq 0$. Consider the quantities defined in \cref{rmk:formulas for xi in ker nonzero}. We compute
    \begin{IEEEeqnarray*}{rCls+x*}
        0
        & =    & \operatorname{ind} \mathbf{D} - 2 & \quad [\text{by assumption}] \\
        & =    & 2 c_1(E,\mathbf{A}_{\mathbf{z}})  & \quad [\text{by Equation \eqref{eq:chern and index}}] \\
        & =    & 2 Z(\xi) + 2 Z_{\infty}(\xi)      & \quad [\text{by Equation \eqref{eq:c1 and asy vanishing}}] \\
        & \geq & 0                                 & \quad [\text{by definition of $Z$ and $Z_{\infty}$}],
    \end{IEEEeqnarray*}
    which implies that $Z(\xi) = 0$. This gives the desired contradiction, because
    \begin{IEEEeqnarray*}{rCls+x*}
        0
        & =    & Z(\xi)                                             & \quad [\text{by the previous computation}] \\
        & =    & \sum_{z \in \xi^{-1}(0)} \operatorname{ord}(\xi;z) & \quad [\text{by definition of $Z$}] \\
        & \geq & \operatorname{ord}(\xi;w)                          & \quad [\text{since $\xi_w = \mathbf{E}(\xi) = 0$}] \\
        & >    & 0                                                  & \quad [\text{by \cref{rmk:formulas for xi in ker nonzero}}]. & \qedhere
    \end{IEEEeqnarray*}
\end{proof}

\section{Cauchy--Riemann operators as sections}
\label{sec:functional analytic setup}

In this section, we phrase the notion of a map $u \colon \dot{\Sigma} \longrightarrow \hat{X}$ being holomorphic in terms of $u$ being in the zero set of a section $\overline{\partial} \colon \mathcal{T} \times \mathcal{B} \longrightarrow \mathcal{E}$ (see \cref{def:bundle for cr op,def:cauchy riemann operator}). The goal of this point of view is that we can then think of moduli spaces of holomorphic curves in $\hat{X}$ as the zero set of the section $\overline{\partial}$. To see if such a moduli space is regular near $(j, u)$, one needs to consider the linearization $\mathbf{L}_{(j,u)}$ of $\overline{\partial}$ at $(j,u)$ (see \cref{def:linearized cr op}), and prove that it is surjective. We will see that a suitable restriction of $\mathbf{L}_{(j,u)}$ is a real linear Cauchy--Riemann operator (\cref{lem:D is a rlcro}), and therefore we can use the theory from the last section to show that $\mathbf{L}_{(j,u)}$ is surjective in some particular cases (\cref{lem:Du is surjective case n is 1,lem:DX surj implies DY surj}). 

\begin{definition}
    \label{def:asymptotic marker}
    Let $(\Sigma,j)$ be a Riemann surface and $z \in \Sigma$ be a puncture. An \textbf{asymptotic marker} at $z$ is a half-line $v \in (T_z \Sigma \setminus \{0\}) / \R_{> 0}$.
\end{definition}

\begin{definition}
    \label{def:moduli space of curves with asymtotic marker}
    Let $(X, \omega, \lambda)$ be a symplectic cobordism, $J \in \mathcal{J}(X)$ be a cylindrical almost complex structure on $\hat{X}$, and $\Gamma^{\pm} = (\gamma^{\pm}_1, \ldots, \gamma^{\pm}_{p^{\pm}})$ be tuples of Reeb orbits on $\partial^{\pm} X$. Let $\mathcal{M}^{\$,J}_X(\Gamma^+, \Gamma^-)$ be the moduli space of (equivalence classes of) tuples 
    \begin{IEEEeqnarray*}{c+x*}
        (\Sigma, j, \mathbf{z}, \mathbf{v}, u), \qquad \mathbf{z} = \mathbf{z}^+ \cup \mathbf{z}^-, \qquad \mathbf{v} = \mathbf{v}^+ \cup \mathbf{v}^{-}
    \end{IEEEeqnarray*}
    where $(\Sigma, j, \mathbf{z}, u)$ is as in \cref{def:asy cyl holomorphic curve} and $\mathbf{v}^{\pm} = \{v^{\pm}_1, \ldots, v^{\pm}_{p^{\pm}}\}$ is a set of asymptotic markers on $\mathbf{z}^{\pm} = \{z^{\pm}_1, \ldots, z^{\pm}_{p^{\pm}}\}$ such that
    \begin{IEEEeqnarray*}{c+x*}
        \lim_{t \to 0^+} u(c(t)) = (\pm \infty, \gamma^{\pm}_i(0))
    \end{IEEEeqnarray*}
    for every $i = 1, \ldots, p^{\pm}$ and every path $c$ in $\Sigma$ with $c(t) = z^{\pm}_i$ and $\dot{c}(0) = v^{\pm}_i$. Two such tuples $(\Sigma_0, j_0, \mathbf{z}_0, \mathbf{v}_0, u_0)$ and $(\Sigma_1, j_1, \mathbf{z}_1, \mathbf{v}_1, u_1)$ are equivalent if there exists a biholomorphism $\phi \colon \Sigma_0 \longrightarrow \Sigma_1$ such that
    \begin{IEEEeqnarray*}{rCls+x*}
        u_1 \circ \phi                         & = & u_0, \\
        \phi(z^{\pm}_{0,i})                    & = & z^{\pm}_{1,i}, \\
        \dv \phi (z^{\pm}_{0,i}) v_{0,i}^{\pm} & = & v_{1,i}^{\pm}.
    \end{IEEEeqnarray*}
\end{definition}

\begin{remark}
    \label{rmk:moduli space may assume sigma is sphere}
    Consider the sphere $S^2$, without any specified almost complex structure. Let $\mathbf{z}^{\pm} = \{z^{\pm}_1, \ldots, z^{\pm}_{p^{\pm}}\} \subset S^2$ be sets of punctures and $\mathbf{v}^{\pm} = \{v^{\pm}_1, \ldots, v^{\pm}_{p^{\pm}}\}$ be corresponding sets of asymptotic markers. Then,
    \begin{IEEEeqnarray*}{c+x*}
        \mathcal{M}^{\$, J}_{X}(\Gamma^+, \Gamma^-)
        \cong
        \left\{
            (j, u)
            \ \middle\vert
            \begin{array}{l}
                j \text{ is an almost complex structure on }S^2, \\
                u \colon (\dot{S}^2, j) \longrightarrow (\hat{X}, J) \text{ is as in \cref{def:asy cyl holomorphic curve}}
            \end{array}
        \right\} / \sim,
    \end{IEEEeqnarray*}
    where two tuples $(j_0, u_0)$ and $(j_1, u_1)$ are equivalent if there exists a biholomorphism $\phi \colon (S^2, j_0) \longrightarrow (S^2, j_1)$ such that
    \begin{IEEEeqnarray*}{rCls+x*}
        u_1 \circ \phi                     & = & u_0, \\
        \phi(z^{\pm}_{i})                  & = & z^{\pm}_{i}, \\
        \dv \phi (z^{\pm}_{i}) v_{i}^{\pm} & = & v_{i}^{\pm}.
    \end{IEEEeqnarray*}
\end{remark}

\begin{remark}
    \label{rmk:counts of moduli spaces with or without asy markers}
    There is a surjective map $\pi^{\$} \colon \mathcal{M}^{\$, J}_{X}(\Gamma^+, \Gamma^-) \longrightarrow \mathcal{M}^{J}_{X}(\Gamma^+, \Gamma^-)$ given by forgetting the asymptotic markers. By \cite[Proposition 11.1]{wendlLecturesSymplecticField2016}, for every $u \in \mathcal{M}^{J}_{X}(\Gamma^+, \Gamma^-)$ the preimage $(\pi^{\$})^{-1}(u)$ contains exactly
    \begin{IEEEeqnarray*}{c+x*}
        \frac{\bigproduct_{\gamma \in \Gamma^+ \cup \Gamma^-} m(\gamma)}{|\operatorname{Aut}(u)|}
    \end{IEEEeqnarray*}
    elements, where $m(\gamma)$ is the multiplicity of the Reeb orbit $\gamma$ and $\operatorname{Aut}(u)$ is the automorphism group of $u = (\Sigma, j, \mathbf{z}, u)$, i.e. an element of $\operatorname{Aut}(u)$ is a biholomorphism $\phi \colon \Sigma \longrightarrow \Sigma$ such that $u \circ \phi = u$ and $\phi(z_i^{\pm}) = z_i^{\pm}$ for every $i$.
\end{remark}

We will work with the following assumptions. Let $\Sigma = S^2$, (without any specified almost complex structure). Let $\mathbf{z} = \{z_1, \ldots, z_p\} \subset \Sigma$ be a set of punctures and $\mathbf{v} = \{v_1, \ldots, v_p\}$ be a corresponding set of asymptotic markers. Assume also that we have a set $\mathbf{j} = \{j_1, \ldots, j_p\}$, where $j_i$ is an almost complex structure defined on a neighbourhood of $z_i$ for every $i = 1, \ldots,p$. For every $i$, there are cylindrical coordinates $(s, t)$ on $\dot{\Sigma}$ near $z_i$ as in \cref{def:punctures asy markers cyl ends}, with the additional property that $v_i$ agrees with the direction $t = 0$. We will also assume that $\mathcal{T} \subset \mathcal{J}(\Sigma)$ is a Teichmüller slice as in \cite[Section 3.1]{wendlAutomaticTransversalityOrbifolds2010}, where $\mathcal{J}(\Sigma)$ denotes the set of almost complex structures on $\Sigma = S^2$. Finally, let $(X, \lambda)$ be a nondegenerate Liouville domain of dimension $2n$ and $J \in \mathcal{J}(X)$ be an admissible almost complex structure on $\hat{X}$.

\begin{definition}
    Let $\gamma$ be an unparametrized simple Reeb orbit of $\partial X$. An \textbf{admissible parametrization} near $\gamma$ is a diffeomorphism $\phi \colon S^1 \times D^{2n-2} \longrightarrow O$, where $O \subset \partial X$ is an open neighbourhood of $\gamma$ and 
    \begin{IEEEeqnarray*}{c+x*}
        D^{2n-2} \coloneqq \{(z^1,\ldots,z^{n-1}) \in \C^{n-1} \mid |z^1| < 1, \ldots, |z^{n-1}| < 1 \}   
    \end{IEEEeqnarray*}
    is the polydisk, such that $t \longmapsto \phi(t,0)$ is a parametrization of $\gamma$. In this case, we denote by $(\vartheta, \zeta) = \phi^{-1} \colon O \longrightarrow S^1 \times D^{2n-2}$ the coordinates near $\gamma$.
\end{definition}

Let $\Gamma = (\gamma_{1},\ldots,\gamma_{p})$ be a tuple of (unparametrized) Reeb orbits in $\partial X$. Denote by $m_i$ the multiplicity of $\gamma_i$ and by $T_i$ the period of the simple Reeb orbit underlying $\gamma_i$ (so, the period of $\gamma_i$ is $m_i T_i$). For every $i = 1,\ldots,p $, choose once and for all an admissible parametrization $\phi_i \colon S^1 \times D^{2n-2} \longrightarrow O_i$ near the simple Reeb orbit underlying $\gamma_i$.

\begin{definition}
    \label{def:bundle for cr op}
    We define a vector bundle $\pi \colon \mathcal{E} \longrightarrow \mathcal{T} \times \mathcal{B}$ as follows. Let $\mathcal{B}$ be the set of maps $u \colon \dot{\Sigma} \longrightarrow \hat{X}$ of class $W^{k,p}_{\mathrm{loc}}$ satisfying the following property for every puncture $z_i$. Write $u$ with respect to the cylindrical coordinates $(s,t)$ defined from $(z_i, v_i)$. First, we require that $u(s,t) \in \R_{\geq 0} \times O_i$ for $s$ big enough. Write $u$ with respect to the coordinates $(\vartheta, \zeta)$ near $\gamma$ on the target and cylindrical coordinates $(s,t)$ on the domain:
    \begin{IEEEeqnarray*}{rCls+x*}
        u(s,t)
        & = & (\pi_{\R} \circ u(s,t), \pi_{\partial X} \circ u (s,t)) \\
        & = & (\pi_{\R} \circ u(s,t), \vartheta(s,t), \zeta(s,t)).
    \end{IEEEeqnarray*}
    Finally, we require that there exists $a \in \R$ such that the map
    \begin{IEEEeqnarray*}{c+x*}
        (s,t) \longmapsto (\pi_{\R} \circ u(s,t), \vartheta(s,t), \zeta(s,t)) - (m_i T_i s + a, m_i T_i t, 0)
    \end{IEEEeqnarray*}
    is of class $W^{k,p,\delta}$. The fibre, total space, projection and zero section are defined by
    \begin{IEEEeqnarray*}{rCls+x*}
        \mathcal{E}_{(j,u)} & \coloneqq & W^{k-1,p,\delta}(\Hom^{0,1}((T \dot{\Sigma}, j), (u^* T \hat{X}, J))), \quad \text{for every } (j,u) \in \mathcal{T} \times \mathcal{B}, \\
        \mathcal{E}         & \coloneqq & \bigcoproduct_{(j,u) \in \mathcal{T} \times \mathcal{B}} \mathcal{E}_{(j,u)} = \{ (j, u, \xi) \mid (j,u) \in \mathcal{T} \times \mathcal{B}, \, \xi \in \mathcal{E}_{(j,u)} \}, \\
        \pi(j,u, \eta)      & \coloneqq & (j,u), \\
        z(j,u)              & \coloneqq & (j,u,0).
    \end{IEEEeqnarray*}
\end{definition}

\begin{definition}
    \label{def:cauchy riemann operator}
    The \textbf{Cauchy--Riemann operators} are the sections
    \begin{IEEEeqnarray*}{rClCrCl}
        \overline{\partial}_j \colon \mathcal{B}                  & \longrightarrow & \mathcal{E}, & \qquad & \overline{\partial}_j(u) & \coloneqq & \frac{1}{2} (T u + J \circ Tu \circ j) \in \mathcal{E}_{(j,u)}, \\
        \overline{\partial} \colon \mathcal{T} \times \mathcal{B} & \longrightarrow & \mathcal{E}, & \qquad & \overline{\partial}(j,u) & \coloneqq & \overline{\partial}_j(u).
    \end{IEEEeqnarray*}
\end{definition}

\begin{definition}
    \label{def:linearized cr op}
    Let $(j,u) \in \mathcal{T} \times \mathcal{B}$ be such that $\overline{\partial}(j ,u) = 0$. Define the \textbf{vertical projection}
    \begin{IEEEeqnarray*}{c+x*}
        P_{(j,u)} \colon T_{(j,u,0)} \mathcal{E} \longrightarrow \mathcal{E}_{(j,u)}, \qquad P_{(j,u)} (\eta) \coloneqq \eta - \dv (z \circ \pi)(j,u,0) \eta.
    \end{IEEEeqnarray*}
    The \textbf{linearized Cauchy--Riemann operators} are the linear maps
    \begin{IEEEeqnarray*}{rCls+x*}
        \mathbf{D}_{(j,u)} & \coloneqq & P_{(j,u)} \circ \dv (\overline{\partial}_j)(u) \colon T_u \mathcal{B} \longrightarrow \mathcal{E}_{(j,u)}, \\
        \mathbf{L}_{(j,u)} & \coloneqq & P_{(j,u)} \circ \dv (\overline{\partial})(j,u) \colon T_j \mathcal{T} \oplus T_u \mathcal{B} \longrightarrow \mathcal{E}_{(j,u)}.
    \end{IEEEeqnarray*}
    Define also the restriction
    \begin{IEEEeqnarray*}{c+x*}
        \mathbf{F}_{(j,u)} \coloneqq \mathbf{L}_{(j,u)}|_{T_j \mathcal{T}} \colon T_j \mathcal{T} \longrightarrow \mathcal{E}_{(j,u)}.
    \end{IEEEeqnarray*}
\end{definition}

\begin{remark}
    \label{rmk:tangent of base of bundle}
    Choose a smooth function $\beta \colon \R \longrightarrow [0,1]$ such that $\beta(s) = 0$ if $s < 0$, $\beta(s) = 1$ if $s > 1$ and $0 \leq \beta'(s) \leq 2$. Consider the Liouville vector field $\hat{Z}^{X} \in \mathfrak{X}(\hat{X})$ and the Reeb vector field $R^{\partial X} \in \mathfrak{X}(\partial X)$. For every puncture $z$, let $(s,t)$ be the cylindrical coordinates near $z$ and define sections
    \begin{IEEEeqnarray*}{rClCrCl}
        \hat{Z}^X_z      & \in & \Gamma(u^* T \hat{X}), & \quad & \hat{Z}^X_z(s,t)      & = & \beta(s) \hat{Z}^X(u(s,t)), \\
        R^{\partial X}_z & \in & \Gamma(u^* T \hat{X}), & \quad & R^{\partial X}_z(s,t) & = & \beta(s) R^{\partial X}(u(s,t)).
    \end{IEEEeqnarray*}
    Denote $V = \bigoplus_{i=1}^{p} \spn \{\hat{Z}^X_{z_i}, R^{\partial X}_{z_i}\}$. Then, the tangent space of $\mathcal{B}$ is given by 
    \begin{IEEEeqnarray*}{c+x*}
        T_u \mathcal{B} = V \oplus W^{k,p,\delta}(\dot{\Sigma}, u^* T \hat{X}).
    \end{IEEEeqnarray*}
\end{remark}

\begin{definition}
    \label{def:conjugate and restriction operators}
    Let $(j,u) \in \mathcal{T} \times \mathcal{B}$ be such that $\overline{\partial}(j,u) = 0$ and consider the linearized Cauchy--Riemann operator $\mathbf{D}_{(j,u)}$. Choose a smooth function $f \colon \dot{\Sigma} \longrightarrow \R$ such that $f(s,t) = \delta s$ on every cylindrical end of $\dot{\Sigma}$. Define the \textbf{restriction} of $\mathbf{D}_{(j,u)}$, denoted $\mathbf{D}_{\delta}$, and the \textbf{conjugation} of $\mathbf{D}_{(j,u)}$, denoted $\mathbf{D}_0$, to be the unique maps such that the diagram
    \begin{IEEEeqnarray*}{c+x*}
        \begin{tikzcd}
            T_u \mathcal{B} \ar[d, swap, "\mathbf{D}_{(j,u)}"] & W^{k,p,\delta}(u^* T \hat{X}) \ar[d, "\mathbf{D}_{\delta}"] \ar[l, hook'] \ar[r, hook, two heads, "\xi \mapsto e^f \xi"] & W^{k,p}(u^* T \hat{X}) \ar[d, "\mathbf{D}_0"] \\
            \mathcal{E}_{(j,u)} \ar[r, equals]                 & W^{k-1,p,\delta}(\Hom^{0,1}(T \dot{\Sigma}, u^* T \hat{X})) \ar[r, hook, two heads, swap, "\eta \mapsto e^f \eta"]       & W^{k-1,p}(\Hom^{0,1}(T \dot{\Sigma}, u^* T \hat{X}))
        \end{tikzcd}
    \end{IEEEeqnarray*}
    commutes.
\end{definition}

\begin{lemma}
    \label{lem:D is a rlcro}
    The maps $\mathbf{D}_\delta$ and $\mathbf{D}_0$ are real linear Cauchy--Riemann operators.
\end{lemma}
\begin{proof}
    By \cite[Proposition 3.1.1]{mcduffHolomorphicCurvesSymplectic2012}, the map $\mathbf{D}_{\delta}$ is given by the equation
    \begin{IEEEeqnarray*}{c+x*}
        \mathbf{D}_{\delta} \xi = \frac{1}{2} \p{}{}{\nabla \xi + J(u) \nabla \xi \circ j} - \frac{1}{2} J(u) (\nabla_{\xi} J)(u) \partial(u),
    \end{IEEEeqnarray*}
    where $\nabla$ is the Levi-Civita connection on $\hat{X}$ associated to the Riemannian metric determined by $J$ and $\edv \hat{\lambda}$. Since $\nabla \colon \mathfrak{X}(\Sigma) \times \Gamma(u^* T \hat{X}) \longrightarrow \Gamma(u^* T \hat{X})$ satisfies the Leibniz rule with respect to the $\Gamma(u^* T \hat{X})$ argument, $\mathbf{D}_{\delta}$ is a real linear Cauchy--Riemann operator. We show that $\mathbf{D}_0$ satisfies the Leibniz rule.
    \begin{IEEEeqnarray*}{rCls+x*}
        \mathbf{D}_0 (g \xi)
        & = & e^f \mathbf{D}_{\delta} (e^{-f} g \xi)                                     & \quad [\text{by definition of $\mathbf{D}_{\delta}$}] \\
        & = & g e^f \mathbf{D}_{\delta} (e^{-f} \xi) + \xi \otimes \overline{\partial} g & \quad [\text{$\mathbf{D}_{\delta}$ obeys the Leibniz rule}] \\
        & = & g \mathbf{D}_{0} (\xi) + \xi \otimes \overline{\partial} g                 & \quad [\text{by definition of $\mathbf{D}_{\delta}$}]. & \qedhere
    \end{IEEEeqnarray*}
\end{proof}

\begin{lemma}
    \label{lem:Du is surjective case n is 1}
    If $n=1$ then $\mathbf{L}_{(j,u)}$ is surjective.
\end{lemma}
\begin{proof}
    Let $\tau_1$ be a global complex trivialization of $u^* T \hat{X}$ extending to an asymptotic unitary trivialization near the punctures. Let $\tau_2$ be the unitary trivialization of $u^* T \hat{X}$ near the punctures which is induced from the decomposition $T_{(r,x)}(\R \times \partial X) = \p{<}{}{\partial_r} \oplus \p{<}{}{R^{\partial X}_x}$. It is shown in the proof of \cite[Lemma 7.10]{wendlLecturesSymplecticField2016} that the operator $\mathbf{D}_0$ is asymptotic at $z_i$ to $- J \partial_t + \delta$, which is nondegenerate and has Conley--Zehnder index $\conleyzehnder^{\tau_2}(- J \partial_t + \delta) = -1$. Therefore, every $z_i$ is an odd puncture and $\# \mathbf{z}_0 = 0$. We show that $c_1^{\tau_2}(u^* T \hat{X}) = \sum_{i=1}^{p} m_i$, where $m_i$ is the multiplicity of the asymptotic Reeb orbit $\gamma_i$:%
    \begin{IEEEeqnarray*}{rCls+x*}
        c_1^{\tau_2}(u^* T \hat{X})
        & = & c_1^{\tau_1}(u^* T \hat{X}) + \sum_{i=1}^{p} \deg(\tau_1|_{E_{z_i}} \circ (\tau_2|_{E_{z_i}})^{-1}) & \quad [\text{by \cite[Exercise 5.3]{wendlLecturesSymplecticField2016}}] \\
        & = & \sum_{i=1}^{p} \deg(\tau_1|_{E_{z_i}} \circ (\tau_2|_{E_{z_i}})^{-1})                               & \quad [\text{by \cref{def:relative first chern number}}] \\
        & = & \sum_{i=1}^{p} m_i,
    \end{IEEEeqnarray*}
    where in the last equality we have used the fact that if $(s,t)$ are the cylindrical coordinates near $z_i$, then for $s$ large enough the map $t \longmapsto \tau_1|_{u(s,t)} \circ (\tau_2|_{u(s,t)})^{-1}$ winds around the origin $m_i$ times. We show that $\operatorname{ind} \mathbf{D}_0 \geq 2$.
    \begin{IEEEeqnarray*}{rCls+x*}
        \operatorname{ind} \mathbf{D}_0
        & =    & n \chi(\dot{\Sigma}) + 2 c_1^{\tau_2}(u^* T \hat{X}) + \sum_{i=1}^{p} \conleyzehnder^{\tau_2}(- J \partial_t + \delta) & \quad [\text{by \cref{thm:riemann roch with punctures}}] \\
        & =    & 2 + 2 \sum_{i=1}^{p} (m_i - 1)                                                                                         & \quad [\text{since $n = 1$ and $g = 0$}] \\
        & \geq & 2                                                                                                                      & \quad [\text{since $m_i \geq 1$ for every $i$}].
    \end{IEEEeqnarray*}
    By \cref{lem:conditions for D surjective genus zero}, this implies that $\mathbf{D}_0$ is surjective. By \cref{def:conjugate and restriction operators}, the operator $\mathbf{D}_{(j,u)}$ is also surjective. Therefore, $\mathbf{L}_{(j,u)} = \mathbf{F}_{(j,u)} + \mathbf{D}_{(j,u)}$ is also surjective.
\end{proof}

From now until the end of this section, let $(X, \lambda^X)$ be a Liouville domain of dimension $2n$ and $(Y, \lambda^Y)$ be a Liouville domain of dimension $2n + 2$ such that 
\begin{enumerate}
    \item $X \subset Y$ and $\partial X \subset \partial Y$;
    \item the inclusion $\iota \colon X \longrightarrow Y$ is a Liouville embedding;
    \item if $x \in X$ then $Z_x^{X} = Z_x^{Y}$;
    \item if $x \in \partial X$ then $R_x^{\partial X} = R^{\partial Y}_x$.
\end{enumerate}
In this case, we have an inclusion of completions $\hat{X} \subset \hat{Y}$ as sets. By assumption, $Z^X$ is $\iota$-related to $Z^Y$, which implies that there is a map $\hat{\iota} \colon \hat{X} \longrightarrow \hat{Y}$ on the level of completions. Since in this case $\hat{X} \subset \hat{Y}$ and by \cref{def:embedding on completions coming from Liouville embedding}, $\hat{\iota}$ is the inclusion. Assume that $J^X \in \mathcal{J}(X)$ and $J^Y \in \mathcal{J}(Y)$ are almost complex structures on $\hat{X}$ and $\hat{Y}$ respectively, such that $\hat{\iota} \colon \hat{X} \longrightarrow \hat{Y}$ is holomorphic. As before, let $\Gamma = (\gamma_{1},\ldots,\gamma_{p})$ be a tuple of unparametrized Reeb orbits in $\partial X$. Notice that each $\gamma_i$ can also be seen as a Reeb orbit in $\partial Y$. For every $i = 1,\ldots,p$, choose once and for all admissible parametrizations $\phi_i^X \colon S^1 \times D^{2n-2} \longrightarrow O_i^X$ and $\phi_i^Y \colon S^1 \times D^{2n} \longrightarrow O_i^Y$ near $\gamma_i$ with the property that the diagram
\begin{IEEEeqnarray*}{c+x*}
    \begin{tikzcd}
        S^1 \times D^{2n - 2} \ar[r, hook, two heads, "\phi^X_i"] \ar[d, hook] & O^X_i \ar[r, hook] \ar[d, hook, dashed, "\exists !"] & \partial X \ar[d, hook, "\iota_{\partial Y, \partial X}"] \\
        S^1 \times D^{2n} \ar[r, hook, two heads, "\phi^Y_i"]                  & O^Y_i \ar[r, hook]                                   & \partial Y
    \end{tikzcd}
\end{IEEEeqnarray*}
commutes. We will consider the bundle of \cref{def:bundle for cr op} as well as the Cauchy--Riemann operator and its linearization for both $X$ and $Y$. We will use the notation
\begin{IEEEeqnarray*}{rClCrClCrCl}
    \pi^X \colon \mathcal{E}X & \longrightarrow & \mathcal{T} \times \mathcal{B}X, & \qquad & \overline{\partial}\vphantom{\partial}^X \colon \mathcal{T} \times \mathcal{B}X & \longrightarrow & \mathcal{E} X, & \qquad & \mathbf{L}^X_{(j,u)} \colon T_j \mathcal{T} \oplus T_u \mathcal{B} X & \longrightarrow & \mathcal{E}_{(j,u)} X, \\
    \pi^Y \colon \mathcal{E}Y & \longrightarrow & \mathcal{T} \times \mathcal{B}Y, & \qquad & \overline{\partial}\vphantom{\partial}^Y \colon \mathcal{T} \times \mathcal{B}Y & \longrightarrow & \mathcal{E} Y, & \qquad & \mathbf{L}^Y_{(j,w)} \colon T_j \mathcal{T} \oplus T_w \mathcal{B} Y & \longrightarrow & \mathcal{E}_{(j,w)} Y
\end{IEEEeqnarray*}
to distinguish the bundles and maps for $X$ and $Y$. Define maps
\begin{IEEEeqnarray*}{rClCrCl}
    \mathcal{B}\iota \colon \mathcal{B} X & \longrightarrow & \mathcal{B}Y, & \quad & \mathcal{B}\iota(u)        & \coloneqq & \hat{\iota} \circ u, \\
    \mathcal{E}\iota \colon \mathcal{E} X & \longrightarrow & \mathcal{E}Y, & \quad & \mathcal{E}\iota(j,u,\eta) & \coloneqq & (j, \hat{\iota} \circ u, T \hat{\iota} \circ \eta).
\end{IEEEeqnarray*}
Then, the diagrams
\begin{IEEEeqnarray*}{c+x*}
    \begin{tikzcd}
        \mathcal{E}X \ar[r, "\pi^X"] \ar[d, swap, "\mathcal{E}\iota"]                                                                                & \mathcal{T} \times \mathcal{B}X \ar[d, "\id_{\mathcal{T}} \times \mathcal{B}\iota"] & & \mathcal{T} \times \mathcal{B}X \ar[d, swap, "\id_{\mathcal{T}} \times \mathcal{B}\iota"] \ar[r, "z^X"] & \mathcal{E}X \ar[d, "\mathcal{E}\iota"] \\
        \mathcal{E}Y \ar[r, swap, "\pi^Y"]                                                                                                           & \mathcal{T} \times \mathcal{B}Y                                                     & & \mathcal{T} \times \mathcal{B}Y \ar[r, swap, "z^Y"]                                                     & \mathcal{E}Y                            \\
        \mathcal{T} \times \mathcal{B}X \ar[r, "\overline{\partial}\vphantom{\partial}^X"] \ar[d, swap, "\id_{\mathcal{T}} \times \mathcal{B}\iota"] & \mathcal{E}X \ar[d, "\mathcal{E}\iota"]                                             & & (z^X)^* T \mathcal{E} X \ar[r, "P^X"] \ar[d, swap, "T \mathcal{E} \iota"]                               & \mathcal{E} X \ar[d, "\mathcal{E} \iota"] \\
        \mathcal{T} \times \mathcal{B}Y \ar[r, swap, "\overline{\partial}\vphantom{\partial}^Y"]                                                     & \mathcal{E}Y                                                                        & & (z^Y)^* T \mathcal{E} Y \ar[r, swap, "P^Y"]                                                             & \mathcal{E} Y
    \end{tikzcd}
\end{IEEEeqnarray*}
commute. By the chain rule, the diagram
\begin{IEEEeqnarray}{c+x*}
    \plabel{eq:diag naturality of lcro}
    \begin{tikzcd}
        T_u \mathcal{B} X                     \ar[rr, bend left = 40, "\mathbf{D}^X_{(j,u)}"]                          \ar[r, "\dv \overline{\partial}\vphantom{\partial}^X_j(u)"]                         \ar[d, swap, "\dv(\mathcal{B} \iota)(u)"] & T_{(j,u,0)} \mathcal{E} X                     \ar[r, "P_{(j,u)}^X"]                         \ar[d, "\dv(\mathcal{E}\iota)(\overline{\partial}\vphantom{\partial}^X_j(u))"] & \mathcal{E}_{(j,u)} X                    \ar[d, "\mathcal{E}_{(j,u)} \iota"] \\
        T_{\hat{\iota} \circ u} \mathcal{B} Y \ar[rr, swap, bend right = 40, "\mathbf{D}^Y_{(j,\hat{\iota} \circ u)}"] \ar[r, swap, "\dv \overline{\partial}\vphantom{\partial}^Y_j(\hat{\iota} \circ u)"]                                           & T_{(j, \hat{\iota} \circ u, 0)} \mathcal{E} Y \ar[r, swap, "P^Y_{(j,\hat{\iota} \circ u)}"]                                                                                & \mathcal{E}_{(j, \hat{\iota} \circ u)} Y 
    \end{tikzcd}
\end{IEEEeqnarray}
is also commutative whenever $\overline{\partial}\vphantom{\partial}^X(j,u) = 0$.

\begin{remark}
    \label{rmk:splittings of B and E}
    Consider the formula for the tangent space of $\mathcal{B}X$ from \cref{rmk:tangent of base of bundle}. By the assumptions on the Liouville domains $X$ and $Y$, we have that $V^X = V^Y$. Also, the diagrams
    \begin{IEEEeqnarray*}{c+x*}
        \begin{tikzcd}
            T_u \mathcal{B} X \ar[r, hook]                                                  & T_{u} \mathcal{B} Y                                                & W^{k,p,\delta}(u^* (T \hat{X})^{\perp}) \ar[l, hook'] \ar[d, equals] \\
            W^{k,p,\delta}(u^* T \hat{X}) \ar[r, hook] \ar[d, two heads, hook] \ar[u, hook] & W^{k,p,\delta}(u^* T \hat{Y}) \ar[u, hook] \ar[d, two heads, hook] & W^{k,p,\delta}(u^* (T \hat{X})^{\perp}) \ar[l, hook'] \ar[d, two heads, hook] \\
            W^{k,p}(u^* T \hat{X}) \ar[r, hook]                                             & W^{k,p}(u^* T \hat{Y})                                             & W^{k,p}(u^* (T \hat{X})^{\perp}) \ar[l, hook']
        \end{tikzcd} \\
        \begin{tikzcd}
            \mathcal{E}_{(j,u)} X \ar[r, hook] \ar[d, hook, two heads] & \mathcal{E}_{(j,u)} Y \ar[d, hook, two heads] & \Omega^{k-1,p,\delta}_j(u^*(T \hat{X})^{\perp}) \ar[d, hook, two heads] \ar[l, hook'] \\
            \Omega^{k-1,p}_j(u^*T \hat{X}) \ar[r, hook]                & \Omega^{k-1,p}_j(u^*T \hat{Y})                & \Omega^{k-1,p}_j(u^*(T \hat{X})^{\perp}) \ar[l, hook']
        \end{tikzcd}
    \end{IEEEeqnarray*}
    commute, where for shortness we are using the notation
    \begin{IEEEeqnarray*}{c+x*}
        \Omega^{k,p}_{j}(E) = W^{k,p}(\Hom^{0,1}((T \dot{\Sigma}, j), (E, J)))
    \end{IEEEeqnarray*}
    for any complex vector bundle $(E, J) \longrightarrow \dot{\Sigma}$. In both diagrams, the middle term of every row is the direct sum of the left and right terms. In addition, the vertical maps in the middle of both diagrams are block diagonal when written with respect to these decompositions.
\end{remark}

\begin{definition}
    Let $z_0 \in \dot{\Sigma}$. Define the \textbf{evaluation map}
    \begin{IEEEeqnarray*}{rrCl}
        \operatorname{ev}^X \colon & \mathcal{B} X & \longrightarrow & \hat{X} \\
                                   & u             & \longmapsto     & u(z_0)
    \end{IEEEeqnarray*}
    as well as its derivative $\mathbf{E}^X_u \coloneqq \dv (\operatorname{ev}^{X})(u) \colon T_u \mathcal{B} X \longrightarrow T_{u(z_0)} \hat{X}$.
\end{definition}

In the following lemma, we show that if a holomorphic curve $u$ in $X$ is regular (in $X$) then the corresponding holomorphic curve $\hat{\iota} \circ u$ in $Y$ is also regular. See also \cite[Proposition A.1]{mcduffSymplecticCapacitiesUnperturbed2022} for a similar result.

\begin{lemma}
    \label{lem:DX surj implies DY surj}
    Let $u \in \mathcal{B}X$ be holomorphic and denote $\hat{\iota} \circ u \in \mathcal{B} Y$ simply by $u$. Assume that the normal Conley--Zehnder index of every asymptotic Reeb orbit $\gamma_i$ is $1$.
    \begin{enumerate}
        \item \label{lem:DX surj implies DY surj 1} If $\mathbf{L}_{(j,u)}^X$ is surjective then so is $\mathbf{L}^Y_{(j,u)}$.
        \item \label{lem:DX surj implies DY surj 2} If $\mathbf{L}_{(j,u)}^X \oplus \mathbf{E}^X_u$ is surjective then so is $\mathbf{L}^Y_{(j,u)} \oplus \mathbf{E}^Y_u$.
    \end{enumerate}
\end{lemma}
\begin{proof}
    Consider the decomposition $T_x \hat{Y} = T_x \hat{X} \oplus (T_x \hat{X})^{\perp}$ for $x \in \hat{X}$. Let $\tau$ be a global complex trivialization of $u^* T \hat{Y}$, extending to an asymptotic unitary trivialization near the punctures, and such that $\tau$ restricts to a trivialization of $u^* T \hat{X}$ and $u^* (T \hat{X})^{\perp}$. By \cref{rmk:splittings of B and E}, there are splittings
    \begin{IEEEeqnarray*}{rCls+x*}
        T_u \mathcal{B} Y & = & T_u \mathcal{B} X \oplus T_u^{\perp} \mathcal{B} X, \\
        \mathcal{E}_{(j,u)} Y & = & \mathcal{E}_{(j,u)} X \oplus \mathcal{E}_{(j,u)}^{\perp} X.
    \end{IEEEeqnarray*}
    We can write the maps
    \begin{IEEEeqnarray*}{rCl}
        \mathbf{L}_{(j,u)}^Y & \colon & T_j \mathcal{T} \oplus T_u \mathcal{B} X \oplus T_u^{\perp} \mathcal{B} X \longrightarrow \mathcal{E}_{(j,u)} X \oplus \mathcal{E}_{(j,u)}^{\perp} X, \\
        \mathbf{D}_{(j,u)}^Y & \colon & T_u \mathcal{B} X \oplus T_u^{\perp} \mathcal{B} X                        \longrightarrow \mathcal{E}_{(j,u)} X \oplus \mathcal{E}_{(j,u)}^{\perp} X, \\
        \mathbf{L}_{(j,u)}^X & \colon & T_j \mathcal{T} \oplus T_u \mathcal{B} X                                  \longrightarrow \mathcal{E}_{(j,u)} X, \\
        \mathbf{F}_{(j,u)}^Y & \colon & T_j \mathcal{T}                                                           \longrightarrow \mathcal{E}_{(j,u)} X \oplus \mathcal{E}_{(j,u)}^{\perp} X, \\
        \mathbf{E}_{u}^Y     & \colon & T_u \mathcal{B} X \oplus T_u^{\perp} \mathcal{B} X                        \longrightarrow T_x \hat{X} \oplus (T_x \hat{X})^{\perp}
    \end{IEEEeqnarray*}
    as block matrices
    \begin{IEEEeqnarray}{rCl}
        \mathbf{L}_{(j,u)}^Y
        & = &
        \begin{bmatrix}
            \mathbf{F}^X_{(j,u)} & \mathbf{D}^X_{(j,u)} & \mathbf{D}^{TN}_{(j,u)} \\
            0                    & 0                    & \mathbf{D}^{NN}_{(j,u)}
        \end{bmatrix}, \plabel{eq:decomposition of cr ops 1}\\
        \mathbf{D}_{(j,u)}^Y
        & = &
        \begin{bmatrix}
            \mathbf{D}^X_{(j,u)} & \mathbf{D}^{TN}_{(j,u)} \\
            0                    & \mathbf{D}^{NN}_{(j,u)}
        \end{bmatrix}, \plabel{eq:decomposition of cr ops 2}\\
        \mathbf{L}_{(j,u)}^X
        & = &
        \begin{bmatrix}
            \mathbf{F}^X_{(j,u)} & \mathbf{D}^X_{(j,u)}
        \end{bmatrix}, \plabel{eq:decomposition of cr ops 3}\\
        \mathbf{F}_{(j,u)}^Y
        & = &
        \begin{bmatrix}
            \mathbf{F}^X_{(j,u)} \\
            0                    
        \end{bmatrix}, \plabel{eq:decomposition of cr ops 4}\\
        \mathbf{E}_{u}^Y
        & = &
        \begin{bmatrix}
            \mathbf{E}^X_{u} & 0 \\
            0                & \mathbf{E}^{NN}_{u}
        \end{bmatrix}, \plabel{eq:decomposition of cr ops 5}
    \end{IEEEeqnarray}
    where \eqref{eq:decomposition of cr ops 5} follows by definition of the evaluation map, \eqref{eq:decomposition of cr ops 4} is true since $\mathbf{F}^{Y}_{(j,u)}$ is given by the formula $\mathbf{F}^{Y}_{(j,u)}(y) = \frac{1}{2} (J \circ T u \circ y)$, \eqref{eq:decomposition of cr ops 2} follows because diagram \eqref{eq:diag naturality of lcro} commutes, and \eqref{eq:decomposition of cr ops 3} and \eqref{eq:decomposition of cr ops 1} then follow by \cref{def:linearized cr op}. Let $\mathbf{D}^{NN}_\delta$ be the restriction and $\mathbf{D}_0^{NN}$ be the conjugation of $\mathbf{D}^{NN}_{(j,u)}$ (as in \cref{def:conjugate and restriction operators}). Denote by $\mathbf{B}^{NN}_{\gamma_i}$ the asymptotic operator of $\mathbf{D}^{NN}_{\delta}$ at $z_i$. Then the asymptotic operator of $\mathbf{D}^{NN}_0$ at $z_i$ is $\mathbf{B}^{NN}_{\gamma_i} + \delta$, which by assumption has Conley--Zehnder index equal to $1$. We show that $\operatorname{ind} \mathbf{D}_0^{NN} = 2$.
    \begin{IEEEeqnarray*}{rCls+x*}
        \operatorname{ind} \mathbf{D}_0^{NN}
        & = & \chi(\dot{\Sigma}) + 2 c_1^{\tau}(u^* T \hat{X}) + \sum_{i=1}^{p} \conleyzehnder^{\tau}(\mathbf{B}^{NN}_{{\gamma_i}} + \delta) & \quad [\text{by \cref{thm:riemann roch with punctures}}] \\
        & = & 2                                                                                                                              & \quad [\text{since $\conleyzehnder^{\tau}(\mathbf{B}^{NN}_{{\gamma_i}} + \delta) = 1$}].
    \end{IEEEeqnarray*}
    We prove \ref{lem:DX surj implies DY surj 1}.
    \begin{IEEEeqnarray*}{rCls+x*}
        \operatorname{ind} \mathbf{D}_0^{NN} = 2
        & \Longrightarrow & \mathbf{D}_0^{NN}       \text{ is surjective} & \quad [\text{by \cref{lem:conditions for D surjective genus zero}}] \\
        & \Longrightarrow & \mathbf{D}_\delta^{NN}  \text{ is surjective} & \quad [\text{$\mathbf{D}_0^{NN}$ and $\mathbf{D}_{\delta}^{NN}$ are conjugated}] \\
        & \Longrightarrow & \mathbf{D}_{(j,u)}^{NN} \text{ is surjective} & \quad [\text{$\mathbf{D}_{\delta}^Y$ is a restriction of $\mathbf{D}_{(j,u)}^Y$}] \\
        & \Longrightarrow & \mathbf{L}_{(j,u)}^Y    \text{ is surjective} & \quad [\text{$\mathbf{L}_{(j,u)}^X$ is surjective by assumption}].
    \end{IEEEeqnarray*}
    We prove \ref{lem:DX surj implies DY surj 2}.
    \begin{IEEEeqnarray*}{rCls+x*}
        \IEEEeqnarraymulticol{3}{l}{\operatorname{ind} \mathbf{D}_0^{NN} = 2}\\ \quad
        & \Longrightarrow & \mathbf{D}_0^{NN}       \oplus \mathbf{E}_u^{NN} \text{ is surjective} & \quad [\text{by \cref{lem:D plus E is surjective}}] \\
        & \Longrightarrow & \mathbf{D}_\delta^{NN}  \oplus \mathbf{E}_u^{NN} \text{ is surjective} & \quad [\text{$\mathbf{D}_0^{NN} \oplus \mathbf{E}^{NN}_u$ and $\mathbf{D}_{\delta}^{NN} \oplus \mathbf{E}^{NN}_{u}$ are conjugated}] \\
        & \Longrightarrow & \mathbf{D}_{(j,u)}^{NN} \oplus \mathbf{E}_u^{NN} \text{ is surjective} & \quad [\text{$\mathbf{D}_{\delta}^Y \oplus \mathbf{E}^{Y}_{u}$ is a restriction of $\mathbf{D}_{(j,u)}^Y \oplus \mathbf{E}^{Y}_u$}] \\
        & \Longrightarrow & \mathbf{L}_{(j,u)}^Y    \oplus \mathbf{E}_u^{Y}  \text{ is surjective} & \quad [\text{$\mathbf{L}_{(j,u)}^X \oplus \mathbf{E}_u^{X}$ is surjective by assumption}].                                                & \qedhere
    \end{IEEEeqnarray*}
\end{proof}

\section{Moduli spaces of curves in ellipsoids}
\label{sec:augmentation map of an ellipsoid}

We now use the techniques explained in the past two sections to compute the augmentation map of an ellipsoid (\cref{thm:augmentation is nonzero}). The proof of this theorem consists in an explicit count of curves in the ellipsoid satisfying a tangency constraint (\cref{lem:moduli spaces of ellipsoids have 1 element}) together with the fact that the moduli space of such curves is transversely cut out (\cref{prp:moduli spaces without point constraint are tco,prp:moduli spaces w point are tco,prp:moduli spaces w tangency are tco}). Therefore, the explicit count agrees with the virtual count. We now state the assumptions for this section.

Let $a_1 < \cdots < a_n \in \R_{> 0}$ be rationally linearly independent and consider the ellipsoid $E(a_1,\ldots,a_n) \subset \C^n$. By \cite[Section 2.1]{guttSymplecticCapacitiesPositive2018}, $\partial E(a_1, \ldots, a_n)$ has exactly $n$ simple Reeb orbits $\gamma_1, \ldots, \gamma_n$, which satisfy
\begin{IEEEeqnarray}{rCls+x*}
    \gamma_j(t)                & = & \sqrt{\frac{a_j}{\pi}} e^{\frac{2 \pi i t}{a_j}} e_j, \\
    \mathcal{A}(\gamma^m_j)    & = & m a_j, \\
    \conleyzehnder(\gamma^m_j) & = & n - 1 + 2 \sum_{i=1}^{n} \p{L}{2}{\frac{m a_j}{a_i}}, \plabel{eq:cz of reeb in ellipsoid}
\end{IEEEeqnarray}
where $\gamma_j \colon \R / a_j \Z \longrightarrow \partial E(a_1, \ldots, a_n)$ and $e_j$ is the $j$th vector of the canonical basis of $\C^n$ as a vector space over $\C$. For simplicity, for every $\ell = 1, \ldots, n$ denote $E_\ell = E(a_1,\ldots,a_\ell) \subset \C^\ell$. Notice that $\gamma_1$ is a Reeb orbit of $\partial E_1, \ldots, \partial E_n$. Define maps
\begin{IEEEeqnarray*}{rClCrCl}
    \iota_{\ell} \colon \C^{\ell} & \longrightarrow & \C^{\ell + 1}, & \quad & \iota_\ell(z_1,\ldots,z_\ell) & \coloneqq & (z_1,\ldots,z_\ell,0) \\
    h_{\ell} \colon \C^{\ell}     & \longrightarrow & \C,            & \quad & h_\ell(z_1,\ldots,z_\ell)     & \coloneqq & z_1.
\end{IEEEeqnarray*}
The maps $\iota_{\ell} \colon E_\ell \longrightarrow E_{\ell+1}$ are Liouville embeddings satisfying the assumptions in \cref{sec:functional analytic setup}. Define also
\begin{IEEEeqnarray*}{rCls+x*}
    x_\ell   & \coloneqq & 0 \in \C^\ell, \\
    D_{\ell} & \coloneqq & \{ (z_1,\ldots,z_\ell) \in \C^{\ell} \mid z_1 = 0 \} = h_{\ell}^{-1}(0).
\end{IEEEeqnarray*}
Choose an admissible almost complex structure $J_{\ell} \in \mathcal{J}(E_\ell, D_\ell)$ on $\hat{E}_{\ell}$ such that $J_{\ell}$ is the canonical almost complex structure of $\C^\ell$ near $0$. We assume that the almost complex structures are chosen in such a way that $\hat{\iota}_{\ell} \colon \hat{E}_{\ell} \longrightarrow \hat{E}_{\ell + 1}$ is holomorphic and also such that there exists a biholomorphism $\varphi \colon \hat{E}_1 \longrightarrow \C$ such that $\varphi(z) = z$ for $z$ near $0 \in \C$ (see \cref{lem:biholomorphism explicit} below). Let $m \in \Z_{\geq 1}$ and assume that $m a_1 < a_2 < \cdots < a_n$.

Consider the sphere $S^2$, without any specified almost complex structure, with a puncture $z_1 \in S^2$ and an asymptotic marker $v_1 \in (T_{z_1} S^2 \setminus \{0\}) / \R_{> 0}$, and also a marked point $z_0 \in \dot{S}^2 = S^2 \setminus \{z_1\}$. For $k \in \Z_{\geq 0}$, denote%
\begin{IEEEeqnarray*}{lCls+x*}
    \mathcal{M}^{\ell,(k)}_{\mathrm{p}} 
    & \coloneqq & \mathcal{M}_{E_{\ell}}^{\$, J_{\ell}}(\gamma^m_1)\p{<}{}{\mathcal{T}^{(k)}x_\ell}_{\mathrm{p}} \\
    & \coloneqq & \left\{
        (j, u)
        \ \middle\vert
        \begin{array}{l}
            j \text{ is an almost complex structure on }S^2, \\
            u \colon (\dot{S}^2, j) \longrightarrow (\hat{E}_\ell, J_\ell) \text{ is as in \cref{def:asy cyl holomorphic curve}}, \\
            u(z_0) = x_\ell \text{ and $u$ has contact order $k$ to $D_\ell$ at $x_\ell$}
        \end{array}
    \right\}.
\end{IEEEeqnarray*}
Here, the subscript $\mathrm{p}$ means that the moduli space consists of parametrized curves, i.e. we are not quotienting by biholomorphisms. Denote the moduli spaces of regular curves and of unparametrized curves by
\begin{IEEEeqnarray*}{lCls+x*}
    \mathcal{M}^{\ell,(k)}_{\mathrm{p,reg}} & \coloneqq & \mathcal{M}_{E_{\ell}}^{\$, J_{\ell}}(\gamma^m_1)\p{<}{}{\mathcal{T}^{(k)}x_\ell}_{\mathrm{p,reg}}, \\
    \mathcal{M}^{\ell,(k)}                  & \coloneqq & \mathcal{M}_{E_{\ell}}^{\$, J_{\ell}}(\gamma^m_1)\p{<}{}{\mathcal{T}^{(k)}x_\ell} \coloneqq \mathcal{M}^{\ell,(k)}_{\mathrm{p}} / \sim.
\end{IEEEeqnarray*}
Here, $\mathcal{M}^{\ell,(0)} \coloneqq \mathcal{M}_{E_{\ell}}^{\$, J_{\ell}}(\gamma^m_1)\p{<}{}{\mathcal{T}^{(0)}x_\ell} \coloneqq \mathcal{M}_{E_{\ell}}^{\$, J_{\ell}}(\gamma^m_1)$ and analogously for $\mathcal{M}^{\ell,(0)}_{\mathrm{p,reg}}$ and $\mathcal{M}^{\ell,(0)}_{\mathrm{p}}$.

\begin{lemma}
    \phantomsection\label{lem:biholomorphism explicit}
    For any $a > 0$, there exists an almost complex structure $J$ on $\hat{B}(a)$ and a biholomorphism $\varphi \colon \hat{B}(a) \longrightarrow \C$ such that
    \begin{enumerate}
        \item \label{lem:biholomorphism explicit 1} $J$ is cylindrical on $\R_{\geq 0} \times \partial B(a)$;
        \item \label{lem:biholomorphism explicit 2} $J$ is the canonical almost complex structure of $\C$ near $0 \in B(a) \subset \C$;
        \item \label{lem:biholomorphism explicit 3} $\varphi(z) = z$ for $z$ near $0 \in B(a) \subset \C$.
    \end{enumerate}
\end{lemma}
\begin{proof}
    Choose $\rho_0 < 0$ and let $g \colon \R \longrightarrow \R_{>0}$ be a function such that $g(\rho) = a/4 \pi$ for $\rho \leq \rho_0$ and $g(\rho) = 1$ for $\rho \geq 0$. For $(\rho, w) \in \R \times \partial B(a)$, define
    \begin{IEEEeqnarray*}{rCls+x*}
        f(\rho)                         & \coloneqq & \exp \p{}{2}{\frac{\rho_0}{2} + \frac{2 \pi}{a} \int_{\rho_0}^{\rho} g(\sigma) \edv \sigma}, \\
        J_{(\rho, w)} (\partial_{\rho}) & \coloneqq & g (\rho) R^{\partial B(a)}_{w}, \\
        \varphi(\rho, w)                & \coloneqq & f(\rho) w.
    \end{IEEEeqnarray*}
    Property \ref{lem:biholomorphism explicit 1} follows from the fact that $g(\rho) = 1$ for $\rho \geq 0$. Consider the Liouville vector field of $\C$, which is denoted by $Z$ and given by $Z(w) = w/2$. Let $\Phi \colon \R \times \partial B(a) \longrightarrow \C$ be the map given by $\Phi(\rho, w) = \phi^\rho_Z(w) = \exp(\rho/2) w$. By definition of completion, $\Phi|_{B(a) \setminus \{0\}} \colon B(a) \setminus \{0\} \longrightarrow \C$ is the inclusion. To prove property \ref{lem:biholomorphism explicit 3}, it suffices to show that $\varphi(\rho, w) = \Phi(\rho, w)$ for every $(\rho, w) \in \R_{\leq \rho_0} \times \partial B(a)$. For this, simply note that
    \begin{IEEEeqnarray*}{rCls+x*}
        f(\rho)
        & = & \exp \p{}{2}{\frac{\rho_0}{2} + \frac{2 \pi}{a} \int_{\rho_0}^{\rho} g(\sigma) \edv \sigma} & \quad [\text{by definition of $f$}] \\
        & = & \exp \p{}{2}{\frac{\rho_0}{2} + \frac{2 \pi}{a} (\rho - \rho_0) \frac{a}{4 \pi} }           & \quad [\text{$\rho \leq \rho_0$ implies $g(\rho) = a / 4 \pi$}] \\
        & = & \exp \p{}{2}{\frac{\rho}{2}}.
    \end{IEEEeqnarray*}
    Therefore, $\varphi(z) = z$ for $z$ near $0 \in B(a) \subset \C$, and in particular $\varphi$ can be extended smoothly to a map $\varphi \colon \hat{B}(a) \longrightarrow \C$. We show that $\varphi$ is holomorphic.
    \begin{IEEEeqnarray*}{rCls+x*}
        j \circ \dv \varphi(\rho, w) (\partial_{\rho}) 
        & = & j \p{}{2}{\pdv{}{\rho} \p{}{1}{f(\rho) |w|} \pdv{}{r}\Big|_{\varphi(\rho, w)}}                & \quad [\text{by definition of $\varphi$}] \\
        & = & \frac{2 \pi}{a} \, g(\rho) \, j \p{}{2}{ f(\rho) |w| \pdv{}{r}\Big|_{\varphi(\rho, w)}}       & \quad [\text{by definition of $f$}] \\
        & = & \frac{2 \pi}{a} \, g(\rho) \, j \p{}{2}{ |\varphi(\rho,w)| \pdv{}{r}\Big|_{\varphi(\rho, w)}} & \quad [\text{by definition of $\varphi$}] \\
        & = & \frac{2 \pi}{a} \, g(\rho) \, \pdv{}{\theta}\Big|_{\varphi(\rho, w)}                          & \quad [\text{by definition of $j$}] \\
        & = & g(\rho) \, \dv \varphi(\rho, w) (R^{\partial B(a)}_w)                                         & \quad [\text{by \cite[Equation (2.2)]{guttSymplecticCapacitiesPositive2018}}] \\
        & = & \dv \varphi(\rho, w) \circ J (\partial_{\rho})                                                & \quad [\text{by definition of $J$}],
    \end{IEEEeqnarray*}
    Where $(r, \theta)$ are the polar coordinates of $\C$. Since $\varphi$ is holomorphic and $\varphi$ is the identity near the origin, we conclude that $J$ is the canonical almost complex structure of $\C$ near the origin. In particular, $J$ can be extended smoothly to an almost complex structure on $\hat{B}(a)$, which proves \ref{lem:biholomorphism explicit 2}. Finally, we show that $\varphi$ is a diffeomorphism. For this, it suffices to show that $\Phi^{-1} \circ \varphi \colon \R \times \partial B(a) \longrightarrow \R \times \partial B(a)$ is a diffeomorphism. This map is given by $\Phi^{-1} \circ \varphi(\rho, w) = (2 \ln(f(\rho)), w)$. Since
    \begin{IEEEeqnarray*}{c+x*}
        \odv{}{\rho} (2 \ln(f(\rho))) = 2 \frac{f'(\rho)}{f(\rho)} = \frac{4 \pi}{a} g(\rho) > 0,
    \end{IEEEeqnarray*}
    $\varphi$ is a diffeomorphism.
\end{proof}

\begin{lemma}
    \label{lem:psi j}
    Let $\operatorname{inv} \colon \overline{\C} \longrightarrow \overline{\C}$ be the map given by $\operatorname{inv}(z) = 1/z$ and consider the vector $V \coloneqq \dv \operatorname{inv}(0) \partial_x \in T_{\infty} \overline{\C}$. For every $j \in \mathcal{T}$ there exists a unique biholomorphism $\psi_j \colon (\overline{\C}, j_0) \longrightarrow (S^2, j)$ such that
    \begin{IEEEeqnarray*}{c+x*}
        \psi_j(0) = z_0, \qquad \psi_j(\infty) = z_1, \qquad \dv \psi_j(\infty) V = \frac{v_1}{\| v_1 \|},
    \end{IEEEeqnarray*}
    where $\| \cdot \|$ is the norm coming from the canonical Riemannian metric on $S^2$ as the sphere of radius $1$ in $\R^3$.
\end{lemma}
\begin{proof}
    By the uniformization theorem \cite[Theorem XII.0.1]{desaint-gervaisUniformizationRiemannSurfaces2016}, there exists a biholomorphism $\phi \colon (S^2, j) \longrightarrow (\overline{\C}, j_0)$. Since there exists a unique Möbius transformation $\psi_0 \colon (\overline{\C}, j_0) \longrightarrow (\overline{\C}, j_0)$ such that
    \begin{IEEEeqnarray*}{c+x*}
        \psi_0(0) = \phi(z_0), \qquad \psi_0(\infty) = \phi(z_1), \qquad \dv \psi_0 (\infty) V = \dv \phi(z_1) \frac{v_1}{\| v_1 \|},
    \end{IEEEeqnarray*}
    the result follows.
\end{proof}

We will denote also by $\psi_j$ the restriction $\psi_j \colon (\C, j_0) \longrightarrow (S^2, j)$.

\begin{lemma}
    \label{lem:u is a polynomial}
    If $(j,u) \in \mathcal{M}^{1,(0)}$ then $\varphi \circ u \circ \psi_j \colon \C \longrightarrow \C$ is a polynomial of degree $m$.
\end{lemma}
\begin{proof}
    Since $u$ is positively asymptotic to $\gamma^m_1$, the map $\varphi \circ u \circ \psi_j$ goes to $\infty$ as $z$ goes to $\infty$. Therefore, $\varphi \circ u \circ \psi_j$ is a polynomial. Again using the fact that $u$ is positively asymptotic to $\gamma^m_1$, we conclude that for $r$ big enough the path $\theta \longmapsto \varphi \circ u \circ \psi_j(r e^{i \theta})$ winds around the origin $m$ times. This implies that the degree of $\varphi \circ u \circ \psi_j$ is $m$.
\end{proof}

\begin{lemma}
    \label{lem:normal cz is one}
    For every $\ell = 1,\ldots,n-1$, view $\gamma^m_1$ as a Reeb orbit of $\partial E_{\ell} \subset \partial E_{\ell + 1}$. The normal Conley--Zehnder index of $\gamma^m_1$ is $1$.
\end{lemma}
\begin{proof}
    By \cite[Equation (2.2)]{guttSymplecticCapacitiesPositive2018}, the Reeb vector field of $\partial E_{\ell + 1}$ is given by
    \begin{IEEEeqnarray*}{c+x*}
        R^{\partial E_{\ell + 1}} = 2 \pi \sum_{j=1}^{\ell+1} \frac{1}{a_j} \pdv{}{\theta_{j}},
    \end{IEEEeqnarray*}
    where $\theta_j$ denotes the angular polar coordinate of the $j$th summand of $\C^{\ell+1}$. Therefore, the flow of $R^{\partial E_{\ell + 1}}$ is given by
    \begin{IEEEeqnarray*}{rrCl}
        \phi^{t}_{R} \colon & \partial E_{\ell+1}     & \longrightarrow & \partial E_{\ell+1} \\
                                                  & (z_1,\ldots,z_{\ell+1}) & \longmapsto     & \p{}{2}{e^{\frac{2 \pi i}{a_1}} z_1, \ldots, e^{\frac{2 \pi i}{a_{\ell+1}}} z_{\ell+1}}.
    \end{IEEEeqnarray*}
    The diagram
    \begin{IEEEeqnarray*}{c+x*}
        \begin{tikzcd}
            \xi^{\partial E_{\ell}}_{\gamma^m_1(0)} \ar[r] \ar[d, swap, "\dv \phi^t_{R}(\gamma^m_1(0))"] & \xi^{\partial E_{\ell+1}}_{\gamma^m_1(0)} \ar[d, "\dv \phi^t_{R}(\gamma^m_1(0))"] & \big(\xi^{\partial E_{\ell+1}}_{\gamma^m_1(0)}\big)^{\perp} \ar[l] \ar[d, "\dv \phi^t_{R}(\gamma^m_1(0))"] \ar[r, equals] & \C \ar[d, "\times \exp \p{}{1}{\frac{2 \pi i t}{a_{\ell+1}}}"] \\
            \xi^{\partial E_{\ell}}_{\gamma^m_1(t)} \ar[r]                                               & \xi^{\partial E_{\ell+1}}_{\gamma^m_1(t)}                                         & \big(\xi^{\partial E_{\ell+1}}_{\gamma^m_1(t)}\big)^{\perp} \ar[l] \ar[r, equals]                                         & \C
        \end{tikzcd}
    \end{IEEEeqnarray*}
    commutes. Define a path $A_{\gamma^m_1} \colon [0,m a_1] \longrightarrow \operatorname{Sp}(2)$ by $A_{\gamma^m_1}(t) = \exp (t J_0 S)$, where
    \begin{IEEEeqnarray*}{c+x*}
        S = \frac{2 \pi}{a_{\ell + 1}}
        \begin{bmatrix}
            1 & 0 \\
            0 & 1
        \end{bmatrix}.
    \end{IEEEeqnarray*}
    The only eigenvalue of $S$ is $2 \pi / a_{\ell+1}$, which has multiplicity $2$. Therefore, the signature of $S$ is $\signature S = 2$. These facts allow us to compute $\conleyzehnder^{\perp}(\gamma^m_1)$ using \cref{prp:gutts cz formula}:
    \begin{IEEEeqnarray*}{rCls+x*}
        \conleyzehnder^{\perp}(\gamma^m_1)
        & = & \conleyzehnder(A_{\gamma^m_1})                                                                                                & \quad [\text{by definition of $\conleyzehnder^{\perp}$}] \\
        & = & \p{}{2}{\frac{1}{2} + \p{L}{2}{\sqrt{\frac{2 \pi}{a_{\ell + 1}}\frac{2 \pi}{a_{\ell + 1}}} \frac{m a_1}{2 \pi}}} \signature S & \quad [\text{by \cref{prp:gutts cz formula}}] \\
        & = & \frac{1}{2} \signature S                                                                                                      & \quad [\text{since $m a_1 < a_2 < \cdots < a_n$}] \\
        & = & 1                                                                                                                             & \quad [\text{by the discussion above}].                    & \qedhere
    \end{IEEEeqnarray*}
\end{proof}

\begin{lemma}
    \label{lem:unique reeb orbit with cz equal to}
    If $\ell = 1,\ldots,n$ then $\gamma^m_1$ is the unique Reeb orbit of $\partial E_{\ell}$ such that $\conleyzehnder(\gamma^m_1) = \ell - 1 + 2m$.    
\end{lemma}
\begin{proof}
    First, notice that
    \begin{IEEEeqnarray*}{rCls+x*}
        \conleyzehnder(\gamma^m_1)
        & = & \ell - 1 + 2 \sum_{j=1}^{\ell} \p{L}{2}{\frac{m a_1}{a_j}} & \quad [\text{by equation \eqref{eq:cz of reeb in ellipsoid}}] \\
        & = & \ell - 1 + 2 m                                          & \quad [\text{since $m a_1 < a_2 < \cdots < a_n$}].
    \end{IEEEeqnarray*}
    Conversely, let $\gamma = \gamma^k_i$ be a Reeb orbit of $\partial E_\ell$ with $\conleyzehnder(\gamma) = \ell - 1 + 2m$. By equation \eqref{eq:cz of reeb in ellipsoid}, this implies that 
    \begin{IEEEeqnarray}{c+x*}
        \label{eq:k is sum of floors}
        m = \sum_{j=1}^{\ell} \p{L}{2}{\frac{k a_i}{a_j}}.
    \end{IEEEeqnarray}
    We show that $i = 1$. Assume by contradiction otherwise. Then
    \begin{IEEEeqnarray*}{rCls+x*}
        m
        & =    & \sum_{1 \leq j \leq \ell} \p{L}{2}{\frac{k a_i}{a_j}}                          & \quad [\text{by equation \eqref{eq:k is sum of floors}}] \\
        & \geq & \sum_{1 \leq j \leq i} \p{L}{2}{\frac{k a_i}{a_j}}                             & \quad [\text{since every term in the sum is $\geq 0$}] \\
        & =    & \p{L}{2}{\frac{k a_i}{a_1}} + \sum_{1 < j < i} \p{L}{2}{\frac{k a_i}{a_j}} + k & \quad [\text{since by assumption, $i > 1$}] \\
        & \geq & (m + i - 1) k                                                                  & \quad [\text{$m a_1 < a_2 < \cdots < a_i$}] \\
        & >    & m k                                                                            & \quad [\text{since by assumption, $i > 1$}],
    \end{IEEEeqnarray*}
    which is a contradiction, and therefore $i = 1$. We show that $k = m$, using the fact that $m \geq \lfloor k a_i / a_1 \rfloor = k$.
    \begin{IEEEeqnarray*}{rCls+x*}
        m
        & = & \sum_{1 \leq j \leq \ell} \p{L}{2}{\frac{k a_1}{a_j}}     & \quad [\text{by equation \eqref{eq:k is sum of floors} and since $i = 1$}] \\
        & = & k + \sum_{2 \leq j \leq \ell} \p{L}{2}{\frac{k a_1}{a_j}} & \\
        & = & k                                                         & \quad [\text{since $k \leq m$ and $k a_1 \leq m a_1 < a_1 < \cdots < a_n$}]. & \qedhere
    \end{IEEEeqnarray*}
\end{proof}

\begin{lemma}
    \label{lem:lch of ellipsoid}
    The module $CH_{n - 1 + 2m}(E_n)$ is the free $\Q$-module generated by $\gamma^m_1$.
\end{lemma}
\begin{proof}
    By equation \eqref{eq:cz of reeb in ellipsoid}, every Reeb orbit of $\partial E_n$ is good. We claim that the differential $\partial \colon CC(E_n) \longrightarrow CC(E_n)$ is zero. Assume by contradiction that there exists a Reeb orbit $\gamma$ such that $\partial \gamma \neq 0$. By definition of $\partial$, this implies that there exist Reeb orbits $\eta, \alpha_1, \ldots, \alpha_p$ such that
    \begin{IEEEeqnarray*}{rCls+x*}
        0 & \neq & \#^{\mathrm{vir}} \overline{\mathcal{M}}^{J_n}_{\partial E_n}(\gamma; \eta, \alpha_1, \ldots, \alpha_p), \\
        0 & \neq & \#^{\mathrm{vir}} \overline{\mathcal{M}}^{J_n}_{E_n}(\alpha_j), \quad \text{for } j=1,\ldots,p.
    \end{IEEEeqnarray*}
    By assumption on the virtual perturbation scheme,
    \begin{IEEEeqnarray*}{rCls+x*}
        0 & =   & \operatorname{virdim} \overline{\mathcal{M}}^{J_n}_{E_n}(\alpha_j) = n - 3 + \conleyzehnder(\alpha_j) \quad \text{for every } j = 1,\ldots,p, \\ \\
        0 & =   & \operatorname{virdim} \overline{\mathcal{M}}^{J_n}_{\partial E_n}(\gamma; \eta, \alpha_1, \ldots, \alpha_p) \\
          & =   & (n-3)(2 - (2+p)) + \conleyzehnder(\gamma) - \conleyzehnder(\eta) - \sum_{j=1}^{p} \conleyzehnder(\alpha_j) - 1 \\
          & =   & \conleyzehnder(\gamma) - \conleyzehnder(\eta) - 1 \\
          & \in & 1 + 2 \Z,
    \end{IEEEeqnarray*}
    where in the last line we used equation \eqref{eq:cz of reeb in ellipsoid}. This gives the desired contradiction, and we conclude that $\partial \colon CC(E_n) \longrightarrow CC(E_n)$ is zero. Therefore, $CH(E_n) = CC(E_n)$ is the free $\Q$-module generated by the Reeb orbits of $\partial E_n$. By \cref{lem:unique reeb orbit with cz equal to}, $\gamma^m_1$ is the unique Reeb orbit of $\partial E_n$ with $\conleyzehnder(\gamma^m_1) = n - 1 + 2m$, from which the result follows.
\end{proof}

\begin{lemma}
    \phantomsection\label{lem:moduli spaces of ellipsoids are all equal}
    If $\ell = 1,\ldots,n$ and $k \in \Z_{\geq 1}$ then $\mathcal{M}^{\ell,(k)}_{\mathrm{p}} = \mathcal{M}^{1,(k)}_{\mathrm{p}}$ and $\mathcal{M}^{\ell,(k)} = \mathcal{M}^{1,(k)}$.
\end{lemma}
\begin{proof}
    It suffices to show that $\mathcal{M}^{\ell,(k)}_{\mathrm{p}} = \mathcal{M}^{\ell+1,(k)}_{\mathrm{p}}$ for every $\ell = 1,\ldots,n-1$. The inclusion $\mathcal{M}^{\ell,(k)}_{\mathrm{p}} \subset \mathcal{M}^{\ell+1,(k)}_{\mathrm{p}}$ follows from the fact that the inclusion $\hat{E}_\ell \hookrightarrow \hat{E}_{\ell+1}$ is holomorphic and the assumptions on the symplectic divisors. To prove that $\mathcal{M}^{\ell+1,(k)}_{\mathrm{p}} \subset \mathcal{M}^{\ell,(k)}_{\mathrm{p}}$, it suffices to assume that $(j,u) \in \mathcal{M}^{\ell+1,(k)}_{\mathrm{p}}$ and to show that the image of $u$ is contained in $\hat{E}_\ell \subset \hat{E}_{\ell+1}$. Since $u$ has contact order $k$ to $D_{\ell+1}$ at $x_{\ell+1} = \iota_{\ell}(x_{\ell})$, we conclude that $u$ is not disjoint from $\hat{E}_\ell$. By \cref{lem:stabilization 2}, $u$ is contained in $\hat{E}_\ell$.
\end{proof}

We now prove that the moduli spaces $\mathcal{M}^{\ell,(k)}$ are regular. The proof strategy is as follows.
\begin{enumerate}
    \item \cref{prp:moduli spaces without point constraint are tco} deals with the moduli spaces $\mathcal{M}^{1,(0)}$. We show that the linearized Cauchy--Riemann operator is surjective using \cref{lem:Du is surjective case n is 1}.
    \item \cref{prp:moduli spaces w point are tco} deals with the moduli spaces $\mathcal{M}^{\ell,(1)}$. Here, we need to consider the linearized Cauchy--Riemann operator together with an evaluation map. We show inductively that this map is surjective using \cref{lem:DX surj implies DY surj}.
    \item Finally, \cref{prp:moduli spaces w tangency are tco} deals with the moduli spaces $\mathcal{M}^{\ell,(k)}$. We now need to consider the jet evaluation map. We prove inductively that this map is surjective by writing it explicitly.
\end{enumerate}

\begin{proposition}
    \label{prp:moduli spaces without point constraint are tco}
    The moduli spaces $\mathcal{M}^{1,(0)}_{\mathrm{p}}$ and $\mathcal{M}^{1,(0)}$ are transversely cut out.
\end{proposition}
\begin{proof}
    It is enough to show that $\mathcal{M}^{1,(0)}_{\mathrm{p}}$ is transversely cut out, since this implies that $\mathcal{M}^{1,(0)}$ is transversely cut out as well. Recall that $\mathcal{M}^{1,(0)}_{\mathrm{p}}$ can be written as the zero set of the Cauchy--Riemann operator $\overline{\partial}\vphantom{\partial}^{1} \colon \mathcal{T} \times \mathcal{B} E_{1} \longrightarrow \mathcal{E} E_{1}$. It suffices to assume that $(j,u) \in (\overline{\partial}\vphantom{\partial}^{1})^{-1}(0)$ and to prove that the linearization 
    \begin{IEEEeqnarray*}{c+x*}
        \mathbf{L}_{(j,u)}^1 \colon T_j \mathcal{T} \oplus T_u \mathcal{B} E_1 \longrightarrow \mathcal{E}_{(j,u)} E_1
    \end{IEEEeqnarray*}
    is surjective. This follows from \cref{lem:Du is surjective case n is 1}.
\end{proof}

\begin{proposition}
    \label{prp:moduli spaces w point are tco}
    If $\ell = 1,\ldots,n$ then $\mathcal{M}^{\ell,(1)}_{\mathrm{p}}$ and $\mathcal{M}^{\ell,(1)}$ are transversely cut out.
\end{proposition}
\begin{proof}
    We will use the notation of \cref{sec:functional analytic setup} with $X = E_{\ell}$ and $Y = E_{\ell + 1}$. We will show by induction on $\ell$ that $\mathcal{M}^{\ell,(1)}_{\mathrm{p}}$ is transversely cut out. This implies that $\mathcal{M}^{\ell,(1)}$ is transversely cut out as well.

    We prove the base case. By \cref{prp:moduli spaces without point constraint are tco}, $\mathcal{M}^{1,(0)}_{\mathrm{p}}$ is a smooth manifold. Consider the evaluation map
    \begin{IEEEeqnarray*}{rrCl}
        \operatorname{ev}^{1} \colon & \mathcal{M}^{1,(0)}_{\mathrm{p}} & \longrightarrow & \hat{E}_1 \\
                                     & (j,u)                            & \longmapsto     & u(z_0).
    \end{IEEEeqnarray*}
    Notice that $\mathcal{M}^{1,(1)}_{\mathrm{p}} = (\operatorname{ev}^1)^{-1}(x_1)$. We wish to show that the linearized evaluation map $\mathbf{E}^1_{(j,u)} = \dv (\operatorname{ev}^1)(j,u) \colon T_{(j,u)} \mathcal{M}^{1,(0)}_{\mathrm{p}} \longrightarrow T_{u(z_0)} \hat{E}_1$ is surjective whenever $u(z_0) = \operatorname{ev}^{1}(j,u) = x_1$. There are commutative diagrams
    \begin{IEEEeqnarray*}{c+x*}
        \begin{tikzcd}
            \mathcal{M}^{1,(0)}_{\mathrm{p}} \ar[r, two heads, "\Phi"] \ar[d, swap, "\operatorname{ev}^1"] & \mathcal{M} \ar[d, "\operatorname{ev}_{\mathcal{M}}"] & \mathcal{C} \ar[l, swap, hook', two heads, "\mathcal{P}"] \ar[d, "\operatorname{ev}_{\mathcal{C}}"] & & T_{(j,u)} \mathcal{M}^{1,(0)}_{\mathrm{p}} \ar[r, two heads, "{\dv \Phi(j,u)}"] \ar[d, swap, "{\mathbf{E}^1_{(j,u)}}"] & T_f \mathcal{M} \ar[d, "\mathbf{E}_{\mathcal{M}}"] & \C^{m+1} \ar[l, swap, hook', two heads, "\dv \mathcal{P}(a)"] \ar[d, "\mathbf{E}_{\mathcal{C}}"] \\
            \hat{E}_1 \ar[r, hook, two heads, swap, "\varphi"]                                             & \C \ar[r, equals]                                     & \C                                                                                                  & & T_{x_1} \hat{E}_1 \ar[r, hook, two heads, swap, "\dv \varphi(x_1)"]                                                    & \C \ar[r, equals]                                  & \C
        \end{tikzcd}
    \end{IEEEeqnarray*}
    where
    \begin{IEEEeqnarray*}{rCls+x*}
        \mathcal{M}                                     & \coloneqq & \{f \colon \C \longrightarrow \C \mid f \text{ is a polynomial of degree }m \}, \\
        \mathcal{C}                                     & \coloneqq & \{(a_0,\ldots,a_m) \in \C^{m+1} \mid a_m \neq 0\}, \\
        \Phi(j,u)                                       & \coloneqq & \varphi \circ u \circ \psi_j, \\
        \operatorname{ev}_{\mathcal{M}}(f)              & \coloneqq & f(0), \\
        \operatorname{ev}_{\mathcal{C}}(a_0,\ldots,a_m) & \coloneqq & a_0, \\
        \mathcal{P}(a_0,\ldots,a_m)(z)                  & \coloneqq & a_0 + a_1 z + \cdots + a_m z^m,
    \end{IEEEeqnarray*}
    and the diagram on the right is obtained by linearizing the one on the left. The map $\Phi$ is well-defined by \cref{lem:u is a polynomial}. Since $\mathbf{E}_{\mathcal{C}}(a_0,\ldots,a_m) = a_0$ is surjective, $\mathbf{E}^1_u$ is surjective as well. This finishes the proof of the base case.

    We prove the induction step, i.e. that if $\mathcal{M}^{\ell,(1)}_p$ is transversely cut out then so is $\mathcal{M}^{\ell+1,(1)}_p$. We prove that $\mathcal{M}^{\ell,(1)}_{\mathrm{p,reg}} \subset \mathcal{M}^{\ell+1,(1)}_{\mathrm{p,reg}}$. For this, assume that $(j,u) \in \mathcal{M}^{\ell,(1)}_{\mathrm{p}}$ is such that $\mathbf{L}_{(j,u)}^\ell \oplus \mathbf{E}_u^\ell \colon T_j \mathcal{T} \oplus T_{u} \mathcal{B} E_\ell \longrightarrow \mathcal{E}_{(j,u)} E_\ell \oplus T_{x_\ell} \hat{E}_\ell$ is surjective. By \cref{lem:DX surj implies DY surj},
    \begin{IEEEeqnarray*}{c+x*}
        \mathbf{L}_{(j,u)}^{\ell+1} \oplus \mathbf{E}_u^{\ell+1} \colon T_j \mathcal{T} \oplus T_{u} \mathcal{B} E_{\ell+1} \longrightarrow \mathcal{E}_{(j,u)} E_{\ell+1} \oplus T_{x_{\ell+1}} \hat{E}_{\ell+1}
    \end{IEEEeqnarray*} 
    is also surjective, which means that $(j,u) \in \mathcal{M}^{\ell+1,(1)}_{\mathrm{p,reg}}$. This concludes the proof of $\mathcal{M}^{\ell,(1)}_{\mathrm{p,reg}} \subset \mathcal{M}^{\ell+1,(1)}_{\mathrm{p,reg}}$. Finally, we show that $\mathcal{M}^{\ell+1,(1)}_{\mathrm{p,reg}} = \mathcal{M}^{\ell+1,(1)}_{\mathrm{p}}$.
    \begin{IEEEeqnarray*}{rCls+x*}
        \mathcal{M}^{\ell+1,(1)}_{\mathrm{p,reg}}
        & \subset & \mathcal{M}^{\ell+1,(1)}_{\mathrm{p}}     & \quad [\text{since regular curves form a subset}] \\
        & =       & \mathcal{M}^{\ell,(1)}_{\mathrm{p}}       & \quad [\text{by \cref{lem:moduli spaces of ellipsoids are all equal}}] \\
        & =       & \mathcal{M}^{\ell,(1)}_{\mathrm{p,reg}}   & \quad [\text{by the induction hypothesis}] \\
        & \subset & \mathcal{M}^{\ell+1,(1)}_{\mathrm{p,reg}} & \quad [\text{proven above}].                                             & \qedhere
    \end{IEEEeqnarray*}
\end{proof}

\begin{proposition}
    \label{prp:moduli spaces w tangency are tco}
    If $\ell = 1,\ldots, n$ and $k = 1,\ldots,m$ then $\mathcal{M}^{\ell,(k)}_{\mathrm{p}}$ and $\mathcal{M}^{\ell,(k)}$ are transversely cut out.
\end{proposition}
\begin{proof}
    By \cref{prp:moduli spaces w point are tco}, $\mathcal{M}^{\ell,(1)}_{\mathrm{p}}$ is a smooth manifold. Consider the jet evaluation map
    \begin{IEEEeqnarray*}{rrCl}
        j^{\ell,(k)} \colon & \mathcal{M}^{\ell,(1)}_{\mathrm{p}} & \longrightarrow & \C^{k-1} \\
                            & (j,u)                               & \longmapsto     & ((h_{\ell} \circ u \circ \psi_j)^{(1)}(0), \ldots, (h_{\ell} \circ u \circ \psi_j)^{(k-1)}(0)).
    \end{IEEEeqnarray*}
    The moduli space $\mathcal{M}^{\ell,(k)}_{\mathrm{p}}$ is given by $\mathcal{M}^{\ell,(k)}_{\mathrm{p}} = (j^{\ell,(k)})^{-1}(0)$. We will prove by induction on $\ell$ that $\mathcal{M}^{\ell,(k)}_{\mathrm{p}}$ is transversely cut out. This shows that $\mathcal{M}^{\ell,(k)}$ is transversely cut out as well. Define $\mathbf{J}^{\ell,(k)}_{(j,u)} \coloneqq \dv(j^{\ell,(k)})(j,u) \colon T_{(j,u)} \mathcal{M}^{\ell,(1)}_{\mathrm{p}} \longrightarrow \C^{k-1}$.

    We prove the base case, i.e. that $\mathcal{M}^{1,(k)}_{\mathrm{p}}$ is transversely cut out. For this, it suffices to assume that $(j,u) \in \mathcal{M}^{1,(1)}_{\mathrm{p}}$ is such that $j^{1,(k)}(j,u) = 0$ and to prove that $\mathbf{J}^{1,(k)}_{(j,u)}$ is surjective. There are commutative diagrams
    \begin{IEEEeqnarray*}{c+x*}
        \begin{tikzcd}
            \mathcal{M}^{1,(1)}_{\mathrm{p}} \ar[r, two heads, "\Phi"] \ar[d, swap, "j^{1,(k)}"] & \mathcal{M} \ar[d, "j^{(k)}_{\mathcal{M}}"] & \mathcal{C} \ar[l, swap, hook', two heads, "\mathcal{P}"] \ar[d, "j^{(k)}_{\mathcal{C}}"] & & T_{(j,u)} \mathcal{M}^{1,(1)}_{\mathrm{p}} \ar[r, two heads, "{\dv \Phi(j,u)}"] \ar[d, swap, "{\mathbf{J}^{1,(k)}_{(j,u)}}"] & T_f \mathcal{M} \ar[d, "\mathbf{J}^{(k)}_{\mathcal{M}}"] & \C^{m} \ar[l, swap, hook', two heads, "\dv \mathcal{P}(a)"] \ar[d, "\mathbf{J}^{(k)}_{\mathcal{C}}"] \\
            \C^{k-1} \ar[r, equals]                                                              & \C^{k-1} \ar[r, equals]                     & \C^{k-1}                                                                                  & & \C^{k-1} \ar[r, equals]                                                                                                      & \C^{k-1} \ar[r, equals]                                  & \C^{k-1}
        \end{tikzcd}
    \end{IEEEeqnarray*}
    where
    \begin{IEEEeqnarray*}{rCls+x*}
        \mathcal{M}                                     & \coloneqq & \{f \colon \C \longrightarrow \C \mid f \text{ is a polynomial of degree }m \text{ with }f(0)=0 \}, \\
        \mathcal{C}                                     & \coloneqq & \{(a_1,\ldots,a_m) \in \C^{m} \mid a_m \neq 0\}, \\
        \Phi(j,u)                                       & \coloneqq & \varphi \circ u \circ \psi_j, \\
        j^{(k)}_{\mathcal{M}}(f)                        & \coloneqq & (f^{(1)}(0),\ldots,f^{(k-1)}(0)), \\
        j^{(k)}_{\mathcal{C}}(a_1,\ldots,a_m)           & \coloneqq & (a_1,\ldots,(k-1)! a_{k-1}), \\
        \mathcal{P}(a_1,\ldots,a_m)(z)                  & \coloneqq & a_1 z + \cdots + a_m z^m,
    \end{IEEEeqnarray*}
    and the diagram on the right is obtained by linearizing the one on the left. The map $\Phi$ is well-defined by \cref{lem:u is a polynomial}. Since $\mathbf{J}^{(k)}_{\mathcal{C}}(a_1,\ldots,a_m) = (a_1,\ldots,(k-1)! a_{k-1})$ is surjective, $\mathbf{J}^{1,(k)}_u$ is surjective as well. This finishes the proof of the base case.

    We prove the induction step, i.e. that if $\mathcal{M}^{\ell,(k)}_{\mathrm{p}}$ is transversely cut out then so is $\mathcal{M}^{\ell+1,(k)}_{\mathrm{p}}$. We show that $\mathcal{M}^{\ell,(k)}_{\mathrm{p,reg}} \subset \mathcal{M}^{\ell+1,(k)}_{\mathrm{p,reg}}$. For this, it suffices to assume that $(j,u) \in \mathcal{M}^{\ell,(k)}_{\mathrm{p}}$ is such that $\mathbf{J}^{\ell,(k)}_{(j,u)}$ is surjective, and to prove that $\mathbf{J}^{\ell+1,(k)}_{(j,u)}$ is surjective as well. This follows because the diagrams
    \begin{IEEEeqnarray*}{c+x*}
        \begin{tikzcd}
            \mathcal{M}^{\ell,(1)}_{\mathrm{p}} \ar[d] \ar[dr, "j^{\ell,(k)}"]   &          & & T_{(j,u)} \mathcal{M}^{\ell,(1)}_{\mathrm{p}} \ar[d] \ar[dr, "\mathbf{J}^{\ell,(k)}_u"] \\
            \mathcal{M}^{\ell+1,(1)}_{\mathrm{p}} \ar[r, swap, "j^{\ell+1,(k)}"] & \C^{k-1} & & T_{(j,u)} \mathcal{M}^{\ell+1,(1)}_{\mathrm{p}} \ar[r, swap, "\mathbf{J}_u^{\ell+1,(k)}"] & \C^{k-1}
        \end{tikzcd}
    \end{IEEEeqnarray*}
    commute. Finally, we show that $\mathcal{M}^{\ell+1,(k)}_{\mathrm{p,reg}} = \mathcal{M}^{\ell+1,(k)}_{\mathrm{p}}$.
    \begin{IEEEeqnarray*}{rCls+x*}
        \mathcal{M}^{\ell+1,(k)}_{\mathrm{p,reg}}
        & \subset & \mathcal{M}^{\ell+1,(k)}_{\mathrm{p}}     & \quad [\text{since regular curves form a subset}] \\
        & =       & \mathcal{M}^{\ell,(k)}_{\mathrm{p}}       & \quad [\text{by \cref{lem:moduli spaces of ellipsoids are all equal}}] \\
        & =       & \mathcal{M}^{\ell,(k)}_{\mathrm{p,reg}}   & \quad [\text{by the induction hypothesis}] \\
        & \subset & \mathcal{M}^{\ell+1,(k)}_{\mathrm{p,reg}} & \quad [\text{proven above}].                                             & \qedhere
    \end{IEEEeqnarray*}
\end{proof}

\begin{proposition}
    \label{lem:moduli spaces of ellipsoids have 1 element}
    If $\ell = 1,\ldots,n$ then $\#^{\mathrm{vir}} \overline{\mathcal{M}}^{\ell,(m)} = \# \overline{\mathcal{M}}^{\ell,(m)} = 1$.
\end{proposition}
\begin{proof}
    By assumption on the perturbation scheme and \cref{prp:moduli spaces w tangency are tco}, $\#^{\mathrm{vir}} \overline{\mathcal{M}}^{\ell,(m)} = \# \overline{\mathcal{M}}^{\ell,(m)}$. Again by \cref{prp:moduli spaces w tangency are tco}, the moduli space $\mathcal{M}^{\ell,(m)}$ is transversely cut out and%
    \begin{IEEEeqnarray*}{c}
        \dim \mathcal{M}^{\ell,(m)} = (n -3)(2 - 1) + \conleyzehnder(\gamma_1^m) - 2 \ell - 2 m + 4 = 0,
    \end{IEEEeqnarray*}
    where in the second equality we have used \cref{lem:unique reeb orbit with cz equal to}. This implies that $\mathcal{M}^{\ell,(m)}$ is compact, and in particular $\# \overline{\mathcal{M}}^{\ell,(m)} = \# \mathcal{M}^{\ell,(m)}$. By \cref{lem:moduli spaces of ellipsoids are all equal}, $\# \mathcal{M}^{\ell,(m)} = \# \mathcal{M}^{1,(m)}$. It remains to show that $\# \mathcal{M}^{1,(m)} = 1$. For this, notice that $\mathcal{M}^{1,(m)}$ is the set of equivalence classes of pairs $(j,u)$, where $j$ is an almost complex structure on $\Sigma = S^2$ and $u \colon (\dot{\Sigma}, j) \longrightarrow (\hat{E}_1, J_1)$ is a holomorphic map such that 
    \begin{enumerate}
        \item $u(z_0) = x_1$ and $u$ has contact order $m$ to $D_1$ at $x_1$;
        \item if $(s,t)$ are the cylindrical coordinates on $\dot{\Sigma}$ near $z_1$ such that $v_1$ agrees with the direction $t = 0$, then
            \begin{IEEEeqnarray*}{rrCls+x*}
                \lim_{s \to +\infty} & \pi_{\R} \circ u(s,t)           & = & + \infty, \\
                \lim_{s \to +\infty} & \pi_{\partial E_1} \circ u(s,t) & = & \gamma_1 (a_1 m t).
            \end{IEEEeqnarray*}
    \end{enumerate}
    Here, two pairs $(j_0, u_0)$ and $(j_1, u_1)$ are equivalent if there exists a biholomorphism $\phi \colon (\Sigma, j_0) \longrightarrow (\Sigma, j_1)$ such that
    \begin{IEEEeqnarray*}{c+x*}
        \phi(z_0) = z_0, \qquad \phi(z_1) = z_1, \qquad \dv \phi(z_1) v_1 = v_1.
    \end{IEEEeqnarray*}
    We claim that any two pairs $(j_0, u_0)$ and $(j_1, u_1)$ are equivalent. By \cref{lem:u is a polynomial}, the maps $\varphi \circ u_0 \circ \psi_{j_0}$ and $\varphi \circ u_1 \circ \psi_{j_1}$ are polynomials of degree $m$:
    \begin{IEEEeqnarray*}{rCls+x*}
        \varphi \circ u_0 \circ \psi_{j_0} (z) & = & a_0 + \cdots + a_m z^m, \\
        \varphi \circ u_1 \circ \psi_{j_1} (z) & = & b_0 + \cdots + b_m z^m.
    \end{IEEEeqnarray*}
    Since $u_0$ and $u_1$ have contact order $m$ to $D_1$ at $x_1$, for every $\nu = 0,\ldots,m-1$ we have
    \begin{IEEEeqnarray*}{rCls+x*}
        0 & = & (\varphi \circ u_0 \circ \psi_{j_0})^{(\nu)}(0) = \nu! a_{\nu}, \\
        0 & = & (\varphi \circ u_1 \circ \psi_{j_1})^{(\nu)}(0) = \nu! b_{\nu}.
    \end{IEEEeqnarray*}
    Since $u_0$ and $u_1$ have the same asymptotic behaviour, $\operatorname{arg}(a_m) = \operatorname{arg}(b_m)$. Hence, there exists $\lambda \in \R_{>0}$ such that $\lambda^m b_m = a_m$. Then, 
    \begin{IEEEeqnarray*}{c+x*}
        u_1 \circ \psi_{j_1} (\lambda z) = u_0 \circ \psi_{j_0} (z).
    \end{IEEEeqnarray*}
    Therefore, $(j_0, u_0)$ and $(j_1, u_1)$ are equivalent and $\# \mathcal{M}^{1,(m)} = 1$.
\end{proof}

\begin{remark}
    In \cite[Proposition 3.4]{cieliebakPuncturedHolomorphicCurves2018}, Cieliebak and Mohnke show that the signed count of the moduli space of holomorphic curves in $\C P^n$ in the homology class $[\C P^1]$ which satisfy a tangency condition $\p{<}{}{\mathcal{T}^{(n)}x}$ equals $(n-1)!$. It is unclear how this count relates to the one of \cref{lem:moduli spaces of ellipsoids have 1 element}.
\end{remark}

Finally, we will use the results of this section to compute the augmentation map of the ellipsoid $E_n$.

\begin{theorem}
    \label{thm:augmentation is nonzero}
    The augmentation map $\epsilon_m \colon CH_{n - 1 + 2m}(E_n) \longrightarrow \Q$ is an isomorphism.
\end{theorem}
\begin{proof}
    By \cref{lem:moduli spaces of ellipsoids have 1 element}, \cref{rmk:counts of moduli spaces with or without asy markers} and definition of the augmentation map, we have $\epsilon_m(\gamma^m_1) \neq 0$. By \cref{lem:lch of ellipsoid}, $\epsilon_m$ is an isomorphism.
\end{proof}

\section{Computations using contact homology}

Finally, we use the tools developed in this chapter to prove \cref{conj:the conjecture} (see \cref{thm:my main theorem}). The proof we give is the same as that of \cref{lem:computation of cl}, with the update that we will use the capacity $\mathfrak{g}^{\leq 1}_{k}$ to prove that 
\begin{IEEEeqnarray*}{c+x*}
    \tilde{\mathfrak{g}}^{\leq 1}_k(X) \leq \mathfrak{g}^{\leq 1}_k(X) = \cgh{k}(X)
\end{IEEEeqnarray*}
for any nondegenerate Liouville domain $X$. Notice that in \cref{lem:computation of cl}, $\tilde{\mathfrak{g}}^{\leq 1}_k(X) \leq \cgh{k}(X)$ held because by assumption $X$ was a $4$-dimensional convex toric domain. We start by showing that $\tilde{\mathfrak{g}}^{\leq \ell}_k(X) \leq \mathfrak{g}^{\leq \ell}_k(X)$. This result has already been proven in \cite[Section 3.4]{mcduffSymplecticCapacitiesUnperturbed2022}, but we include a proof for the sake of completeness.

\begin{theorem}[{\cite[Section 3.4]{mcduffSymplecticCapacitiesUnperturbed2022}}]
    \phantomsection\label{thm:g tilde vs g hat}
    If $X$ is a Liouville domain then 
    \begin{IEEEeqnarray*}{c+x*}
        \tilde{\mathfrak{g}}^{\leq \ell}_k(X) \leq {\mathfrak{g}}^{\leq \ell}_k(X).
    \end{IEEEeqnarray*}
\end{theorem}
\begin{proof}
    By \cref{lem:can prove ineqs for ndg}, we may assume that $X$ is nondegenerate. Choose a point $x \in \itr X$ and a symplectic divisor $D$ through $x$. Let $J \in \mathcal{J}(X,D)$ be an almost complex structure on $\hat{X}$ and consider the bar complex $\mathcal{B}(CC(X)[-1])$, computed with respect to $J$. Suppose that $a > 0$ is such that the augmentation map 
    \begin{IEEEeqnarray*}{c+x*}
        \epsilon_k \colon H(\mathcal{A}^{\leq a} \mathcal{B}^{\leq \ell}(CC(X)[-1])) \longrightarrow \Q
    \end{IEEEeqnarray*}
    is nonzero. By \cref{thm:g tilde two definitions}, it is enough to show that there exists a word of Reeb orbits $\Gamma = (\gamma_1,\ldots,\gamma_p)$ such that 
    \begin{IEEEeqnarray*}{c+x*}
        p \leq \ell, \qquad \mathcal{A}(\Gamma) \leq a, \qquad \overline{\mathcal{M}}^{J}_{X}(\Gamma)\p{<}{}{\mathcal{T}^{(k)}x} \neq \varnothing.
    \end{IEEEeqnarray*}
    Choose a homology class $\beta \in H(\mathcal{A}^{\leq a} \mathcal{B}^{\leq \ell}(CC(X)[-1]))$ such that $\epsilon_k(\beta) \neq 0$. The element $\beta$ can be written as a finite linear combination of Reeb orbits $\Gamma = (\gamma_1,\ldots,\gamma_p)$, where every word has length $p \leq \ell$ and action $\mathcal{A}(\Gamma) \leq a$. One of the words in this linear combination, say $\Gamma = (\gamma_1,\ldots,\gamma_{p})$, is such that $\#^{\mathrm{vir}} \overline{\mathcal{M}}^{J}_{X}(\Gamma)\p{<}{}{\mathcal{T}^{(k)}x} \neq 0$. By assumption on the virtual perturbation scheme, $\overline{\mathcal{M}}^{J}_{X}(\Gamma)\p{<}{}{\mathcal{T}^{(k)}x}$ is nonempty. 
\end{proof}

\begin{theorem}
    \label{thm:g hat vs gh}
    If $X$ is a Liouville domain such that $\pi_1(X) = 0$ and $2 c_1(TX) = 0$ then%
    \begin{IEEEeqnarray*}{c+x*}
        {\mathfrak{g}}^{\leq 1}_k(X) = \cgh{k}(X).
    \end{IEEEeqnarray*}
\end{theorem}
\begin{proof}
    By \cref{lem:can prove ineqs for ndg}, we may assume that $X$ is nondegenerate. Let $E = E(a_1,\ldots,a_n)$ be an ellipsoid as in \cref{sec:augmentation map of an ellipsoid} such that there exists a strict exact symplectic embedding $\phi \colon E \longrightarrow X$. In \cite{bourgeoisEquivariantSymplecticHomology2016}, Bourgeois--Oancea define an isomorphism between linearized contact homology and positive $S^1$-equivariant contact homology, which we will denote by $\Phi_{\mathrm{BO}}$. This isomorphism commutes with the Viterbo transfer maps and respects the action filtration. In addition, the Viterbo transfer maps in linearized contact homology commute with the augmentation maps of \cref{def:augmentation map}. Therefore, there is a commutative diagram
    \begin{IEEEeqnarray*}{c+x*}
        \begin{tikzcd}
            SH^{S^1,(\varepsilon,a]}_{n - 1 + 2k}(X) \ar[r, "\iota^{S^1,a}"] \ar[d, hook, two heads, swap, "\Phi_{\mathrm{BO}}^a"] & SH^{S^1,+}_{n - 1 + 2k}(X) \ar[r, "\phi_!^{S^1}"] \ar[d, hook, two heads, "\Phi_{\mathrm{BO}}"] & SH^{S^1,+}_{n - 1 + 2k}(E) \ar[d, hook, two heads, "\Phi_{\mathrm{BO}}"] \\
            CH^{a}_{n - 1 + 2k}(X) \ar[r, "\iota^{a}"] \ar[d, equals]                                & CH_{n - 1 + 2k}(X) \ar[r, "\phi_{!}"] \ar[d, equals]                  & CH_{n - 1 + 2k}(E) \ar[d, hook, two heads, "{\epsilon}^E_k"] \\
            CH^{a}_{n - 1 + 2k}(X) \ar[r, swap, "\iota^{a}"]                                         & CH_{n - 1 + 2k}(X) \ar[r, swap, "{\epsilon}_k^X"]                 & \Q
        \end{tikzcd}
    \end{IEEEeqnarray*}
    Here, the map ${\epsilon}_k^E$ is nonzero, or equivalently an isomorphism, by \cref{thm:augmentation is nonzero}. Then,%
    \begin{IEEEeqnarray*}{rCls+x*}
        \cgh{k}(X)
        & = & \inf \{ a > 0 \mid \phi_!^{S^1} \circ \iota^{S^1,a} \neq 0 \}       & \quad [\text{by \cref{def:ck alternative}}] \\
        & = & \inf \{ a > 0 \mid {\epsilon}_k^X \circ \iota^{a} \neq 0 \} & \quad [\text{since the diagram commutes}] \\
        & = & {\mathfrak{g}}^{\leq 1}_k(X)                                & \quad [\text{by \cref{def:capacities glk}}]. & \qedhere
    \end{IEEEeqnarray*}
\end{proof}

\begin{theorem}
    \phantomsection\label{thm:my main theorem}
    Under \cref{assumption}, if $X_\Omega$ is a convex or concave toric domain then%
    \begin{IEEEeqnarray*}{c+x*}
        c_L(X_{\Omega}) = \delta_\Omega.
    \end{IEEEeqnarray*}
\end{theorem}
\begin{proof}
    Since $X_{\Omega}$ is concave or convex, we have $X_{\Omega} \subset N(\delta_\Omega)$. For every $k \in \Z_{\geq 1}$,
    \begin{IEEEeqnarray*}{rCls+x*}
        \delta_\Omega
        & \leq & c_P(X_{\Omega})                                         & \quad [\text{by \cref{lem:c square geq delta}}] \\
        & \leq & c_L(X_{\Omega})                                         & \quad [\text{by \cref{lem:c square leq c lag}}] \\
        & \leq & \frac{\tilde{\mathfrak{g}}^{\leq 1}_{k}(X_{\Omega})}{k} & \quad [\text{by \cref{thm:lagrangian vs g tilde}}] \\
        & \leq & \frac{{\mathfrak{g}}^{\leq 1}_{k}(X_{\Omega})}{k}       & \quad [\text{by \cref{thm:g tilde vs g hat}}] \\
        & =    & \frac{\cgh{k}(X_{\Omega})}{k}                           & \quad [\text{by \cref{thm:g hat vs gh}}] \\
        & \leq & \frac{\cgh{k}(N(\delta_\Omega))}{k}                     & \quad [\text{since $X_{\Omega} \subset N(\delta_\Omega)$}] \\
        & =    & \frac{\delta_\Omega(k+n-1)}{k}                          & \quad [\text{by \cref{lem:cgh of nondisjoint union of cylinders}}].
    \end{IEEEeqnarray*}
    The result follows by taking the infimum over $k$.
\end{proof}

\AtEndDocument{
    \bibliographystyle{alpha}
    \bibliography{thesis}

\newcommand{\etalchar}[1]{$^{#1}$}
\begin{thebibliography}{BEH{\etalchar{+}}03}

\bibitem[AB95]{austinMorseBottTheoryEquivariant1995}
D.~M. Austin and P.~J. Braam.
\newblock \href{https://doi.org/10.1007/978-3-0348-9217-9_8}{Morse-Bott theory
  and equivariant cohomology}.
\newblock In Helmut Hofer, Clifford~H. Taubes, Alan Weinstein, and Eduard
  Zehnder, editors, {\em The Floer Memorial Volume}, Progress in Mathematics,
  pages 123--183. Birkh\"auser, Basel, 1995.

\bibitem[AK14]{andersenTQFTQuantumTeichmuller2014}
J{\o}rgen~Ellegaard Andersen and Rinat Kashaev.
\newblock \href{https://doi.org/10.1007/s00220-014-2073-2}{A TQFT from quantum
  Teichm\"uller theory}.
\newblock {\em Communications in Mathematical Physics}, 330(3):887--934,
  September 2014.

\bibitem[AS10]{abouzaidOpenStringAnalogue2010}
Mohammed Abouzaid and Paul Seidel.
\newblock \href{http://www.msp.org/gt/2010/14-2/p01.xhtml}{An open string
  analogue of Viterbo functoriality}.
\newblock {\em Geometry \& Topology}, 14(2):627--718, February 2010.

\bibitem[BEH{\etalchar{+}}03]{bourgeoisCompactnessResultsSymplectic2003}
Fr{\'e}d{\'e}ric Bourgeois, Yakov Eliashberg, Helmut Hofer, Krzysztof Wysocki,
  and Eduard Zehnder.
\newblock \href{https://msp.org/gt/2003/7-2/p09.xhtml}{Compactness results in
  symplectic field theory}.
\newblock {\em Geometry \& Topology}, 7(2):799--888, December 2003.

\bibitem[BH18]{baoDefinitionCylindricalContact2018}
Erkao Bao and Ko~Honda.
\newblock
  \href{https://onlinelibrary.wiley.com/doi/abs/10.1112/topo.12077}{Definition
  of cylindrical contact homology in dimension three}.
\newblock {\em Journal of Topology}, 11(4):1002--1053, 2018.

\bibitem[BH21]{baoSemiglobalKuranishiCharts2021}
Erkao Bao and Ko~Honda.
\newblock \href{http://arxiv.org/abs/1512.00580}{Semi-global Kuranishi charts
  and the definition of contact homology}.
\newblock {\em arXiv:1512.00580 [math]}, September 2021.

\bibitem[BM04]{bourgeoisCoherentOrientationsSymplectic2004}
Fr{\'e}d{\'e}ric Bourgeois and Klaus Mohnke.
\newblock \href{https://doi.org/10.1007/s00209-004-0656-x}{Coherent
  orientations in symplectic field theory}.
\newblock {\em Mathematische Zeitschrift}, 248(1):123--146, September 2004.

\bibitem[BO09]{bourgeoisExactSequenceContact2009}
Fr{\'e}d{\'e}ric Bourgeois and Alexandru Oancea.
\newblock \href{http://link.springer.com/10.1007/s00222-008-0159-1}{An exact
  sequence for contact- and symplectic homology}.
\newblock {\em Inventiones mathematicae}, 175(3):611--680, March 2009.

\bibitem[BO10]{bourgeoisFredholmTheoryTransversality2010}
Fr{\'e}d{\'e}ric Bourgeois and Alexandru Oancea.
\newblock
  \href{https://www.ems-ph.org/journals/show_abstract.php?issn=1435-9855&vol=12&iss=5&rank=5}{Fredholm
  theory and transversality for the parametrized and for the
  ${{S^1}}$-invariant symplectic action}.
\newblock {\em Journal of the European Mathematical Society}, 12(5):1181--1229,
  August 2010.

\bibitem[BO13]{bourgeoisGysinExactSequence2013}
Fr{\'e}d{\'e}ric Bourgeois and Alexandru Oancea.
\newblock
  \href{https://www.worldscientific.com/doi/abs/10.1142/S1793525313500210}{The
  Gysin exact sequence for ${{S^1}}$-equivariant symplectic homology}.
\newblock {\em Journal of Topology and Analysis}, 05(04):361--407, December
  2013.

\bibitem[BO16]{bourgeoisEquivariantSymplecticHomology2016}
Fr{\'e}d{\'e}ric Bourgeois and Alexandru Oancea.
\newblock
  \href{https://academic.oup.com/imrn/article-lookup/doi/10.1093/imrn/rnw029}{${{S^1}}$-equivariant
  symplectic homology and linearized contact homology}.
\newblock {\em International Mathematics Research Notices}, page rnw029, June
  2016.

\bibitem[CM05]{cieliebakCompactnessPuncturedHolomorphic2005}
Kai Cieliebak and Klaus Mohnke.
\newblock \href{https://projecteuclid.org/euclid.jsg/1154467631}{Compactness
  for punctured holomorphic curves}.
\newblock {\em Journal of Symplectic Geometry}, 3(4):589--654, December 2005.

\bibitem[CM07]{cieliebakSymplecticHypersurfacesTransversality2007}
Kai Cieliebak and Klaus Mohnke.
\newblock
  \href{http://www.intlpress.com/site/pub/pages/journals/items/jsg/content/vols/0005/0003/a002/}{Symplectic
  hypersurfaces and transversality in Gromov\textendash Witten theory}.
\newblock {\em Journal of Symplectic Geometry}, 5(3):281--356, 2007.

\bibitem[CM18]{cieliebakPuncturedHolomorphicCurves2018}
Kai Cieliebak and Klaus Mohnke.
\newblock \href{http://link.springer.com/10.1007/s00222-017-0767-8}{Punctured
  holomorphic curves and Lagrangian embeddings}.
\newblock {\em Inventiones mathematicae}, 212(1):213--295, April 2018.

\bibitem[CO18]{cieliebakSymplecticHomologyEilenberg2018}
Kai Cieliebak and Alexandru Oancea.
\newblock \href{https://msp.org/agt/2018/18-4/p03.xhtml}{Symplectic homology
  and the Eilenberg\textendash Steenrod axioms}.
\newblock {\em Algebraic \& Geometric Topology}, 18(4):1953--2130, April 2018.

\bibitem[dB16]{desaint-gervaisUniformizationRiemannSurfaces2016}
Henri~Paul {de Saint-Gervais} and Robert Burns.
\newblock {\em \href{http://www.ems-ph.org/doi/10.4171/145}{Uniformization of
  Riemann surfaces: revisiting a hundred-year-old theorem}}.
\newblock European Mathematical Society Publishing House, Z\"urich,
  Switzerland, January 2016.

\bibitem[dS08]{silvaLecturesSymplecticGeometry2008}
Ana~Cannas da~Silva.
\newblock {\em
  \href{https://link.springer.com/book/10.1007/978-3-540-45330-7}{Lectures on
  symplectic geometry}}.
\newblock Number 1764 in Lecture notes in mathematics. Springer, Berlin,
  Heidelberg, corrected 2nd printing edition, 2008.

\bibitem[EH89]{ekelandSymplecticTopologyHamiltonian1989}
Ivar Ekeland and Helmut Hofer.
\newblock \href{http://link.springer.com/10.1007/BF01215653}{Symplectic
  topology and Hamiltonian dynamics}.
\newblock {\em Mathematische Zeitschrift}, 200(3):355--378, September 1989.

\bibitem[EH90]{ekelandSymplecticTopologyHamiltonian1990}
Ivar Ekeland and Helmut Hofer.
\newblock \href{https://doi.org/10.1007/BF02570756}{Symplectic topology and
  Hamiltonian dynamics II}.
\newblock {\em Mathematische Zeitschrift}, 203(1):553--567, January 1990.

\bibitem[Eva10]{evansPartialDifferentialEquations2010}
Lawrence~C. Evans.
\newblock {\em \href{https://bookstore.ams.org/gsm-19-r/}{Partial differential
  equations}}.
\newblock Number v. 19 in Graduate studies in mathematics. American
  Mathematical Society, Providence, R.I, 2nd ed edition, 2010.

\bibitem[FHS95]{floerTransversalityEllipticMorse1995}
Andreas Floer, Helmut Hofer, and Dietmar Salamon.
\newblock
  \href{https://projecteuclid.org/journals/duke-mathematical-journal/volume-80/issue-1/Transversality-in-elliptic-Morse-theory-for-the-symplectic-action/10.1215/S0012-7094-95-08010-7.full}{Transversality
  in elliptic Morse theory for the symplectic action}.
\newblock {\em Duke Mathematical Journal}, 80(1):251--292, October 1995.

\bibitem[FHW94]{floerApplicationsSymplecticHomology1994}
Andreas Floer, Helmut Hofer, and Krzysztof Wysocki.
\newblock \href{http://link.springer.com/10.1007/BF02571962}{Applications of
  symplectic homology I}.
\newblock {\em Mathematische Zeitschrift}, 217(1):577--606, September 1994.

\bibitem[Flo88]{floerUnregularizedGradientFlow1988}
Andreas Floer.
\newblock
  \href{https://onlinelibrary.wiley.com/doi/abs/10.1002/cpa.3160410603}{The
  unregularized gradient flow of the symplectic action}.
\newblock {\em Communications on Pure and Applied Mathematics}, 41(6):775--813,
  1988.

\bibitem[FS07]{frauenfelderHamiltonianDynamicsConvex2007}
Urs Frauenfelder and Felix Schlenk.
\newblock \href{https://doi.org/10.1007/s11856-007-0037-3}{Hamiltonian dynamics
  on convex symplectic manifolds}.
\newblock {\em Israel Journal of Mathematics}, 159(1):1--56, June 2007.

\bibitem[Fv18]{frauenfelderRestrictedThreeBodyProblem2018}
Urs Frauenfelder and Otto {van Koert}.
\newblock {\em \href{http://link.springer.com/10.1007/978-3-319-72278-8}{The
  restricted three-body problem and holomorphic curves}}.
\newblock Pathways in Mathematics. Springer International Publishing, Cham,
  2018.

\bibitem[Gei08]{geigesIntroductionContactTopology2008}
Hansj{\"o}rg Geiges.
\newblock {\em
  \href{https://www.cambridge.org/core/books/an-introduction-to-contact-topology/F851B2A2E7E78C6B9967A18A6641B40C}{An
  introduction to contact topology}}.
\newblock Cambridge Studies in Advanced Mathematics. Cambridge University
  Press, Cambridge, 2008.

\bibitem[GH18]{guttSymplecticCapacitiesPositive2018}
Jean Gutt and Michael Hutchings.
\newblock \href{https://msp.org/agt/2018/18-6/p11.xhtml}{Symplectic capacities
  from positive ${{S^1}}$-equivariant symplectic homology}.
\newblock {\em Algebraic \& Geometric Topology}, 18(6):3537--3600, October
  2018.

\bibitem[Gro85]{gromovPseudoHolomorphicCurves1985}
Mikhael Gromov.
\newblock \href{http://link.springer.com/10.1007/BF01388806}{Pseudo holomorphic
  curves in symplectic manifolds}.
\newblock {\em Inventiones Mathematicae}, 82(2):307--347, June 1985.

\bibitem[GS18]{ginzburgFilteredSymplecticHomology2018}
Viktor~L. Ginzburg and Jeongmin Shon.
\newblock
  \href{https://www.worldscientific.com/doi/abs/10.1142/S0129167X18500714}{On
  the filtered symplectic homology of prequantization bundles}.
\newblock {\em International Journal of Mathematics}, 29(11):1850071, October
  2018.

\bibitem[Gut12]{guttConleyZehnderIndex2012}
Jean Gutt.
\newblock \href{http://arxiv.org/abs/1201.3728}{The Conley\textendash Zehnder
  index for a path of symplectic matrices}.
\newblock {\em arXiv:1201.3728 [math]}, January 2012.

\bibitem[Gut14]{guttMinimalNumberPeriodic2014}
Jean Gutt.
\newblock {\em \href{https://tel.archives-ouvertes.fr/tel-01016954v2}{On the
  minimal number of periodic Reeb orbits on a contact manifold}}.
\newblock PhD thesis, Universit\'e de Strasbourg and Universit\'e libre de
  Bruxelles, 2014.

\bibitem[Gut17]{guttPositiveEquivariantSymplectic2017}
Jean Gutt.
\newblock
  \href{https://www.intlpress.com/site/pub/pages/journals/items/jsg/content/vols/0015/0004/a003/}{The
  positive equivariant symplectic homology as an invariant for some contact
  manifolds}.
\newblock {\em Journal of Symplectic Geometry}, 15(4):1019--1069, December
  2017.

\bibitem[HN16]{hutchingsCylindricalContactHomology2016}
Michael Hutchings and Jo~Nelson.
\newblock
  \href{https://www.intlpress.com/site/pub/pages/journals/items/jsg/content/vols/0014/0004/a001/}{Cylindrical
  contact homology for dynamically convex contact forms in three dimensions}.
\newblock {\em Journal of Symplectic Geometry}, 14(4):983--1012, December 2016.

\bibitem[Hof93]{hoferEstimatesEnergySymplectic1993}
Helmut Hofer.
\newblock \href{https://doi.org/10.1007/BF02565809}{Estimates for the energy of
  a symplectic map}.
\newblock {\em Commentarii Mathematici Helvetici}, 68(1):48--72, December 1993.

\bibitem[Hut11]{hutchingsQuantitativeEmbeddedContact2011}
Michael Hutchings.
\newblock
  \href{https://projecteuclid.org/journals/journal-of-differential-geometry/volume-88/issue-2/Quantitative-Embedded-Contact-Homology/10.4310/jdg/1320067647.full}{Quantitative
  embedded contact homology}.
\newblock {\em Journal of Differential Geometry}, 88(2), June 2011.

\bibitem[HWZ21]{hoferPolyfoldFredholmTheory2021}
Helmut Hofer, Krzysztof Wysocki, and Eduard Zehnder.
\newblock {\em
  \href{https://link.springer.com/book/10.1007/978-3-030-78007-4}{Polyfold and
  Fredholm theory}}.
\newblock Number 3. Folge, Volume 72 in Ergebnisse der Mathematik und ihrer
  Grenzgebiete. Springer, Cham, Switzerland, 2021.

\bibitem[HZ90]{hoferNewCapacitySymplectic1990}
Helmut Hofer and Eduard Zehnder.
\newblock
  \href{https://www.sciencedirect.com/science/article/pii/B9780125742498500237}{A
  new capacity for symplectic manifolds}.
\newblock In Paul~H. Rabinowitz and Eduard Zehnder, editors, {\em Analysis, et
  Cetera}, pages 405--427. Academic Press, January 1990.

\bibitem[HZ11]{hoferSymplecticInvariantsHamiltonian2011}
Helmut Hofer and Eduard Zehnder.
\newblock {\em
  \href{http://link.springer.com/10.1007/978-3-0348-0104-1}{Symplectic
  invariants and Hamiltonian dynamics}}.
\newblock Springer Basel, Basel, 2011.

\bibitem[Iri21]{irieSymplecticHomologyFiberwise2021}
Kei Irie.
\newblock \href{http://arxiv.org/abs/1907.09749}{Symplectic homology of
  fiberwise convex sets and homology of loop spaces}.
\newblock {\em arXiv:1907.09749 [math]}, June 2021.

\bibitem[Ish18]{ishikawaConstructionGeneralSymplectic2018}
Suguru Ishikawa.
\newblock \href{http://arxiv.org/abs/1807.09455}{Construction of general
  symplectic field theory}.
\newblock {\em arXiv:1807.09455 [math]}, August 2018.

\bibitem[MS12]{mcduffHolomorphicCurvesSymplectic2012}
Dusa McDuff and Dietmar Salamon.
\newblock {\em
  \href{https://doi.org/http://dx.doi.org/10.1090/coll/052}{$J$-holomorphic
  curves and symplectic topology}}, volume~52 of {\em Colloquium Publications}.
\newblock American Mathematical Society, Providence, RI, 2012.

\bibitem[MS17]{mcduffIntroductionSymplecticTopology2017}
Dusa McDuff and Dietmar Salamon.
\newblock {\em
  \href{https://oxford.universitypressscholarship.com/view/10.1093/oso/9780198794899.001.0001/oso-9780198794899}{Introduction
  to symplectic topology}}.
\newblock Number~27 in Oxford graduate texts in mathematics. Oxford University
  Press, Oxford New York, NY, third edition edition, 2017.

\bibitem[MS22]{mcduffSymplecticCapacitiesUnperturbed2022}
Dusa McDuff and Kyler Siegel.
\newblock \href{http://arxiv.org/abs/2111.00515}{Symplectic capacities,
  unperturbed curves, and convex toric domains}.
\newblock {\em arXiv:2111.00515 [math]}, February 2022.

\bibitem[Nel15]{nelsonAutomaticTransversalityContact2015}
Jo~Nelson.
\newblock \href{https://doi.org/10.1007/s12188-015-0112-3}{Automatic
  transversality in contact homology I: regularity}.
\newblock {\em Abhandlungen aus dem Mathematischen Seminar der Universit\"at
  Hamburg}, 85(2):125--179, October 2015.

\bibitem[Oan04]{oanceaSurveyFloerHomology2004}
Alexandru Oancea.
\newblock
  \href{http://ensaios.sbm.org.br/wp-content/uploads/sites/14/2021/09/EM_7_Oancea.pdf}{A
  survey of Floer homology for manifolds with contact type boundary or
  symplectic homology}.
\newblock {\em Ensaios Matem\'aticos}, 7(2), 2004.

\bibitem[Oh02a]{ohMinimaxTheorySpectral2002}
Yong-Geun Oh.
\newblock \href{http://arxiv.org/abs/math/0206092}{Mini-max theory, spectral
  invariants and geometry of the Hamiltonian diffeomorphism group}.
\newblock {\em arXiv:math/0206092}, July 2002.

\bibitem[Oh02b]{ohChainLevelFloer2002}
Yong-Geun Oh.
\newblock
  \href{https://www.intlpress.com/site/pub/pages/journals/items/ajm/content/vols/0006/0004/a001/index.php}{Chain
  level Floer theory and Hofer's geometry of the Hamiltonian diffeomorphism
  group}.
\newblock {\em Asian Journal of Mathematics}, 6(4):579--624, December 2002.

\bibitem[Oh05]{ohSpectralInvariantsLength2005}
Youg-Geun Oh.
\newblock
  \href{https://www.intlpress.com/site/pub/pages/journals/items/ajm/content/vols/0009/0001/a001/}{Spectral
  invariants and the length minimizing property of Hamiltonian paths}.
\newblock {\em Asian Journal of Mathematics}, 9(1):1--18, March 2005.

\bibitem[Par16]{pardonAlgebraicApproachVirtual2016}
John Pardon.
\newblock \href{https://msp.org/gt/2016/20-2/p04.xhtml}{An algebraic approach
  to virtual fundamental cycles on moduli spaces of pseudo-holomorphic curves}.
\newblock {\em Geometry \& Topology}, 20(2):779--1034, April 2016.

\bibitem[Par19]{pardonContactHomologyVirtual2019}
John Pardon.
\newblock
  \href{https://www.ams.org/jams/2019-32-03/S0894-0347-2019-00924-0/}{Contact
  homology and virtual fundamental cycles}.
\newblock {\em Journal of the American Mathematical Society}, 32(3):825--919,
  July 2019.

\bibitem[RF10]{roydenRealAnalysis2010}
Halsey Royden and Patrick Fitzpatrick.
\newblock {\em
  \href{https://www.pearson.com/us/higher-education/program/Royden-Real-Analysis-4th-Edition/PGM28345.html}{Real
  analysis}}.
\newblock Prentice Hall, Boston, 2010.

\bibitem[Rie16]{riehlCategoryTheoryContext2016}
Emily Riehl.
\newblock {\em
  \href{https://store.doverpublications.com/048680903x.html}{Category theory in
  context}}.
\newblock Aurora: Dover modern math originals. Dover Publications, Mineola, NY,
  2016.

\bibitem[Rit13]{ritterTopologicalQuantumField2013}
Alexander~F. Ritter.
\newblock
  \href{https://londmathsoc.onlinelibrary.wiley.com/doi/abs/10.1112/jtopol/jts038}{Topological
  quantum field theory structure on symplectic cohomology}.
\newblock {\em Journal of Topology}, 6(2):391--489, 2013.

\bibitem[Rot88]{rotmanIntroductionAlgebraicTopology1988}
Joseph~J. Rotman.
\newblock {\em \href{http://link.springer.com/10.1007/978-1-4612-4576-6}{An
  introduction to algebraic topology}}, volume 119 of {\em Graduate Texts in
  Mathematics}.
\newblock Springer New York, New York, NY, 1988.

\bibitem[Rot09]{rotmanIntroductionHomologicalAlgebra2009}
Joseph~J. Rotman.
\newblock {\em \href{https://link.springer.com/book/10.1007/b98977}{An
  introduction to homological algebra}}.
\newblock Universitext. Springer, New York, NY, 2nd ed edition, 2009.

\bibitem[Sal99]{salamonLecturesFloerHomology1999}
Dietmar Salamon.
\newblock \href{https://www.ams.org/books/pcms/007/}{Lectures on Floer
  homology}.
\newblock In Yakov Eliashberg and Lisa Traynor, editors, {\em Symplectic
  geometry and topology}, volume~7 of {\em IAS/Park City Mathematics Series}.
  American Mathematical Society, Providence, RI, 1999.

\bibitem[Sch00]{schwarzActionSpectrumClosed2000}
Matthias Schwarz.
\newblock \href{http://msp.berkeley.edu/pjm/2000/193-2/p10.xhtml}{On the action
  spectrum for closed symplectically aspherical manifolds}.
\newblock {\em Pacific Journal of Mathematics}, 193(2):419--461, April 2000.

\bibitem[Sch08]{schlenkEmbeddingProblemsSymplectic2008}
Felix Schlenk.
\newblock {\em
  \href{https://www.degruyter.com/document/doi/10.1515/9783110199697/html}{Embedding
  problems in symplectic geometry}}, volume~40 of {\em De Gruyter Expositions
  in Mathematics}.
\newblock De Gruyter, Berlin, New York, August 2008.

\bibitem[Sei08]{seidelBiasedViewSymplectic2008}
Paul Seidel.
\newblock
  \href{https://projecteuclid.org/ebooks/current-developments-in-mathematics/Current-Developments-in-Mathematics-2006/chapter/A-biased-view-of-symplectic-cohomology/cdm/1223654543}{A
  biased view of symplectic cohomology}.
\newblock {\em Current Developments in Mathematics, 2006}, 2006:211--254,
  February 2008.

\bibitem[Sie20]{siegelHigherSymplecticCapacities2020}
Kyler Siegel.
\newblock \href{http://arxiv.org/abs/1902.01490}{Higher symplectic capacities}.
\newblock {\em arXiv:1902.01490 [math-ph]}, February 2020.

\bibitem[Sma65]{smaleInfiniteDimensionalVersion1965}
Stephen Smale.
\newblock \href{https://www.jstor.org/stable/2373250?origin=crossref}{An
  infinite dimensional version of Sard's theorem}.
\newblock {\em American Journal of Mathematics}, 87(4):861, October 1965.

\bibitem[SZ92]{salamonMorseTheoryPeriodic1992}
Dietmar Salamon and Eduard Zehnder.
\newblock \href{http://doi.wiley.com/10.1002/cpa.3160451004}{Morse theory for
  periodic solutions of Hamiltonian systems and the Maslov index}.
\newblock {\em Communications on Pure and Applied Mathematics},
  45(10):1303--1360, December 1992.

\bibitem[Vit90]{viterboNewObstructionEmbedding1990}
Claude Viterbo.
\newblock \href{http://link.springer.com/10.1007/BF01231188}{A new obstruction
  to embedding Lagrangian tori}.
\newblock {\em Inventiones Mathematicae}, 100(1):301--320, December 1990.

\bibitem[Vit92]{viterboSymplecticTopologyGeometry1992}
Claude Viterbo.
\newblock \href{https://doi.org/10.1007/BF01444643}{Symplectic topology as the
  geometry of generating functions}.
\newblock {\em Mathematische Annalen}, 292(1):685--710, March 1992.

\bibitem[Vit99]{viterboFunctorsComputationsFloer1999}
Claude Viterbo.
\newblock \href{http://link.springer.com/10.1007/s000390050106}{Functors and
  computations in Floer homology with applications, I}.
\newblock {\em Geometric And Functional Analysis}, 9(5):985--1033, December
  1999.

\bibitem[Wen10]{wendlAutomaticTransversalityOrbifolds2010}
Chris Wendl.
\newblock \href{http://www.ems-ph.org/doi/10.4171/CMH/199}{Automatic
  transversality and orbifolds of punctured holomorphic curves in dimension
  four}.
\newblock {\em Commentarii Mathematici Helvetici}, pages 347--407, 2010.

\bibitem[Wen16]{wendlLecturesSymplecticField2016}
Chris Wendl.
\newblock \href{http://arxiv.org/abs/1612.01009}{Lectures on symplectic field
  theory}.
\newblock {\em arXiv:1612.01009 [math]}, December 2016.

\end{thebibliography}
}

\end{document}